\documentclass{amsart}

\newtheorem{theorem}{Theorem}[section]
\newtheorem{lemma}[theorem]{Lemma}

\theoremstyle{definition}
\newtheorem{definition}[theorem]{Definition}

\theoremstyle{remark}
\newtheorem{remark}[theorem]{Remark}

\numberwithin{equation}{section}

\usepackage{verbatim}
\usepackage{pgf}
\usepackage{tikz}
\usepackage{amscd}

\usepackage{hyperref}

\newtheorem{corollary}[theorem]{Corollary}

\begin{document}

\title[Heat flow and Yang-Mills theory]{Heat flow on the moduli space of flat connections and Yang-Mills theory}
\author{R\'emi Janner}
\address{Department of Mathematics, ETH, R\"amistrasse 101, 8092 Zurich, Switzerland}
\address{Alpiq Management AG, Bahnhofquai 12, 4601 Olten, Switzerland}
\email{remi.janner@gmail.com}
\thanks{The author want to express his gratitude to Dietmar Salamon for valuable discussions and for suggesting this problem to him. For partial financial support we are most grateful to the ETH Zurich (grant TH-0106-1).}
\subjclass[2010]{58E05, 58E15, 53C22, 53C07, 53D20}
\keywords{adiabatic process, flat connection, heat flow, moduli space, Yang-Mills flow.}

\begin{abstract}
It is known that there is a bijection between the perturbed closed geodesics, 
below a given energy level, on the moduli space of flat connections $\mathcal M$ 
and families of perturbed Yang-Mills connections depending on a parameter 
$\varepsilon$. In this paper we study the heat flow on the loop space on 
$\mathcal M$ and the Yang-Mills $L^2$-flows for a $3$-manifold $N$ with 
partial rescaled metrics. Our main result is that the bounded Morse homology 
of the loop space on $\mathcal M$ is isomorphic to the bounded Morse homologies 
of the connections space of $N$.
\end{abstract}

 \maketitle


\tableofcontents

\section{Introduction and background}

In 1983 Atiyah and Bott (cf. \cite{MR702806}) introduced the moduli space of flat 
connections for a principal bundle $P\to\Sigma$ over a surface $(\Sigma,g_\Sigma)$. 
This moduli space can be seen as an infinite dimensional symplectic reduction; 
in fact, the conformal structure on the surface defines an almost complex structure, and 
with it together with the scalar product on the 1-forms one can also obtain a symplectic form. 
In case we pick a principal non trivial $\textrm{\bf SO}(3)$-bundle, then the moduli 
space $\mathcal M^g(P)$, defined as the quotient between the space of the flat connections
 $\mathcal A_0(P)\subset \mathcal A(P)$ and the identity component of the gauge 
group $\mathcal G_0(P)$, is a smooth compact symplectic manifold of dimension 
$6g-6$ (cf. \cite{MR1297130}) where $g$ denotes the genus of the surface 
and $\mathcal A(P)$ the set of the connections of the bundle.\\


\noindent{\bf Critical connections.} On the one side, we consider the loop space $\mathcal L(\mathcal M^g(P))$ of the manifold $\mathcal M^g(P)$ and if we want to compute the perturbed energy functional of a loop $\gamma$, where the perturbation comes from an equivariant Hamiltonian map, we can pick a lift $A(t)\in \mathcal A_0(P)$ of $\gamma$ and the unique loop $\Psi(t)\in\Omega^0(\Sigma,\mathfrak g_P)$ of $0$-forms satisfying the condition\footnote{$\Omega^k(\Sigma,\mathfrak g_P)$ denotes the space of $k$-forms on $\Sigma$ with values in the adjoint bundle $\mathfrak g_P:=P\times _{\mathrm{Ad}}\mathfrak g$ defined by the equivalence classes $[pg,\xi]\equiv[p,\mathrm{Ad}_g\xi]\equiv[p,g\xi g^{-1}]$ for $p\in P$, $g\in G$ and $\xi\in \mathfrak g$.}
\begin{equation}\label{cond:eqg0}
d_A^*\left(\partial_tA-d_A\Psi\right)=0.
\end{equation}
In this case, we can see $A(t)+\Psi(t) dt$ as a connection on $\Sigma \times S^1$ and the perturbed energy functional can be written as 
\begin{equation}\label{intro:EH}
E^H(A)=\frac 12 \int_0^1\left(\left\|\partial_tA-d_A\Psi\right\|^2_{L^2(\Sigma)}-H_t(A)\right) dt
\end{equation}
where $H_t:\mathcal A(P)\to \mathbb R$ is a generic equivariant Hamiltonian map. The critical loops of (\ref{intro:EH}) have to satisfy the equation 
\begin{equation}\label{intro:eqg0}
\pi_A\left(-\nabla_t(\partial_t A-d_A \Psi)-*X_t(A)\right)=0,
\end{equation}
where $\nabla_t:=\partial_t+[\Psi,.]$, where $\pi_A$ denotes the projection of the 1-forms in to the linear space of the harmonic 1-forms $H_A^1$ which corresponds to the tangent space of the manifold $\mathcal M^g(P)$ at the point $[A]$ and where the time-dependent Hamiltonian vector field $X_t$ is defined such that, for any connection $A\in \mathcal A(P)$ and any 1-form $\alpha$, $dH_t(A)\alpha=\int_\Sigma \langle X_t(A)\wedge \alpha\rangle$. The equation (\ref{intro:eqg0}) can be therefore also written as
\begin{equation}\label{intro:eqg1}
-\nabla_t(\partial_t A-d_A \Psi)-*X_t(A)-d_A^* \omega=0;
\end{equation}
the 2-form  $\omega(t)\in\Omega^2(\Sigma,\mathfrak g_P)$ is defined uniquely by the identity
\begin{equation*}
d_Ad_A^*\omega=\left[\left(\partial_t A-d_A \Psi\right) \wedge\left(\partial_t A-d_A \Psi\right )\right]-d_A*X_t(A)
\end{equation*}
and $d_A^*\omega$ corresponds to the non-harmonic part of 
$-\nabla_t(\partial_t A-d_A \Psi)-*X_t(A)$.

On the other side, if we pick the 3-manifold $\Sigma \times S^1$ with the metric $\varepsilon^2 g_\Sigma \oplus g_{S^1}$ for a positive parameter $\varepsilon$ and if we consider the principal $\textrm{\bf SO}(3)$-bundle $P\times S^1\to \Sigma\times S^1$ such that the restriction $P\times \{s\}\to \Sigma\times \{s\}$ is non trivial, then the perturbed Yang-Mills functional is
 \begin{equation}\label{intro:YME}
 \mathcal{YM}^{\varepsilon,H}(\Xi)
  =\frac12\int_{0}^1 \left(\frac 1 {\varepsilon^2}\|F_{A}\|_{L^2(\Sigma)}^2
  +\|\partial_{t}A-d_{A}\Psi\|_{L^2(\Sigma)}^2-H_{t}(A)\right) dt;
\end{equation}
where the connection $\Xi\in \mathcal A(P\times S^1)$ can be written as $=A+\Psi \,dt$ with $A(t)\in\mathcal A(P)$, $\Psi(t)\in\Omega^0(\Sigma,\mathfrak g_P)$; in fact, the curvature of $\Xi$ is $F_{\Xi}=F_{A}-(\partial_{t}A-d_{A}\Psi)\wedge dt$. A perturbed Yang-Mills connection $\Xi^\varepsilon\in\mathcal A(P\times S^1)$ has therefore to satisfy the two conditions
 \begin{equation}\label{intro:YMeq1}
 \frac1{\varepsilon^2}d_{A^\varepsilon}^*F_{A^\varepsilon}-
  \nabla_t(\partial_t A^\varepsilon-d_{A^\varepsilon}\Psi^\varepsilon)
 -*X_{t}(A^\varepsilon)=0,
  \end{equation}
 \begin{equation}\label{intro:YMeq2}
 d_{A^\varepsilon}^*(\partial_tA^\varepsilon-d_{A}\Psi^\varepsilon)=0.
\end{equation}

Next, we consider the sets of critical connections below an energy level $b$, i.e.
\begin{equation*}
\begin{split}
\mathrm{Crit}^b_{E^H}:=\big\{A+\Psi dt\in \mathcal L(\mathcal A_0(P)\otimes &\Omega^0(\Sigma,\mathfrak g_P)\wedge dt)|\,E^H(A)\leq b, (\ref{cond:eqg0}),(\ref{intro:eqg0})\big\},
\end{split}
\end{equation*}
\begin{equation*}
\begin{split}
\mathrm{Crit}^b_{\mathcal{YM}^{\varepsilon,H}}:=
\big\{\Xi^\varepsilon\in& \mathcal A(P\times S^1)|\,\mathcal {YM}^{\varepsilon,H }\left(\Xi^\varepsilon\right)\leq b, (\ref{intro:YMeq1}),(\ref{intro:YMeq2})\big\}.
\end{split}
\end{equation*}
We can now define a map between the perturbed geodesics, $\mathrm{Crit}^b_{E^{H}}$, and the set of the perturbed Yang-Mills connections, $\mathrm{Crit}^b_{\mathcal{YM}^{\varepsilon,H}}$, with energy less 
than $b$ provided that the parameter $\varepsilon$ is small enough 
(cf. \cite{MR1715156}); this map can also be defined uniquely, it is 
bijective and maps perturbed geodesics to perturbed Yang-Mills connections
 with the same Morse index (cf. \cite{remyj6}):

\begin{theorem}[cf. \cite{remyj6}, theorem 1.1]\label{thm:mainthm}
We assume that the Jacobi operators of all the perturbed geodesics are invertible and we choose a regular value $b$ of the energy $E^H$ and $p\geq 2$. Then there are two positive constants $\varepsilon_0$ and $c$ such that the following holds. For every $\varepsilon\in(0,\varepsilon_0)$ there is a unique gauge equivariant map
\begin{equation*}
\mathcal T^{\varepsilon,b}:\mathrm{Crit}^b_{E^H}
\to \mathrm{Crit}^b_{\mathcal{YM}^{\varepsilon,H}}
\end{equation*}
satisfying, for $\Xi^0 \in \mathrm{Crit}^b_{E^H}$, 
\begin{equation}\label{intro:A:a1}
d_{\Xi^0 }^{*_\varepsilon}\left(\mathcal T^{\varepsilon,b }(\Xi^0)-\Xi^0\right)=0,\quad\left\| \mathcal T^{\varepsilon,b}(\Xi^0)-\Xi^0
\right\|_{\Xi^0,2,p,\varepsilon, \Sigma\times S^1}
\leq c \varepsilon^2.
\end{equation}
Furthermore, this map is bijective and $\textrm{index}_{E^H}(\Xi^0)=\textrm{index}_{\mathcal {YM}^{\varepsilon,H}}(\mathcal T^{\varepsilon,b}(\Xi^0))$.
\end{theorem}
\noindent The norm $\|\cdot\|_{\Xi^0,2,p,\varepsilon, \Sigma\times S^1}$ is introduced in the appendix \ref{appendix:norm}. 
\begin{remark}\label{thm:existence:crit:est}
With the same assumptions of the last theorem we can also conclude the following estimates (cf. \cite{remyj6}, theorem 9.1 and lemma 9.6). We consider the unique solution 
$$\alpha_0^\varepsilon(t)\in \textrm{im } \left(d_{A^0(t)}^*:\Omega^2(\Sigma,\mathfrak g_P)\to \Omega^1(\Sigma,\mathfrak g_P)\right)$$
of the equation
\begin{equation}\label{eq:firststep}
d_{A^0}^*d_{A^0}\alpha_0^\varepsilon
=\varepsilon^2\nabla_t(\partial_tA^0-d_{A^0}\Psi^0)+\varepsilon^2
*X_t(A^0),
\end{equation}
then, for $\alpha+\psi dt:=\mathcal T^{\varepsilon,b }(\Xi^0)-\Xi^0$,
\begin{equation}\label{eq:cbb2}
\left\| (1-\pi_{A^0})(\alpha-\alpha_0^\varepsilon)
\right\|_{\Xi^0,2,p,\varepsilon, \Sigma\times S^1}+\varepsilon\left\|\psi dt\right\|_{\Xi^0,2,p,\varepsilon, \Sigma\times S^1} \leq c\varepsilon^{4},
\end{equation}
\begin{equation}\label{eq:cbb2cxy}
\left\|\pi_{A^0}(\alpha)
\right\|_{\Xi^0,2,p,1, \Sigma\times S^1}+\|\alpha_0^\varepsilon\|_{\Xi^0,2,p,1, \Sigma\times S^1} \leq c\varepsilon^{2}.
\end{equation}
\end{remark}

\vspace{0.4cm}
\noindent{\bf Bijection between the flows.} The theorem \ref{thm:mainthm} allows us to identify the critical connections and thus the next natural step is to prove a bijection between the flow lines.

On the one side, every map $[\Xi]: S^1\times \mathbb R \to\mathcal M^g(P)$ can be seen as a connection $\Xi=A+\Psi dt+\Phi ds\in \mathcal A(P\times S^1\times \mathbb R)$ which satisfies
\begin{equation}\label{intro:eqgeoflow333}
\begin{split}
F_A=0, \quad \partial_t A-d_A\Psi\in H_A^1,  \quad \partial_s A-d_A\Phi\in H_A^1
\end{split}
\end{equation}
and if we have a map $A:S^1\times \mathbb R \to\mathcal A_0(P)$, 
the second and the third condition of (\ref{intro:eqgeoflow333}) 
yield to unique 0-forms $\Psi, \Phi\in \Omega^0(\Sigma\times S^1\times \mathbb R,\mathfrak g_P)$. In order to achieve the transversality condition for the heat flow we need to choose a generic abstract perturbation on the loop space instead of the Hamiltonian one. Furthermore, $[\Xi]$ is a heat flow between the perturbed geodesics $\Xi_\pm \in \mathrm{Crit}_{E^H}^b$, $b\in \mathbb R$, if it satisfies the flow equation for the functional $E^H$, i.e.
\begin{equation}\label{intro:eqgeoflow}
\begin{split}
\partial_s A-d_A\Phi -\pi_A\big(\nabla_t&\left(\partial_t A-d_A\Psi\right)+*X_t(A)\big)=0,\\
\lim_{s\to\pm \infty}\Xi(&s)=\Xi_\pm.
\end{split}
\end{equation}

On the other side, a perturbed, $\varepsilon$-dependent, Yang-Mills flow between 
two perturbed Yang-Mills connections 
$\Xi_\pm\in \mathrm{Crit}_{\mathcal {YM}^{\varepsilon,H}}^b$ can be considered as a 
connection $\Xi:=A+\Psi dt+\Phi ds$ on the 4-manifold 
$\Sigma\times S^1\times \mathbb R$, where $\Phi\in \Omega^0(\Sigma\times S^1\times \mathbb R,\mathfrak g_P)$ makes the equations gauge invariant, and it satisfies the equations
\begin{equation}\label{intro:eq}
\begin{split}
\partial_s A-d_A\Phi + \frac 1{\varepsilon^2}d_A^*F_A- \nabla_t\left(\partial_t A-d_A\Psi\right)-*X_t(A)=0,\\
\partial_s\Psi-\nabla_t\Phi- \frac 1{\varepsilon^2} d_A^*\left(\partial_t A-d_A\Psi\right)=0,\quad  \lim_{s\to \pm \infty} \Xi=\Xi_\pm.
\end{split}
\end{equation}
In the following we denote by $\mathcal M^0(\Xi_-,\Xi_+)$ (respectively by $\mathcal M^\varepsilon(\Xi_-,\Xi_+)$) the moduli space of the solutions of $(\ref{intro:eqgeoflow})$ (respectively of (\ref{intro:eq})). We can therefore expect a bijective relation also between the flows of the two functionals for $\varepsilon$ small enough.

\begin{theorem}\label{flow:thm:main1}
We assume that the energy functional $E^H$ is Morse-Smale and we choose $p>2$ and a regular value $b > 0$ of $E^H$. There are constants $\varepsilon_0, c>0$ such  that the following holds. For every $\varepsilon \in (0,\varepsilon_0)$, every pair $\Xi^{0}_\pm:=A^0_\pm+\Psi^0_\pm dt \in \mathrm {Crit}^b_{E^H}$ with index difference 1, there exists a unique map
$$\mathcal R^{\varepsilon,b}: \mathcal M^0\left(\Xi_-^0,\Xi_+^0\right)\to \mathcal M^\varepsilon\left(\mathcal T^{\varepsilon,b}(\Xi_-^0),\mathcal T^{\varepsilon,b}(\Xi_+^0)\right)$$
satisfying for each $\Xi^0\in\mathcal M^0\left(\Xi_-^0,\Xi_+^0\right)$
\begin{equation}\label{Rthm:c1}
d_{\Xi^0}^{*_\varepsilon }\left(\mathcal R^{\varepsilon,b}(\Xi^0)-\mathcal K_2^\varepsilon(\Xi^0)\right)=0, \quad \mathcal R^{\varepsilon,b}(\Xi^0)-\mathcal K_2^\varepsilon(\Xi^0) \in \textrm{im } \left(\mathcal D^\varepsilon(\mathcal K_2(\Xi^0))\right)^*,
\end{equation}
\begin{equation}\label{Rthm:c2}
\left\|\mathcal R^{\varepsilon,b}(\Xi^0)-\mathcal K_2^\varepsilon(\Xi^0)\right\|_{1,2;p,1}\leq c\varepsilon^2.
\end{equation}
Furthermore, $\mathcal R^{\varepsilon,b}$ is bijective.
\end{theorem}

In the last theorem the connection $\mathcal K_2^\varepsilon(\Xi^0)$ should be 
seen as a first approximation of the Yang-Mills flow and the norm 
$\|\cdot\|_{1,2;p,1}$ is defined in the section \ref{flow:subsection:norm}.\\

\noindent{\bf Isomorphism between the homologies.} The theorem \ref{thm:mainthm} assures a bijection between the critical connections 
with the same index and the theorem \ref{flow:thm:main1} between the 
flows and thus we can compare the Morse homologies defined using the
 $L^2$-flow of the two functionals below a energy level $b$. In the loop space case the homology is well defined by the work of Weber (cf. \cite{Weberhab}) and in the Yang-Mills case we know that the flow exists in the case when the base manifold is two or three dimensional (cf. \cite{MR1179335}) or when we have a symmetry of codimension 3 on the base manifold of higher dimension (cf. \cite{Rade:fv}), but no results about the Morse-Smale transversality or the orientation of the unstable manifolds are known and therefore a priori $HM_*\left(\mathcal A^{\varepsilon,b}\left(P\times S^1\right)/\mathcal G_0\left(P\times S^1\right)\right)$ might even not be defined; in our case, it makes sense because the unstable manifolds of $\mathcal A^{\varepsilon,b}\left(P\times S^1\right)/\mathcal G_0\left(P\times S^1\right)$ inherit these properties from the unstable manifolds of $\mathcal L^b\mathcal M^g(P)$. Here $\mathcal G_0\left(P\times S^1\right)$ denotes the loop group on the gauge group $\mathcal G_0(P)$; $\mathcal L^b\mathcal M^g(P)\subset \mathcal L\mathcal M^g(P)$ and $\mathcal A^{\varepsilon,b}\left(P\times S^1\right)\subset \mathcal A\left(P\times S^1\right)$ are respectively the subsets where $E^H\leq b$ and $\mathcal {YM}^{\varepsilon,H}\leq b$. Thus, we can prove the following theorem.

\begin{theorem}\label{thm:main}
We assume that the energy functional $E^H$ is Morse-Smale. For every regular value $b>0$ of $E^H$ there is a positive constant $\varepsilon_0$ such that, for $0<\varepsilon<\varepsilon_0$, the inclusion $\mathcal L^b\mathcal M^g(P)\to \mathcal A^{\varepsilon, b}\left(P\times S^1\right)/\mathcal G_0\left(P\times S^1\right)$ induces an isomorphism
$$HM_*\left(\mathcal L^b\mathcal M^g(P),\mathbb Z_2\right)\cong HM_*\left(\mathcal A^{\varepsilon,b}\left(P\times S^1\right)/\mathcal G_0\left(P\times S^1\right),\mathbb Z_2\right).$$  
\end{theorem}
Another way to approach this problem could have been to consider the $W^{1,2}$-flows; 
in this case both homologies are well defined since the Palais-Smale 
condition is satisfied in both cases (cf. \cite{MR943798}) and thus by the 
general Morse theory (cf. \cite{MR2276948}) the transvervality may be achieved. 
It is also interesting to remark that the Morse homology of 
the loop space defined by the heat flow is isomorphic to its singular homology by the work of Weber (cf. \cite{Weberhab}) and to the Floer homology of the cotangent 
bundle $T^*\mathcal M^g(P)$ using the Hamiltonian $H_V$ given by the kinetic plus 
the potential energy and considering only orbits with the action bounded by $b$ as showed for the general case by Viterbo (cf. \cite{vit}), by Salamon and Weber 
(cf. \cite{MR2276534}) and by Abbondandolo and Schwarz (cf. \cite{MR2190223}, \cite{MR2276949}). The result of this paper enter therefore in a bigger picture as discussed by Swoboda and the author in \cite{elliptic}, where a new elliptic Yang-Mills theory is introduced and where it is conjectured that the elliptic Yang-Mills homology is isomorphic to $HM_*\left(\mathcal A^{\varepsilon,b}\left(P\times S^1\right)/\mathcal G_0\left(P\times S^1\right),\mathbb Z_2\right)$.\\


\noindent\textbf{Outline.} In the sections \ref{sec:geoflow}-\ref{sec:ymflow} we will discuss the 
heat flow and the Yang-Mills flow equations; after that we will define the 
$\varepsilon$-dependent norms (section \ref{flow:subsection:norm}) and introduce the Morse homologies (section \ref{flow:section:morse}). In the 
sections \ref{flow:section:linest} and \ref{flow:section:qest} we will show respectively 
some linear and quadratic estimates that we will need in the section 
\ref{flow:section:firstapprox} to construct an approximation of a perturbed 
Yang-Mills flow starting from a perturbed geodesic flow and in the section 
\ref{flow:section:themap} in order to define uniquely the map $\mathcal R^{\varepsilon,b}$ 
using a contraction argument. The next four sections are of preparatory nature for 
the proof of the surjectivity of $\mathcal R^{\varepsilon,b}$ 
(section \ref{flow:section:surj}); in fact in the section \ref{flow:section:apriori} 
we will show some a priori estimates for the perturbed Yang-Mills flow and then we will prove 
 an estimate for the $L^\infty$-norm of their curvature terms 
(section \ref{flow:section:linfty}), the uniformly exponential convergence of the 
flows (section \ref{flow:section:expconv}) and two theorems that allow to choose the right 
relative Coulomb gauge (section \ref{ch:rcg}). The surjectivity 
(section \ref{flow:section:surj}) will be showed using an indirect argument; in fact we will prove 
that any sequence of perturbed Yang-Mills flows
 $\Xi^{\varepsilon_\nu}$, $\varepsilon_\nu\to 0$, which is not in the image of the 
map $\mathcal R^{\varepsilon,b}$, has a subsequence which converges (modulo gauge) 
by the implicit function theorem to a perturbed geodesic flow and thus, by the 
uniqueness property of $\mathcal R^{\varepsilon,b}$, it is in the image of this map which is a contradiction. In the last section we will prove first the theorem \ref{flow:thm:main1}, 
which follows easily from the definition \ref{defiR} of the map 
$\mathcal R^{\varepsilon,b}$ and its surjectivity (theorem \ref{flow:thm:surj}), 
and then the theorem \ref{thm:main}.

%
%

\section{Geodesic flow}\label{sec:geoflow}

Every continuously differentiable map $[\Xi]: S^1\times \mathbb R \to\mathcal M^g(P)$ can be seen as a connection $\Xi=A+\Psi dt+\Phi ds\in \mathcal A(P\times S^1\times \mathbb R)$ which satisfies the following conditions
\begin{equation}\label{flow:eqgeoflow333}
\begin{split}
F_A=0, \quad \partial_t A-d_A\Psi\in H_A^1,  \quad \partial_s A-d_A\Phi\in H_A^1.
\end{split}
\end{equation}
In fact, for any $[\Xi]$ we can choose a lift $A:S^1\times \mathbb R \to\mathcal A_0(P)$; the second and the third condition of (\ref{flow:eqgeoflow333}) yield to unique 0-forms $\Psi(t,s), \Phi(t,s)\in \Omega^0(P,\mathfrak g_P)$. One can also consider $\Phi$ to have an exponential convergence as $|s|\to \infty$ (cf. \cite{Davisthesis}). The connection $\Xi$ is clearly not uniquely defined, but for every two connections $\Xi_1$ and $\Xi_2$ with the above properties there is a map $u\in \mathcal G_0(P\times S^1\times \mathbb R)$ such that $u^*\Xi_1=\Xi_2$, the existence and the uniqueness of $u$ follow from the definition of $\mathcal M^g(P)$ and from the equivariance of the conditions (\ref{flow:eqgeoflow333}). The gauge group $\mathcal G_0\left(P\times S^1\times \mathbb R\right)$ is defined as the set of smooth maps $g: S^1\times \mathbb R\to \mathcal G_0(P)$.\\

Furthermore, $[\Xi]$ is a heat flow between the perturbed geodesics $\Xi_\pm \in \mathrm{Crit}_{E^H}^b$, $b\in \mathbb R$, if it satisfies the flow equation for the functional $E^H$, i.e.
\begin{equation}\label{flow:eqgeoflow}
\begin{split}
\partial_s A-d_A\Phi -\pi_A\big(\nabla_t&\left(\partial_t A-d_A\Psi\right)+*X_t(A)\big)=0,\\
\lim_{s\to\pm \infty}\Xi(&s)=\Xi_\pm
\end{split}
\end{equation}
where $\nabla_t:=\partial_t+[\Psi,\cdot]$ and the perturbation term $X_t$ will be discussed in the next section. Since
\begin{align*}
d_A^*\left(\nabla_t\left(\partial_t A-d_A\Psi\right)+*X_t(A)\right)
=&\nabla_td_A^*\left(\partial_t A-d_A\Psi\right)+*d_AX_t(A)\\
&+*\left[\left(\partial_t A-d_A\Psi\right)\wedge*\left(\partial_t A-d_A\Psi\right)\right]=0,
\end{align*}
we can write the first line of (\ref{flow:eqgeoflow}) as the pair of equations
\begin{equation}\label{flow:eqgeoflowdfs}
\begin{split}
\partial_s A-d_A\Phi -\nabla_t\left(\partial_t A-d_A\Psi\right)-*X_t(A)+d_A^*\omega=0,\\
d_Ad_A^*\omega=-\left[\left(\partial_t A-d_A\Psi\right)\wedge \left(\partial_t A-d_A\Psi\right)\right]+d_A*X_t(A)
\end{split}
\end{equation}
where $\omega(t,s)\in \Omega^2(\Sigma,\mathfrak g_P)$ is uniquely defined by the second condition which is obtained deriving the first one by $d_A$ using the commutation formula (\ref{commform}) and (\ref{flow:eqgeoflow333}).

\begin{lemma}
We have the following two commutation formulas:
\begin{equation}\label{commform}
[d_A,\nabla_t]=-[(\partial_t A-d_A\Psi)\wedge\cdot\,];
\end{equation}
\begin{equation}\label{commform2}
[d_A^*,\nabla_t]=*[(\partial_t A-d_A\Psi)\wedge*\,\cdot\,].
\end{equation}
\end{lemma}

\begin{proof} The lemma follows from the definitions of the operators using the Jacobi identity for the super Lie bracket operator.\end{proof}

The linearised operator for a heat flow $\Xi$ is
\begin{equation}\label{flow:op:lingeo}
\begin{split}
\mathcal D^0(\Xi)(\pi_A(\alpha))
:= &\pi_A\big(\nabla_s\pi_A(\alpha)-2[\psi_0, (\partial_t A-d_A\Psi)]\\
&\qquad-\nabla_t\nabla_t\pi_A(\alpha)- d*X_t(A)\pi_A(\alpha)+*[\alpha\wedge*\omega]\big)
\end{split}
\end{equation}
for any $\alpha: S^1\times\mathbb R\to \Omega^1(\Sigma,\mathfrak g_P)$ and where $\nabla_s:=\partial_s+[\Phi,\cdot]$ and $\psi_0$ is defined uniquely by $d_A^*d_A\psi_0=-2*[\pi_A(\alpha)\wedge*(\partial_t A-d_A\Psi)]$. This last formula can be obtained in the same  way as the Jacobi operator for perturbed geodesics (cf. \cite{MR1715156} or \cite{remyj6}).

\section{Perturbation}\label{sec:perturbation}

In order to achieve the transversality we have to perturb the equations and for this purpose we choose an abstract perturbation on $\mathcal L(\mathcal M^g(P))$.\\

First, we choose a perturbation $\bar {\mathcal V}:\mathcal L(\mathcal M^g(P))\to \mathbb R$ on the loop space of $\mathcal M^(P)$. We assume that $\bar{\mathcal V}$ satisfies the following condition (see condition (V4), section 2 of \cite{MR2276534} or condition (V3), section 1 of \cite{Weberhab}). For any two integers $k>0$ and $l\geq 0$ there is a constant $c=c(k,l)$ such that
\begin{equation}\label{flow:intro:cond}
\left|\nabla_t^l \nabla_s^k\bar{\mathcal V}(A)\right|
\leq c\sum_{k_j, l_j}\left(\Pi_{j, l_j>0}\left|\nabla_t^{l_j}\nabla_s^{k_j}A\right|\right)\Pi_{j, l_j=0}\left(\left|\nabla_s^{k_j}A\right|+\left\|\nabla_s^{k_j}A\right\|_{L^{p_j}}\right)
\end{equation}
for every smooth map $A:\mathbb R\to \mathcal L(\mathcal A_0(P)): s\mapsto A(s,\cdot)$ \footnote{$\Psi(t,s),\Phi(t,s)\in \Omega^0(P,\mathfrak g_P)$ are uniquely defined by $d_A^*\left(\partial_tA-d_A\Psi\right)=0$ and $d_A^*\left(\partial_sA-d_A\Phi\right)=0$.} and every $(s,t)\in \mathbb R\times S^1$; here $p_j\geq 1$ and $\sum_{l_j=0}\frac 1{p_j}=1$; the sum runs over all partitions $k_1+\dots+k_m=k$ and $l_1+\dots+l_m\leq l$ such that $k_j+l_j\geq 1$ for all $j$. For $k=0$ the same inequality holds with an additional summand $c$ on the right. In addition, $\nabla_s^{k_j}A$ and $\nabla_t^{l_j}A$ should be interpreted as $\nabla_s^{k_j-1}\left(\partial_sA-d_A\Phi\right)$ and $\nabla_t^{l_j-1}\left(\partial_tA-d_A\Psi\right)$.\\

Using the results of Weber (cf. \cite{Weberhab}, theorem 1.13)we can consider the following. For any regular value $b$ of the energy functional $E^H$ there is a Banach manifold $\mathcal O^b$ of perturbations $\bar {\mathcal V}$ that satisfy the above condition (\ref{flow:intro:cond}) and such that $E^H+\bar{\mathcal V}$ have the same critical loops as $E^H$. Moreover there is a residual subset $\mathcal O_{reg}^b\subset \mathcal O^b$ such that the perturbed functional $E^H+\bar {\mathcal V}$ is Morse-Smale below the energy level $b$ if $\bar {\mathcal V}\in \mathcal O_{reg}^b$. From now on all the computations are done using a generic perturbation of this kind. Next, we define $*X_t(A)\subset \Omega^1(\Sigma,\mathfrak g_P)$, $A\in \mathcal L(\mathcal A_0(P))$, $t\in S^1$, by
\begin{equation}\label{flow:intro:x}
\int_0^1\langle *X_t(A), \partial_sA_0(t,0)\rangle dt:=\frac d{ds}\Big|_{s=0}\left(\bar{\mathcal V}(A(s))+\int_0^1H_t(A(s)) dt\right)
\end{equation}
for every smooth variation $A(t,s)$ of $A(t)$.\\

Furthermore, using the results of Weber (cf. \cite{Weberhab}, theorems 1.7, 1.8), for any constant $b$ there are positive constants $c,\rho, c_0,c_1,c_2,\dots$ such that the following holds. If the connection $\Xi=A+\Psi dt+\Phi ds$ satisfies (\ref{flow:eqgeoflow}) and $E^H(A(\cdot,s))\leq b$, then for every $s$
\begin{equation}\label{flow:web:w1}
\|\partial_tA-d_A\Psi\|_{L^\infty}+  \|\nabla_t(\partial_tA-d_A\Psi)\|_{L^\infty}\leq c,
\end{equation}
\begin{equation}\label{flow:web:w2}
\|\partial_sA-d_A\Phi\|_{L^\infty}+  \|\nabla_t(\partial_sA-d_A\Phi)\|_{L^\infty}+  \|\nabla_s(\partial_sA-d_A\Phi)\|_{L^\infty}\leq c,
\end{equation}
\begin{equation}\label{flow:web:w3}
\|\partial_sA-d_A\Phi\|_{C^k(S^1\times[T,\infty))}\leq c_ke^{-\rho T},
\end{equation}
\begin{equation}\label{flow:web:w4}
\|\partial_sA-d_A\Phi\|_{C^k(S^1\times(-\infty,-T])}\leq c_ke^{-\rho T},
\end{equation}
for every $T\geq1$. Moreover $[A]$ converges to a perturbed geodesic in $C^2(S^1)$ as $s\to \pm\infty$.\\

Next, we need to choose an extension $\mathcal V:\mathcal L(\mathcal A(P)/\mathcal G_0(P))\to \mathbb R$, $\mathcal V(A)=\bar{\mathcal V}(A)$ for $A\in \mathcal L(\mathcal A_0(P))$, such that $\mathcal V$ satisfies (\ref{flow:intro:cond}) for any smooth map $A:\mathbb R\to\mathcal L(\mathcal A(P)$ with $\|F_{A(s)}\|_{L^2(\Sigma)}\leq \delta_0$ for every $s\in \mathbb R$, where $\delta_0$ is chosen such that the lemmas \ref{lemma76dt94} and \ref{lemma82dt94} hold for $p=2$ and $q=4$.\\
Another possibility is to extend $\bar{\mathcal V}$ in the following way. We choose a smooth map $\rho:[0,\infty)\to [0,1]$ with the property that $\rho(x)=0$ if $x\geq \delta_0$ and $\rho(x)=1$ if $x\leq \frac {2\delta_0}3$. Next, we define $\mathcal V:\mathcal L(\mathcal A(P)/\mathcal G_0(P))\to \mathbb R$ by
\begin{equation}
\mathcal V(A)=\rho\left(\sup_{t\in S^1}\|F_A(t,s)\|_{L^2}\right)\bar {\mathcal V}\left(A+*d_{A}\eta(A)\right)
\end{equation}
where $\eta(A)$ is the unique $0$-form, by lemma \ref{lemma82dt94}, which satisfies $$F_{A+*d_A\eta(A)}=0\quad \|d_A\eta(A)\|_{L^4(\Sigma)}\leq c\|F_A\|_{L^2(\Sigma)}.$$ We define $X_t(A)$ in the same way as in (\ref{flow:intro:x}).

\section{Yang-Mills flow}\label{sec:ymflow}
The Yang-Mills flow equations for $A(s)+\Psi(s) dt\in \mathcal A(P\times S^1)$ and for an $\varepsilon$-independent metric are
\begin{equation}\label{flow:intro:woeps}
\begin{split}
&\partial_s A-d_A\Phi +d_A^*F_A- \nabla_t\left(\partial_t A-d_A\Psi\right)-*X_t(A)=0,\\
&\partial_s\Psi-\nabla_t\Phi- d_A^*\left(\partial_t A-d_A\Psi\right)=0
\end{split}
\end{equation}
where $\Phi(t,s)\in \Omega^0(\Sigma,\mathfrak g_P)$ in order to make the equations gauge invariant. We can consider the $s$-dependent connection $A(s)+\Psi(s) dt$ together with the 0-form $\Phi$ as a connection $\Xi:=A+\Psi dt+\Phi ds$ on the 4-manifold $\Sigma\times S^1\times \mathbb R$. In our case we shrink the metric on $\Sigma$ by $\varepsilon^2$ and therefore the adjoint of the  exterior derivative $d_A$ contribute with a factor $\frac 1{\varepsilon^2}$ to the flow equation and if we consider the flow lines between two perturbed Yang-Mills connections $\Xi_\pm\in \mathrm{Crit}_{E^H}^b$ we have the equation
\begin{equation}\label{flow:eq}
\begin{split}
\partial_s A-d_A\Phi + \frac 1{\varepsilon^2}d_A^*F_A- \nabla_t\left(\partial_t A-d_A\Psi\right)-*X_t(A)=0,\\
\partial_s\Psi-\nabla_t\Phi- \frac 1{\varepsilon^2} d_A^*\left(\partial_t A-d_A\Psi\right)=0,\quad  \lim_{s\to \pm \infty} \Xi=\Xi_\pm.
\end{split}
\end{equation}
Another viewpoint to see the last equation is to consider (\ref{flow:intro:woeps})
for the connection
\begin{equation}
\tilde A(t,s)+\tilde \Psi(t,s) dt+\tilde \Phi (t,s) ds=A(\varepsilon t,\varepsilon^2s)+\varepsilon \Psi(\varepsilon t,\varepsilon^2 s) dt+\varepsilon^2\Phi(\varepsilon t,\varepsilon^2s) ds, 
\end{equation}
which is equivalent to (\ref{flow:eq}) for $(t,s)\in \left[0,\frac 1 \varepsilon\right]\times \mathbb R$.

\section{Norms}\label{flow:subsection:norm}

We choose a reference connection $\Xi:=A+\Psi dt+\Phi ds \in \mathcal A (\Sigma\times S^1\times \mathbb R)$; let $\xi:=\alpha+\psi dt+\phi ds \in \Omega^{1}(\Sigma\times S^1\times\mathbb R,\mathfrak g_P)$, $\alpha(t,s),\psi(t,s),\phi(t,s) \in \Omega^{j}(\Sigma,\mathfrak g_P)$, $j=0,1$, then
\begin{equation*}
\|\alpha+\psi dt+\phi ds\|_{\infty,\varepsilon }:=\|\alpha\|_{L^\infty}+\varepsilon\|\psi\|_{L^\infty }+\varepsilon^2\|\phi\|_{L^\infty},
\end{equation*}
\begin{equation*}
\|\alpha+\psi dt+\phi ds\|_{0,p,\varepsilon }^p:=\int_{S^1\times \mathbb R}\left(\|\alpha\|_{L^p(\Sigma)}^p+\varepsilon^p\|\psi\|_{L^p(\Sigma)}^p+\varepsilon^{2p}\|\phi\|_{L^p(\Sigma)}^p\right) dt\,ds,
\end{equation*}
\begin{equation*}
\begin{split}
\|\alpha&+\psi dt+\phi ds\|_{1,p,\varepsilon }^p:=\|\alpha+\psi dt+\phi ds\|_{0,p,\varepsilon }^p\\
&+\int_{S^1\times \mathbb R}\left(\|d_A^*\alpha\|_{L^p(\Sigma)}^p+\|d_A\alpha\|_{L^p(\Sigma)}^p+\varepsilon^p \|\nabla_t\alpha\|_{L^p(\Sigma)}^p+\varepsilon^{2p} \|\nabla_s\alpha\|_{L^p(\Sigma)}^p\right) dt\,ds\\
&+\int_{S^1\times \mathbb R}\left(\varepsilon^{p}\|d_A\psi \|_{L^p(\Sigma)}^p+\varepsilon^{2p} \|\nabla_t\psi \|_{L^p(\Sigma)}^p+\varepsilon^{3p} \|\nabla_s\psi \|_{L^p(\Sigma)}^p\right) dt\,ds\\
&+\int_{S^1\times \mathbb R}\left(\varepsilon^{2p}\|d_A\phi \|_{L^p(\Sigma)}^p+\varepsilon^{3p} \|\nabla_t\phi \|_{L^p(\Sigma)}^p+\varepsilon^{4p} \|\nabla_s\phi \|_{L^p(\Sigma)}^p\right) dt\,ds,
\end{split}
\end{equation*}
\begin{equation}\label{flow:definorm12}
\begin{split}
\|\alpha&+\psi dt+\phi ds\|_{1,2;p,\varepsilon }^p:=\|\alpha+\psi dt+\phi ds\|_{1,p,\varepsilon }^p\\
&+\int_{S^1\times \mathbb R}\left(\|d_A^*d_A\alpha \|_{L^p(\Sigma)}^p+\varepsilon^p \|\nabla_t d_A \alpha\|_{L^p(\Sigma)}^p\right) dt\,ds\\
&+\int_{S^1\times \mathbb R}\left(\varepsilon \|d_A^*(d_A\psi-\nabla_t \alpha) \|_{L^p(\Sigma)}^p+\varepsilon^{2} \|\nabla_t(d_A\psi-\nabla_t\alpha)\|_{L^p(\Sigma)}^p\right) dt\,ds
\end{split}
\end{equation}
where $\nabla_t:=\partial_t+[\Psi, \cdot]$ and $\nabla_s:=\partial_s+[\Phi,\cdot]$. The $\varepsilon$-dependent norms are created using the following simple rule that is given from the linearisation $\mathcal D^\varepsilon$ of the Yang-Mills flow equations. For every $\nabla_t$ and every 0-form $\psi$, which descends from a 1-form in the $t$-direction, we put an $\varepsilon$ in front of the norm; for every $\nabla_s$ and every 0-form $\phi$, coming from a 1-form in the $s$-direction, we multiply by $\varepsilon^2$. The definition (\ref{flow:definorm12}) contains, in the first line, all the $0$-order $L^p$-norms and the $L^p$-norms of all the first derivatives; in the last two lines we can find the $L^p$-norms of some second derivatives. These can be interpreted in the following way. We split $\alpha+\psi dt$ in two orthogonal components $\alpha_i+\psi_i dt\in \textrm{im } d_\Xi$ and $\alpha_k+\psi_k dt\in \ker d_\Xi^{*_\varepsilon}$; on the one side, if $\alpha+\psi dt \in \ker d_{A+\Psi dt }^{*_\varepsilon}$, then
\begin{equation*}
\begin{split}
\varepsilon\|d_A^*d_A\psi-\varepsilon^2\nabla_t\nabla_t\psi \|_{L^p}=&\varepsilon\|d_A^*d_A\psi-\nabla_t d_A^*\alpha \|_{L^p}\\
\leq& \varepsilon\|d_A^*(d_A\psi-\nabla_t\alpha) \|_{L^p}+\varepsilon\|[(\partial_t A-d_A\Psi)\wedge *\alpha]\|_{L^p},
\end{split}
\end{equation*}
\begin{equation*}
\begin{split}
\|d_Ad_A^*\alpha -\varepsilon^2\nabla_t\nabla_t\alpha \|_{L^p}=&\varepsilon^2\| d_A \nabla_t\psi-\nabla_t\nabla_t\alpha \|_{L^p}\\
\leq& \varepsilon^2\|\nabla_t(d_A\psi-\nabla_t\alpha) \|_{L^p}+\varepsilon^2\|[(\partial_t A-d_A\Psi),\psi]\|_{L^p}
\end{split}
\end{equation*}
and thus $\varepsilon \|d_A^*d_A\psi\|_{L^p}$, $\varepsilon^3\|\nabla_t\nabla_t\psi\|_{L^p}$, $\|d_A^*d_A\alpha \|_{L^p}$ and $\varepsilon^2\|\nabla_t\nabla_t\alpha \|_{L^p}$ can be estimates by $\|\alpha+\psi dt+\phi ds\|_{1,2;p,\varepsilon }$ as we will discuss in the section \ref{flow:section:linest}. On the other side, if $\alpha+\psi dt =d_A\gamma+\nabla_t\gamma dt$, $\gamma \in \Omega^0(\Sigma\times S^1\times \mathbb R,\mathfrak g_P)$, then 
$$d_A^*d_A\alpha=d_A^*[F_A,\gamma],\quad d_A^*(d_A\psi-\nabla_t\alpha)=-d_A^*[(\partial_t A-d_A\Psi),\gamma],$$
$$\nabla_t d_A\alpha=\nabla_t[F_A,\gamma],\quad \nabla_t(d_A\psi-\nabla_t\alpha)=-\nabla_t[(\partial_t A-d_A\Psi),\gamma];$$
therefore, under some extra conditions on the curvature $F_A-(\partial_t A-d_A\Psi) dt$, for example that $F_A=0$ and $\partial_t A-d_A\Psi$ is smooth, the last two lines of (\ref{flow:definorm12}) can be estimate with the first two if $\alpha+\psi dt \in \textrm{im } d_\Xi$. Thus, $\|\cdot\|_{1,2;p,\varepsilon}$ considers the $L^p$-norm of $\xi$ and of its derivatives, but the $L^p$-norm of the second derivatives in the $\Sigma\times S^1$-directions only for the $\ker d_{A+\Psi dt}^{*_\varepsilon}$-part of $\xi$. This orthogonal splitting plays a fundamental role in the proof of the linear estimates of the section \ref{flow:section:linest}. Next, we can define the Sobolev spaces 
$$W^{1,2;p}:=W^{1,2;p}\left(\Sigma\times S^1\times \mathbb R, T^*(\Sigma\times S^1\times \mathbb R)\otimes \mathfrak g_{P\times S^1\times \mathbb R} \right),$$
$$W^{1,p}:=W^{1,p}\left(\Sigma\times S^1\times \mathbb R, T^*(\Sigma\times S^1\times \mathbb R)\otimes \mathfrak g_{P\times S^1\times \mathbb R} \right)$$ as the completion, respect to the norm $\|\cdot\|_{1,2;p,1}$ and $\|\cdot\|_{1,p,1}$, of the 1-forms $$\Omega^1\left(\Sigma\times S^1\times \mathbb R, \mathfrak g_{P\times S^1\times \mathbb R}\right)$$ with compact support; we denote by $W^{1,2;p}(\Xi_-,\Xi_+)$ (respectively $W^{1,p}$) the space of all connections $\Xi$ that satisfy $\Xi-\Xi_0 \in W^{1,2;p}$ (respectively $\Xi-\Xi_0 \in W^{1,p}$) for a smooth connection $\Xi_0\in\mathcal A(P\times S^1\times \mathbb R)$ and the limit conditions $\lim_{s\to \pm\infty}\Xi=\Xi_\pm$ (respectively without limit condition). Furthermore, we denote by $\mathcal G_0^{2,p}\left(P\times S^1\times \mathbb R\right)$ the completion of $\mathcal G_0\left(P\times S^1\times \mathbb R\right)$ with respect to the Sobolev $W^{1,p}$-norm on 1-forms, i.e. $g\in\mathcal G_0^{2,p}\left(P\times S^1\times \mathbb R\right)$ if $g^{-1}d_{\Sigma\times S^1\times \mathbb R}g\in W^{1,p}$ and  by $\mathcal G_0^{1,2;p}\left(P\times S^1\times \mathbb R\right)$ the completion respect to the norm $\|\cdot\|_{1,2;p,\varepsilon}$ on the $1$-forms. In addition, we denote by $\bar {\mathcal G}_{0}^{1,2;2}(P\times S^1\times \mathbb R)$ the gauge group such that an element $g$ is locally in $\mathcal G_0^{1,2;2}\left(P\times S^1\times \mathbb R\right)$, i.e. we allow also elements that do not vanish at $\pm \infty$. We conclude this section proving the following Sobolev estimates.

\begin{theorem}[Sobolev estimate]\label{flow:thm:sob}
We choose $1\leq p,q<\infty$. Then there is a constant $c_S$ such that for any $\xi\in W^{1,p}$, $0<\varepsilon\leq1$:
\begin{enumerate}
\item If $-\frac 4q\leq 1-\frac 4p$, then
\begin{equation}
\|\xi\|_{0,q,\varepsilon}\leq c_S \varepsilon^{\frac3q-\frac3p}\|\xi\|_{1,p,\varepsilon }.
\end{equation}
\item If $0<1-\frac 4p$, then
\begin{equation}
\|\xi\|_{\infty,\varepsilon}\leq c_S \varepsilon^{-\frac 3p }\|\xi\|_{1,p,\varepsilon}.
\end{equation}
\end{enumerate}
\end{theorem}

\begin{proof}
Analogously as for the lemma 4.1 in \cite{MR1283871}, we can define $\bar \xi=\bar \alpha+\bar \psi dt+ \bar \phi ds$ by $\bar\alpha(t,s)=\alpha(\varepsilon t,\varepsilon^2s)$, $\bar\psi(t,s)=\varepsilon \psi(\varepsilon t,\varepsilon^2s)$ and $\bar\phi(t,s)=\varepsilon^2\phi(\varepsilon t,\varepsilon^2s)$. Thus, $\|\xi\|_{1,p,\varepsilon}=\varepsilon^{\frac 3p}\|\bar\xi\|_{W^{1,p}}$ for $0\leq t\leq \varepsilon^{-1}$, $\varepsilon\in (0,1]$. Therefore the theorem follows from the standard Sobolev's inequality.
\end{proof}

\section{Morse homologies}\label{flow:section:morse}

In this section we want to define the Morse homologies defined using the heat flow and the Yang-Mills $L^2$-flow. We start with the Morse homology of the loop space on $\mathcal M^g(P)$. First, we introduce the moduli spaces
 \begin{align*}
\mathcal M^0(\Xi_-,\Xi_+):=&\{ \Xi \in W^{1,2;2}(\Xi_-,\Xi_+); \Xi \textrm{ satisfies } (\ref{flow:eqgeoflow333}), (\ref{flow:eqgeoflow}) \},\\
\bar{\mathcal M}^0(\Xi_-,\Xi_+):=&\mathcal M^0(\Xi_-,\Xi_+)/PG_{\infty}
\end{align*}
where
$$PG_{\infty}:=\{g\in \mathcal G_0^{1,2;2}\left(P\times S^1\times \mathbb R\right); \exists S>0 \textrm{ such that for } |s|\geq S, g(s)=1\}.$$

Now, we choose a regular value $b$ of the energy functional $E^H$. In order to define the Morse homology of the  loop space $\mathcal L^b\mathcal M^g(P)$ in our Morse-Bott setting, where a critical loop is an equivalence class $[A+\Psi dt]$ of perturbed geodesics with $A+\Psi dt\in \mathrm{Crit}_{E^H}^b$, we need to count the flow lines between the critical loops with Morse index difference $1$ in the following way. We define $\left[{\mathrm{Crit}}_{E^H}^b\right]:=\mathrm{Crit}_{E^H}^b/\mathcal G_0(P\times S^1)$ and we consider the space of flow lines between two loops $\gamma_\pm\in \left[{\mathrm{Crit}}_{E^H}^b\right]$:
$$\mathcal {FL}^0(\gamma_-,\gamma_+)=\{ \Xi \in W^{1,2;2}(\Xi_-,\Xi_+); \Xi \textrm{ satisfies } (\ref{flow:eqgeoflow333}), (\ref{flow:eqgeoflow}), [A_\pm]=\gamma_\pm\}$$
and thus the moduli space 
$$\mathcal M^0(\gamma_-,\gamma_+):=\mathcal {FL}^0(\gamma_-,\gamma_+)/ \bar {\mathcal G}_{0}^{1,2;2}(P\times S^1\times \mathbb R);$$
then we organise the critical loops of $[\mathrm{Crit}_{E^H}^b]$ in a chain complex where 
$$C_k^{E^H,b}:= \oplus_{\gamma\in [\mathrm{Crit}_{E^H}^b], index_{E^H}(\gamma)=k} \mathbb Z_2 \gamma$$
and the boundary operator $\partial_k^{E^H}: C_{k}^{E^H,b}\to C_{k-1}^{E^H,b}$ by 
$$\partial_k^{E^H}\gamma_-:=\sum_{\gamma_+ \in C_{k-1}^{E^H,b}} \left(\sharp_{\mathbb Z_2}\left( {\mathcal M}^0(\gamma_-,\gamma_+)/\mathbb R\right)\right)\gamma_+.$$
If the functional $E^H$ satisfies the transversality condition, then $\partial_{k+1}^{E^H}\partial_k^{E^H}=0$ and in this case we can define the Morse homology
\begin{equation}\label{Morseg}
HM_*\left(\mathcal L^b\mathcal M^g(P),\mathbb Z_2\right):= \ker \partial_{*}^{E^H}/\textrm{im }\partial_{*+1}^{E^H}.
\end{equation}
As we have already mentioned, by the work of Weber (cf. \cite{Weberhab}), for a generic perturbation, the transversality condition is satisfied and thus the Morse homology of the loop space   $HM_*\left(\mathcal L^b\mathcal M^g(P),\mathbb Z_2\right)$ is well defined.

\begin{remark}
For any two perturbed geodesics $\gamma_\pm\in [\mathrm{Crit}_{E^H}^b]$ and any two representatives $\Xi_\pm$ we can identify the moduli spaces ${\mathcal M}^0(\gamma_-,\gamma_+)$ and $\bar{\mathcal M}^0(\Xi_-,\Xi_+)$, in particular we have
$$\sharp_{\mathbb Z_2}\left({\mathcal M}^0(\gamma_-,\gamma_+)/\mathbb R\right)=\sharp_{\mathbb Z_2}\left(\bar{\mathcal M}^0(\Xi_-,\Xi_+)/\mathbb R\right).$$
\end{remark}

Next, we define the Morse homology for the Yang-Mills case. First, we denote by $\mathcal M^\varepsilon(\Xi_-,\Xi_+)$ and by $\bar{\mathcal M}^\varepsilon(\Xi_-,\Xi_+)$ the moduli spaces
\begin{align*}
\mathcal M^\varepsilon(\Xi_-,\Xi_+):=&\{ \Xi\in W^{1,2;2}(P\times S^1\times \mathbb R); \Xi \textrm{ satisfies } (\ref{flow:eq})\},\\
\bar{\mathcal M}^\varepsilon(\Xi_-,\Xi_+):=&\mathcal M^\varepsilon(\Xi_-,\Xi_+)/PG_{\infty}.
\end{align*}
Also in this case we can define a Morse homology for 
$$\mathcal A^{\varepsilon,b}\left(P\times S^1\right)/\mathcal G_0\left(P\times S^1\right);$$ 
in order to do that we consider the chain complex $C_k^{\mathcal{YM}^{\varepsilon,H},b}:= \oplus_{\theta\in [\mathrm{Crit}_{\mathcal{YM}^{\varepsilon,H}}^b]} \mathbb Z_2 \theta$, where 
$$[\mathrm{Crit}_{\mathcal{YM}^{\varepsilon,H}}^b]:=
\mathrm{Crit}_{\mathcal{YM}^{\varepsilon,H}}^b/\mathcal G_0({\Sigma\times S^1}),$$ with the boundary operator $\partial_k^{\mathcal{YM}^{\varepsilon,H}}: C_{k}^{\mathcal{YM}^{\varepsilon,H},b}\to C_{k-1}^{\mathcal{YM}^{\varepsilon,H},b}$ defined by 
$$\partial_k^{\mathcal{YM}^{\varepsilon,H}}\theta_-:=\sum_{\theta_+ \in C_{k-1}^{\mathcal{YM}^{\varepsilon,H},b}} \sharp_{\mathbb Z_2}\left(\mathcal M^\varepsilon(\theta_-,\theta_+)/\mathbb R\right)\theta_+$$
where $\mathcal M^\varepsilon(\theta_-,\theta_+)$ is the moduli space
$$\mathcal M^\varepsilon(\theta_-,\theta_+):=\mathcal {FL}^\varepsilon(\theta_-,\theta_+)/\bar{\mathcal G}_{0}^{1,2;2}(P\times S^1\times \mathbb R),$$
$$\mathcal {FL}^\varepsilon(\theta_-,\theta_+)=\{ \Xi \in W^{1,2;2} (\Xi_-,\Xi_+); \Xi \textrm{ satisfies } (\ref{flow:eq}), [\Xi_\pm]=\theta_\pm\}.$$
The functional $\mathcal{YM}^{\varepsilon,H}$ will inherits the transversality property of $E^H$ provided that $\varepsilon$ is small enough, in this case $\partial_{k+1}^{\mathcal{YM}^{\varepsilon,H}}\partial_k^{\mathcal{YM}^{\varepsilon,H}}=0$ and thus we can define the Morse homology
\begin{equation}\label{Morseym}
HM_*\left(\mathcal A^{\varepsilon,b}\left(P\times S^1\right)/\mathcal G_0\left(P\times S^1\right),\mathbb Z_2\right):= \ker \partial_{*}^{\mathcal{YM}^{\varepsilon,H}}/\textrm{im }\partial_{*+1}^{\mathcal{YM}^{\varepsilon,H}}.\end{equation}

\begin{remark}
Also in this case, for any two orbits of perturbed Yang-Mills connections $\theta_\pm\in [\mathrm{Crit}_{\mathcal {YM}^{\varepsilon,H}}^b]$ and any two representatives $\Xi_\pm \in\mathrm{Crit}_{\mathcal {YM}^{\varepsilon,H}}^b$ we can identify the moduli spaces ${\mathcal M}^\varepsilon(\theta_-,\theta_+)$ and $\bar{\mathcal M}^\varepsilon(\Xi_-,\Xi_+)$; in particular we have
$$\sharp_{\mathbb Z_2}\left({\mathcal M}^\varepsilon(\theta_-,\theta_+)/\mathbb R\right)=\sharp_{\mathbb Z_2}\left(\bar{\mathcal M}^\varepsilon(\Xi_-,\Xi_+)/\mathbb R\right).$$
\end{remark}

The aim of this paper is to show that the two Morse homologies (\ref{Morseg}) and (\ref{Morseym}) are isomorph and we give the proof in the section \ref{c:flow:mt}. In order to do this we need to show that there is a bijective map 
$$\mathcal R^{b,\varepsilon}: \mathcal M^0(\Xi_-,\Xi_+)\to\mathcal M^\varepsilon\left(\mathcal T^{b,\varepsilon}(\Xi_-),\mathcal T^{b,\varepsilon}(\Xi_+)\right)$$
for each regular value $b$ of $E^H$, every pair $\Xi_-, \Xi_+\in \mathrm{Crit}^{b}_{E^H}$ with index difference 1 and for $\varepsilon$ sufficiently small; for this purpose, we will proceed in the following way. In section \ref{flow:section:linest} we will prove some linear estimates using the linear operator, for a $1$-form $\xi=\alpha+\psi dt+\phi ds\in W^{1,2;p}$
\begin{equation}\label{linopflow}
\mathcal D^\varepsilon(\Xi)(\xi):=\mathcal D_1^\varepsilon(\Xi)(\xi)+\mathcal D_2^\varepsilon(\Xi)(\xi) dt+\mathcal D_3^\varepsilon(\Xi)(\xi) ds
\end{equation}
where the first two terms are the linearization of (\ref{flow:eq}), i.e.
\begin{equation}\label{flow:op:linym12}
\begin{split}
\mathcal D_1^\varepsilon(\Xi)(\xi)
:=& \nabla_s\alpha -d_A\phi +\frac 1{\varepsilon^2} d_A^*d_A\alpha +\frac 1{\varepsilon^2}*[\alpha\wedge *F_A]\\
&-\nabla_t\nabla_t\alpha+d_A\nabla_t\psi-2[\psi,(\partial_t A-d_A\Psi)]-d*X_t(A)\alpha,\\
\mathcal D_2^\varepsilon(\Xi)(\xi)
:=& \nabla_s\psi-\nabla_t\phi+\frac 2{\varepsilon^2}*[\alpha\wedge*(\partial_t A-d_A\Psi)]\\
&-\frac 1{\varepsilon^2}\nabla_t d_A^*\alpha+\frac 1{\varepsilon^2}d_A^*d_A\psi.
\end{split}
\end{equation}
and the third one is, for a fixed reference connection $\Xi^0=A^0+\Psi^0dt+\Phi^0 ds$,
\begin{equation}\label{flow:op:linym3}
\mathcal D_3^\varepsilon(\Xi)(\xi):= \nabla_s^{\Phi^0}\phi-\frac 1{\varepsilon^4}d_{A^0}^*\alpha+\frac 1{\varepsilon^2}\nabla_t^{\Psi^0}\psi.
\end{equation}
The linear operator $\mathcal D^\varepsilon(\Xi)$ can also be seen as the linearisation of the map 
$$\mathcal F^\varepsilon(\Xi):= \mathcal F^\varepsilon_1(\Xi)+\mathcal F^\varepsilon_2(\Xi) dt+\mathcal F^\varepsilon_3(\Xi) ds$$ 
where
\begin{equation}\label{flow:eqfds}
\begin{split}
\mathcal F^\varepsilon_1(\Xi):=&\partial_s A-d_A\Phi + \frac 1{\varepsilon^2}d_A^*F_A- \nabla_t\left(\partial_t A-d_A\Psi\right),\\
\mathcal F^\varepsilon_2(\Xi):=&\partial_s\Psi-\nabla_t\Phi- \frac 1{\varepsilon^2} d_A^*\left(\partial_t A-d_A\Psi\right),\\
\mathcal F_3^\varepsilon(\Xi):=& \nabla_s^{\Phi^0}(\Phi-\Phi_2)-\frac 1{\varepsilon^4}d_{A^0}^*(A-A_2)+\frac 1{\varepsilon^2}\nabla_t^{\Psi^0}(\Psi-\Psi_2)
\end{split}
\end{equation}
and $A_2+\Psi_2 dt+\Phi_2 ds:=\mathcal K_2^\varepsilon(A^0+\Psi^0 dt+\Psi^0 ds)$; we will discuss the map $\mathcal K_2^\varepsilon$ in the section \ref{flow:section:firstapprox}.\\

After computing some quadratic estimates in section \ref{flow:section:qest}, we will prove the existence and the local uniqueness of the map $\mathcal R^{b,\varepsilon}$ (section \ref{flow:section:themap}). In the following sections \ref{flow:section:apriori}, \ref{flow:section:linfty} and \ref{flow:section:expconv} we will prove some a priori estimates that we will use in the section \ref{flow:section:surj} in order to prove the surjectivity of $\mathcal R^{b,\varepsilon}$. The section \ref{ch:rcg} is devoted to prove the Coulomb gauge condition theorem.\\

%
%

\section{Linear estimates for the Yang-Mills flow operator}\label{flow:section:linest}
 As we already mentioned, by the Weber's regularity theorem (cf. \cite{Weberhab}, theorem 1.13), we can assume that the energy functional $E^H$ is Morse-Smale. In this section we will prove a linear estimate, theorem \ref{flow:thme:linest333}, for the operator $\mathcal D^\varepsilon (\Xi)$ for a perturbed geodesic flow $\Xi=A+\Phi dt+\Psi ds$.The main idea is to divide the linear operator respect to the orthogonal splitting $\ker d_{A+\Psi dt }^{*_\varepsilon}\oplus \textrm{im } d_{A+\Psi dt}$ and to use different linear estimates on the two parts. In order to investigate this we need to decompose, in a unique way, every 1-form $\xi=\alpha+\psi dt$ as $(\alpha_k+\psi_k dt)+(\alpha_i+\psi_i dt)$ where $\frac 1{\varepsilon^2}d_A^*\alpha_k-\nabla_t\psi_k=0$ and $\alpha_i+\psi_i dt=d_A\gamma+\nabla_t\gamma dt$ for a $0$-form $\gamma$. Formally, first, we solve the equation
\begin{equation}\label{flow:cond:split0}
\frac 1{\varepsilon^2}d_A^*d_A\gamma-\nabla_t\nabla_t \gamma=\frac 1{\varepsilon^2} d_A^* \alpha-\nabla_t\psi
\end{equation}
which has a unique solution $\gamma$ whose existence and uniqueness can be proved as in the lemma 6.4 of \cite{MR1283871}; then we define
\begin{equation}\label{flow:cond:split}
\alpha_i+\psi_i dt:=d_A\gamma+\nabla_t\gamma dt,\quad \alpha_k+\psi_k dt:=(\alpha+\psi dt)-(\alpha_i+\psi_i dt)
\end{equation}
and since the splitting is orthogonal, 
\begin{equation}
\|d_A\gamma+\nabla_t\gamma dt\|_{0,p,\varepsilon}\leq \|\alpha+\psi dt\|_{0,p,\varepsilon},\quad \|\alpha_k+\psi_k dt\|_{0,p,\varepsilon}\leq \|\alpha+\psi dt\|_{0,p,\varepsilon}.
\end{equation}
By definition and using the commutation formulas (\ref{commform}), (\ref{commform2}) we have also that 
\begin{equation}\label{flow:linest:eqfff}
\begin{split}
d_A\nabla_t\psi_k=&\frac 1{\varepsilon^2} d_Ad_A^*\alpha_k,\\
d_A\nabla_t\psi_i=&\nabla_t\nabla_t\alpha_i-2[(\partial_t A-d_A\Psi),\psi_i]-[\nabla_t(\partial_t A-d_A\Psi),\gamma],\\
\nabla_t d_A^*\alpha_i=&d_A^*\nabla_t d_A\gamma+*[\alpha_i\wedge*(\partial_t A-d_A\Psi)]\\
=&d_A^*d_A\psi_i+*[\alpha_i\wedge*(\partial_t A-d_A\Psi)]+d_A^*[(\partial_t A-d_A\Psi),\gamma]\\
=&d_A^*d_A\psi_i+2*[\alpha_i\wedge*(\partial_t A-d_A\Psi)]-*[d_A*(\partial_t A-d_A\Psi),\gamma].
\end{split}
\end{equation}
Now, we  can write the components of the linear operator using this splitting. On the one hand, the first component of $\mathcal D^\varepsilon(\Xi)(\xi+\phi ds)$, defined by (\ref{flow:op:linym12}), is, using the identities (\ref{flow:linest:eqfff}),
\begin{align*}
\mathcal D_1^\varepsilon(\Xi)(\xi+\phi ds)=&\nabla_s\alpha_k +\frac 1{\varepsilon^2} d_A^*d_A\alpha _k+\frac 1{\varepsilon^2} d_Ad_A^*\alpha_k-\nabla_t\nabla_t\alpha_k\\
&+\nabla_s\alpha_i-d_A\phi-[\nabla_t(\partial_t A-d_A\Psi),\gamma]\\
&-2[\psi_k,(\partial_t A-d_A\Psi)]-d*X_t(A)\alpha;
\end{align*}
in the other hand, using (\ref{flow:linest:eqfff}), the second component becomes
\begin{align*}
\mathcal D_2^\varepsilon(\Xi)(\xi+\phi ds)
=& \nabla_s\psi_k-\nabla_t\nabla_t\psi_k+\frac 1{\varepsilon^2}d_A^*d_A\psi_k+\frac 2{\varepsilon^2}*[\alpha_k\wedge*(\partial_t A-d_A\Psi)]\\
&+ \nabla_s\psi_i-\nabla_t\phi.
\end{align*}
The third component is the easiest to investigate, because it depends only on $\alpha_i+\psi_i dt$ and on $\phi$:
\begin{equation*}
\mathcal D_3^\varepsilon(\Xi)(\xi+\phi ds)=\nabla_s\phi-\frac 1{\varepsilon^4}d_A^*\alpha+\frac 1{\varepsilon^2}\nabla_t\psi
=  \nabla_s\phi-\frac 1{\varepsilon^4}d_A^*\alpha_i+\frac 1{\varepsilon^2}\nabla_t\psi_i.
\end{equation*}
Next, the idea is to consider $\mathcal D^\varepsilon(\Xi)(\xi+\phi ds)$ as the sum of the following three operators
\begin{equation*}
\begin{split}
\mathcal D^{\varepsilon,1}(\Xi)(\xi+\phi ds):=&\nabla_s\alpha_k +\frac 1{\varepsilon^2} d_A^*d_A\alpha _k+\frac 1{\varepsilon^2} d_Ad_A^*\alpha_k-[\psi_k,(\partial_t A-d_A\Psi)]\\
&-\nabla_t\nabla_t\alpha_k+\left( \nabla_s\psi_k-\nabla_t\nabla_t\psi_k+\frac 1{\varepsilon^2}d_A^*d_A\psi_k\right) dt\\
&+\frac 1{\varepsilon^2}*[\alpha_k\wedge*(\partial_t A-d_A\Psi)]\, dt,\\
\mathcal D^{\varepsilon,2}(\Xi)(\xi+\phi ds):=&\nabla_s\alpha_i-d_A\phi+ \left( \nabla_s\psi_i-\nabla_t\phi\right) dt\\
&+\left(\nabla_s\phi-\frac 1{\varepsilon^4}d_A^*\alpha+\frac 1{\varepsilon^2}\nabla_t\psi\right) ds,\\
\textrm{Rest}^\varepsilon(\Xi)(\xi+\phi ds):=&-[\nabla_t(\partial_t A-d_A\Psi),\gamma]-[\psi_k,(\partial_t A-d_A\Psi)]\\
&-d*X_t(A)\alpha+\frac 1{\varepsilon^2}*[\alpha_k\wedge*(\partial_t A-d_A\Psi)] \,dt
\end{split}
\end{equation*}
and to project them in to the two parts of the orthogonal splitting $\textrm{im } d_{A+\Psi dt} \oplus \ker d_{A+\Psi dt}^{*_\varepsilon}$. The result is that the important part of $\mathcal D^{\varepsilon,1}(\Xi)$ lies in $\ker d_{A+\Psi dt}^{*_\varepsilon}$ and that of $\mathcal D^{\varepsilon,2}(\Xi)$ in $\textrm{im } d_{A+\Psi dt}$ as is showed in the next lemma; in other words, we interchange the operator $\mathcal D^\varepsilon(\Xi)$ with the projection in the two parts of the splitting. We recall that, by (\ref{flow:web:w1}) and (\ref{flow:web:w2}), we can assume that
\begin{equation}\label{flow:aprioriest:geo}
\left\|\partial_tA-d_A\Psi \right\|_{L^\infty}+\left\|\nabla_t\left(\partial_tA-d_A\Psi\right) \right\|_{L^\infty}+\left\|\partial_sA-d_A\Phi \right\|_{L^\infty}\leq c_0
\end{equation}
for a positive constant $c_0$.
\begin{lemma}\label{flow:lemma:kerim2}
We choose $b,p>0$. For any geodesic flow $\Xi=A+\Psi dt+\Phi ds \in \mathcal M^0(\Xi_-,\Xi_+)$, $\Xi_\pm\in \mathrm{Crit}^b_{E^H}$, there exists a positive constant $c$ such that
\begin{equation*}
\left \| \Pi_{\textrm{im } d_{A+\Psi dt}} \mathcal D^{\varepsilon,1}(\Xi)(\xi+\phi ds)\right\|_{0,p,\varepsilon}\leq c\left( \|\alpha_k \|_{L^p}+\|\psi_k\|_{L^p}+\|\nabla_t\alpha_k\|_{L^p}\right),
\end{equation*}
\begin{equation*}
\left \| \left(1-\Pi_{\textrm{im } d_{A+\Psi dt}}\right)\left(\nabla_s\alpha_i -d_A\phi +\left( \nabla_s\psi_i-\nabla_t\phi\right) dt\right)\right\|_{0,p,\varepsilon}\leq c \|\alpha_i \|_{L^p},
\end{equation*}
\begin{equation*}
\varepsilon^2\left\|\mathrm {Rest}^\varepsilon(\Xi)(\xi+\phi ds)\right\|_{0,p,\varepsilon}\leq c\varepsilon \left\|\xi\right\|_{0,p,\varepsilon}
\end{equation*}
for all $\xi+\phi ds \in W^{1,2;p}$ and using the splitting $\xi=:(\alpha_k+\psi_k dt)+(\alpha_i+\psi_i dt)$ defined by (\ref{flow:cond:split0}) and (\ref{flow:cond:split}). We denote by $\Pi_{\textrm{im } d_{A+\Psi dt}}$ the projection in to the linear subspace $\textrm{im } d_{A+\Psi dt}$.
\end{lemma}

\begin{proof}First, we remark that
\begin{equation*}
\begin{split}
&\langle\nabla_s\alpha_k+\frac1{\varepsilon^2}d_A^*d_A\alpha_k,d_A\omega\rangle+\varepsilon^2\langle\nabla_s\psi_k,\nabla_t\omega\rangle\\
&\qquad=\langle\nabla_s(d_A^*\alpha_k-\varepsilon^2\nabla_t\psi_k),\omega\rangle+\frac 1{\varepsilon^2}*\left[F_A\wedge *d_A\alpha_k\right]\\
&\qquad\quad+\langle*[(\partial_s A-d_A\Phi)\wedge*\alpha_k]+\varepsilon^2[(\partial_s\Psi-\partial_t\Phi+[\Psi,\Phi]),\psi_k],\omega\rangle\\
&\qquad=\langle*[(\partial_s A-d_A\Phi)\wedge*\alpha_k]+\varepsilon^2[(\partial_s\Psi-\partial_t\Phi+[\Psi,\Phi]),\psi_k],\omega\rangle,
\end{split}
\end{equation*}
where we used the commutation formulas (\ref{commform2}) and
\begin{equation}\label{commform3}
 [\nabla_s,\nabla_t]\omega=[(\partial_s\Psi-\partial_t\Phi+[\Psi,\Phi]),\omega]
\end{equation}
for any $0$-form $\omega\in\Omega^0(\Sigma\times S^1\times \mathbb R,\mathfrak g_P)$.
\begin{equation*}
\begin{split}
&\langle\frac 1{\varepsilon^2} d_Ad_A^*\alpha_k-d_A\nabla_t\psi_k,d_A\omega\rangle=0,\\
&\langle\frac1{\varepsilon^2}\nabla_t d_A^*\alpha_k-\nabla_t\nabla_t\psi_k,\nabla_t\omega\rangle=0,\\
&\langle-\nabla_t\nabla_t\alpha_k,d_A\omega\rangle+\langle -d_A^*\nabla_t\alpha_k,\nabla_t\omega\rangle
=-*[(\partial_t A-d_A\Psi)\wedge*\nabla_t\alpha_k],\omega\rangle,\\
&\langle \nabla_t d_A\psi_k,d_A\omega\rangle+\langle d_A^*d_A\psi_k,\nabla_t\omega\rangle\\
&\qquad=\langle *[(\partial_t A-d_A\Psi)\wedge *d_A\psi_k],\omega\rangle
=\langle -*[*(\partial_t A-d_A\Psi)\wedge\psi_k],d_A\omega\rangle
\end{split}
\end{equation*}
where for the last step we used that $d_{A}^*\left(\partial_t A-d_A \Psi\right)=0$. Next, we choose $q$ such that $\frac1p+\frac1q=1$. Then
\begin{equation*}
\begin{split}
\Big\| \Pi_{\textrm{im } d_{A+\Psi dt}} &\mathcal D^{\varepsilon,1}(\Xi)(\xi+\phi ds)\Big\|_{0,p,\varepsilon}\\
\leq &\sup_{\omega\in\Omega^0(\Sigma\times S^1\times \mathbb R)}\frac{\langle \mathcal D^{\varepsilon,1}(\Xi)(\xi+\phi ds),d_A\omega+\nabla_t\omega dt\rangle}{\|d_A\omega+\nabla_t\omega dt\|_{q} }\\
\leq & c\left( \|\alpha_k \|_{L^p}+\varepsilon^2\|\psi_k\|_{L^p}+\|\nabla_t\alpha_k\|_{L^p}\right).
\end{split}
\end{equation*}
The last estimate follows directly from the next identities, from the H\"older's inequality, from $\|\omega\|_{L^q}\leq c \|d_A\omega\|_{L^q}$ and from lemma \ref{flow:lemma:othereq1}. The second estimate of the lemma follows from the identity
\begin{equation*}
\begin{split}
\nabla_s\alpha_i -d_A\phi &+\left( \nabla_s\psi_i-\nabla_t\phi\right) dt=\nabla_s d_A\gamma -d_A\phi +\left( \nabla_s\nabla_t\gamma-\nabla_t\phi\right) dt\\
=& d_{A+\Psi dt}\left(\nabla_s\gamma-\phi\right)+[(\partial_s A-d_A\Phi), \gamma]-[(\partial_t\Phi-\partial_s\Psi-[\Phi,\Psi]),\gamma]\,dt,
\end{split}
\end{equation*}
from the a priori estimate (\ref{flow:aprioriest:geo}) and from $\|\gamma\|_{L^p}\leq c\| d_A\gamma\|_{L^p}=c\| \alpha_i \|_{L^p}$. The third estimate follows directly from the definition of $\mathrm{Rest}^\varepsilon$ and the $L^\infty$-bound (\ref{flow:aprioriest:geo}) for the curvature terms $\partial_tA-d_A\Psi$, $\nabla_t(\partial_tA-d_A\Psi)$.
\end{proof}

\begin{lemma}\label{flow:lemma:othereq1} We choose a regular value $b$ of $E^H$. There is a positive constant $c$ such that for any perturbed geodesic flow $A+\Psi dt+ \Phi ds\in \mathcal M^0(\Xi_-,\Xi_+)$, $\Xi_\pm\in \mathrm {Crit}_{E^H}^b$, 
\begin{equation}
d_A^*d_A\left(\partial_s\Psi-\nabla_t\Phi \right)=2*[B_s\wedge*B_t],
\end{equation}
\begin{equation}
\|\partial_s\Psi-\nabla_t\Phi \|_{L^\infty}\leq c
\end{equation}
hold.
\end{lemma}

\begin{proof}
We define $B_t=\partial_t A-d_A\Psi$ und $B_s=\partial_s A-d_A\Phi$ then
$d_A^*B_t=d_A^* B_s=0$ and therefore
\begin{equation*}
\begin{split}
\nabla_s d_A^*B_t=&*[B_s\wedge*B_t]+d_A^*\nabla_s B_t=0\\
\nabla_t d_A^*B_s=&*[B_t\wedge*B_s]+d_A^*\nabla_t B_s\\
=&-*[B_s\wedge*B_t]+d_A^*\nabla_t B_s=0\\
\end{split}
\end{equation*}
yields to $d_A^*\nabla_t B_s-d_A^*\nabla_s B_t=2*[B_s\wedge*B_t]$ where
\begin{equation*}
d_A^*\nabla_t B_s-d_A^*\nabla_s B_t=d_A^*d_A\left(\nabla_s\Psi-\nabla_t\Phi -[\Phi,\Psi]\right).
\end{equation*}
Finally, we can finish the proof of the lemma, i.e.
\begin{equation*}
d_A^*d_A\left(\partial_s\Psi-\nabla_t\Phi \right)=d_A^*d_A\left(\nabla_s\Psi-\nabla_t\Phi -[\Phi,\Psi]\right)=2*[B_s\wedge*B_t].
\end{equation*}
Furthermore, for a positive constant $c$
$$\|\partial_s\Psi-\nabla_t\Phi \|_{L^\infty}\leq 8 \|B_s\|_{L^\infty}\|B_t\|_{L^\infty}\leq c
$$
by (\ref{flow:aprioriest:geo}).
\end{proof}

\begin{theorem}\label{flow:thm:3eqbasic}
We choose a regular value $b$ of $E^H$, then there are two positive constants $c$ and $\varepsilon_0$ such that the following holds. For any $\Xi=A+\Psi dt+\Phi ds\in \mathcal M^{0}(\Xi_-,\Xi_+)$, $\Xi_\pm\in \mathrm{Crit}_{E^H}^b$, any 1-form $\xi=\alpha+\psi dt+\phi ds \in W^{1,2;p}$ and for $0<\varepsilon<\varepsilon_0$
\begin{equation}
\|\xi \|_{1,2;p,\varepsilon}\leq c\varepsilon^2\left\|\mathcal D^\varepsilon(\Xi)(\xi)\right\|_{0,p,\varepsilon }+c\|\pi_A(\alpha)\|_{L^p},
\end{equation}
\begin{equation}
\begin{split}
\|(1 -\pi_A)(\xi) \|_{1,2;p,\varepsilon}\leq &c\varepsilon^2\left\|\mathcal D^\varepsilon(\Xi)(\xi)\right\|_{0,p,\varepsilon }+c\varepsilon\|\pi_A(\alpha)\|_{L^p}\\
&+c\varepsilon^2\|\nabla_s\pi_A(\alpha)\|_{L^p}+c\varepsilon^2\|\nabla_t\nabla_t\pi_A(\alpha)\|_{L^p},
\end{split}
\end{equation}
\begin{equation}
\begin{split}
\| (1-\pi_A)\alpha \|_{1,2;p,\varepsilon} \leq& c\varepsilon^2\left(\|\mathcal D^\varepsilon(\Xi)\xi\|_{0,p,\varepsilon }+\|\nabla_s\pi_A(\alpha)\|_{L^p}+\|\pi_A(\alpha)\|_{L^p}\right)\\
&+c\varepsilon^2\left( \|\nabla_t\pi_A(\alpha)\|_{L^p}+\|\nabla_t\nabla_t\pi_A(\alpha)\|_{L^p}\right).
\end{split}
\end{equation}
\end{theorem}

In order to prove the last statement we need the next two theorems that will be proven in the subsections \ref{flow:subsection:proofthm1}, \ref{flow:subsection:proofthm2}.

\begin{theorem}\label{flow:thm:linestpar}
We choose a regular value $b$ of $E^H$, then there are two positive constants $c$ and $\varepsilon_0$ such that the following holds. For any $\Xi=A+\Psi dt+\Phi ds\in \mathcal M^{0}(\Xi_-,\Xi_+)$, $\Xi_\pm\in \mathrm{Crit}_{E^H}^b$, any 1-form\footnote{A $i$-form $\gamma$, $i=0,1$, is an element of $W^{j,l,k;p}$, if $j$ derivatives of $\gamma$ in the $\Sigma$-direction, $l$ derivatives in the $S^1$-direction and $k$ derivatives in the $\mathbb R$ direction are in $L^p$.} $\alpha\in W^{2,2,1;p}$, any 0-form $\psi\in W^{2,2,1;p}$ and for $0<\varepsilon<\varepsilon_0$
\begin{equation} \label{flow:thm:linestpateq1}
\begin{split}
\|\alpha\|_{L^p}&+\|d_A\alpha\|_{L^p}+\|d_A^*\alpha\|_{L^p}+\|d_A^*d_A\alpha\|_{L^p}+\|d_Ad_A^*\alpha\|_{L^p}\\
&+\varepsilon \|\nabla_t \alpha\|_{L^p}+\varepsilon^2\|\nabla_t\nabla_t \alpha\|_{L^p}+\varepsilon\|d_A\nabla_t \alpha\|_{L^p}+\varepsilon\|\nabla_t d_A\alpha\|_{L^p}\\
&+\varepsilon\|d_A^*\nabla_t \alpha\|_{L^p}+\varepsilon\|\nabla_t d_A^*\alpha\|_{L^p}+\varepsilon^2\|\nabla_s\alpha\|_{L^p}\\
\leq& c\left\|\left(\varepsilon^2\nabla_s-\varepsilon^2\nabla_t^2+\Delta_A\right)\alpha\right\|_{L^p}+c\|\pi_A(\alpha)\|_{L^p},
\end{split}
\end{equation}
\begin{equation}
\begin{split}
\|\psi \|_{L^p}&+\|d_A\psi \|_{L^p}+\|d_A^*d_A\psi\|_{L^p}+\varepsilon \|\nabla_t \psi\|_{L^p}+\varepsilon^2\|\nabla_t\nabla_t \psi\|_{L^p}\\
&+\varepsilon\|d_A\nabla_t \psi\|_{L^p}+\varepsilon\|\nabla_t d_A\psi\|_{L^p}+\varepsilon^2\|\nabla_s\alpha\|_{L^p}\\
\leq&c\left\|\left(\varepsilon^2\nabla_s-\varepsilon^2\nabla_t^2+\Delta_A\right)\psi\right\|_{L^p}.
\end{split}
\end{equation}
\end{theorem}

\begin{theorem}\label{flow:thm:wavefin}
We choose a regular value $b$ of $E^H$, then there are two positive constants $c$ and $\varepsilon_0$ such that the following holds. For any $\Xi=A+\Psi dt+\Phi ds\in \mathcal M^{0}(\Xi_-,\Xi_+)$, $\Xi_\pm\in \mathrm{Crit}_{E^H}^b$, any 1-form $\alpha+\psi dt\in W^{1,1,1;p}\cap \textrm{im } d_{A+\Psi dt }$, any 0-form $\phi\in W^{1,1,1;p}$ and any $0<\varepsilon<\varepsilon_0$
\begin{equation}
\begin{split}
\|\alpha\|_{L^p}&+\|d_A^*\alpha\|_{L^p}+\varepsilon^2\|\nabla_s\alpha\|_{L^p}+\varepsilon\|\nabla_t\alpha\|_{L^p}+\varepsilon \|\psi\|_{L^p}+\varepsilon^2\|\nabla_t\psi\|_{L^p}\\
&+\varepsilon^3\|\nabla_s\psi\|_{L^p}+\varepsilon^2 \|\phi \|_{L^p}+\varepsilon^2 \|d_A\phi\|_{L^p}+\varepsilon^3 \|\nabla_t\phi\|_{L^p}+\varepsilon^4 \|\nabla_s\phi \|_{L^p}\\
\leq &c\varepsilon^2\left\|\nabla_s\alpha-d_A\phi\right\|_{L^p}+c\varepsilon^3\left\|\nabla_s\psi-\nabla_t\phi\right\|_{L^p}\\
&+c\varepsilon^4\left\|\nabla_s\phi-\frac1{\varepsilon^4}d_A^*\alpha+\frac1{\varepsilon^2}\nabla_t\psi\right\|_{L^p}.
\end{split}
\end{equation}
\end{theorem}

\begin{proof}[Proof of theorem \ref{flow:thm:3eqbasic}]
By theorem \ref{flow:thm:linestpar} and the lemma \ref{flow:lemma:kerim2}
\begin{equation*}
\begin{split}
\|\alpha_k+\psi_k dt\|_{1,2;p,\varepsilon}\leq& c \|\varepsilon^2\nabla_s\alpha_k-\varepsilon^2\nabla_t^2\alpha_k+\Delta_A\alpha_k\|_{L^p}+c\|\pi_A(\alpha)\|_{L^p}\\
&+c\varepsilon \|\varepsilon^2\nabla_s\psi_k-\varepsilon^2\nabla_t^2\psi_k+d_A^*d_A\psi_k \|_{L^p}\\
\leq&c \varepsilon^2 \left |\left |\left(1-\Pi_{\textrm{im } d_{A+\Psi dt} }\right)\mathcal D^\varepsilon(\Xi)(\alpha+\psi dt+\phi ds)\right |\right |_{0,p,\varepsilon}\\
&+c\|\pi_A(\alpha)\|_{L^p}+c \varepsilon\|\alpha\|_{L^p}+c\varepsilon^2\|\psi_k\|_{L^p}+c\varepsilon^2\|\nabla_t\alpha_k\|_{L^p}
\end{split}
\end{equation*}
and by theorem \ref{flow:thm:wavefin} and the lemma \ref{flow:lemma:kerim2}
\begin{equation}\label{flow:eq:sdlsd}
\begin{split}
\|\alpha_i+\psi_i dt\|_{1,2;p,\varepsilon}
\leq&c \varepsilon^2 \left |\left |\Pi_{\textrm{im } d_{A+\Psi dt} }\mathcal D^\varepsilon(\Xi)(\alpha+\psi dt+\phi ds) \right |\right |_{0,p,\varepsilon}\\
&+c\varepsilon^2\|\mathcal D^\varepsilon_3(\Xi)(\alpha+\psi dt+\phi ds)\|_{0,p,\varepsilon }\\
&+c \varepsilon \|\alpha\|_{L^p}+c\varepsilon^2\|\psi_k\|_{L^p}+c\varepsilon^2\|\nabla_t\alpha_k\|_{L^p};
\end{split}
\end{equation}
for $\varepsilon$ small enough we can conclude therefore that
\begin{equation*}
\|\xi \|_{1,2;p,\varepsilon}\leq c\varepsilon^2\left\|\mathcal D^\varepsilon(\Xi)(\xi)\right\|_{0,p,\varepsilon }+c\|\pi_A(\alpha)\|_{L^p}.
\end{equation*}
The second estimate of the theorem follows from (\ref{flow:eq:sdlsd}) and
\begin{equation*}
\begin{split}
\|(1-\pi_A)\alpha_k&+\psi_k dt\|_{1,2;p,\varepsilon}\\
\leq& c \|\varepsilon^2\nabla_s(1-\pi_A)\alpha_k-\varepsilon^2\nabla_t^2(1-\pi_A)\alpha_k+\Delta_A(1-\pi_A)\alpha_k\|_{L^p}\\
&+c\varepsilon \|\varepsilon^2\nabla_s\psi_k-\varepsilon^2\nabla_t^2\psi_k+d_A^*d_A\psi_k \|_{L^p}\\
\leq&c \varepsilon^2 \left |\left | \left(1-\Pi_{\textrm{im } d_{A+\Psi dt}}\right) \mathcal D^\varepsilon(\Xi)(\alpha+\psi dt+\phi ds) \right |\right |_{0,2,\varepsilon}+c \varepsilon\|\alpha\|_{L^p}\\
&+c\varepsilon^2\|\psi_k\|_{L^p}+c\varepsilon^2\|\nabla_t\alpha_k\|_{L^p}+c\varepsilon^2\|(\nabla_s-\nabla_t\nabla_t)\pi_A(\xi)\|_{L^p}.
\end{split}
\end{equation*}
In order to prove the third estimate we need the following one. We choose $q$ such that $\frac 1p+\frac 1{q}=1$, then
\begin{equation*}
\begin{split}
\Big|\Big|\Pi_{\textrm{im } d_{A+\Psi dt} }d_A^*\nabla_t&\pi_A(\alpha) dt\Big|\Big|_{0,p,\varepsilon }
\leq \sup_{\gamma\in \Omega^0} \frac {\varepsilon^2\langle d_A^*\nabla_t\pi_A(\alpha), \nabla_t\gamma\rangle }{\|d_A\gamma+\nabla_t\gamma dt \|_{L^{q}} }\\
= & \sup_{\gamma\in \Omega^0} \frac {\varepsilon^2\langle \nabla_t\pi_A(\alpha), d_A\nabla_t\gamma\rangle }{\|d_A\gamma+\nabla_t\gamma dt \|_{L^{q}} }\\
= & \sup_{\gamma\in \Omega^0} \frac {\varepsilon^2\langle \nabla_t\pi_A(\alpha),\nabla_t d_A\gamma\rangle-\varepsilon^2\langle \nabla_t\pi_A(\alpha),[(\partial_t A-d_A\Psi),\gamma]\rangle }{\|d_A\gamma+\nabla_t\gamma dt \|_{L^{q}} }\\
\leq & \sup_{\gamma\in \Omega^0} \frac {\varepsilon^2\| \nabla_t\nabla_t\pi_A(\alpha)\|_{L^p}\|d_A\gamma\|_{L^{q}}+c \varepsilon^2\| \nabla_t\pi_A(\alpha)\|_{L^p}\|\gamma\|_{L^{q}} }{\|d_A\gamma+\nabla_t\gamma dt \|_{L^{q}} }\\
\leq & \varepsilon^2\| \nabla_t\nabla_t\pi_A(\alpha)\|_{L^p} +c \varepsilon^2\| \nabla_t\pi_A(\alpha)\|_{L^p}.
\end{split}
\end{equation*}
where the last inequality follows from the estimate $\|\gamma\|_{L^q}\leq c\|d_A\gamma\|_{L^q}$.
Thus, since by (\ref{flow:aprioriest:geo})
\begin{align*}
\varepsilon^2\left\|\Pi_{\textrm{im } d_{A+\Psi dt}}\mathrm{Rest}^\varepsilon(\Xi)(\xi+\phi ds)\right\|_{0,p,\varepsilon}
\leq& c\varepsilon^2\|\alpha\|_{L^p}+c\varepsilon^2\|\psi\|_{L^p}+c\varepsilon\|(1-\pi_A)\alpha\|_{L^p}\\
&+c\left|\left|\Pi_{\textrm{im } d_{A+\Psi dt} }d_A^*\nabla_t\pi_A(\alpha) dt\right|\right|_{0,p,\varepsilon },
\end{align*}

\begin{equation*}
\begin{split}
\|\alpha_i+\psi_i dt \|_{1,2;p,\varepsilon }\leq &c\varepsilon^2\left\|\Pi_{\textrm{im } d_{A+\Psi dt} }\mathcal D^\varepsilon(\xi+\phi ds) \right\|_{0,2,\varepsilon}\\
&+c\varepsilon^2\left\|\mathcal D^\varepsilon_3(\xi+\phi ds) \right\|_{0,2,\varepsilon}+c\varepsilon^2\|\psi\|_{L^p}+c\varepsilon^2\|\alpha \|_{L^p}\\
&+c\varepsilon^2\|\nabla_t\pi_A(\alpha)\|_{L^p}+c\varepsilon^2\|\nabla_t\nabla_t\pi_A(\alpha)\|_{L^p},
\end{split}
\end{equation*}
\begin{equation*}
\begin{split}
\|\alpha_k-\pi_A(\alpha)\|_{1,2;p,\varepsilon }\leq &c\varepsilon^2\left\|\left(1-\Pi_{\textrm{im } d_{A+\Psi dt}}\right)\mathcal D^\varepsilon_1((1-\pi_A)\xi+\phi ds) \right\|_{L^p}\\
\leq &c\varepsilon^2\left(\left\|\mathcal D^\varepsilon_1(\xi+\phi ds) \right\|_{L^p} +\|\psi\|_{L^p}+\|\alpha \|_{L^p}\right)\\
&+c\varepsilon^2\left( \|\nabla_s\pi_A(\alpha)\|_{L^p}+\|\nabla_t\pi_A(\alpha)\|_{L^p}+\|\nabla_t\nabla_t\pi_A(\alpha)\|_{L^p}\right)
\end{split}
\end{equation*}
and finally we can conclude that
\begin{equation*}
\begin{split}
\| (1-\pi_A)\alpha \|_{1,2;p,\varepsilon} \leq& c\varepsilon^2\left(\|\mathcal D^\varepsilon(\xi)\|_{0,p,\varepsilon }+\|\nabla_s\pi_A(\alpha)\|_{L^p}+\|\pi_A(\alpha)\|_{L^p}\right)\\
&+c\varepsilon^2\left( \|\nabla_t\pi_A(\alpha)\|_{L^p}+\|\nabla_t\nabla_t\pi_A(\alpha)\|_{L^p}\right).
\end{split}
\end{equation*}
\end{proof}

The next goal is to improve the theorem \ref{flow:thm:3eqbasic} in the sense that we want to estimate the norms using only the operator $\mathcal D^\varepsilon(\Xi)$ (theorem \ref{flow:thme:linest333}) and in order to do this we need to use the properties of the geodesic flow (lemma \ref{flow:lemma:saweb2}). We define
\begin{equation*}
\omega(A):=d_A\left(d_A^*d_A\right)^{-1}(\nabla_t(\partial_t A-d_A\Psi)+*X_t(A)).
\end{equation*}

\begin{lemma}\label{flow:lemma:diffD}
We choose a regular value $b$ of $E^H$, then there are two positive constants $c$ and $\varepsilon_0$ such that the following holds. For any $\Xi=A+\Psi dt+\Phi ds\in \mathcal M^{0}(\Xi_-,\Xi_+)$, $\Xi_\pm\in \mathrm{Crit}_{E^H}^b$, any 1-form $\xi =\alpha+\psi dz+\phi ds\in W^{1,2;p}$ and for $0<\varepsilon<\varepsilon_0$
\begin{equation}
\begin{split}
\| \pi_A(\mathcal D^\varepsilon(\Xi)&(\xi)+*[\alpha,*\omega(A)])-\mathcal D^0(\Xi)\pi_A(\xi)\|_{L^p}\\
\leq& c \|(1-\pi_A)\alpha+\psi dt \|_{L^p}+c\|\nabla_t(1-\pi_A)\alpha\|_{L^p}+\varepsilon^2 \|\psi\|_{L^p}\\
&+c\varepsilon^2 \|\nabla_t\nabla_t\pi_A(\alpha) \|_{L^p}+c\varepsilon^2 \|\nabla_t\pi_A(\alpha)\|_{L^p}+c\varepsilon^2 \|\pi_A(\alpha)\|_{L^p}\\
&+c\varepsilon^2\|\nabla_s\pi_A(\alpha) \|_{L^p}+c\varepsilon^2 \|\mathcal D^\varepsilon_2(\Xi)(\xi) \|_{L^p},
\end{split}
\end{equation}
\begin{equation}
\begin{split}
\| \pi_A((\mathcal D^\varepsilon(\Xi)^*&(\xi)+*[\alpha,*\omega(A)])-(\mathcal D^0(\Xi))^*\pi_A(\xi)\|_{L^p}\\
\leq& c \|(1-\pi_A)\alpha+\psi dt \|_{L^p}+c\|\nabla_t(1-\pi_A)\alpha\|_{L^p}+\varepsilon^2 \|\psi\|_{L^p}\\
&+c\varepsilon^2 \|\nabla_t\nabla_t\pi_A(\alpha) \|_{L^p}+c\varepsilon^2 \|\nabla_t\pi_A(\alpha)\|_{L^p}+c\varepsilon^2 \|\pi_A(\alpha)\|_{L^p}\\
&+c\varepsilon^2\|\nabla_s\pi_A(\alpha) \|_{L^p}+c\varepsilon^2 \|(\mathcal D^\varepsilon_2(\Xi))^*(\xi) \|_{L^p}.
\end{split}
\end{equation}

\end{lemma}
\begin{proof}
By definition we have that
\begin{equation*}
\begin{split}
\pi_A\mathcal D^\varepsilon(\Xi)(\xi)
:=& \pi_A\left(\nabla_s\alpha-\nabla_t\nabla_t\alpha-2[\psi, (\partial_t A-d_A\Psi)]- d*X_t(A)\alpha\right),\\
\mathcal D^0(\Xi)(\pi_A(\xi))
:= &\pi_A\big(\nabla_s\pi_A(\alpha)-2[\psi_0, (\partial_t A-d_A\Psi)]\\
&\qquad-\nabla_t\nabla_t\pi_A(\alpha)- d*X_t(A)\pi_A(\alpha)+*[\pi_A(\alpha)\wedge*\omega(A)]\big)
\end{split}
\end{equation*}
where $d_A^*d_A\psi_0=-2*[\pi_A(\alpha)\wedge*(\partial_t A-d_A\Psi)]$; therefore
\begin{equation*}
\begin{split}
\| \pi_A(\mathcal D^\varepsilon(\Xi)(\xi)&+*[\alpha\wedge*\omega])-\mathcal D^0(\Xi)\pi_A(\xi)\|_{L^p}\\
\leq&\|\pi_A\left(\nabla_s(1-\pi_A)\alpha-\nabla_t\nabla_t(1-\pi_A)\alpha\right) \|_{L^p}+c\|(1-\pi_A)\alpha\|_{L^p}\\
&+\|\pi_A\left(2[(\psi-\psi_0), (\partial_t A-d_A\Psi)]\right)\|_{L^p}\\
\leq& c \|(1-\pi_A)\alpha\|_{L^p}+c\|\nabla_t(1-\pi_A)\alpha\|_{L^p}+c\|\psi-\psi_0 \|_{L^p}
\end{split}
\end{equation*}
where we used the commutation formula and the uniform $L^\infty$ bound of the curvatures in order to drop the derivative $\nabla_s$ and a derivative $\nabla_t$. Next, we split the 1-form $\alpha+\psi dt=(\alpha_k+\psi_k dt)+(\alpha_i+\psi_i dt)$ in the same way as (\ref{flow:cond:split}). We can easily remark that
\begin{equation*}
\|\alpha_i+\psi_i dt\|_{L^p}+\|\alpha_k+\psi_k dt\|_{L^p}\leq  2 \|\alpha+\psi dt\|_{L^p},
\end{equation*}
\begin{equation*}
\|\alpha_i+\psi_i dt\|_{L^p}\leq  \|(1-\pi_A)\alpha+\psi dt\|_{L^p}.
\end{equation*}
Furthermore, since $\|\psi_i\|_{L^p}\leq c \| d_A\psi_i\|_{L^p}$, by the commutation formula
\begin{align*}
\|\psi_i\|_{L^p}\leq &\|(1-\pi_A)\alpha\|_{L^p}
+\|\nabla_t\alpha_i\|_{L^p}\\
\leq&\|(1-\pi_A)\alpha\|_{L^p}
+\|\nabla_t(1-\pi_A)\alpha\|_{L^p}+\|d_{A}^*\nabla_t(1-\pi_A)\alpha_k\|_{L^p}\\
\intertext{$\alpha_k+\psi_k dt$ lies in the kernel of $d_{A+\Psi ds}$, thus}
\leq&\|(1-\pi_A)\alpha\|_{L^p}
+\|\nabla_t(1-\pi_A)\alpha\|_{L^p}+\varepsilon^2\|\nabla_t\nabla_t\psi_k\|_{L^p}\\
\leq&\|(1-\pi_A)\alpha\|_{L^p}
+\|\nabla_t(1-\pi_A)\alpha\|_{L^p}\\
&+\varepsilon^2\|\nabla_t\nabla_t(\psi_k-\psi_0)\|_{L^p}
+\varepsilon^2 \|\nabla_t\nabla_t\psi_0\|_{L^p}\\
\intertext{by the definition of $\psi_0$}
\leq&\|(1-\pi_A)\alpha\|_{L^p}
+\|\nabla_t(1-\pi_A)\alpha\|_{L^p}+\varepsilon^2\|\nabla_t\nabla_t(\psi_k-\psi_0)\|_{L^p}\\
&+\varepsilon^2 \|\nabla_t\nabla_t\pi_A\alpha\|_{L^p}+\varepsilon^2 \|\pi_A(\alpha)\|_{L^p}+\varepsilon^2\|\nabla_t\pi_A(\alpha)\|_{L^p}.
\end{align*}
Moreover, the theorem \ref{flow:thm:linestpar} yields to
\begin{align*}
\|\psi_k-\psi_0\|_{L^p}&+\varepsilon^2\|\nabla_t\nabla_t(\psi_k-\psi_0)\|_{L^p}
\\
\leq &c  \|(\varepsilon^2\nabla_s-\varepsilon^2\nabla_t\nabla_t+d_A^*d_A)(\psi_k-\psi_0)\|_{L^p}\\
\intertext{by the definition of $\mathcal D_2^\varepsilon$ and by lemma \ref{flow:lemma:kerim2}}
\leq& c \varepsilon^2 \|\mathcal D_2^\varepsilon(\Xi)(\alpha+\psi dt) \|_{L^p}+c\varepsilon^2 \|\nabla_t\nabla_t\psi_0 \|_{L^p}\\
&+c\|(1-\pi_A)\alpha\|_{L^p}+c\varepsilon^2\|\nabla_s\psi_0\|_{L^p}+c\varepsilon^2\|\alpha\|_{L^p}\\
&+c\varepsilon^2\|\psi_k\|_{L^p}+c\varepsilon^2\|\nabla_t\alpha\|_{L^p}+c\varepsilon^2\|\nabla_t\nabla_t\pi_A(\alpha)\|_{L^p}\\
\leq& c \varepsilon^2 \|\mathcal D_2^\varepsilon(\Xi)(\alpha+\psi dt) \|_{L^p}+c\varepsilon^2 \|\nabla_t\nabla_t\pi_A(\alpha) \|_{L^p}\\
&+c\|(1-\pi_A)\alpha\|_{L^p}+c\varepsilon^2 \|\nabla_t\pi_A(\alpha)\|_{L^p}+c\varepsilon^2 \|\pi_A(\alpha)\|_{L^p}\\
&+c\varepsilon^2\|\nabla_s\pi_A(\alpha) \|_{L^p}+c\varepsilon^2\|\psi_k\|_{L^p},
\end{align*}
where the last two steps follows from the definition of $\psi_0$. Thus, by the last estimates we can conclude that
\begin{equation*}
\begin{split}
\| \pi_A(\mathcal D^\varepsilon(\Xi)(\xi)&+*[\alpha\wedge*\omega])-\mathcal D^0(\Xi)\pi_A(\xi)\|_{L^p}\\
\leq& c \|(1-\pi_A)\alpha \|_{L^p}+c\|\nabla_t(1-\pi_A)\alpha\|_{L^p}+c\varepsilon^2 \|\nabla_t\nabla_t\pi_A(\alpha) \|_{L^p}\\
&+c\varepsilon^2 \|\nabla_t\pi_A(\alpha)\|_{L^p}+c\varepsilon^2 \|\pi_A(\alpha)\|_{L^p}+c\varepsilon^2\|\nabla_s\pi_A(\alpha) \|_{L^p}\\
&+c\varepsilon^2 \|\mathcal D_2^\varepsilon(\Xi)(\alpha+\psi dt) \|_{L^p}+\varepsilon^2\|\psi\|_{L^p}
\end{split}
\end{equation*}
The second estimate of the lemma follows exactly in the same way.
\end{proof}

The following lemma was proved by Salamon and Weber in \cite{MR2276534} (lemma D.7).

\begin{lemma} \label{flow:lemma:saweb2}
Assume $E^H$ is Morse-Smale and let $\Xi=A+\Psi dt+\Phi ds \in\mathcal M^0(\Xi_-,\Xi_+)$, $\Xi_\pm\in \mathrm{Crit}_{E^H}^b$. Then, for every $p>1$, there is 
a constant $c>0$ such that
\begin{equation}
\|\alpha \|_{L^p}+\|\nabla_s\alpha \|_{L^p}+\|\nabla_t\nabla_t\alpha\|\leq c \|\mathcal D^0(\Xi)^*(\alpha)\|_{L^p},
\end{equation}
\begin{equation}
\|\alpha \|_{L^p}+\|\nabla_s\alpha \|_{L^p}+\|\nabla_t\nabla_t\alpha\|\leq c \left( \|\alpha-(\mathcal D^0(\Xi))^*(\eta)\|_{L^p}+\|\mathcal D^0(\Xi)(\alpha)\|_{L^p}\right)
\end{equation}
for all compactly supported vector fields $\alpha,\eta\in H_{A}^1$.
\end{lemma}

\begin{theorem}\label{flow:thme:linest333}
We choose a regular value $b$ of $E^H$, then there are two positive constants $c$ and $\varepsilon_0$ such that the following holds. For any $\Xi=A+\Psi dt+\Phi ds\in \mathcal M^{0}(\Xi_-,\Xi_+)$, $\Xi_\pm\in \mathrm{Crit}_{E^H}^b$, and any $0<\varepsilon<\varepsilon_0$ the following estimates hold.
\begin{equation}
\begin{split}
\| \pi_A(\alpha) \|_{1,2;p,1} \leq &c\varepsilon \|\mathcal D^\varepsilon(\Xi)(\xi)\|_{0,p,\varepsilon}+c\|\pi_A(\mathcal D^\varepsilon(\Xi)(\xi)+*[\alpha,*\omega(A)])\|_{L^p}\\
&+c\|\pi_A(\xi-(\mathcal D^0(\Xi))^*(\pi_A(\eta)))\|_{L^p},
\end{split}
\end{equation}
\begin{equation}
\begin{split}
\|(1-\pi_A)\xi\|_{1,2;p,\varepsilon }\leq &c\varepsilon^2\|\mathcal D^\varepsilon(\Xi)(\xi)\|_{0,p,\varepsilon}+c\varepsilon \|\pi_A(\mathcal D^\varepsilon(\Xi)(\xi)+*[\alpha,*\omega(A)])\|_{L^p}\\
&+c\varepsilon \|\pi_A(\xi-(\mathcal D^0(\Xi))^*(\pi_A(\eta)))\|_{L^p},
\end{split}
\end{equation}
\begin{equation}
\begin{split}
\|(1-\pi_A)\alpha \|_{1,2;p,\varepsilon }\leq &c\varepsilon^2\|\mathcal D^\varepsilon(\Xi)(\xi)\|_{0,p,\varepsilon}+c\varepsilon^2 \|\pi_A(\xi-(\mathcal D^0(\Xi))^*(\pi_A(\eta)))\|_{L^p}
\end{split}
\end{equation}
for all compactly supported $1$-forms $\xi,\eta \in W^{1,2;p}$ and $0<\varepsilon\leq \varepsilon_0$. 
\end{theorem}

\begin{proof}
By theorem \ref{flow:thm:3eqbasic} and by lemma \ref{flow:lemma:saweb2} we have that
\begin{align*}
\|\xi\|_{1,2;p,\varepsilon}+\|\nabla_s\pi_A(\xi)\|_{L^p}&+\|\nabla_t\pi_A(\xi)\|_{L^p}+\|\nabla_t\nabla_t\pi_A(\alpha)\|_{L^p}\leq c\varepsilon\|\mathcal D^\varepsilon(\Xi)(\xi)\|_{0,p,\varepsilon }\\
&+c\|\pi_A(\alpha-\mathcal D^0(\Xi)^*(\eta))\|_{L^p}+c\|\mathcal D^0(\Xi)(\pi_A(\alpha))\|_{L^p}
\end{align*}
and thus always by lemma \ref{flow:lemma:saweb2} and by the last estimate
\begin{align*}
\|(1-\pi_A)(\xi)\|_{1,2;p,\varepsilon }\leq& c\varepsilon^2 \|\mathcal D^\varepsilon(\Xi)(\xi)\|_{0,p,\varepsilon}+c\varepsilon\|\pi_A(\alpha-\mathcal D^0(\Xi)^*(\eta))\|_{L^p}\\
&+c\varepsilon \|\mathcal D^0(\Xi)(\pi_A(\alpha))\|_{L^p},\\
\|(1-\pi_A)\alpha \|_{1,2;p,\varepsilon }\leq& c\varepsilon^2  \|\mathcal D^\varepsilon(\Xi)(\xi)\|_{0,p,\varepsilon}+c\varepsilon^2 \|\pi_A(\alpha-\mathcal D^0(\Xi)^*(\eta))\|_{L^p}\\
&+c\varepsilon^2 \|\mathcal D^0(\Xi)(\pi_A(\alpha))\|_{L^p}.
\end{align*}
Therefore, with the lemma \ref{flow:lemma:diffD} we obtain
\begin{align*}
\| \xi \|_{1,2;p,\varepsilon}&+\|\nabla_s\pi_A(\xi)\|_{L^p}+\|\nabla_t\nabla_t\pi_A(\alpha)\|_{L^p} \\
\leq &c \varepsilon \|\mathcal D^\varepsilon(\Xi)(\xi)\|_{0,p,\varepsilon}+c\|\pi_A(\mathcal D^\varepsilon(\Xi)(\xi)+*[\alpha,*\omega])\|_{L^p}\\
&+c\|\pi_A(\xi-\mathcal (D^\varepsilon(\Xi))^*(\eta)-*[\eta,*\omega])\|_{L^p}\\
&+\frac c\varepsilon \|(1-\pi_A)\xi \|_{1,2;p,\varepsilon },\\
\|(1-\pi_A)\xi\|_{1,2;p,\varepsilon }\leq &c\varepsilon^2\|\mathcal D^\varepsilon(\Xi)(\xi)\|_{0,p,\varepsilon}+c\varepsilon \|\pi_A(\mathcal D^\varepsilon(\Xi)(\xi)+*[\alpha,*\omega])\|_{L^p}\\
&+c\varepsilon \|\pi_A(\xi-\mathcal (D^\varepsilon(\Xi))^*(\eta)-*[\eta,*\omega])\|_{L^p}\\
&+c\|(1-\pi_A)\alpha \|_{1,2;p,\varepsilon },\\
\|(1-\pi_A)\alpha \|_{1,2;p,\varepsilon }\leq &c\varepsilon^2\|\mathcal D^\varepsilon(\Xi)(\xi)\|_{0,p,\varepsilon}+c\varepsilon^2 \|\pi_A(\mathcal D^\varepsilon(\Xi)(\xi)+*[\alpha,*\omega])\|_{L^p}\\
&+c\varepsilon ^2\|\pi_A(\xi-\mathcal (D^\varepsilon(\Xi))^*(\eta)-*[\eta,*\omega])\|_{L^p}
\end{align*}
and thus the theorem follows combining these last three estimates.

\end{proof}

In the same way as theorem \ref{flow:thme:linest333} one can prove the following theorem.
\begin{theorem}\label{flow:thme:linest333tre}
We choose a regular value $b$ of $E^H$, then there are two positive constants $c$ and $\varepsilon_0$ such that the following holds. For any $\Xi=A+\Psi dt+\Phi ds\in \mathcal M^{0}(\Xi_-,\Xi_+)$, $\Xi_\pm\in \mathrm{Crit}_{E^H}^b$, and any $0<\varepsilon<\varepsilon_0$ the following estimates hold.
\begin{equation}
\begin{split}
\|(1-\pi_A)\xi\|_{1,2;p,\varepsilon }&+\varepsilon\| \pi_A(\alpha) \|_{1,2;p,1}\\
\leq &c\varepsilon^2\|(\mathcal D^\varepsilon(\Xi))^*(\xi)\|_{0,p,\varepsilon}+c\varepsilon \|\pi_A((\mathcal D^\varepsilon(\Xi))^*(\xi)+*[\alpha,*\omega(A)])\|_{L^p},
\end{split}
\end{equation}
\begin{equation}
\begin{split}
\|(1-\pi_A)\alpha \|_{1,2;p,\varepsilon }\leq &c\varepsilon^2\|(\mathcal D^\varepsilon(\Xi))^*(\xi)\|_{0,p,\varepsilon}
\end{split}
\end{equation}
for all compactly supported $1$-forms $\xi \in W^{1,2;p}$ and $0<\varepsilon\leq \varepsilon_0$. 
\end{theorem}

%
%
\subsection{Proof of the theorem \ref{flow:thm:linestpar} }\label{flow:subsection:proofthm1}

We will use the following criterion to prove our estimates (cf. \cite{MR2030823}, theorem C.2).

\begin{theorem}[Marcinkiewicz, Mihlin]\label{flow:thm:MM}
Let $m:\mathbb R^n\to \mathbb C$ be a measurable function that for some constant $c_0$ satisfies
\begin{equation}\label{flow:thmMM:eq1}
\left|x_{i_1}x_{i_2}... x_{i_s}\frac {\partial^sm }{\partial_{x_{i_1}}\partial_{x_{i_2}}...\partial_{x_{i_s}} }\right|\leq c_0
\end{equation}
for all integers $0\leq s\leq n$ and $1\leq i_1<i_2<...<i_s\leq n$. We define $T_m:L^2(\mathbb R^2)\to L^2(\mathbb R^2)$ by
\begin{equation*}
T_mf:=\mathcal F^{-1} (m\mathcal F(f))
\end{equation*} 
where $f\in L^2(\mathbb R^2)$ and $\mathcal F: L^2(\mathbb R^n,\mathbb C)\to L^2(\mathbb R^n,\mathbb C)$ is the Fourier transformation given by
$$(\mathcal F f)(y_1,...,y_{n}):= \frac 1{2\pi}\int _{-\infty}^\infty\dots \int_{-\infty}^\infty e^{-i(y_0x_0+\dots+y_nx_n)}f(x_1,\dots, x_n) dx_1\dots dx_n$$
for $f\in L^2(\mathbb R^n,\mathbb C)\cap L^1(\mathbb R^n,\mathbb C)$. Then $m$ is an $L^p$-multiplier for all $1<p<\infty$, i.e. there exists a constant $c$ such that whenever $f\in L^p(\mathbb R^n)\cap L^2(\mathbb R^2)$ then $T_mf\in L^p(\mathbb R^n)$ and 
\begin{equation}
\|T_mf\|_{L^p}\leq c \|f\|_{L^p}.
\end{equation}
\end{theorem}

\begin{corollary}\label{flow:cor:MM1}
For every $p>1$ there is a positive constant $c$ such that
\begin{equation*}
\left\| \partial_s u\right\|_{L^p}+\sum_{i,j=0}^{n-1}\left\|\partial_{x_i}\partial_{x_j} u\right\|_{L^p}
\leq c \left\| \left(\partial_s-\sum_{i=0}^{n-1}\partial_{x_i}\partial_{x_i}\right)u\right\|_{L^p}
\end{equation*}
for every $u\in W_0^{1,2,p}(\mathbb R\times \mathbb R^n)\cap W_0^{1,2,2}(\mathbb R\times \mathbb R^n)$.
\end{corollary}

\begin{proof}
We define $f\in L^p(\mathbb R\times \mathbb R^n)\cap L^2(\mathbb R\times \mathbb R^n)$ by $f=\left(\partial_s-\sum_{i=0}^{n-1}\partial_{x_i}\partial_{x_i}\right)u$ and thus
$$\mathcal F (f)= \left( i\sigma+ \sum_{i=0}^{n-1}y_i^2\right) \mathcal F(u)$$
and therefore
$$\mathcal F(\partial_su)=\frac {i\sigma }{i\sigma+ \sum_{i=0}^{n-1}y_i^2 }\mathcal F (f)=:m_{s}(\sigma,y_0,...,y_{n-1})\mathcal F (f),$$
$$\mathcal F(\partial_{x_i}\partial_{x_j}u)=\frac {y_i y_j }{i\sigma+ \sum_{i=0}^{n-1}y_i^2 }\mathcal F (f)=:m_{y_i}(\sigma,y_0,...,y_{n-1})\mathcal F (f).$$
The multipliers $m_{s}(\sigma,y_0,...,y_{n-1})$ and $m_{y_i}(\sigma,y_0,...,y_{n-1})$ satisfy the condition $(\ref{flow:thmMM:eq1})$ and therefore we can apply the theorem \ref{flow:thm:MM} and conclude the proof.
\end{proof}

We denote $d_Ad_A^*+d_A^*d_A$ by $\Delta_A$.

\begin{lemma}\label{flow:fjdfjdkjdkj}
We choose a connection $A_0\in\mathcal A_0(P)$, then there is a positive constant $c$ such that for any 0- or 1-form $\alpha$ with compact support:
\begin{equation}\label{fdssgasfaaaa}
\begin{split}
\|\partial_s\alpha&\|_{L^p( \Sigma\times\mathbb R^2)}
+\|\partial_t^2\alpha\|_{L^p(\Sigma\times\mathbb R^2)} 
+\| d_{A_0}d_{A_0}^*\alpha\|_{L^p(\Sigma\times\mathbb R^2)}\\
&+\| d_{A_0}^*d_{A_0}\alpha\|_{L^p(\Sigma\times\mathbb R^2)}
+\| \partial_t d_{A_0}^*\alpha\|_{L^p(\Sigma\times\mathbb R^2)}+\| \partial_t d_{A_0}\alpha\|_{L^p(\Sigma\times\mathbb R^2)}\\
\leq &c\left\|\left(\partial_s-\partial_t^2+\Delta_{A_0}\right)\alpha\right\|_{L^p(\Sigma\times\mathbb R^2)}+c \|\alpha\|_{L^p(\Sigma\times\mathbb R^2)}\\
&+c \|\partial_t\alpha\|_{L^p(\Sigma\times\mathbb R^2)}+c \|d_{A_0}\alpha\|_{L^p(\Sigma\times\mathbb R^2)}+c \|d_{A_0}^*\alpha\|_{L^p(\Sigma\times\mathbb R^2)}.
\end{split}
\end{equation}
\end{lemma}

\begin{proof}
The previous corollary continues to holds if we consider a metric closed to a constant metric. Therefore we can pick a finite atlas $\{\mathbb R^2\times V_i,\varphi_i:\mathbb R^2\times V_i\to \mathbb R^2\times\Sigma\}_{i\in I}$ and a partition of the unity $\{\rho_i\}_{i\in I}\subset C^\infty(\mathbb R^2\times \Sigma,[0,1])$, $\sum_{i\in I}\rho_i(x)=1$ for every $x\in \mathbb R^2\times\Sigma$ and $\mathrm{supp}(\rho_i)\subset \varphi_i(\mathbb R^2\times V_i)$ for any $i\in I$, and apply the corollary \ref{flow:cor:MM1} for $(\rho_i\circ\varphi_i)\alpha_i$ where $\alpha_i$ is the local representations of $\alpha$ on $\mathbb R^2\times V_i$. Summing all the estimates and considering the smooth constant connection $A_0(t,s)$ we obtain (\ref{fdssgasfaaaa}).
\end{proof}

\begin{lemma}\label{flow:fjdfjdkjdkj22}
We choose a flat connection $A_0\in\mathcal A_0(P)$, then there is a positive constant $c$ such that for any 0- or 1-form $\alpha$ with compact support:
\begin{equation}
\begin{split}
\varepsilon^2\|\partial_s\alpha&\|_{L^p( \Sigma\times\mathbb R^2)}
+\varepsilon^2\|\partial_t^2\alpha\|_{L^p(\Sigma\times\mathbb R^2)} 
+\| d_{A_0}d_{A_0}^*\alpha\|_{L^p(\Sigma\times\mathbb R^2)}\\
&+\| d_{A_0}^*d_{A_0}\alpha\|_{L^p(\Sigma\times\mathbb R^2)}
+\varepsilon\| \partial_t d_{A_0}^*\alpha\|_{L^p(\Sigma\times\mathbb R^2)}+\varepsilon\| \partial_t d_{A_0}\alpha\|_{L^p(\Sigma\times\mathbb R^2)}\\
\leq &c\varepsilon^2\left\|\left(\partial_s-\partial_t^2+\frac 1{\varepsilon^2}\Delta_{A_0}\right)\alpha\right\|_{L^p(\Sigma\times\mathbb R^2)}+c \|\alpha\|_{L^p(\Sigma\times\mathbb R^2)}\\
&+c\varepsilon \|\partial_t\alpha\|_{L^p(\Sigma\times\mathbb R^2)}+c \|d_{A_0}\alpha\|_{L^p(\Sigma\times\mathbb R^2)}+c \|d_{A_0}^*\alpha\|_{L^p(\Sigma\times\mathbb R^2)}
\end{split}
\end{equation}
\end{lemma}

\begin{proof}
The lemma follows from the previous lemma \ref{flow:fjdfjdkjdkj} using the rescaling $\bar \alpha(t,s):=\alpha(\varepsilon t,\varepsilon^2s)$ .
\end{proof}

\begin{lemma}\label{flow:lemma:c2}
We choose a flat connection $A_0 \in\mathcal A_0(P)$ and a constant $c_0$, then there is a positive constant $c$ such that following holds. For any connection $A\in \mathcal A^{2,p}(P\times \mathbb R^2)$ which satisfies
\begin{equation}\label{flow:esttLmm}
\sup_{(s,t)\in \mathbb R^2}\left(\|A(s,t)-A_0\|_{C^1}+\varepsilon\|\partial_tA\|_{L^\infty}\right)\leq c_0
\end{equation}
 and for any 0- or 1-form $\alpha$ with compact support:
\begin{equation}
\begin{split}
\varepsilon^2\|\partial_s\alpha&\|_{L^p( \Sigma\times\mathbb R^2)}
+\varepsilon^2\|\partial_t^2\alpha\|_{L^p(\Sigma\times\mathbb R^2)} 
+\| d_{A}d_{A}^*\alpha\|_{L^p(\Sigma\times\mathbb R^2)}\\
&+\| d_{A}^*d_{A}\alpha\|_{L^p(\Sigma\times\mathbb R^2)}
+\varepsilon\| \partial_t d_{A}^*\alpha\|_{L^p(\Sigma\times\mathbb R^2)}+\varepsilon\| \partial_t d_{A}\alpha\|_{L^p(\Sigma\times\mathbb R^2)}\\
\leq &c\varepsilon^2\left\|\left(\partial_s-\partial_t^2+\frac 1{\varepsilon^2}d_{A}d_{A}^*+\frac 1 {\varepsilon^2}d_{A}^*d_{A}\right)\alpha\right\|_{L^p(\Sigma\times\mathbb R^2)}+c \|\alpha\|_{L^p(\Sigma\times\mathbb R^2)}.
\end{split}
\end{equation}
\end{lemma}

\begin{proof}
This lemma follow directly from the lemma \ref{flow:fjdfjdkjdkj} using the assumption (\ref{flow:esttLmm}) and the lemma \ref{lemma:dAalphaest}.
\end{proof}

\begin{lemma}\label{flow:cor:alpha}
We choose a regular value $b$ of $E^H$, then there is a positive constant $c$ such that the following holds. For any $\Xi=A+\Psi dt+\Phi ds\in \mathcal M^{0}(\Xi_-,\Xi_+)$, $\Xi_-,\Xi_+\in \mathrm{Crit}_{E^H}^b$, and any 0- or 1-form $\alpha\in W^{2,2,1;p}$
\begin{equation}\label{flow:cor:alpha:eq21}
\begin{split}
\|\alpha\|_{L^p}&+\|d_A\alpha\|_{L^p}+\|d_A^*\alpha\|_{L^p}+\|d_A^*d_A\alpha\|_{L^p}+\|d_Ad_A^*\alpha\|_{L^p}+\varepsilon \|\nabla_t \alpha\|_{L^p}\\
&+\varepsilon^2\|\nabla_t\nabla_t \alpha\|_{L^p}+\varepsilon^2\|\nabla_s\alpha\|_{L^p}+\varepsilon\| \nabla_t d_{A}^*\alpha\|_{L^p}+\varepsilon\| \nabla_t d_{A}\alpha\|_{L^p}\\
\leq &c\left\|\left(\varepsilon^2\nabla_s-\varepsilon^2\nabla_t^2+\Delta_A\right)\alpha\right\|_{L^p}+c\|\alpha\|_{L^p}.
\end{split}
\end{equation}
\end{lemma}

\begin{proof} We choose a finite atlas $\{B_i,\varphi_i:B_i\to S^1\times \mathbb R\}_{i\in I}$ of $S^1\times \mathbb R$ such that the condition (\ref{flow:esttLmm}) is satisfied for every chart; we can cover the two ends of the cylinder with two chart each because $A(t,s)$ converges exponentially to $A_\pm$ as $s\to \pm\infty$ and thus for $s_0$ big enough
\begin{align*}
\sup_{(s,t)\in S^1\times [s_0,\infty)}\left(\|A(s,t)-A_+\|_{C^1}+\varepsilon\|\partial_tA\|_{L^\infty}\right)\leq &c_0,\\
\sup_{(s,t)\in S^1\times (-\infty,s_0]}\left(\|A(s,t)-A_-\|_{C^1}+\varepsilon\|\partial_tA\|_{L^\infty}\right)\leq& c_0.
\end{align*}
Furthermore, we take a partition of the unity $\sum_{i\in \mathbb N}\rho_i(t,s)=1$, $\rho_i(t,s)\in[0,1]$ and $\mathrm{supp}( \rho_i)\subset \varphi(B_i)$; next, collecting the estimate given by the lemma \ref{flow:lemma:c2} on every chart $B_i\times \Sigma$ for $(\rho_i\circ \varphi_i)\alpha_i$, where $\alpha_i$ is the representation of $\alpha$ on $B_i\times \Sigma$, we obtain
\begin{equation}
\begin{split}
\varepsilon^2\|\partial_s\alpha\|_{L^p}&+\varepsilon^2\|\partial_t\partial_t \alpha\|_{L^p}+ \|d_A^*d_A\alpha\|_{L^p}+\|d_Ad_A^*\alpha\|_{L^p}\\
&+\varepsilon\| \partial_t d_{A}^*\alpha\|_{L^p}+\varepsilon\| \partial_t d_{A}\alpha\|_{L^p}\\
\leq &c\left\|\left(\varepsilon^2\partial_s-\varepsilon^2\partial_t^2+\Delta_A\right)\alpha\right\|_{L^p}+c\|\alpha\|_{L^p}+c\varepsilon^2\|\partial_t\alpha\|_{L^p}.
\end{split}
\end{equation}
 Since $\|\Psi\|_{L^\infty}+\|\partial_t\Psi\|_{L^\infty}+\|\Phi\|_{L^\infty}\leq c_1$, we have
 \begin{equation}
\begin{split}
\varepsilon^2\|\nabla_s\alpha\|_{L^p}&+\varepsilon^2\|\nabla_t\nabla_t \alpha\|_{L^p}+ \|d_A^*d_A\alpha\|_{L^p}+\|d_Ad_A^*\alpha\|_{L^p}\\
&+\varepsilon\| \nabla_t d_{A}^*\alpha\|_{L^p}+\varepsilon\| \nabla_t d_{A}\alpha\|_{L^p}\\
\leq &c\left\|\left(\varepsilon^2\nabla_s-\varepsilon^2\nabla_t^2+\Delta_A\right)\alpha\right\|_{L^p}+c\|\alpha\|_{L^p}\\
&+c\varepsilon^2\|\nabla_t\alpha\|_{L^p}+c\varepsilon\|  d_{A}^*\alpha\|_{L^p}+c\varepsilon\|  d_{A}\alpha\|_{L^p}.
\end{split}
\end{equation}
 The estimate (\ref{flow:cor:alpha:eq21}) follows then from the lemmas \ref{flow:lemma40} and \ref{lemma:dAalphaest}.
\end{proof}

\begin{lemma} \label{flow:lemma:lplppia}
We choose a regular value $b$ of $E^H$, then there are two positive constants $c$ and $\varepsilon_0$ such that the following holds. For any $\Xi=A+\Psi dt+\Phi ds\in \mathcal M^{0}(\Xi_-,\Xi_+)$, $\Xi_\pm\in \mathrm{Crit}_{E^H}^b$, any $i$-form $\xi\in W^{2,2,1;p}$, $i=0,1$ and $0<\varepsilon<\varepsilon_0$
\begin{equation}
\begin{split}
\int_{S^1\times\mathbb R} \|\xi\|^p_{L^2(\Sigma)}  \,dt\,ds\leq &c \int_{S^1\times\mathbb R}\|\varepsilon^2\partial_s\xi-\varepsilon^2\partial_t^2\xi+\Delta_A\xi \|^p_{L^2(\Sigma)}dt\,ds\\
&+c\int_{S^1\times \mathbb R}\|\pi_A(\xi)\|^p_{L^2(\Sigma)}dt\,ds.
\end{split}
\end{equation}
\end{lemma}

\begin{proof}

In this proof we denote the norm $\|\cdot\|_{L^2(\Sigma) }$ by $\|\cdot\|$. If we consider only the Laplace part of the operator, we obtain that
\begin{equation*}
\begin{split}
\int_{S^1\times\mathbb R}&\|\xi\|^{p-2}\langle\xi,-\varepsilon^2\partial^2_t\xi+\Delta_A\xi\rangle ds\,dt\\
=&\int_{S^1\times\mathbb R}\|\xi\|^{p-2}\left(\varepsilon^2\|\partial_t\xi\|^2+\|d_A\xi\|^2+\|d_A^*\xi\|^2\right)ds\,dt\\
&+\int_{S^1\times\mathbb R}(p-2)\|\xi\|^{p-4}\langle\xi,\partial_t\xi\rangle^2ds\,dt
\end{split}
\end{equation*}
and thus
\begin{equation}\label{flow:dghhga}
\begin{split}
\int_{S^1\times\mathbb R}&\|\xi\|^{p-2}\left(\varepsilon^2\|\partial_t\xi\|^2+\|d_A\xi\|^2+\|d_A^*\xi\|^2\right)ds\,dt\\
\leq &\int_{S^1\times\mathbb R}\|\xi\|^{p-2}\langle\xi,-\varepsilon^2\partial^2_t\xi+\Delta_A\xi\rangle ds\,dt\\
= &\int_{S^1\times\mathbb R}\|\xi\|^{p-2}\langle\xi,\varepsilon^2\partial_s\xi-\varepsilon^2\partial^2_t\xi+\Delta_A\xi\rangle ds\,dt\\
\leq&\int_{S^1\times\mathbb R}\|\xi\|^{p-1}\|\varepsilon^2\partial_s\xi-\varepsilon^2\partial^2_t\xi+\Delta_A\xi\|\,ds\,dt\\
\leq &\left(\int_{S^1\times\mathbb R}\|\xi\|^p ds\,dt\right)^{\frac{p-1}p}\left(\int_{S^1\times\mathbb R}\|\varepsilon^2\partial_s\xi-\varepsilon^2\partial^2_t\xi+\Delta_A\xi\|^p ds\,dt\right)^{\frac{1}p}
\end{split}
\end{equation}
where the second step follows from
\begin{equation*}
\int_{S^1\times\mathbb R}\|\xi\|^{p-2}\langle\xi,\partial_s\xi\rangle ds\,dt
=\frac 1p\int_{S^1\times\mathbb R}\partial_s\|\xi\|^p ds\,dt=0,
\end{equation*}
the third from the Cauchy-Schwarz inequality and the fourth from the H\"older's inequality. Therefore, by lemma \ref{flow:lemma:lpl2}
\begin{align*}
\int_{S^1\times\mathbb R}&\|\xi\|^p ds\,dt\leq \int_{S^1\times\mathbb R}\|\xi\|^{p-2}\left(\|d_A\xi\|^2+\|d_A^*\xi\|^2+\|\pi_A(\xi)\|^2\right)ds\,dt\\
\intertext{by (\ref{flow:dghhga}) we have that}
\leq& \left(\int_{S^1\times\mathbb R}\|\xi\|^p ds\,dt\right)^{\frac{p-1}p}\left(\int_{S^1\times\mathbb R}\|\varepsilon^2\partial_s\xi-\varepsilon^2\partial^2_t\xi+\Delta_A\xi\|^p ds\,dt\right)^{\frac{1}p}\\
&+\int_{S^1\times\mathbb R}\|\xi\|^{p-1}\|\pi_A(\xi)\|\,ds\,dt\\
\intertext{and by the H\"older's inequality}
\leq& \left(\int_{S^1\times\mathbb R}\|\xi\|^p ds\,dt\right)^{\frac{p-1}p}\left(\int_{S^1\times\mathbb R}\|\varepsilon^2\partial_s\xi-\varepsilon^2\partial^2_t\xi+\Delta_A\xi\|^p ds\,dt\right)^{\frac{1}p}\\
&+\left(\int_{S^1\times\mathbb R}\|\xi\|^p ds\,dt\right)^{\frac{p-1}p}\left(\int_{S^1\times\mathbb R}\|\pi_A(\xi)\|^p ds\,dt\right)^{\frac{1}p};
\end{align*}
thus, we can conclude that
\begin{equation*}
\int_{S^1\times\mathbb R}\|\xi\|^p ds\,dt\leq c\int_{S^1\times\mathbb R}\left( \|\varepsilon^2\partial_s\xi-\varepsilon^2\partial^2_t\xi+\Delta_A\xi\|^p+\|\pi_A(\xi)\|^p\right)ds\,dt.
\end{equation*}
and hence we finished the proof of the lemma.
\end{proof}

\begin{proof}[Proof of theorem \ref{flow:thm:linestpar}]
By lemma \ref{flow:lemma:lpl2}, for any $\delta>0$ there is a $c_0$ such that
\begin{equation*}
\begin{split}
\|\alpha\|_{L^p}^p\leq &\delta\left(\|d_A\alpha\|_{L^p}^p+\|d_A*\alpha\|_{L^p}^p\right)+ c_0\int_{S^1\times\mathbb R }\|\alpha\|_{L^2}^p dt\,ds \\
\leq &\delta\left(\|d_A\alpha\|_{L^p}^p+\|d_A*\alpha\|_{L^p}^p\right)+ c_0c_1\int_{S^1\times\mathbb R }\|\pi_A(\alpha)\|_{L^2}^p dt\,ds\\
&+ c_0c_1\int_{S^1\times\mathbb R}\|\varepsilon^2\partial_s\alpha-\varepsilon^2\partial_t^2\alpha+\Delta_A\alpha\|^p_{L^2}dt\,ds\\
\leq &\delta\left(\|d_A\alpha\|_{L^p}^p+\|d_A*\alpha\|_{L^p}^p\right)+ c_0c_1c_2\|\pi_A(\alpha)\|_{L^p}^p\\
&+ c_0c_1c_2\|\varepsilon^2\partial_s\alpha-\varepsilon^2\partial_t^2\alpha+\Delta_A\alpha \|^p_{L^p}\\
\leq &\delta\left(\|d_A\alpha\|_{L^p}^p+\|d_A*\alpha\|_{L^p}^p\right)+ c_0c_1c_2\|\pi_A(\alpha)\|_{L^p}^p\\
&+ c_0c_1c_2\|\varepsilon^2\nabla_s\alpha-\varepsilon^2\nabla_t^2\alpha+\Delta_A\alpha \|^p_{L^p}+c_3\varepsilon^2\|\Psi\|_{L^\infty}\|\nabla_t\alpha \|_{L^p}^p\\
&+ c_3\varepsilon^2\left( \|\Psi\|_{L^\infty}^2+\|\partial_t\Psi\|_{L^\infty}+\|\Phi \|_{L^\infty}\right)  \|\alpha \|_{L^p}^p 
\end{split}
\end{equation*}
where the second step follows form the lemma \ref{flow:lemma:lplppia} and the third by the H\"older's inequality with $c_2:=\left(\int_{\Sigma}\mathrm{dvol}_{\Sigma}\right)^{\frac{p-2}p}$. Therefore if we choose $\delta$ and $\varepsilon$ small enough we can improve the estimate (\ref{flow:cor:alpha:eq21}) of the corollary  \ref{flow:cor:alpha} using the last estimate, i.e.
\begin{equation*}
\begin{split}
\|\alpha\|_{L^p}&+\|d_A\alpha\|_{L^p}+\|d_A^*\alpha\|_{L^p}+\|d_A^*d_A\alpha\|_{L^p}+\|d_Ad_A^*\alpha\|_{L^p}+\varepsilon \|\nabla_t \alpha\|_{L^p}\\
&+\varepsilon^2\|\nabla_t\nabla_t \alpha\|_{L^p}+\varepsilon\|\nabla_t d_A\alpha\|_{L^p}+\varepsilon\|\nabla_t d_A^*\alpha\|_{L^p}+\varepsilon^2\|\nabla_s\alpha\|_{L^p}\\
\leq &\left\|\left(\varepsilon^2\nabla_s-\varepsilon^2\nabla_t^2+\Delta_A\right)\alpha\right\|_{L^p}+c\|\pi_A(\alpha)\|_{L^p},
\end{split}
\end{equation*}
because $\|\Psi\|_{L^\infty}+\|\partial_t\Psi\|_{L^\infty}+\|\Phi \|_{L^\infty}$ is bounded by a constant. Furthermore, the terms $\varepsilon\| d_A\nabla_t\alpha\|_{L^p}$ and $\varepsilon\|d_A^*\nabla_t \alpha\|_{L^p}$ can be estimated by 
$$d_A\nabla_t\alpha\|_{L^p}+\varepsilon\|d_A^*\nabla_t \alpha\|_{L^p}\leq \varepsilon\|\nabla_t d_A\alpha\|_{L^p}+\varepsilon\|\nabla_t d_A^*\alpha\|_{L^p}+c\varepsilon\|\alpha\|_{L^p}$$
using the commutation formulas and because the curvature $\partial_tA-d_A\Psi$ is bounded; we proved therefore (\ref{flow:thm:linestpateq1}) and the second inequality of the theorem can be proved in the same way.
\end{proof}

%
%
\subsection{Proof of the theorem \ref{flow:thm:wavefin} }\label{flow:subsection:proofthm2}

Before proving the theorem we will show some preliminary results, in fact the theorem \ref{flow:thm:wavefin} will then follow from the corollary \ref{flow:cor:wave} and the lemmas \ref{flow:lemma:lpl2} and \ref{flow:lemma:lplppia2}.

\begin{corollary}\label{flow:cor:MM2}
For every $p>1$ there is a positive constant $c$ such that the following holds. For every two maps $\gamma \in W_0^{2,p}(\mathbb R^4, \mathbb R)$, $\phi \in W_0^{1,p}(\mathbb R^4, \mathbb R)$ we have that
\begin{equation}\label{flow:wave:eqcor2}
\begin{split}
\|\partial_{s}\alpha_1\|_{L^p}&+\lambda \|\partial_{x_1}\alpha_1\|_{L^p}+\|\partial_{s}\alpha_2\|_{L^p}+\lambda \|\partial_{x_2}\alpha_2\|_{L^p}
+\lambda \|\partial_{t}\alpha_1\|_{L^p}\\
&+\lambda \|\partial_{t}\alpha_2\|_{L^p}
+\|\partial_s\psi\|_{L^p} +\lambda\|\partial_t\psi\|_{L^p}+\|\partial_s\phi \|_{L^p}\\
&+\|\partial_{x_1}\phi\|_{L^p}
+\|\partial_{x_2}\phi\|_{L^p}+\|\partial_{t}\phi\|_{L^p}\\
\leq &c \|\partial_s\alpha_1-\partial_{x_1}\phi\|_{L^p}+c \|\partial_s\alpha_2-\partial_{x_2}\phi\|_{L^p}+c\|\partial_s\psi-\partial_t\phi\|_{L^p}\\
&+c\|\partial_s\phi+\partial_{x_2}\alpha_1+\partial_{x_1}\alpha_2+\partial_t\psi \|_{L^p}+c\lambda \|\partial_s\gamma\|_{L^p},
\end{split}
\end{equation}
where $\alpha_1=\partial_{x_1}\gamma$, $\alpha_2=\partial_{x_2}\gamma$, $\psi=\partial_{t}\gamma$ and $\lambda\in [0,1]$.
\end{corollary}

\begin{proof} In order to prove this corollary we need to apply the theorem \ref{flow:thm:MM} of Marcinkiewicz and Mihlin stated in the previous subsection and in order to do this we have to define the multipliers and prove the assumption (\ref{flow:thmMM:eq1}); therefore, we look at the following system of equations 
\begin{equation}\label{flow:wave:eqcor1}
f=\left(\begin{array}{c}
f_1\\ f_2\\ f_3\\ f_4
\end{array}
\right):=
\left(\begin{array}{cccc}
\partial_s &0&0&-\partial_{x_1}\\
0&\partial_s&0&-\partial_{x_2}\\
0&0&\partial_s&-\partial_t\\
\partial_{x_1}&\partial_{x_2}&\partial_t&\partial_s
\end{array}\right)
\left(\begin{array}{c}
\alpha_1\\ \alpha_2\\ \psi\\ \phi
\end{array}
\right).
\end{equation}
One can remark that the four lines of (\ref{flow:wave:eqcor1}) correspond to the first four terms in the $L^p$-norm in the right side of the estimate (\ref{flow:wave:eqcor2}). Applying the Fourier transformation to (\ref{flow:wave:eqcor2}) we obtain
\begin{equation*}
\mathcal F(f)=\left(\begin{array}{cccc}
\sigma i &0&0&-y_1 i\\
0&\sigma i &0&-y_2 i\\
0&0&\sigma i &-\tau i\\
y_1 i&y_2 i&\tau i&\sigma i 
\end{array}\right)
\left(\begin{array}{c}
\mathcal F(\alpha_1)\\ \mathcal F(\alpha_2)\\ \mathcal F(\psi)\\ \mathcal F(\phi)
\end{array}
\right)
\end{equation*}
and thus computing its invers:
\begin{align*}
\mathcal F(\alpha_1)=& \frac {-i \left( {\sigma}^{2}+{y_{{2}}}^{2}+{\tau}^{2} \right) }{ \left( {\sigma}^{2}+{y_{{1}}}^{2}+{y_{{2}}}^{2}+{\tau}^{2} \right) \sigma}  \mathcal F(f_1)+{\frac {iy_{{2}}y_{{1}}}{ \left( {\sigma}^{2}+{y_{{1}}}^{2}+{y_{{2}}}^{2}+{\tau}^{2} \right) \sigma}}  \mathcal F(f_2)\\
&+  \frac {i\tau\,y_{{1}}} { \left( {\sigma}^{2}+{y_{{1}}}^{2}+{y_{{2}}}^{2}+{\tau}^{2} \right) \sigma} \mathcal F(f_3)+ \frac {-iy_{{1}}}{{\sigma}^{2}+{y_{{1}}}^{2}+{y_{{2}}}^{2}+{\tau}^{2}} \mathcal F(f_4),
\end{align*}

\begin{align*}
\mathcal F(\phi)=& {\frac {iy_{{1}}}{{\sigma}^{2}+{y_{{1}}}^{2}+{y_{{2}}}^{2}+{\tau}^{2}}} \mathcal F(f_1)+{\frac {iy_{{2}}}{{\sigma}^{2}+{y_{{1}}}^{2}+{y_{{2}}}^{2}+{\tau}^{2}}}\mathcal F(f_2)\\
&+{\frac {i\tau}{{\sigma}^{2}+{y_{{1}}}^{2}+{y_{{2}}}^{2}+{\tau}^{2}}}\mathcal F(f_3)+\frac {-i\sigma}{{\sigma}^{2}+{y_{{1}}}^{2}+{y_{{2}}}^{2}+{\tau}^{2}}\mathcal F(f_4),
\end{align*}

and then

\begin{align*}
\mathcal F(\partial_s\alpha_1)=&\frac { {\sigma}^{2}+{y_{{2}}}^{2}+{\tau}^{2}  }{ {\sigma}^{2}+{y_{{1}}}^{2}+{y_{{2}}}^{2}+{\tau}^{2}  } \mathcal F(f_1)+{\frac {-y_{{2}}y_{{1}}}{  {\sigma}^{2}+{y_{{1}}}^{2}+{y_{{2}}}^{2}+{\tau}^{2}  }}\mathcal F(f_2)\\
&+{\frac {-\tau\,y_{{1}}}{  {\sigma}^{2}+{y_{{1}}}^{2}+{y_{{2}}}^{2}+{\tau}^{2}  }}\mathcal F(f_3)+\frac {y_{{1}}\sigma}{{\sigma}^{2}+{y_{{1}}}^{2}+{y_{{2}}}^{2}+{\tau}^{2}}\mathcal F(f_4),
\end{align*}

\begin{align*}
\mathcal F(\partial_s\phi)=&\frac {-y_{{1}}\sigma}{{\sigma}^{2}+{y_{{1}}}^{2}+{y_{{2}}}^{2}+{\tau}^{2}}\mathcal F(f_1)+{\frac {-y_{{2}}\sigma}{{\sigma}^{2}+{y_{{1}}}^{2}+{y_{{2}}}^{2}+{\tau}^{2}}}\mathcal F(f_2)\\
&+{\frac {-\tau\sigma}{{\sigma}^{2}+{y_{{1}}}^{2}+{y_{{2}}}^{2}+{\tau}^{2}}}\mathcal F(f_3)+{\frac {\sigma^2}{{\sigma}^{2}+{y_{{1}}}^{2}+{y_{{2}}}^{2}+{\tau}^{2}}}\mathcal F(f_4).
\end{align*}
\noindent The formulas for $\mathcal F(\alpha_2)$ and for $\mathcal F(\phi)$, respectively for $\mathcal F(\partial_s\alpha_2)$ and for $\partial_s\mathcal F(\phi)$, are similar to that of $\mathcal F(\alpha_1)$, respectively to that of $\mathcal F(\partial_s\alpha_1)$. Since the multipliers for $\mathcal F(\partial_s\alpha_1)$, $\mathcal F(\partial_s\alpha_2)$,  $\mathcal F(\partial_s\psi)$ and $\mathcal F(\partial_s\phi)$ satisfy the assumption (\ref{flow:thmMM:eq1}) of the theorem \ref{flow:thm:MM}, we can conclude that
\begin{equation}\label{flow:wave:eqcor3}
\begin{split}
\|\partial_{s}\alpha_1\|_{L^p}&+\|\partial_{s}\alpha_2\|_{L^p}+\|\partial_s\psi\|_{L^p} +\|\partial_s\phi \|_{L^p}+\|\partial_{x_1}\phi\|_{L^p}\\
&+\|\partial_{x_2}\phi\|_{L^p}+\|\partial_{t}\phi\|_{L^p}+\|\partial_{x_1}\alpha_1+\partial_{x_2}\alpha_2+\partial_t\psi\|_{L^p}\\
\leq &c \|\partial_s\alpha_1-\partial_{x_1}\phi\|_{L^p}+c \|\partial_s\alpha_2-\partial_{x_2}\phi\|_{L^p}+c\|\partial_s\psi-\partial_t\phi\|_{L^p}\\
&+c\|\partial_s\phi+\partial_{x_2}\alpha_1+\partial_{x_1}\alpha_2+\partial_t\psi \|_{L^p}.
\end{split}
\end{equation}
Next, we use that $\alpha_1=\partial_{x_1}\gamma$, $\alpha_2=\partial_{x_2}\gamma$, $\psi=\partial_{t}\gamma$ and thus
\begin{equation*}
\begin{split}
\|\partial_s\gamma-(\partial_{x_1}^2+\partial_{x_2}^2+\partial_t^2)\gamma\|_{L^p}
\leq &\|(\partial_{x_1}^2+\partial_{x_2}^2+\partial_t^2)\gamma\|_{L^p}+\|\partial_s\gamma\|_{L^p}\\
\leq &\|\partial_{x_1}\alpha_1+\partial_{x_2}\alpha_2+\partial_t\psi\|_{L^p}+\|\partial_s\gamma\|_{L^p}.
\end{split}
\end{equation*}
Therefore by corollary \ref{flow:cor:MM1}, it follow that
\begin{equation}\label{flow:wave:eqcor66}
\begin{split}
\lambda\|\partial_{x_1}\alpha_1\|_{L^p}+&\lambda \|\partial_{x_2}\alpha_2\|_{L^p}+\lambda \|\partial_{t}\alpha_1\|_{L^p}+\lambda \|\partial_{t}\alpha_2\|_{L^p}+\lambda \|\partial_t\psi\|_{L^p}\\
\leq &c \lambda \|\partial_s\alpha_1-\partial_{x_1}\phi\|_{L^p}+c\lambda  \|\partial_s\alpha_2-\partial_{x_2}\phi\|_{L^p}+c\lambda \|\partial_s\psi-\partial_t\phi\|_{L^p}\\
&+c\lambda\|\partial_s\phi+\partial_{x_2}\alpha_1+\partial_{x_1}\alpha_2+\partial_t\psi \|_{L^p}+c\lambda \|\partial_s\gamma\|_{L^p}.
\end{split}
\end{equation}
Therefore the theorem follows combining (\ref{flow:wave:eqcor3}) and (\ref{flow:wave:eqcor3}). 
\end{proof}

\begin{lemma}\label{flow:cor:wave}
We choose a regular value $b$ of $E^H$, then there is a positive constant $c$ such that the following holds. For any $\Xi=A+\Psi dt+\Phi ds\in \mathcal M^{0}(\Xi_-,\Xi_+)$, $\Xi_\pm\in \mathrm{Crit}_{E^H}^b$, and any 1-form $\alpha+\psi dt=d_{A+\Psi dt }\gamma \in W^{1,2;p}\cap \textrm{im } d_{A+\Psi dt }$
\begin{equation}\label{floe:ds}
\begin{split}
\|\alpha\|_{L^p}&+\|d_A^*\alpha\|_{L^p}+\varepsilon^2\|\nabla_s\alpha\|_{L^p}+\varepsilon\|\nabla_t\alpha\|_{L^p}+\varepsilon \|\psi\|_{L^p}+\varepsilon^2\|\nabla_t\psi\|_{L^p}\\
&+\varepsilon^3\|\nabla_s\psi\|_{L^p}
+\varepsilon^2 \|\phi \|_{L^p}+\varepsilon^2 \|d_A\phi\|_{L^p}+\varepsilon^3 \|\nabla_t\phi\|_{L^p}+\varepsilon^4 \|\nabla_s\phi \|_{L^p}\\
\leq &c\varepsilon^2\left\|\nabla_s\alpha-d_A\phi\right\|_{L^p}+c\varepsilon^3\left\|\nabla_s\psi-\nabla_t\phi\right\|_{L^p}\\
&+c\varepsilon^4\left\|\nabla_s\phi-\frac1{\varepsilon^4}d_A^*\alpha+\frac1{\varepsilon^2}\nabla_t\psi\right\|_{L^p}
+c\|\alpha\|_{L^p}+c\varepsilon^2 \|\phi \|_{L^p}.
\end{split}
\end{equation}
\end{lemma}

\begin{proof}
In the same way that the lemma \ref{flow:cor:alpha} follows from the corollary \ref{flow:cor:MM1}, the corollary \ref{flow:cor:MM2} implies

\begin{align*}
\|\alpha\|_{L^p}&+\varepsilon^2\|\nabla_{s}\alpha\|_{L^p}+\lambda \|d_A^*\alpha\|_{L^p}+\lambda\varepsilon\|\nabla_t \alpha\|_{L^p}+\varepsilon^3\|\nabla_s\psi\|_{L^p} +\lambda\varepsilon^2\|\nabla_t\psi\|_{L^p}\\
&+\varepsilon^2\|\phi\|_{L^p}+\varepsilon^4\|\nabla_s\phi \|_{L^p}+\varepsilon^2\|d_A\phi\|_{L^p}+\varepsilon^3\|\nabla_{t}\phi\|_{L^p}\\
\leq &c \varepsilon^2\|\nabla_s\alpha-d_A\phi\|_{L^p}+c\varepsilon^3\|\nabla_s\psi-\nabla_t\phi\|_{L^p}+c\lambda\varepsilon^2\|\nabla_s\gamma\|_{L^p}\\
&+c\varepsilon^4\left\|\nabla_s\phi+\frac 1{\varepsilon^4}d_A^*\alpha_1+\frac1{\varepsilon^2}\nabla_t\psi\right \|_{L^p}+c\|\alpha\|_{L^p}+c\varepsilon^2\|\phi\|_{L^p}
\end{align*}
The term $\varepsilon^2\|\nabla_s\gamma\|_{L^p}$ can be estimate by $c\varepsilon^2\|\nabla_s\alpha\|_{L^p}+c\varepsilon^2\|\alpha\|_{L^p}$ by the lemma \ref{lemma76dt94} and the commutation formula and thus using the last estimate for $\lambda=0$
\begin{align*}
\varepsilon^2\|\nabla_s\gamma\|_{L^p}\leq &c \varepsilon^2\|\nabla_s\alpha-d_A\phi\|_{L^p}+c\varepsilon^3\|\nabla_s\psi-\nabla_t\phi\|_{L^p}\\
&+c\varepsilon^4\left\|\nabla_s\phi+\frac 1{\varepsilon^4}d_A^*\alpha_1+\frac1{\varepsilon^2}\nabla_t\psi\right \|_{L^p}+c\|\alpha\|_{L^p}+c\varepsilon^2\|\phi\|_{L^p}.
\end{align*}
Furthermore, by the lemma \ref{lemma76dt94} and the commutation formula
$$\varepsilon\|\psi\|_{L^p}\leq c\varepsilon\|d_A\psi\|_{L^p}
\leq c\|\alpha\|_{L^p}+c\varepsilon\|\nabla_t\alpha\|_{L^p}$$
and finally collecting the last thee estimates we obtain (\ref{floe:ds}).
\end{proof}

\begin{lemma} \label{flow:lemma:lplppia2}
We choose a regular value $b$ of $E^H$ and a $\delta>0$, then there are two positive constants $c$ and $\varepsilon_0$ such that the following holds. For any $\Xi=A+\Psi dt+\Phi ds\in \mathcal M^{0}(\Xi_-,\Xi_+)$, $\Xi_\pm\in \mathrm{Crit}_{E^H}^b$, any 1-form $\xi:=\alpha+\psi dt+\phi dt \in W^{1,1,1;p}$, where $\alpha+\psi dt\in \textrm{im } d_{A+\Psi dt}$, and any $0<\varepsilon<\varepsilon_0$
\begin{equation}
\begin{split}
\int_{S^1\times\mathbb R}& \|\xi\|^{p-2}_{L^2(\Sigma)}\left(\|\alpha\|^2_{L^2(\Sigma)}+\varepsilon^4\|\phi\|_{L^2(\Sigma)}^2\right)  dt\,ds\\
\leq& \int_{S^1\times\mathbb R} \left(c\|\varepsilon^2\partial_s\alpha-\varepsilon^2d_A\phi \|^p_{L^2(\Sigma) } +c\varepsilon^p \|\varepsilon^2\partial_s\psi-\varepsilon^2\partial_t\phi \|^p_{L^2(\Sigma) }\right)dt\,ds\\
&+\int_{S^1\times\mathbb R} \left(c\varepsilon^{2p} \left\|\varepsilon^2\partial_s\phi-\frac 1{\varepsilon^2}d_A^*\alpha+\partial_t\psi \right\|^p_{L^2(\Sigma) }+\delta  \|\xi\|^{p}_{L^2(\Sigma)} \right)dt\,ds.
\end{split}
\end{equation}
\end{lemma}

\begin{proof}
In this proof we denote the norm $\|\cdot\|_{L^2(\Sigma) }$ by $\|\cdot\|$. We consider $\xi=\alpha+\psi dt+\phi ds$ where $\alpha+\psi dt=d_{A+\Psi dt}\gamma$ and

\begin{equation*}
\bar \xi=\bar \alpha+\bar\psi dt+\bar \phi ds= D\xi=\left(\begin{array}{ccc}\varepsilon^2\partial_s&0& -\varepsilon^2d_A\\0&\varepsilon^2\partial_s&-\varepsilon^2\partial_t\\ -\frac 1{\varepsilon^2}d_A^*&\partial_t&\varepsilon^2\partial_s\end{array}\right)
\left(\begin{array}{c}\alpha\\\psi\\\phi\end{array}\right);
\end{equation*}
thus $D^*D\xi$ can be written in the following way
\begin{equation*}
\begin{split}
D^*\bar \xi= \left(\begin{array}{ccc}-\varepsilon^2 \partial_s&0& -\varepsilon^2d_A\\0&-\varepsilon^2\partial_s&-\varepsilon^2\partial_t\\ -\frac 1{\varepsilon^2}d_A^*&\partial_t&-\varepsilon^2\partial_s\end{array}\right)
\left(\begin{array}{ccc}\varepsilon^2\partial_s&0& -\varepsilon^2d_A\\0&\varepsilon^2\partial_s&-\varepsilon^2\partial_t\\ -\frac 1{\varepsilon^2}d_A^*&\partial_t&\varepsilon^2\partial_s\end{array}\right)
\left(\begin{array}{c}\alpha\\\psi\\\phi\end{array}\right)\\
= \left(\begin{array}{ccc}-\varepsilon^4 \partial_s^2+d_Ad_A^*&-\varepsilon^2d_A\partial_t&\varepsilon^4(\partial_s d_A-d_A\partial_s)\\
\partial_t d_A^*&-\varepsilon^4\partial_s^2-\varepsilon^2\partial_t^2&0\\ -d_A^*\partial_s+\partial_s d_A^*&0&d_A^*d_A-\varepsilon^2\partial_t^2-\varepsilon^4\partial_s^2\end{array}\right)
\left(\begin{array}{c}\alpha\\\psi\\\phi\end{array}\right).
\end{split}
\end{equation*}
We define
\begin{equation*}
\begin{split}
B:=&\frac12 \|d_A^*\alpha-\varepsilon^2\partial_t\psi \|^2+\varepsilon^4\|\partial_s\alpha\|^2+\varepsilon^6\|\partial_s\psi\|^2+\varepsilon^4\|d_A\phi \|^2\\
&+\frac12\left(\|d_A^*\alpha \|^2+\|\varepsilon^2\partial_t\psi \|^2\right)+\varepsilon^6\|\partial_t\phi\|^2+\varepsilon^8\|\partial_s\phi\|^2
\end{split}
\end{equation*}
\begin{equation*}
G^p:=\|\varepsilon^2\partial_s\alpha-\varepsilon^2d_A\phi \|^p +\varepsilon^p \|\varepsilon^2\partial_s\psi-\varepsilon^2\partial_t\phi \|^p+\varepsilon^{2p} \left\|\varepsilon^2\partial_s\phi-\frac 1{\varepsilon^2}d_A^*\alpha+\partial_t\psi \right\|^p
\end{equation*}
Using the partial integration we obtain
\begin{equation}\label{flow:est:waved1}
\begin{split}
\int_{S^1\times\mathbb R}&\|\xi\|^{p-2}\langle \alpha, (-\varepsilon^4\partial_s^2+d_Ad_A^*)\alpha-\varepsilon^2d_A\partial_t\psi\rangle dt\,ds\\
&+\int_{S^1\times\mathbb R}\|\xi\|^{p-2} \varepsilon^2\langle\psi,(-\varepsilon^4\partial_s^2-\varepsilon^2\partial_t^2)\psi+\partial_t d_A^*\alpha\rangle dt\, ds\\
&+\int_{S^1\times\mathbb R}\|\xi\|^{p-2}\varepsilon^4\langle\phi,(d_A^*d_A-\varepsilon^2\partial_t^2-\varepsilon^4\partial_s^2)\phi\rangle  dt\,ds\\
=&\int_{S^1\times\mathbb R}\|\xi\|^{p-2}B\,dt\,ds\\
&+(p-2)\int_{S^1\times\mathbb R}\|\xi\|^{p-4}\left(\langle\alpha,\varepsilon^2\partial_s\alpha\rangle+\varepsilon^2\langle\psi, \varepsilon^2\partial_s \psi\rangle+\varepsilon^4\langle\phi, \varepsilon^2\partial_s\phi\rangle\right)^2dt\,ds\\
&+\int_{S^1\times\mathbb R}\|\xi\|^{p-2}\varepsilon^2\partial_t\langle\psi,-\varepsilon^2 \partial_t \psi\rangle dt\,ds+\int_{S^1\times\mathbb R}\|\xi\|^{p-2}\varepsilon^2\langle d_A\psi, \partial_t \alpha\rangle dt\,ds
\end{split}
\end{equation}
whose last term, using that $\alpha+\psi dt=d_A\gamma+\partial_t\gamma dt$, can be estimate as follows
\begin{equation}\label{flow:est:waved2}
\begin{split}
\int_{S^1\times\mathbb R}&\|\xi\|^{p-2}\varepsilon^2\langle d_A\psi, \partial_t \alpha\rangle dt\,ds=\int_{S^1\times\mathbb R}\|\xi\|^{p-2}\varepsilon^2\langle d_A\nabla_t\gamma, \partial_t d_A\gamma\rangle dt\,ds\\
\geq&\int_{S^1\times\mathbb R}\|\xi\|^{p-2}\varepsilon^2\|d_A\psi\|^2-c\varepsilon^2\int_{S^1\times\mathbb R}\|\xi\|^{p-1}(\|\alpha\|+\|d_A^*\alpha\|)dt\,ds.
\end{split}
\end{equation}
Since the penultimate line of (\ref{flow:est:waved1}) is positive, (\ref{flow:est:waved1}) and (\ref{flow:est:waved2}) yield
\begin{align*}
\int_{S^1\times\mathbb R}&\|\xi\|^{p-2}B dt\,ds\\
\leq&\int_{S^1\times\mathbb R}\|\xi\|^{p-2}\left(\langle \alpha, (D^*\bar\xi)_1\rangle+\varepsilon^2 \langle \psi, (D^*\bar\xi)_2\rangle+\varepsilon^4\langle \phi, (D^*\bar\xi)_3\rangle\right)dt\,ds\\
&-\int_{S^1\times\mathbb R } \|\xi\|^{p-2} \varepsilon^4\left(\langle \alpha,[\partial_s A,\phi]\rangle-\langle\phi,*[\partial_s A\wedge*\alpha]\rangle\right)dt\,ds\\
&+c\varepsilon^2\int_{S^1\times\mathbb R}\|\xi\|^{p-1}(\|\alpha\|+\|d_A^*\alpha\|)dt\,ds\\
\intertext{integrating by parts the first line after the inequality and using the Cauchy-Schwarz inequality, we obtain}
\leq&c\int_{S^1\times\mathbb R}\|\xi\|^{p-2} B^{\frac12} \left(\|\varepsilon^2\bar\alpha \| +\varepsilon \|\varepsilon^2\bar\psi \|+\varepsilon^{2} \left\|\varepsilon^2\bar\phi \right\|\right) dt\,ds\\
&+c\varepsilon^2\int_{S^1\times\mathbb R } \|\xi\|^{p} dt\,ds
+c\varepsilon^2\int_{S^1\times\mathbb R}\|\xi\|^{p-1}\|d_A^*\alpha\|\,dt\,ds\\
\intertext{Since $2ab\leq a^2+b^2$ for any $a,b\in\mathbb R$, choosing $\varepsilon$ small enough}
\leq&c\int_{S^1\times\mathbb R}\|\xi\|^{p-2}\left(\|\varepsilon^2\bar\alpha \| +\varepsilon \|\varepsilon^2\bar\psi \|+\varepsilon^{2} \left\|\varepsilon^2\bar\phi \right\|\right) ^2dt\,ds\\
&+ \frac12\int_{S^1\times\mathbb R}\|\xi\|^{p-2} B dt\,ds+c\varepsilon^2\int_{S^1\times\mathbb R } \|\xi\|^{p} dt\,ds.
\end{align*}
The last estimate implies that
\begin{align*}
\frac12\int_{S^1\times\mathbb R}&\|\xi\|^{p-2}B\,dt\,ds
\leq c\int_{S^1\times\mathbb R}\|\xi\|^{p-2}G^2ds+\varepsilon^2\int_{S^1\times\mathbb R } \|\xi\|^{p} dt\,ds\\
\leq& c\left(\int_{S^1\times\mathbb R}\|\xi\|^{p}ds\right)^{1-\frac2p}\left(\int_{S^1\times\mathbb R}G^p ds\right)^{\frac2p}+c\varepsilon^2\int_{S^1\times\mathbb R } \|\xi\|^{p} dt\,ds\\
\leq& c\int_{S^1\times\mathbb R}G^p ds+\delta \int_{S^1\times\mathbb R } \|\xi\|^{p} dt\,ds,
\end{align*}
where in the second step we use the H\"older's estimate and in the third the estimate $ab\leq \frac{a^r}r+\frac {b^q}q$, $\frac 1r+\frac 1q=1$, with $r=\frac p2$. Finally,
\begin{equation*}
\begin{split}
\int_{S^1\times\mathbb R} &\|\xi\|^{p-2}\left(\|\alpha\|^2+\varepsilon^4\|\phi\|^2\right)  dt\,ds
\leq\int_{S^1\times\mathbb R}\|\xi\|^{p-2}\left(\|d_A^*\alpha\|^2+\varepsilon^4\|d_A\phi \|^2\right)ds\\
\leq& c\int_{S^1\times\mathbb R}G^p ds+\delta \int_{S^1\times\mathbb R } \|\xi\|^{p} dt\,ds
\end{split}
\end{equation*}
and thus the lemma is proved.
\end{proof}

\begin{proof}[Proof of theorem \ref{flow:thm:wavefin}]
By lemma \ref{flow:lemma:lpl2}, for any $\delta>0$ there is a positive constant $c_0$ such that
\begin{align*}
 \|\alpha\|_{L^p}^p&+\varepsilon^{2p}\|\phi\|_{L^p}^p\leq \delta\left( \|d_A\alpha\|_{L^p}^p+\|d_A^*\alpha\|_{L^p}^p+\varepsilon^{2p}\|d_A\phi\|_{L^p}^p\right)\\
&+c_0\int_{S^1\times \mathbb R}\left( \|\alpha\|_{L^2}^p+\varepsilon^{2p}\|\phi\|_{L^2}^p\right) dt\, ds\\
\intertext{since $\alpha=d_A\gamma$ and the connection is flat on $\Sigma$, $\|d_A\alpha\|_{L^p}$ vanishes. By the lemma \ref{flow:lemma:lplppia2} we have then}
\leq& \delta\left(\|d_A^*\alpha\|_{L^p}^p+\varepsilon^{2p}\|d_A\phi\|_{L^p}^p\right)\\
&+ c_0c_1\int_{S^1\times\mathbb R} \left(\|\varepsilon^2\partial_s\alpha-\varepsilon^2d_A\phi \|^p_{L^2(\Sigma) } +\varepsilon^p \|\varepsilon^2\partial_s\psi-\varepsilon^2\partial_t\phi \|^p_{L^2(\Sigma) }\right)dt\,ds\\
&+\int_{S^1\times\mathbb R} \left(c_0c_1\varepsilon^{2p} \left\|\varepsilon^2\partial_s\phi-\frac 1{\varepsilon^2}d_A^*\alpha+\partial_t\psi \right\|^p_{L^2(\Sigma) }+\delta  \|\xi\|^{p}_{L^2(\Sigma)} \right)dt\,ds\\
\intertext{and by the H\"older's inequality with $c_2=\left(\int_\Sigma \mathrm{dvol}_\Sigma\right)^{\frac {p-2}p}$}
\leq& \delta\left(\|d_A^*\alpha\|_{L^p}^p+\varepsilon^{2p}\|d_A\phi\|_{L^p}^p+c_2\|\xi\|_{L^p}^p\right)+ c_0c_1c_2\varepsilon^{2p}\|\partial_s\alpha-d_A\phi \|^p_{L^2}\\
& +c_0c_1c_2\varepsilon^{3p} \|\partial_s\psi-\partial_t\phi \|^p_{L^p}
+c_0c_1c_2\varepsilon^{4p} \left\|\partial_s\phi-\frac 1{\varepsilon^4}d_A^*\alpha+\frac 1{\varepsilon^2}\partial_t\psi \right\|^p_{L^p}\\
\intertext{since $\|\Psi\|_{L^\infty}+\|\Phi\|_{L^\infty}\leq c_3$, for $c_0c_1c_3\varepsilon^p\leq\delta$}
\leq& \delta\left(\|d_A^*\alpha\|_{L^p}^p+\varepsilon^{2p}\|d_A\phi\|_{L^p}^p+2c_2\|\xi\|_{L^p}^p\right)+ c_0c_1c_2\varepsilon^{2p}\|\nabla_s\alpha-d_A\phi \|^p_{L^2}\\
& +c_0c_1c_2\varepsilon^{3p} \|\nabla_s\psi-\nabla_t\phi \|^p_{L^p}
+c_0c_1c_2\varepsilon^{4p} \left\|\nabla_s\phi-\frac 1{\varepsilon^4}d_A^*\alpha+\frac 1{\varepsilon^2}\nabla_t\psi \right\|^p_{L^p}.
\end{align*}
Therefore, the theorem follows from the lemma \ref{flow:cor:wave} and the last estimate choosing $\delta$ small enough.
\end{proof}

%
%
\section{Quadratic estimates}\label{flow:section:qest}

In this section we prove the following quadratic estimates.
\begin{lemma}\label{flow:lemma:qe1}
For any $c_0>0$ there are two positive constants $c$ and $\varepsilon_0$ such that, for any $0<\varepsilon<\varepsilon_0$, the  following holds. If two connections $\Xi=A+\Psi dt+\Phi ds$, $\bar\Xi=\bar A+\bar\Psi dt+\bar \Phi ds\in W^{1,p}$, with $\bar\alpha+\bar\psi dt+\bar \phi ds:= \Xi-\bar\Xi$, satisfies $\|\bar\alpha+\bar\psi dt\|_{\infty,\varepsilon }\leq c_0$, then
\begin{equation*}
\begin{split}
\varepsilon^2\big\|\big(\mathcal D^\varepsilon&(\Xi)-\mathcal D^\varepsilon(\bar\Xi)    \big)(\alpha+\psi dt+ \phi ds)\big\|_{0,p,\varepsilon } \\
\leq&c \|\bar\alpha+\bar\psi dt+\bar \phi ds\|_{\infty,\varepsilon }\|\alpha+\psi dt+ \phi ds\|_{1,p,\varepsilon }\\
&+c \|\alpha+\psi dt+ \phi ds\|_{\infty,\varepsilon }\|\bar\alpha+\bar\psi dt\|_{1,p,\varepsilon },
\end{split}
\end{equation*}
\begin{equation*}
\begin{split}
\varepsilon^2\big\|\big(\mathcal D^\varepsilon&(\Xi)-\mathcal D^\varepsilon(\bar\Xi)    \big)(\alpha+\psi dt+ \phi ds)\big\|_{0,p,\varepsilon } \\
\leq&c \|(1-\pi_A)\bar\alpha+\bar\psi dt+\bar \phi ds\|_{\infty,\varepsilon }\|\alpha+\psi dt+ \phi ds\|_{1,p,\varepsilon }\\
&+c \|\bar\alpha+\bar\psi dt+\bar \phi ds\|_{\infty,\varepsilon }\|(1-\pi_A)\alpha+\psi dt+ \phi ds\|_{1,p,\varepsilon }\\
&+c \|\alpha+\psi dt+ \phi ds\|_{0,p,\varepsilon }\left( \|d_A\bar\alpha\|_{L^\infty}+\|d_A^*\bar\alpha\|_{L^\infty}+\varepsilon\|\nabla_t\bar\alpha\|_{L^\infty}\right)\\
&+c \|\alpha+\psi dt+ \phi ds\|_{0,p,\varepsilon }\left( \varepsilon\|d_A\bar\psi\|_{L^\infty}+\varepsilon^2\|\nabla_t\bar\psi \|_{L^\infty}\right),
\end{split}
\end{equation*}
\begin{equation*}
\begin{split}
\varepsilon^2\big|\big|\pi_A\big(\mathcal D^\varepsilon&(\Xi)-\lambda *[\alpha\wedge \omega(A)]-\mathcal D^\varepsilon(\bar \Xi)    \big)(\alpha+\psi dt+ \phi ds)\big|\big|_{L^p} \\
\leq&c \|\bar\alpha+\bar\psi dt+\bar \phi ds\|_{\infty,\varepsilon }\|(1-\pi_A)\alpha+\psi dt+ \phi ds\|_{1,p,\varepsilon }+c\varepsilon^2\|\bar\alpha\|_{L^\infty}\|\alpha\|_{L^p}\\
&+c\varepsilon^2 \|\psi dt\|_{L^p }\|\nabla_t\bar\alpha\|_{L^\infty}+c\varepsilon^2\|\bar\psi\|_{L^\infty}\|\nabla_t\alpha\|_{L^p}+c\varepsilon^2\|\nabla_t\bar\psi\|_{L^\infty}\|\alpha\|_{L^p}\\
&+c\varepsilon^2\left( \|\bar\phi\|_{L^\infty }+\|\bar\alpha\|_{L^\infty}^2\right) \|\pi_A(\alpha) \|_{L^p}+ c \|\alpha\|_{L^\infty}\|d_A\bar\alpha-\lambda\varepsilon^2\omega(A)\|_{L^p},
\end{split}
\end{equation*}
for any $\alpha+\psi dt+ \phi ds\in W^{1,2;p}$ and where $\lambda\in \{0,1\}$.
\end{lemma}

\begin{proof}
The lemma can be proved directly estimating term by term the following identities.
\begin{equation*}
\begin{split}
\big(\mathcal D^\varepsilon_1&(\Xi)-\mathcal D^\varepsilon_1(\bar \Xi)   \big)(\alpha+\psi dt+ \phi ds)=[\bar\phi, \alpha]-[\bar\alpha,\phi]-\frac 1{\varepsilon^2 }[\bar \alpha,*d_A\alpha+[\bar\alpha\wedge \alpha]]\\
&+\frac 1{\varepsilon^2 }*\left[\alpha, *\left(d_A\bar\alpha+\frac12\left[\bar\alpha\wedge\bar\alpha\right]\right)\right]
+\frac 1{\varepsilon^2}d_A^*[\bar\alpha\wedge\alpha]-\left[\bar\psi,(\nabla_t\alpha+[\bar\psi,\alpha])\right]\\
&+[\bar\alpha,(\nabla_t\psi+[\bar\psi,\psi])] +d_A[\bar\psi,\psi]-2[\psi,(\nabla_t\bar\alpha-d_A\bar\psi-[\bar \alpha,\bar\psi])]\\
&-\nabla_t[\bar\psi,\alpha]-(d*X_t(A)-d*X_t(\bar A))\alpha,
\end{split}
\end{equation*}
\begin{equation*}
\begin{split}
\pi_A\big(\mathcal D^\varepsilon_1&(\Xi)-\frac 1{\varepsilon^2}*[\alpha,* \omega(A)]-\mathcal D^\varepsilon_1(\bar \Xi)   \big)(\alpha+\psi dt+ \phi ds)
=\pi_A\Big([\bar\phi, \alpha]-[\bar\alpha,\phi]\\
&-\frac 1{\varepsilon^2 }[\bar \alpha,*d_A\alpha+[\bar\alpha\wedge \alpha]]-\left[\bar\psi,(\nabla_t\alpha+[\bar\psi,\alpha])\right]\\
&+\frac 1{\varepsilon^2 }*\left[\alpha, *\left(d_A\bar\alpha-\lambda\varepsilon^2\omega(A)+\frac12\left[\bar\alpha\wedge\bar\alpha\right]\right)\right]+[\bar\alpha,(\nabla_t\psi+[\bar\psi,\psi])] \\
&-2[\psi,(\nabla_t\bar\alpha-d_A\bar\psi-[\bar \alpha,\bar\psi])]-\nabla_t[\bar\psi,\alpha]-(d*X_t(A)-d*X_t(\bar A))\alpha\Big),
\end{split}
\end{equation*}
\begin{equation*}
\begin{split}
\big(\mathcal D^\varepsilon_2&(\Xi)-\mathcal D^\varepsilon_2(\bar \Xi)   \big)(\alpha+\psi dt+ \phi ds)=[\bar\phi, \psi]-[\bar\psi,\phi]-\frac1 {\varepsilon^2 }*\nabla_t[\bar\alpha\wedge*\alpha]\\
&-\frac 2{\varepsilon^2 }*[\alpha,*(\nabla_t\bar\alpha-d_A\bar\psi-[\bar \alpha,\bar\psi])]+\frac 1{\varepsilon^2}[\bar\psi, (d_A^*\alpha-*[\bar\alpha\wedge*\alpha])]\\
&-*\frac 1{\varepsilon^2}[\bar\alpha\wedge *(d_A\psi+[\alpha,\psi])]+\frac 1{\varepsilon^2}d_A^*[\alpha,\psi],
\end{split}
\end{equation*}
\begin{equation*}
\big(\mathcal D^\varepsilon_3(\Xi)-\mathcal D^\varepsilon_3(\bar \Xi)   \big)(\alpha+\psi dt+ \phi ds)
=[\bar\phi, \phi]+*\frac 1{\varepsilon^4}[\bar\alpha\wedge*\alpha]+\frac 1{\varepsilon^2}[\bar\psi,\psi].
\end{equation*}

\end{proof}

We choose a connection $\Xi=A+\Psi dt+\Phi ds\in W^{1,p}$ and a 1-form $\xi=\alpha+\psi dt+\phi ds\in W^{1,2;p}$, then
$\mathcal F^\varepsilon(\Xi+\xi)
=\mathcal F^\varepsilon(\Xi)+\mathcal D^\varepsilon(\Xi)(\xi)
+\mathcal C^\varepsilon(\Xi)(\xi)
$
and we denote by $\mathcal C_1^\varepsilon(\Xi)$, $\mathcal C_2^\varepsilon(\Xi)$ and $\mathcal C_3^\varepsilon(\Xi)$ the three components of $\mathcal C^\varepsilon(\Xi)=\mathcal C_1^\varepsilon(\Xi)+\mathcal C_2^\varepsilon(\Xi)\, dt+\mathcal C_3^\varepsilon(\Xi)\, ds$; in this case we have the following estimates.

\begin{lemma}\label{flow:lemma:qe2}
For any $c_0>0$ there are constants $c>0$ and $\varepsilon_0>0$ such that, for any $0<\varepsilon<\varepsilon_0$,
\begin{equation}
\varepsilon^2\left\|\mathcal C^\varepsilon(\Xi)(\xi)\right\|_{0,p,\varepsilon } \leq c \|\xi\|_{\infty,\varepsilon }\|\xi\|_{1,p,\varepsilon }
\end{equation}
\begin{equation}
\begin{split}
\varepsilon^2\left\|\pi_A\left(\mathcal C_1^\varepsilon(\Xi)(\xi)\right)\right\|_{0,p,\varepsilon } 
\leq &c \|(1-\pi_A)\xi\|_{\infty,\varepsilon }\left( \|(1-\pi_A)\alpha+\psi dt\|_{1,p,\varepsilon }+\varepsilon\|\nabla_t\alpha \|_{L^p}\right)\\
&+c \|\pi_A\alpha \|_{L^\infty }\|(1-\pi_A)\alpha+\psi dt\|_{1,p,\varepsilon }\\
&+c \|\pi_A\alpha \|_{L^\infty }( \|\pi_A\alpha \|_{L^\infty }+\varepsilon^2) \| \pi_A\alpha \|_{L^p }
\end{split}
\end{equation}
for any $\xi:=\alpha+\psi dt+ \phi ds\in W^{1,2;p}$ and where we assume that $\|\alpha+\psi dt\|_{\infty,\varepsilon }\leq c_0$.
\end{lemma}

\begin{proof}Also this lemma can be showed estimating term by term the identities:
\begin{equation*}
\begin{split}
\mathcal C_1^\varepsilon(\Xi)&(\xi)=[\phi,\alpha]-[\alpha,\phi]+\frac1{\varepsilon^2}d_A^*[\alpha\wedge\alpha]\\
&-\frac 1{\varepsilon^2}*[\alpha\wedge*(d_A\alpha+[\alpha\wedge\alpha])]+\frac1{\varepsilon^2}*\left[\alpha\wedge*\left(d_A\alpha+\frac12[\alpha\wedge\alpha]\right)\right]\\
&-[\psi,(\nabla_t\alpha+[\psi,\alpha])]-2[\psi,(\nabla_t\alpha-d_A\psi-[\alpha,\psi])]\\
&-\nabla_t[\psi,\alpha]+[\alpha,\nabla_t\psi]-(*X_t(A+\alpha)-*X_t(A)-d*X_t(A)\alpha),
\end{split}
\end{equation*}
\begin{equation}\label{xvasvdb}
\begin{split}
\pi_A\big(\mathcal C_1^\varepsilon(\Xi)&(\xi)\big)=\pi_A\Big( [\phi,\alpha]-[\alpha,\phi]-[\psi,(\nabla_t\alpha+[\psi,\alpha])]\\
&-\frac 1{\varepsilon^2}*[\alpha\wedge*(d_A\alpha+[\alpha\wedge\alpha])]+\frac1{\varepsilon^2}*\left[\alpha\wedge*\left(d_A\alpha+\frac12[\alpha\wedge\alpha]\right)\right]\\
&-2[\psi,(\nabla_t\alpha-d_A\psi-[\alpha,\psi])]-\nabla_t[\psi,\alpha]+[\alpha,\nabla_t\psi]\\
&-(*X_t(A+\alpha)-*X_t(A)-d*X_t(A)\alpha)\Big),
\end{split}
\end{equation}
\begin{equation}\label{qestnvdsopapa}
\begin{split}
\mathcal C_2^\varepsilon(\Xi)&(\xi)=[\phi,\psi]-[\psi,\phi]
+\frac 2{\varepsilon^2}*[\alpha\wedge*(\nabla_t\alpha-d_A\psi-[\alpha,\psi])]\\
&-\frac1{\varepsilon^2}[\psi,d_A^*\alpha]-\frac 1{\varepsilon^2}*[\alpha\wedge *(d_A\psi+[\alpha,\psi])]+\frac 1{\varepsilon^2}d_A^*[\alpha,\psi],
\end{split}
\end{equation}
\begin{equation*}
\mathcal C_3^\varepsilon(\Xi)(\alpha+\psi dt+ \phi ds)=0.
\end{equation*}
\end{proof}

%
%

\section{The map $\mathcal K^{\varepsilon}_2$: A first approximation}\label{flow:section:firstapprox}

In the next section, we will construct, using a Newton's iteration, a perturbed Yang-Mills flow $\Xi^\varepsilon\in\mathcal M^\varepsilon\left(\mathcal T^{\varepsilon, b}(\Xi_-),\mathcal T^{\varepsilon, b}(\Xi_+)\right)$ for any perturbed geodesic flow $\Xi^0\in\mathcal M^0\left(\Xi_-,\Xi_+\right)$ and every pair of connections $\Xi_\pm\in \mathrm{Crit}^b_{E^H}$with index difference 1, where $b$ is a regular value of $E^H$. For this purpose, we need to define a connection $\mathcal K^\varepsilon_2(\Xi^0)$ which is an approximate solution of the perturbed Yang-Mills flow equation. $\mathcal K^\varepsilon_2(\Xi^0)$ is constructed in two steps: First, we add to $\Xi_0$ a $1$-form $\alpha_0^\varepsilon(s)+\psi_0^\varepsilon(s) dt$ which satisfies the limit conditions $\lim_{s\to \pm\infty}\Xi^0+\alpha_0^\varepsilon+\psi_0^\varepsilon dt=\mathcal T^{\varepsilon,b}(\Xi_\pm)$ and then we add a second 1-form in order to have an approximated solution of the Yang-Mills flow equations. \\

First, we recall that a connection $\Xi^0:=A^0+\Psi^0 dt+\Phi^0 ds$ descends to a flow line between two perturbed geodesics $\Xi_\pm=A_\pm+\Psi_\pm dt$ when it satisfies the equations (\ref{flow:eqgeoflow}) and (\ref{flow:eqgeoflow333}), i.e.
\begin{equation*}
\begin{split}
\partial_s A^0-d_{A^0}\Phi^0-\pi_{A^0}\left(\nabla_t(\partial_t A^0-d_{A^0}\Psi^0)+*X_s(A^0)\right)=0,\\
d_{A^0}^*(\partial_t A^0-d_{A^0}\Psi^0)=d_{A^0}^*(\partial_s A^0-d_{A^0}\Phi^0)=0.
\end{split}
\end {equation*}
Next, we choose a smooth positive function $\theta$ such that $\theta(s)=0$ for $s\leq1$, $\theta(s)=1$ when $s\geq2$, $0\leq \theta\leq 1$ and $0\leq \partial_s\theta\leq c_0$ with $c_0>0$. Thus, we define $\alpha_0^\varepsilon+\psi_0^\varepsilon dt$ as  
\begin{equation}\label{flow:k2:eqq1}
\begin{split}
\alpha_0^\varepsilon(s)+\psi_0^\varepsilon(s) dt:=&\theta(-s) g(s)^{-1}(\mathcal T^{\varepsilon,b }(A_-+\Psi_- dt)-(A_-+\Psi_- dt))g(s)\\
&+\theta(s) g(s)^{-1}(\mathcal T^{\varepsilon,b }(A_++ \Psi_+ dt)-(A_++\Psi_+ dt))g(s),
\end{split}
\end{equation}
where $g\in\mathcal G_0(P\times S^1\times \mathbb R)$ satisfies 
\begin{equation}\label{flow:k2:eqq2}
g^{-1}\partial_sg=\Phi^0, \quad\lim_{s\to -\infty}g= 1;
\end{equation}
we introduce $g$ in order to make the definition of $\Xi^0+\alpha_0^\varepsilon+\psi_0^\varepsilon dt$ gauge-invariant. 

\begin{lemma}\label{flow:lemma:firstappr}
We choose two constants $b > 0$, $p>2$. There are positive constants  $\varepsilon_0, c$ such  that the following holds. For every $\varepsilon \in (0,\varepsilon_0)$, every pair $\Xi_\pm:=A_\pm+\Psi_\pm dt \in \mathrm {Crit}^b_{E^H}$ that are perturbed closed geodesics of index difference one, there exists a unique equivariant map
\begin{equation*}
\mathcal K_2^\varepsilon: \mathcal M^0 (\Xi_-, \Xi_+) \to W^{1,2;p}\left(\mathcal T^{\varepsilon,b }(\Xi_-), \mathcal T^{\varepsilon,b}(\Xi_+)\right)
\end{equation*}
such that for any $\Xi^0:=A^0+\Psi^0 dt+\Phi^0 ds \in \mathcal M^0 (\Xi_-, \Xi_+)$, with $\alpha_0^\varepsilon+\psi_0^\varepsilon dt$ and $g$ defined as in (\ref{flow:k2:eqq1}) and in (\ref{flow:k2:eqq2}),
\begin{equation}\label{flow:k2:leqw}
\mathcal K_2^\varepsilon(\Xi^0)-(\Xi^0+\alpha_0^\varepsilon+\psi^\varepsilon_0 dt) \in \textrm{im } d_{A^0}^*
\end{equation}
and
\begin{equation}\label{flow:k2:leqw2}
\begin{split}
\frac 1{\varepsilon^2}d_{A^0}^*&d_{A^0}\left(\mathcal K_2^\varepsilon(\Xi^0)-(\Xi^0+\alpha_0^\varepsilon+\psi^\varepsilon_0 dt)\right)\\
=&d_{A^0}^*\left(d_{A^0}d_{A^0}^*\right)^{-1}d_{A^0}\left(\nabla_t(\partial_t A^0-d_{A^0}\Psi^0)+*X_t(A^0)\right)\\
&-\theta(-s) d_{A^0}^*d_{A^0}g^{-1}\left(d_{A_-}d_{A_-}^*\right)^{-1}\left(\nabla_t^{\Psi_-}\left(\partial_tA_--d_{A_-}\Psi_-\right)+*X_t(A_-)\right)g\\
&-\theta(s) d_{A^0}^*d_{A^0}g^{-1}\left(d_{A_+}d_{A_+}^*\right)^{-1}\left(\nabla_t^{\Psi_+}\left(\partial_tA_+-d_{A_+}\Psi_+\right)+*X_t(A_+)\right)g\\
&-\theta(-s)d_{A^0}^*d_{A^0}\pi_{g^{-1}(A_-)g}(\alpha_0^\varepsilon)
-\theta(s)d_{A^0}^*d_{A^0}\pi_{g^{-1}(A_+)g}(\alpha_0^\varepsilon)
\end{split}
\end{equation}
In addition, it satisfies
\begin{equation}\label{flow:k2:leqw3}
\|\mathcal K_2^\varepsilon(\Xi^0)-(\Xi^0+\alpha_0^\varepsilon+\psi^\varepsilon_0 dt)\|_{1,2;p,1}\leq c\varepsilon^2,
\end{equation}
\begin{equation}\label{flow:k2:leqw4}
\|\mathcal F_1^\varepsilon(\mathcal K_2^\varepsilon(\Xi^0))\|_{L^p}\leq c\varepsilon^{2},\quad \|\mathcal F_2^\varepsilon(\mathcal K_2^\varepsilon(\Xi^0))\|_{L^p}\leq c.
\end{equation}
\end{lemma}
\begin{figure}[ht]
\begin{center}
\begin{tikzpicture} 
\filldraw (0,0) circle (1.5pt) node[below,left] {$\Xi_-$}
(6,0) circle (1.5pt) node[below] {$\Xi_+$}
(0,2) circle (1.5pt) node[above, left] {$\Xi_-^\varepsilon$}
(6,1) circle (1.5pt) node[above] {$\Xi_+^\varepsilon$}; 
\draw (0,0) .. controls (1,1) and (2,0.5) .. (3,0); 
\draw (3,0) .. controls (4,-0.5)  and (5,-0.5) .. node[near start,below] {$\Xi^0$}(6,0); 
 
\draw (0,0) .. controls (0,1) and (0,1) .. node[left] {$\alpha_{-}^\varepsilon+\psi_-^\varepsilon dt$}  (0,2); 
\draw (6,0) .. controls (6,0.5) and (6,0.5) .. node[right] {$\alpha_{+}^\varepsilon+\psi_+^\varepsilon dt$}  (6,1); 

\draw[blue] (0,2) .. node[above] {$\Xi^0+\alpha_0^\varepsilon+\phi_0^\varepsilon dt$} controls (1,3) and (2,2.5) ..(2.5,2.2); 
\draw[blue] (2.5,2.2) .. controls (3,1.9) and (2.7,0) ..  (3,0); 
\draw[blue] (3,0) .. controls (3.3,0) and (3,1.1) .. (3.5,0.9); 
\draw[blue] (3.5,0.9) .. controls (4,0.6)  and (5,0.5) .. (6,1); 

\draw[magenta] (0,2) .. controls (1,2.8) and (2,2.5) ..(3,1.5); 
\draw[magenta] (3,1.5) .. controls (4,0.6)  and (5,0.7) ..node[above] {$\mathcal K_2^\varepsilon(\Xi^0)$} (6,1); 
\end{tikzpicture}
\caption{Situation of the lemma \ref{flow:lemma:firstappr}.}
\end{center}
\end{figure}
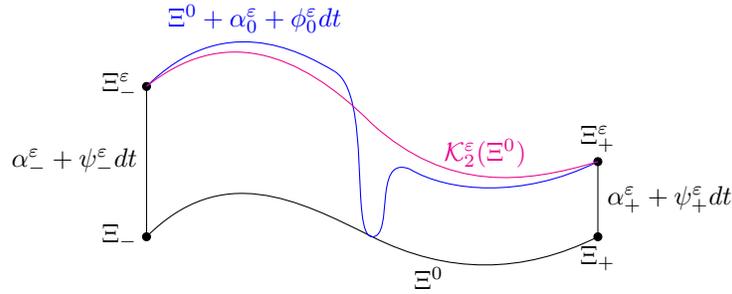
%

%

The critical points of our two functionals $E^H$ und $\mathcal {YM}^{\varepsilon,H}$ do not coincide and therefore the difference $\Xi-\Xi^0$ between a geodesic flow $\Xi^0\in\mathcal M^0\left(\Xi_-,\Xi_+\right)$ and any connection $\Xi$ which converges to $\mathcal T^{\varepsilon, b}(\Xi_\pm)$ as $s$ goes to $\pm\infty$ can not be estimates using the norm defined in the subsection \ref{flow:subsection:norm}. We need therefore another reference connection in order to compute the norms and for this purpose we use $\mathcal K^\varepsilon_2(\Xi^0)$. 

\begin{proof}[Proof of lemma \ref{flow:lemma:firstappr}]
$\mathcal K_2^\varepsilon(\Xi^0)$ is uniquely defined by (\ref{flow:k2:leqw}) and (\ref{flow:k2:leqw2}) because $F_{A^0}=0$ and 
$$d_{A^0}^*d_{A^0}:\textrm{im } d_{A^0}^*\Omega^2(\Sigma,\mathfrak g_P)\to \textrm{im } d_{A^0}^*\Omega^2(\Sigma,\mathfrak g_P)$$
is bijective. Furthermore, the lemma \ref{lemma76dt94}, the commutation formulas (\ref{commform}), (\ref{commform2}) and the estimates of the geodesic flow (\ref{flow:web:w1})-(\ref{flow:web:w4}) yield to (\ref{flow:k2:leqw3}). Therefore we need only to prove (\ref{flow:k2:leqw4}).\\

We define $\Xi_1^\varepsilon:=A_1^\varepsilon+\Psi_1^\varepsilon dt+\Phi_1^\varepsilon ds:= \Xi^0+\alpha_0^\varepsilon+\psi_0^\varepsilon dt$ and we consider
\begin{equation}
\begin{split}
A_\pm(s)+\Psi_\pm(s) dt=& g(s)^{*}(A_\pm+\Psi_\pm dt),\\
\alpha(s)+\psi(s) dt=& \theta(-s) ((A^0(s)+\Psi^0(s) dt)-(A_-(s)+\Psi_-(s) dt))\\
&+\theta(s) ((A^0(s)+\Psi^0(s) dt)-(A_+(s)+\Psi_+(s)dt))
\end{split}
\end{equation}
and
\begin{equation}
\begin{split}
\bar\alpha_1^\varepsilon:=&\theta(-s)\left(d_{A_-}d_{A_-}^*\right)^{-1}\left(\nabla_t^{\Psi_-}\left(\partial_tA_--d_{A_-}\Psi_-\right)+*X_t(A_-)\right)\\
&\theta(s)\left(d_{A_+}d_{A_+}^*\right)^{-1}\left(\nabla_t^{\Psi_+}\left(\partial_tA_+-d_{A_+}\Psi_+\right)+*X_t(A_+)\right).
\end{split}
\end{equation}
Furthermore, we consider $A_2^\varepsilon+\Psi_2^\varepsilon dt+\Phi_2^\varepsilon ds:= \Xi_1^\varepsilon+\alpha_1^\varepsilon:=\mathcal K^\varepsilon_2(\Xi^0)$. If we look at the expansion
\begin{equation}\label{flow:eq:k2:w}
\begin{split}
\mathcal F^\varepsilon_1\left(\mathcal K_2^\varepsilon(\Xi^0)\right)=& \partial_s A_2^\varepsilon-d_{A_2^\varepsilon}\Phi_2^\varepsilon+\frac 1{\varepsilon^2}d_{A_2^\varepsilon}^*F_{A_2^\varepsilon}\\
&-\nabla_t^{\Psi_2^\varepsilon}(\partial_t A_2^\varepsilon-d_{A_2^\varepsilon}\Psi_2^\varepsilon)-*X_t(A_2^\varepsilon)\\
=& \partial_s A^0-d_{A^0}\Phi^0-\pi_{A^0}\left(\nabla_t(\partial_t A^0-d_{A^0}\Psi^0)+*X_s(A^0)\right)\\
&+\left(\nabla_s-\nabla_t^2\right) \alpha_1^\varepsilon-*X_t(A^2)+*X_t(A^0+\alpha_0^\varepsilon)\\
&-\frac 1{\varepsilon^2}*\left[\alpha_1^\varepsilon\wedge *\left(d_{A^0}\alpha_1^\varepsilon+\frac 12[\alpha_1^\varepsilon\wedge \alpha_1^\varepsilon]\right)\right]+\frac 1{2\varepsilon^2}d_{A^0}^*[\alpha_1^\varepsilon\wedge\alpha_1^\varepsilon]\\
&+\frac 1{\varepsilon^2}d_{A^0}^*d_{A^0}\alpha_1^\varepsilon- (1-\pi_{A_0})\left(\nabla_t(\partial_t A^0-d_{A^0}\Psi^0)+*X_t(A^0)\right)\\
&+F_1^\varepsilon+F_2^\varepsilon+F_3^\varepsilon
\end{split}
\end{equation}
we remark that the second line vanishes because $\Xi^0$ is a geodesic flow, the third and the fourth can be estimates by $c\varepsilon^2$ because the $\|\cdot\|_{1,2;p,1}$-norm of $\alpha_1^\varepsilon$ can be estimated by the same factor by (\ref{flow:k2:leqw3}). The fifth line of (\ref{flow:eq:k2:w}) can be written by the definitions as $$-\frac1{\varepsilon^2}d_{A^0}^*d_{A^0}\bar\alpha_1^\varepsilon-\theta(-s)d_{A^0}d_{A^0}\pi_{A_-}(\alpha_0^\varepsilon)-\theta(s)d_{A^0}d_{A^0}\pi_{A_+}(\alpha_0^\varepsilon).$$
Therefore we have that
\begin{align*}
\left\|\mathcal F^\varepsilon_1\left(\mathcal K_2^\varepsilon(\Xi^0)\right)\right\|_{L^p}\leq& c\varepsilon^2+\left\|\frac1{\varepsilon^2}d_{A^0}^*d_{A^0}(\bar\alpha_1^\varepsilon+\pi_{A_-}(\alpha_0^\varepsilon)+\pi_{A_+}(\alpha_0^\varepsilon))-F_1^\varepsilon\right\|_{L^p}\\
&+\left\|F^\varepsilon_2\right\|_{L^p}+\left\|F^\varepsilon_3\right\|_{L^p}.
\end{align*}
For this purpose, we need to investigate $F_1^\varepsilon$, $F_2^\varepsilon$ and $F_3^\varepsilon$. In order to simplify the exposition we evaluate $F_i^\varepsilon(s)$ for $s\leq0$; for $s>0$ the computation is the same, we only need to substitute $A_-+\Psi_- dt$ with $A_++\Psi_+ dt$ and $\theta(-s)$ with $\theta(s)$. If we denote $\partial_tA_--d_{A_-}\Psi_-$ by $B_t^-$, we have
\begin{align*}
F_1^\varepsilon:=&\frac 1{\varepsilon^2} d_{A_-}^*\left(d_{{A_-}}\alpha_0^\varepsilon+\frac 12[\alpha_0^\varepsilon\wedge\alpha_0^\varepsilon]\right)-[\psi_0^\varepsilon,B_t^-]\\
&-\frac 1{\varepsilon^2} *\left [\alpha_0^\varepsilon\wedge\left(d_{{A_-}}\alpha_0^\varepsilon+\frac 12[\alpha_0^\varepsilon\wedge\alpha_0^\varepsilon]\right)\right]\\
&-[\psi_0^\varepsilon,(\nabla_t^{\Psi_-}\alpha_0^\varepsilon-d_{A_-}\psi_0^\varepsilon-[\alpha_0^\varepsilon,\psi_0^\varepsilon])]\\
&-\nabla_t^{\Psi_-}\left(\nabla_t^{\Psi_-}\alpha_0^\varepsilon-d_{{A_-}}\psi_0^\varepsilon-[\alpha_0^\varepsilon,\psi_0^\varepsilon]\right)\\
&-*X_t({A_-}+\alpha_0^\varepsilon)+*X_t({A_-})\\
& -\frac 1{\varepsilon^2}*\left[\alpha\wedge *d_{{A_-}}\alpha_0^\varepsilon\right]
+\frac 1{\varepsilon^2}d_{A^0}^*[\alpha\wedge \alpha_0^\varepsilon]
\end{align*}
and since $\mathcal T^{\varepsilon,b}(\Xi_-)$ is a perturbed Yang-Mills connection and by the definition of $\alpha_0^\varepsilon+\psi_0^\varepsilon dt$ we get that
\begin{equation*}
\begin{split}
0=&\theta(-s)\left(-\nabla_t^{\Psi_-}B_t^--*X_t(A^-)\right)+
\frac 1{\varepsilon^2} d_{A_-}^*\left(d_{{A_-}}\alpha_0^\varepsilon+\frac 12[\alpha_0^\varepsilon\wedge\alpha_0^\varepsilon]\right)\\
&-\frac 1{\varepsilon^2} *\left [\alpha_0^\varepsilon\wedge\left(d_{{A_-}}\alpha_0^\varepsilon+\frac 12[\alpha_0^\varepsilon\wedge\alpha_0^\varepsilon]\right)\right]-[\psi_0^\varepsilon,B_t^-]\\
&-[\psi_0^\varepsilon,(\nabla_t^{\Psi_-}\alpha_0^\varepsilon-d_{{A_-}}\psi_0^\varepsilon-[\alpha_0^\varepsilon,\psi_0^\varepsilon])]-\nabla_t^{\Psi_-}\left(\nabla_t^{\Psi_-}\alpha_0^\varepsilon-d_{A_-}\psi_0^\varepsilon-[\alpha_0^\varepsilon,\psi_0^\varepsilon]\right)\\
&-*X_t({A_-}+\alpha_0^\varepsilon)+*X_t({A_-})+F_{[-2,0]}^\varepsilon
\end{split}
\end{equation*}
where $F_{[-2,0]}^\varepsilon$ contain only quadratic terms with support in $[-2,0]$. Therefore we can write $F_1^\varepsilon$ in the following way by the definition of $\bar\alpha_1^\varepsilon$
\begin{align*}
F_1^\varepsilon=&\theta(-s)\left(\nabla_t^{\Psi_-}B_t^-+*X_t(A^-)\right)+F_{[-2,0]}^\varepsilon\\
& -\frac 1{\varepsilon^2}*\left[\alpha\wedge *d_{{A_-}}\alpha_0^\varepsilon\right]
+\frac 1{\varepsilon^2}d_{A^0}^*[\alpha\wedge \alpha_0^\varepsilon]\\
=&\frac 1{\varepsilon^2} d_{A^0}^*d_{A^0}(\bar\alpha_1^\varepsilon+\pi_{A_-}(\alpha_0^\varepsilon))+F_{[-2,0]}^\varepsilon\\
& -\frac 1{\varepsilon^2}*\left[\alpha\wedge *d_{{A_-}}(\alpha_0^\varepsilon-\bar\alpha_1^\varepsilon)\right]
+\frac 1{\varepsilon^2}d_{A^0}^*[\alpha\wedge ((1-\pi_{A_-})\alpha_0^\varepsilon-\bar\alpha_1^\varepsilon)]
\end{align*}
The two terms are
\begin{align*}
F_2^\varepsilon:=& -\frac 1{\varepsilon^2}*\left[\alpha\wedge *\frac 12[\alpha_0^\varepsilon\wedge\alpha_0^\varepsilon]\right]\\
&-\left[\psi_0^\varepsilon,\left(\nabla_t^{\Psi_-}\alpha-d_{A^-}\psi-[\alpha,\psi]\right)\right]
-\left[\psi_0^\varepsilon,[\psi,\alpha_0^\varepsilon]-[\alpha,\psi_0^\varepsilon]\right]\\
&-\left[\psi,\left(\nabla_t^{\Psi_-}\alpha_0^\varepsilon-d_{A^0}\psi_0^\varepsilon-[\alpha_0^\varepsilon,\psi_0^\varepsilon]\right)\right]
-\nabla_t\left([\psi,\alpha_0^\varepsilon]-[\alpha,\psi_0^\varepsilon]\right)\\
&-*X_t(A^0+\alpha_0^\varepsilon)+*X_t({A_-}+\alpha_0^\varepsilon)+*X_t(A^0)-*X_t(A_-),
\end{align*}
\begin{align*}
F_3^\varepsilon:=&-\nabla_t[\psi_0^\varepsilon, \alpha_1^\varepsilon]
-[\psi_0^\varepsilon, \nabla_t\alpha_1^\varepsilon]\\
&-\frac 1{\varepsilon^2}*\left[\alpha_0^\varepsilon\wedge *\left(d_{A^0}\alpha_1^\varepsilon+\frac 12[\alpha_1^\varepsilon\wedge \alpha_1^\varepsilon]\right)\right]+\frac 1{\varepsilon^2}d_{A^0}^*[\alpha_0^\varepsilon\wedge\alpha_1^\varepsilon]\\
&-\frac 1{\varepsilon^2}*\left[\alpha_1^\varepsilon\wedge *\left(d_{A^0}\alpha_0^\varepsilon+[\alpha_0^\varepsilon\wedge \alpha_1^\varepsilon]\right)\right].
\end{align*}
\noindent This allows us to conclude that, by the a priori estimates (\ref{flow:web:w1})-(\ref{flow:web:w4}),
\begin{align*}
\left\|\mathcal F^\varepsilon_1\left(\mathcal K_2^\varepsilon(\Xi^0)\right)\right\|_{L^p}\leq & c\varepsilon^2+\left\|\frac1{\varepsilon^2}d_{A^0}^*d_{A^0}\bar\alpha_1^\varepsilon-F_1^\varepsilon\right\|_{L^p}+\left\|F^\varepsilon_2\right\|_{L^p}+\left\|F^\varepsilon_3\right\|_{L^p}
\leq  c\varepsilon^2
\end{align*}
in order to see this we need that 
$$\|d_{{A_-}}(\alpha_0^\varepsilon-\bar \alpha_1^\varepsilon) \|_{L^p(\Sigma\times S^1)}+\|(1-\pi_{A_-})\alpha_0^\varepsilon-\bar \alpha_1^\varepsilon  \|_{L^p(\Sigma\times S^1)}\leq c\varepsilon^{4}$$
which holds by the remark \ref{thm:existence:crit:est}. The second Yang-Mills flow equation can be written as
\begin{equation*}
\begin{split}
\mathcal F_2^\varepsilon\left(\mathcal K_2^\varepsilon(\Xi^0)\right)=&\partial_s\Psi^2-\nabla_t\Phi^2-\frac 1{\varepsilon^2}d_{A^2}^*(\partial_t A^2-d_{A^2}\Psi^2)\\
=&\partial_s\Psi^0-\nabla_t\Phi^0-\frac 1{\varepsilon^2}d_{A^0}^*(\partial_t A^0-d_{A^0}\Psi^0)\\
&+\frac 1{\varepsilon^2}*\left[(\alpha_0^\varepsilon+\alpha_1)\wedge*(\partial_t A^0-d_{A^0}\Psi^0)\right]\\
&+\frac 1{\varepsilon^2}*\left[(\alpha_0^\varepsilon+\alpha_1)\wedge*\left(\nabla_t^{\Psi^0}\alpha_0^\varepsilon-d_{A^0}\psi_0^\varepsilon-[\alpha_0^\varepsilon,\psi_0^\varepsilon]\right)\right]\\
&+\frac 1{\varepsilon^2}*\left[(\alpha_0^\varepsilon+\alpha_1)\wedge *\left(\nabla_t^{\Psi^0}\alpha_1-[\alpha_1,\psi_0^\varepsilon]\right)\right]\\
&+d_{A^0}^*\left((\partial_t A^2-d_{A^2}\Psi^2-(\partial_t A^0-d_{A^0}\Psi^0\right)
\end{split}
\end{equation*}
and therefore by the lemma \ref{flow:lemma:othereq1} and the identity $d_{A^0}^*(\partial_t A^0-d_{A^0}\Psi^0)=0$
\begin{equation*}
\begin{split}
\left\|\mathcal F_2^\varepsilon\left(\mathcal K_2^\varepsilon(\Xi^0)\right)\right\|_{L^p}\leq & \left \|2\left[(\partial_s A^0-d_{A^0}\Phi^0)\wedge *(\partial_t A^0-d_{A^0}\Psi^0\right] \right \|_{L^p}\\
&+ \frac c{\varepsilon^2}\left\| \alpha_0^\varepsilon+\alpha_1\right\|_{2,p,\varepsilon} +\left\|d_{A^0}^*d_{A^0}\psi_0^\varepsilon\right\|_{L^p}\leq c.
\end{split}
\end{equation*}
\end{proof}

\begin{theorem}\label{flow:thm:estk2}
We choose a regular value $b$ of $E^H$, $p>2$, then there are two positive constants $c$ and $\varepsilon_0$ such that the following holds. For any $\Xi^0\in\mathcal M^0(\Xi_-,\Xi_+)$ with $\Xi_\pm\in \mathrm{Crit}^b_{E^H}$ the estimates
\begin{equation}
\begin{split}
\|\pi_A( \xi) \|_{1,2;p,1} \leq &c\varepsilon \|\mathcal D^\varepsilon(\mathcal K_2^\varepsilon(\Xi^0))\xi\|_{0,p,\varepsilon}+c\|\pi_A\mathcal D^\varepsilon(\mathcal K_2^\varepsilon(\Xi^0))\xi\|_{L^p},\\
\|(1-\pi_A)\xi\|_{1,2;p,\varepsilon }\leq &c\varepsilon^2\|\mathcal D^\varepsilon(\mathcal K_2^\varepsilon(\Xi^0))\xi\|_{0,p,\varepsilon}+c\varepsilon \|\pi_A\mathcal D^\varepsilon\mathcal K_2^\varepsilon(\Xi^0))\xi\|_{L^p},\\
\|(1-\pi_A)\alpha \|_{1,2;p,\varepsilon }\leq &c\varepsilon^2\|\mathcal D^\varepsilon(\mathcal K_2^\varepsilon(\Xi^0))\xi\|_{0,p,\varepsilon}+c\varepsilon^2 \|\pi_A\mathcal D^\varepsilon\mathcal K_2^\varepsilon(\Xi^0))\xi\|_{L^p}\\
\end{split}
\end{equation}
hold for all compactly supported $1$-form $\xi=\alpha+\psi dt+\phi ds\in W^{1,2;p} ,\eta \in W^{1,2;p}$, $\xi \in \textrm{im }(\mathcal D^\varepsilon(\mathcal K_2^\varepsilon(\Xi^0)))^* $ and $0<\varepsilon\leq \varepsilon_0$.
\end{theorem}

\begin{proof}
By the theorem \ref{thm:mainthm}, by the remark \ref{thm:existence:crit:est} and by the Sobolev theorem \ref{lemma:sobolev} we have that
$$ \varepsilon\left\|(1-\pi_A)\alpha_0^\varepsilon+\psi_0^\varepsilon dt\right\|_{\infty,\varepsilon}+\|d_A\alpha_0^\varepsilon\|_{L^\infty}+\|d_A^*\alpha_0^\varepsilon\|_{L^\infty}\leq c\varepsilon^{4-\frac 1p},$$
$$\varepsilon\|\psi_0^\varepsilon\|_{L^\infty}+\varepsilon^2\|\nabla_t\psi_0^\varepsilon\|_{L^\infty}\leq c\varepsilon^{3-\frac 1p},\quad\|\nabla_t\alpha_0^\varepsilon\|_{L^\infty}+\|\pi_A(\alpha_0^\varepsilon)\|_{L^\infty}\leq c\varepsilon^{2};$$
in addition, by the previous lemma \ref{flow:lemma:firstappr} we know that
$$\left\|\mathcal K_2^\varepsilon(\Xi^0)-(\Xi^0+\alpha_0^\varepsilon+\psi_0^\varepsilon dt)\right\|_{1,2;p,1}\leq c\varepsilon^2.$$
Thus by the quadratic estimates of the lemma \ref{flow:lemma:qe1}, we obtain
\begin{equation*}
 \begin{split}
  \varepsilon^2&\left\|\left(\mathcal D^\varepsilon \left(\mathcal K_2^\varepsilon\left(\Xi^0\right)\right)-\mathcal D^\varepsilon\left(\Xi^0\right)\right)\xi\right\|_{0,p,\varepsilon} \\
   \leq & \varepsilon^2\left\|\left(\mathcal D^\varepsilon \left(\mathcal K_2^\varepsilon\left(\Xi^0\right)\right)-\mathcal D^\varepsilon\left(\Xi^0+\alpha_0^\varepsilon+\psi_0^\varepsilon dt\right)\right)\xi\right\|_{0,p,\varepsilon}\\
&+ \varepsilon^2\left\|\left(\mathcal D^\varepsilon \left(\Xi^0+\alpha_0^\varepsilon+\psi_0^\varepsilon dt\right)-\mathcal D^\varepsilon\left(\Xi^0\right)\right)\xi\right\|_{0,p,\varepsilon}\\
\leq & c\varepsilon^{2-\frac 1p}\|\xi\|_{1,2;p,\varepsilon},
 \end{split}
\end{equation*}

\begin{equation}\label{flow:k2fhods2}
 \begin{split}
&\left\|\pi_{A^0}\left(\mathcal D^\varepsilon \left(\mathcal K_2^\varepsilon\left(\Xi^0\right)\right)\xi-\mathcal D^\varepsilon\left(\Xi^0\right)\xi-*[\alpha,*\omega(A^0)] \right)\right\|_{0,p,\varepsilon}
\leq c\varepsilon^{1-\frac 1p}\|\xi\|_{1,2;p,\varepsilon}
 \end{split}
\end{equation}
where we used that
$$\omega(A^0)= d_{A^0}\left(d_{A^0}^*d_{A^0}\right)^{-1}\left(\nabla_t\left(\partial_tA^0-d_{A^0}\Psi^0\right)+*X_t(A^0)\right)$$
and, with the notation of the proof of theorem \ref{flow:lemma:firstappr},
\begin{align*}
&\left\|\frac 1{\varepsilon^2}d_{A^0}\left(\mathcal K_2^\varepsilon(\Xi^0)-\Xi^0\right)-\omega(A^0)\right\|_{L^\infty}\\
\leq & \theta(-s)\left\|d_{A^0}\left(\mathcal T^{\varepsilon,b}(\Xi_-)-\Xi_--\bar\alpha_1^\varepsilon\right) \right\|_{L^\infty}
+c \theta(-s)\left\|\pi_{A^0}\left(\mathcal T^{\varepsilon,b}(\Xi_-)-\Xi_-\right) \right\|_{L^\infty}\\
&
+\theta(s)\left\|d_{A^0}\left(\mathcal T^{\varepsilon,b}(\Xi_+)-\Xi_+-\bar\alpha_1^\varepsilon\right) \right\|_{L^\infty}
+c\theta(s)\left\|\pi_{A^0}\left(\mathcal T^{\varepsilon,b}(\Xi_+)-\Xi_+\right) \right\|_{L^\infty}
\end{align*}
which is smaller than $c\varepsilon^2$. Furthermore, for $\xi=(\mathcal D^\varepsilon(\mathcal K_2^\varepsilon))^*\eta$, $\eta=\eta_1+\eta_2 dt+\eta_3ds$, we have by the lemmas \ref{flow:lemma:qe1} (for the adjoint operators) and \ref{flow:lemma:diffD} as well as by the theorem \ref{flow:thme:linest333tre} that
\begin{equation}\label{flow:k2fhods}
\begin{split}
&\|\pi_A(\xi-(\mathcal D^0(\Xi^0))^*(\pi_A(\eta)))\|_{L^p}\\
\leq & \|\pi_A((\mathcal D^\varepsilon(\Xi^0))^*(\eta)-*[\eta_1\wedge*\omega(A^0)]-(\mathcal D^0(\Xi^0))^*(\pi_A(\eta)))\|_{L^p}+c\varepsilon^{1-\frac 1p} \|\xi \|_{0,p,\varepsilon}\\
\leq & c\varepsilon^{1-\frac 1p} \|\xi \|_{0,p,\varepsilon}.
\end{split}
\end{equation}
The theorem follows then from the theorem \ref{flow:thme:linest333} and the last computations.
\end{proof}

%
%

\section{The map $\mathcal R^{\varepsilon,b }$ between flows}\label{flow:section:themap}

In this section we will show that, for any pair $\Xi^0_\pm\in \mathrm {Crit}_{E^H}$, any geodesic flow $\Xi^0\in \mathcal M^0 \left(\Xi^{0}_-, \Xi^{0}_+\right)$ can be approximated by a Yang-Mills flow $$\Xi^\varepsilon\in \mathcal M^\varepsilon\left(\mathcal T^{\varepsilon,b }\left(\Xi^{0}_-\right), \mathcal T^{\varepsilon,b}\left(\Xi^{0}_+\right)\right)$$ provided that $\varepsilon$ is small enough. In addition, $\Xi^\varepsilon$ will turn out to be the unique Yang-Mills flow in a ball around $\Xi^0$ of radius $\delta\varepsilon$. Therefore we can define an injective map $\mathcal R^{\varepsilon,b}$ between the flows of the two functionals provided that we choose an energy bound $b$ for the critical connections and $\varepsilon$ small enough.

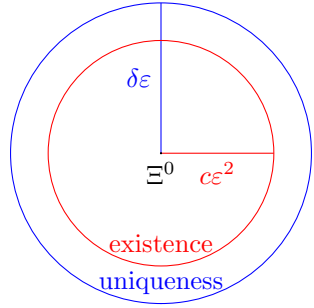
\begin{figure}[ht]
\begin{center}
\begin{tikzpicture}[scale=2] 
\draw[blue] (0,0) circle (1cm); 
\draw[blue] (0,0)--node[left]{$\delta\varepsilon$}(0,1);
\draw[red] (0,0) circle (0.75cm); 
\draw[red] (0,0)--node[below]{$c\varepsilon^2$}(0.75,0);

\draw (0,0) circle (0.1pt) node[below]{$\Xi^0$};

\draw[red] (0,-0.72) circle (0pt) node[above]  {existence} ; 
\draw[blue] (0,-1) circle (0.01pt) node[above]  {uniqueness} ; 
\end{tikzpicture}
\end{center}
\caption{Existence and uniqueness.}
\end{figure}

\begin{theorem}[Existence]\label{flow:thm:existence}
We assume that the energy functional $E^H$ is Morse-Smale and we choose two constants $b > 0$, $p>2$. There are constants $\varepsilon_0, c>0$ such  that the following holds. For every $\varepsilon \in (0,\varepsilon_0)$, every pair $\Xi^{0}_\pm:=A^0_\pm+\Psi^0_\pm dt \in \mathrm {Crit}^b_{E^H}$ that are perturbed closed geodesics of index difference one and every $\Xi^0:=A^0+\Psi^0 dt+\Phi^0 ds \in \mathcal M^0 (\Xi^{0}_-, \Xi^{0}_+)$, there exists a connection $\Xi^\varepsilon\in \mathcal M^\varepsilon\left(\mathcal T^{\varepsilon,b }\left(\Xi^{0}_-\right), \mathcal T^{\varepsilon,b}\left(\Xi^{0}_+\right)\right)$ which satisfies
\begin{equation}
d_{\Xi^0}^{*_\varepsilon }\left(\Xi^\varepsilon-\mathcal K_2^\varepsilon\left(\Xi^0\right)\right)=0, \quad \Xi^\varepsilon-\mathcal K_2^\varepsilon\left(\Xi^0\right) \in \textrm{im } \left(\mathcal D^\varepsilon\left(\mathcal K_2\left(\Xi^0\right)\right)\right)^*,
\end{equation}
\begin{equation}
\left\|\left(1-\pi_{A^0}\right)\left(\Xi^\varepsilon-\mathcal K_2^\varepsilon\left(\Xi^0\right)\right)\right\|_{1,2;p,\varepsilon}+\varepsilon\left\|\pi_{A^0}\left(\Xi^\varepsilon-\mathcal K_2^\varepsilon\left(\Xi^0\right)\right)\right\|_{1,2;p,1}\leq c\varepsilon^3.
\end{equation}
\end{theorem}

\begin{theorem}[Local uniqueness]\label{flow:thm:locuniq}
We choose $p>3$. For every pair $\Xi^{0}_\pm:=A^0_\pm+\Psi^0_\pm dt \in \mathrm {Crit}^b_{E^H}$ that are perturbed closed geodesics of index difference one, every $\Xi^0:=A^0+\Psi^0 dt+\Phi^0 ds \in \mathcal M^0 \left(\Xi^{0}_-, \Xi^{0}_+\right)$ and any $c>0$ there are an $\varepsilon_0>0$ and a $\delta>0$ such that the following holds. If $\Xi^\varepsilon,\bar\Xi^\varepsilon\in \mathcal M^\varepsilon\left(\mathcal T^{\varepsilon,b}\left(\Xi^0_-\right),\mathcal T^{\varepsilon,b}\left(\Xi^0_+\right)\right)$ satisfy
\begin{equation*}
d_{\Xi^0}^{*_\varepsilon}\left(\Xi^\varepsilon-\mathcal K_2^\varepsilon\left(\Xi^0\right)\right)=d_{\Xi^0}^{*_\varepsilon}\left(\bar\Xi^\varepsilon-\mathcal K_2^\varepsilon\left(\Xi^0\right)\right)=0,
\end{equation*}
\begin{equation*}
\Xi^\varepsilon-\mathcal K_2^\varepsilon\left(\Xi^0\right),\, \bar \Xi^\varepsilon-\mathcal K_2^\varepsilon\left(\Xi^0\right) \in \textrm{im } \left(\mathcal D^\varepsilon\left(\mathcal K_2^\varepsilon\left(\Xi^0\right)\right)\right)^*
\end{equation*}
and the estimates
\begin{equation}\label{flow:eq:locuniq:l6}
\left\|\Xi^\varepsilon-\mathcal K_2^\varepsilon\left(\Xi^0\right)\right\|_{1,2;p,\varepsilon}\leq c\varepsilon^2,
\end{equation}
\begin{equation}\label{flow:eq:locuniq:l7}
\left\|\bar \Xi^\varepsilon-\mathcal K_2^\varepsilon\left(\Xi^0\right)\right\|_{1,2;p,\varepsilon}+\left\|\bar \Xi^\varepsilon-\mathcal K_2^\varepsilon\left(\Xi^0\right)\right\|_{L^\infty}\leq \delta\varepsilon,
\end{equation}
then $\bar\Xi^\varepsilon=\Xi^\varepsilon$.
\end{theorem}

\begin{definition}\label{defiR}
We choose $p>3$. For every regular value $b>0$ of the energy $E^H$ there are positive constants $\varepsilon_0$, $\delta$ and $c$ such that the assertion of the theorems \ref{flow:thm:existence} and \ref{flow:thm:locuniq} hold with these constants. Shrink $\varepsilon_0$ such that $c\varepsilon_0+c_0c\varepsilon_0^{1-\frac 3p}<\delta$, where $c_0$ is the constant of the Sobolev's theorem \ref{flow:thm:sob}. Theorems \ref{flow:thm:existence} and \ref{flow:thm:locuniq} assert that, for every pair $\Xi^0_\pm:=A_\pm^0+\Psi^0_\pm dt \in \mathrm{Crit}^b_{E^H}$ that are perturbed closed geodesics of index difference one, every $\Xi^0\in \mathcal M^0\left(\Xi^0_-,\Xi^0_+\right)$ and every $\varepsilon$ with $0<\varepsilon<\varepsilon_0$, there is a unique $\Xi^\varepsilon\in \mathcal M^\varepsilon\left(\mathcal T^{\varepsilon,b}\left(\Xi^0_-\right),\mathcal T^{\varepsilon,b}\left(\Xi^0_+\right)\right)$ satisfying
\begin{equation}\label{flow:defiR1}
d_{\Xi^0}^{*_\varepsilon }\left(\Xi^\varepsilon-\mathcal K_2^\varepsilon\left(\Xi^0\right)\right)=0, \quad \Xi^\varepsilon-\mathcal K_2^\varepsilon\left(\Xi^0\right) \in \textrm{im } \left(\mathcal D^\varepsilon\left(\mathcal K_2\left(\Xi^0\right)\right)\right)^*
\end{equation}
and
\begin{equation}\label{flow:defiR2}
\left\|\Xi^\varepsilon-\mathcal K_2^\varepsilon\left(\Xi^0\right)\right\|_{1,2;p,\varepsilon}\leq c\varepsilon^2.
\end{equation}
We define the map 
\begin{equation*}
\mathcal R ^{\varepsilon,b}: \mathcal M^0\left(\Xi^0_-,\Xi^0_+\right)\to \mathcal M^\varepsilon\left(\mathcal T^{\varepsilon,b}\left(\Xi^0_-\right),\mathcal T^{\varepsilon,b}\left(\Xi^0_+\right)\right)
\end{equation*}
by $\mathcal R ^{\varepsilon,b}\left(\Xi^0\right):=\Xi^\varepsilon$ where $\Xi^\varepsilon\in\mathcal M^\varepsilon\left(\mathcal T^{\varepsilon,b}\left(\Xi^0_-\right),\mathcal T^{\varepsilon,b}\left(\Xi^0_+\right)\right)$ is the unique Yang-Mills flow satisfying (\ref{flow:defiR1}) and (\ref{flow:defiR1}).
\end{definition}

\begin{proof}[Proof of theorem \ref{flow:thm:existence}]
We choose $\Xi^\varepsilon_2:= \mathcal K_2^\varepsilon(\Xi^0)$. By induction we define, for $k\geq3$, $\Xi^\varepsilon_k:=\Xi_{k-1}^\varepsilon+\eta^\varepsilon_{k-1}$, $\eta_k^\varepsilon=\alpha_k^\varepsilon+\psi_k^\varepsilon dt+\phi_k^\varepsilon ds$, where $\eta_k$ is defined by
\begin{equation}\label{flow:eq:exkl1}
\mathcal D^\varepsilon(\mathcal K_2^\varepsilon(\Xi^0))(\eta_{k-1}^\varepsilon)=-\mathcal F^\varepsilon(\Xi^\varepsilon_{k-1}),\quad \eta_{k-1}^\varepsilon \in \textrm{im } (\mathcal D^\varepsilon(\mathcal K_2^\varepsilon(\Xi^0)))^*.
\end{equation}
In addition, one can remark that
\begin{equation}\label{flow:eq:exkl2}
\mathcal F^\varepsilon(\Xi_{k-1}^\varepsilon)=\mathcal F^{\varepsilon}(\Xi^\varepsilon_{k-2})+\mathcal D^\varepsilon(\Xi^\varepsilon_{k-2})(\eta_{k-2}^\varepsilon)+\mathcal C^\varepsilon(\Xi^\varepsilon_{k-2})(\eta_{k-2}^\varepsilon).
\end{equation}
By theorem \ref{flow:thm:estk2} we have the estimate
\begin{align*}
\| (1-\pi_{A^0})&\eta_{k-1} ^\varepsilon\|_{1,2;p,\varepsilon}+\varepsilon\|\pi_A(\alpha_{k-1}^\varepsilon)\|_{1,2;p,1}\\
\leq &c\varepsilon^2\|\mathcal D^\varepsilon(\mathcal K_2^\varepsilon(\Xi^0))(\eta_{k-1}^\varepsilon)\|_{0,p,\varepsilon }+c\varepsilon\left\|\pi_{A^0}(\mathcal D^\varepsilon(\mathcal K_2^\varepsilon(\Xi^0))(\eta_{k-1}^\varepsilon))\right\|_{L^p}\\
= &c\varepsilon^2\|\mathcal F^\varepsilon(\Xi^\varepsilon_{k-1})\|_{0,p,\varepsilon }+c\varepsilon\left\|\pi_{A^0}(\mathcal F^\varepsilon(\Xi^\varepsilon_{k-1}))\right\|_{L^p}
\end{align*}
where the last step follows from (\ref{flow:eq:exkl1}) and by (\ref{flow:eq:exkl2}) we obtain
\begin{align*}
\varepsilon^2\|\mathcal F^\varepsilon&(\Xi^\varepsilon_{k-1})\|_{0,p,\varepsilon }+c\varepsilon\left\|\pi_{A^0}(\mathcal F^\varepsilon(\Xi^\varepsilon_{k-1}))\right\|_{L^p}\\
\leq &c\varepsilon^2\|\mathcal C^\varepsilon(\Xi^\varepsilon_{k-2})(\eta_{k-2}^\varepsilon)\|_{0,p,\varepsilon }+c\varepsilon\left\|\pi_{A^0}(\mathcal C^\varepsilon(\Xi^\varepsilon_{k-2})(\eta_{k-2}^\varepsilon))\right\|_{L^p}\\
&+c\varepsilon^2\left\|\left( \mathcal D^\varepsilon(\Xi_{k-1}^\varepsilon)-\mathcal D^\varepsilon(\mathcal K_2^\varepsilon(\Xi^0))\right) \eta_{k-1}^\varepsilon\right\|_{0,p,\varepsilon }\\
&+c\varepsilon\left\|\pi_{A^0}\left( \mathcal D^\varepsilon(\Xi^\varepsilon_{k-1})-\mathcal D^\varepsilon(\mathcal K_2^\varepsilon(\Xi^0))\right) \eta_{k-1}^\varepsilon\right\|_{L^p}\\
\intertext{and finally using the lemmas  \ref{flow:lemma:qe1} and \ref{flow:lemma:qe2}, we can conclude}
\leq &c\|\eta_{k-2}^\varepsilon\|_{\infty,\varepsilon}\|\alpha_{k-2}^\varepsilon+\psi_{k-2}^\varepsilon \|_{1,1;p,\varepsilon}\\
&+\frac c\varepsilon| |(1-\pi_A)\alpha_{k-2}^\varepsilon+\psi_{k-2}^\varepsilon dt+\phi_{k-2}^\varepsilon ds\|_{\infty,\varepsilon }\|\alpha_{k-2}^\varepsilon+\psi_{k-2}^\varepsilon dt\|_{1,p,\varepsilon}\\
&+\frac c\varepsilon\|\pi_A(\alpha_{k-2}^\varepsilon)\|_{L^\infty }\|(1-\pi_A)\alpha_{k-2}^\varepsilon+\psi_{k-2}^\varepsilon dt+\phi_{k-2}^\varepsilon ds\|_{1,p,\varepsilon}\\
&+\frac c\varepsilon\|\pi_A(\alpha_{k-2}^\varepsilon)\|_{L^\infty }(\|\pi_A(\alpha_{k-2}^\varepsilon)\|_{L^\infty}+\varepsilon^2)\|\pi_A(\alpha_{k-2}^\varepsilon)\|_{L^p}.
\end{align*}
Next, by the estimates of lemma \ref{flow:lemma:firstappr}:
\begin{equation*}
\|\mathcal F_1^\varepsilon(\mathcal K_2^\varepsilon(\Xi^0))\|_{L^p}\leq c\varepsilon^{2},\quad \|\mathcal F_2^\varepsilon(\mathcal K_2^\varepsilon(\Xi^0))\|_{L^p}\leq c,
\end{equation*}
there is a positive constant $c_0$ such that
\begin{equation*}
\| (1-\pi_{A^0})\eta_{2}^\varepsilon \|_{1,2;p,\varepsilon}+\varepsilon\|\pi_A(\alpha_{2}^\varepsilon)\|_{1,2;p,1}\leq  c_0\varepsilon^3.
\end{equation*}

Using the Sobolev's theorem \ref{flow:thm:sob}, one can easily see that, for $\varepsilon$ small enough and $k>3$,  by induction there are two positive constants $c_1$, $c$ such that
\begin{equation*}
\begin{split}
\| (1-\pi_{A^0})\eta_{3} ^\varepsilon\|_{1,2;p,\varepsilon}+\varepsilon\|\pi_A(\alpha_{3}^\varepsilon)\|_{1,2;p,1}\leq &c\varepsilon^{4-\frac 3p },\\
\| (1-\pi_{A^0})\eta_{k}^\varepsilon \|_{1,2;p,\varepsilon}+\varepsilon\|\pi_A(\alpha_{k}^\varepsilon)\|_{1,2;p,1}\leq & 2^{-(k-2)}c_0 \varepsilon^{3 },\\
\varepsilon^2\|\mathcal F^\varepsilon(\Xi_{k}^\varepsilon)\|_{0,p,\varepsilon }+\varepsilon\left\|\pi_{A^0}(\mathcal F^\varepsilon(\Xi_{k}^\varepsilon))\right\|_{L^p}\leq&2^{-(k-2)}c_1 \varepsilon^{3 }.
\end{split}
\end{equation*}
Therefore $\mathcal F^\varepsilon(\Xi^{k})$ converges to $0$ and we can choose $\Xi^\varepsilon:= \mathcal K_2(\Xi^0)+\sum_{k=2}^\infty \eta_k^\varepsilon$. Then
\begin{equation*}
\Xi^\varepsilon-\mathcal K_2^\varepsilon(\Xi^0) \in \textrm{im } (\mathcal D^\varepsilon(\mathcal K_2^\varepsilon(\Xi^0)))^*, \quad \mathcal F^\varepsilon(\Xi^{\varepsilon })=0
\end{equation*}
and
\begin{equation*}
\left\|\left(1-\pi_{A^0}\right)\left(\Xi^\varepsilon-\mathcal K_2^\varepsilon(\Xi^0)\right)\right\|_{1,2;p,\varepsilon}+\varepsilon\left\|\pi_{A^0}\left(\Xi^\varepsilon-\mathcal K_2^\varepsilon(\Xi^0)\right)\right\|_{1,2;p,1}\leq c\varepsilon^3.
\end{equation*}
Since $\mathcal F^\varepsilon(\Xi^\varepsilon)=0$, by the definition of $\mathcal F^\varepsilon_3$, $d_{\Xi^0}^{*_\varepsilon}\left(\Xi^\varepsilon-\mathcal K^\varepsilon_2(\Xi^0)\right)=0$ holds and thus we concluded the proof of the theorem \ref{flow:thm:existence}.
\end{proof}

\begin{proof}[Proof of theorem \ref{flow:thm:locuniq}]
First, we improve the estimate (\ref{flow:eq:locuniq:l7}) and we show that
\begin{equation}\label{flow:eq:locuniq:l8}
\|\bar\Xi^\varepsilon-\mathcal K_2^\varepsilon(\Xi^0)\|_{1,2;p,\varepsilon}+\varepsilon  \|\pi_{A_0}(\bar\Xi^\varepsilon-\mathcal K_2^\varepsilon(\Xi^0))\|_{1,2;p,1}\leq c\varepsilon^3.
\end{equation}
In order to fulfil this task, we consider the identity
\begin{equation}\label{flow:locuniqeq4}
\mathcal D_i^\varepsilon(\mathcal K_2^\varepsilon(\Xi^0))(\bar\Xi^\varepsilon-\mathcal K_2^\varepsilon(\Xi^0))=-C_i^\varepsilon(\mathcal K_2^\varepsilon(\Xi^0))(\bar \Xi^\varepsilon-\mathcal K_2^\varepsilon(\Xi^0))-\mathcal F_i^\varepsilon(\mathcal K_2^\varepsilon(\Xi^0))
\end{equation}
then, by theorem \ref{flow:thm:estk2}
\begin{align*}
\|\bar\Xi^\varepsilon&-\mathcal K_2^\varepsilon(\Xi^0)\|_{1,2;p,\varepsilon}+\varepsilon  \|\pi_{A_0}(\bar\Xi^\varepsilon-\mathcal K_2^\varepsilon(\Xi^0))\|_{1,2;p,1}\\
\leq&c \varepsilon^2\|\mathcal D^\varepsilon(\mathcal K_2^\varepsilon(\Xi^0))(\bar\Xi^\varepsilon-\mathcal K_2^\varepsilon(\Xi^0))\|_{0,p,\varepsilon }\\
&+c\varepsilon \|\pi_{A_0}\mathcal D^\varepsilon(\mathcal K_2^\varepsilon(\Xi^0))(\bar\Xi^\varepsilon-\mathcal K_2^\varepsilon(\Xi^0))\|_{L^p }\\
\intertext{next, we can apply (\ref{flow:locuniqeq4}), i.e.}
\leq&c \varepsilon^2\|\mathcal C^\varepsilon(\mathcal K_2^\varepsilon(\Xi^0))(\bar\Xi^\varepsilon-\mathcal K_2^\varepsilon(\Xi^0))\|_{0,p,\varepsilon }\\
&+c\varepsilon \|\pi_{A_0}\mathcal C^\varepsilon(\mathcal K_2^\varepsilon(\Xi^0))(\bar\Xi^\varepsilon-\mathcal K_2^\varepsilon(\Xi^0))\|_{L^p }\\
&+c \varepsilon^2\|\mathcal F^\varepsilon(\mathcal K_2^\varepsilon(\Xi^0))\|_{0,p,\varepsilon }+c\varepsilon \|\pi_{A_0}\mathcal F^\varepsilon(\mathcal K_2^\varepsilon(\Xi^0))\|_{L^p }\\
\intertext{and the quadratic estimates (\ref{flow:lemma:qe1}) and (\ref{flow:lemma:qe2}): }
\leq &c\varepsilon^3+ c \|\bar\Xi^\varepsilon-\mathcal K_2^\varepsilon(\Xi^0)\|_{\infty,\varepsilon} \|\bar\Xi^\varepsilon-\mathcal K_2^\varepsilon(\Xi^0)\|_{1,1;p,\varepsilon }\\
&+ \frac c{\varepsilon } \|(1-\pi_{A_0})(\bar\Xi^\varepsilon-\mathcal K_2^\varepsilon(\Xi^0))\|_{\infty,\varepsilon }\|(1-\pi_{A_0})(\bar\Xi^\varepsilon-\mathcal K_2^\varepsilon(\Xi^0))\|_{1,1;p,\varepsilon}\\
&+c  \|(1-\pi_{A_0})(\bar\Xi^\varepsilon-\mathcal K_2^\varepsilon(\Xi^0))\|_{\infty,\varepsilon }\|\nabla_t\pi_{A_0}(\bar\Xi^\varepsilon-\mathcal K_2^\varepsilon(\Xi^0))\|_{L^p}\\
&+\frac c{\varepsilon} \|\pi_{A_0}(\bar\Xi^\varepsilon-\mathcal K_2^\varepsilon(\Xi^0))\|_{L^\infty }\|(1-\pi_{A_0})(\bar\Xi^\varepsilon-\mathcal K_2^\varepsilon(\Xi^0))\|_{1,1;p,\varepsilon }\\
&+\frac c{\varepsilon} \|\pi_{A_0}(\bar\Xi^\varepsilon-\mathcal K_2^\varepsilon(\Xi^0))\|_{L^\infty }^2\|\pi_{A_0}(\bar\Xi^\varepsilon-\mathcal K_2^\varepsilon(\Xi^0))\|_{L^p}\\
&+ c\varepsilon \|\pi_{A_0}(\bar\Xi^\varepsilon-\mathcal K_2^\varepsilon(\Xi^0))\|_{L^\infty }\|\pi_{A_0}(\bar\Xi^\varepsilon-\mathcal K_2^\varepsilon(\Xi^0))\|_{L^p}\\
\leq& c\varepsilon^3+ c\delta  \|(1-\pi_{A_0})(\bar\Xi^\varepsilon-\mathcal K_2^\varepsilon(\Xi^0))\|_{1,1;p,\varepsilon}\\
&+c\delta \varepsilon \|\pi_{A_0}(\bar\Xi^\varepsilon-\mathcal K_2^\varepsilon(\Xi^0))\|_{1,1;p,1 }
\end{align*}
where the last inequality follows from the assumptions. Therefore for $\delta$ small enough (\ref{flow:eq:locuniq:l8}) holds. Furthermore, by (\ref{flow:eq:locuniq:l6}) and (\ref{flow:eq:locuniq:l8}), $\|\bar\Xi^\varepsilon-\Xi^\varepsilon\|_{1,2;p,\varepsilon}\leq c\varepsilon^2$. Thus, always by theorem \ref{flow:thm:estk2},
\begin{align*}
\|\Xi^\varepsilon-\bar\Xi^\varepsilon&\|_{1,2;p,\varepsilon}+\varepsilon  \|\pi_{A_0}(\Xi^\varepsilon-\bar\Xi^\varepsilon)\|_{1,2;p,1}\\
\leq&c \varepsilon^2\|\mathcal D^\varepsilon(\mathcal K_2^\varepsilon(\Xi^0))(\Xi^\varepsilon-\bar\Xi^\varepsilon)\|_{0,p,\varepsilon }+c \varepsilon \|\pi_{A_0}\mathcal D^\varepsilon(\mathcal K_2^\varepsilon(\Xi^0))(\Xi^\varepsilon-\bar\Xi^\varepsilon)\|_{L^p }\\
\intertext{and since $\mathcal F^\varepsilon(\bar\Xi^\varepsilon)\mathcal F^\varepsilon(\Xi^\varepsilon)=0$, $\mathcal D_i^\varepsilon(\bar \Xi^\varepsilon)( \Xi^\varepsilon-\bar\Xi^\varepsilon)=-C_i^\varepsilon(\bar \Xi^\varepsilon)( \Xi^\varepsilon-\bar\Xi^\varepsilon)$ and thus we obtain}
\leq&c \left(\varepsilon^2\|C_i^\varepsilon(\bar \Xi^\varepsilon)(\Xi^\varepsilon-\bar\Xi^\varepsilon)\|_{0,p,\varepsilon }+\varepsilon \|\pi_{A_0}C_i^\varepsilon(\bar \Xi^\varepsilon)(\Xi^\varepsilon-\bar\Xi^\varepsilon)\|_{L^p }\right)\\
&+c \varepsilon^2\|(\mathcal D^\varepsilon(\bar \Xi^\varepsilon)-\mathcal D^\varepsilon(\mathcal K_2^\varepsilon(\Xi^0)))(\Xi^\varepsilon-\bar\Xi^\varepsilon)\|_{0,p,\varepsilon }\\
&+c\varepsilon \|\pi_{A_0}(\mathcal D^\varepsilon(\bar\Xi^\varepsilon)-\mathcal D^\varepsilon(\mathcal K_2^\varepsilon(\Xi^0)))(\Xi^\varepsilon-\bar\Xi^\varepsilon)\|_{L^p }\\
\leq &c\varepsilon^{1-\frac 3p}  \left( \|(1-\pi_{A_0})(\Xi^\varepsilon-\bar\Xi^\varepsilon)\|_{1,1;p,\varepsilon}+ \varepsilon \|\pi_{A_0}(\Xi^\varepsilon-\bar\Xi^\varepsilon)\|_{1,1;p,1 }\right)
\end{align*}
where the last inequality follows from the quadratic estimates of the lemmas  \ref{flow:lemma:qe1} and \ref{flow:lemma:qe2}. Hence for $p>3$ and $\varepsilon$ small enough $\bar\Xi^\varepsilon=\Xi^\varepsilon$.
\end{proof}

%
%

\section{A priori estimates for the Yang-Mills flow}\label{flow:section:apriori}

In this section we will prove some a priori estimates, that will be stated in theorem \ref{flow:thm:apriori},  on the curvature for a perturbed Yang-Mills flow. These will then be used to prove the surjectivity of the map $\mathcal R^{\varepsilon, b}$ in the section \ref{flow:section:surj}. \\

In order to simplify the exposition we denote by $\|\cdot\|$ the $L^2$-norm over $\Sigma$ and we introduce the following notation. We choose two perturbed Yang-Mills connections $\Xi_\pm^\varepsilon\in \mathrm{Crit}_{\mathcal {YM}^{\varepsilon,H}}^b$ where $b>0$. For any Yang-Mills flow $\Xi^\varepsilon:=A+\Psi dt+\Phi ds \in\mathcal M^\varepsilon(\Xi^\varepsilon_-,\Xi^\varepsilon_+)$ we define
\begin{equation}\label{flow:eq:apriori:eq1}
B_t:=\partial_t A-d_A\Psi,\quad B_s:=\partial_s A-d_A\Phi, \quad C:=\partial_t\Psi-\partial_s\Phi-[\Psi,\Phi];
\end{equation}
thus, the Yang-Mills flow equations (\ref{flow:intro:woeps}) can be written as
\begin{equation}\label{flow:eq:apriori:eq2}
B_s+\frac 1{\varepsilon^2}d_A^*F_A-\nabla_t B_t-*X_t(A)=0,\quad C-\frac 1{\varepsilon^2}d_A^*B_t=0;
\end{equation}
with this notation, the Bianchi identities are
\begin{equation}\label{flow:eq:apriori:eq3}
\nabla_t F_A=d_AB_t,\quad \nabla_s F_A=d_AB_s,\quad \nabla_t B_s-\nabla_s B_t=d_AC
\end{equation}
and the commutation formulas
\begin{equation}\label{flow:eq:apriori:eq4}
[\nabla_t,d_A]=B_t,\quad [\nabla_s,d_A]=B_s,\quad [\nabla_s,\nabla_t]=C.
\end{equation}

Furthermore, we have the identity
\begin{equation}
\left\|B_s+ C\, dt\right\|_{0,2,\varepsilon }^2
= \mathcal {YM}^{\varepsilon,H }(\Xi_-^\varepsilon)-\mathcal {YM}^{\varepsilon,H }(\Xi_+^\varepsilon)
\end{equation}
which can be showed by the direct computation:
\begin{align*}
\left\| B_s+ C\,dt\right\|_{0,2,\varepsilon }^2
=&\int_{\mathbb R}\| B_s+C\, dt \|^2_{0,2,\varepsilon}\,ds\\
\intertext{and by the Yang-Mills flow equations (\ref{flow:intro:woeps})}
= &\int_{\mathbb R}\int_{\Sigma\times S^1 }\langle B_s,-\frac 1{\varepsilon^2}d_A^*F_A+\nabla_t B_t +*X_t(A)\rangle\,\mathrm{dvol}_{\Sigma} \,dt\,ds\\
&+ \int_{\mathbb R}\int_{\Sigma\times S^1 }\langle C, d_A^*B_t\rangle\,\mathrm{dvol}_{\Sigma}\, dt\,ds\\
= &\int_{\mathbb R}\int_{\Sigma\times S^1 }\left(-\frac 1{\varepsilon^2}\langle d_A B_s,F_A\rangle- \langle\nabla_t B_s, B_t \rangle\right)\mathrm{dvol}_{\Sigma} \,dt\,ds\\
&-\int_{\mathbb R}\partial_s H(A) \, ds+ \int_{\mathbb R}\int_{\Sigma\times S^1 }\langle d_AC, B_t\rangle\,\mathrm{dvol}_{\Sigma} \,dt\,ds\\
\intertext{and by the Bianchi identity (\ref{flow:eq:apriori:eq3})}
= &-\int_{\mathbb R}\int_{\Sigma\times S^1 }\frac 1{\varepsilon^2}\langle \nabla_s F_A,F_A\rangle\mathrm{dvol}_{\Sigma} \,dt\,ds-\int_{\mathbb R}\partial_s H(A) \, ds\\
&-\int_{\mathbb R}\int_{\Sigma\times S^1 } \langle\nabla_sB_t, B_t\rangle\mathrm{dvol}_{\Sigma} \,dt\,ds\\
=& \mathcal {YM}^{\varepsilon,H }(A_-+\Psi_-dt)-\mathcal {YM}^{\varepsilon,H }(A_++\Psi_+dt).
\end{align*}
Furthermore, we denote by $a_i^1$, $i=1,2,3$, the three operators $d_A$, $d_A^*$ and $\varepsilon\nabla_t$, by $a_i^2$, $i=1,\dots ,9$, the nine operators defined combining two operators between $d_A$, $d_A^*$ and $\varepsilon\nabla_t$, by $a_i^3$, $i=1,\dots ,27$, the $27$ operators defined combining three operators between $d_A$, $d_A^*$, $\varepsilon\nabla_t$ and finally we denote by $a_i^4$, $i=1,\dots ,81$, the $81$ operators defined combining four operators between $d_A$, $d_A^*$, $\varepsilon\nabla_t$. In addition we denote by $\mathcal N_j(t,s)$ the norms
$$\mathcal N_j(t,s):=\sum_{i=1,\dots,3^j}\left(\|a_i^jB_s\|^2+\varepsilon^4\|a_i^jC\|^2\right).$$

\begin{theorem}\label{flow:thm:apriori}
We choose an open interval $\Omega\subset \mathbb R$, a compact set $Q\subset \Omega$,  $p>4$ and two constants $b, c_0>0$. There are two positive constants $\varepsilon_0$, $c$ such that the following holds. If a perturbed Yang-Mills flow $\Xi=A+\Psi dt+\Phi ds \in \mathcal M^\varepsilon(\Xi_-,\Xi_+)$, with $\Xi_-, \Xi_+ \in \mathrm{Crit}^b_{\mathcal {YM}^{\varepsilon, H}}$ and $0<\varepsilon<\varepsilon_0$, satisfies
\begin{equation}\label{flow:apriori74}
\sup_{(t,s)\in S^1\times \mathbb R}\left(\|\partial_t A-d_A\Psi\|_{L^4(\Sigma)}+\|\partial_s A-d_A\Phi\|_{L^\infty(\Sigma)}\right)\leq c_0,
\end{equation} 
then
\begin{equation}
\int_{S^1\times Q}\left(\|F_A\|^2+\varepsilon^2\|\nabla_t F_A\|^2+\|d_A^*F_A\|^2\right)dt\,ds\leq c\varepsilon^4,
\end{equation}
\begin{equation}\label{oggi1}
\begin{split}
\sup_{S^1\times Q }&\left(\varepsilon^2\|B_t\|^2+ \|d_A^*d_AB_t\|^2+ \|d_A^*d_Ad_A^*B_t\|^2\right)\\
\leq& c\int_{S^1\times \Omega} \left(\varepsilon^2\|B_s\|^2+\varepsilon^2\|B_t\|^2+\|F_A\|^2+\varepsilon^2c_{\dot X_t(A)}+\varepsilon^4\|C\|^2\right) dt\, ds,
\end{split}
\end{equation}
\begin{equation}\label{oggi2}
\begin{split}
\sup_{S^1\times Q }&\left(\varepsilon^2\|B_s\|^2+ \|d_A^*d_AB_s\|^2+ \|d_Ad_A^*B_s\|^2\right)\\
\leq& c\int_{S^1\times \Omega} \left(\varepsilon^2\|B_s\|^2+\varepsilon^2\|B_t\|^2+\|F_A\|^2+\varepsilon^2c_{\dot X_t(A)}+\varepsilon^4\|C\|^2\right) dt\, ds,
\end{split}
\end{equation}
\begin{equation}\label{flow:eq:apriorig2}
\sup_{S^1\times Q }\left(\|F_A\|+\varepsilon \|\nabla_t F_A\|+\varepsilon^2 \|\nabla_t\nabla_t F_A\|\right)\leq c \varepsilon^{2},
\end{equation}
\begin{equation}\label{flow:eq:apriorig1}
\sup_{(t,s)\in S^1\times \mathbb R}\|F_A\|_{L^\infty(\Sigma)}\leq c \varepsilon^{2}
\end{equation}
where $c_{\dot X_t(A)}$ is a constant which estimates the norm $\|\dot X_t(A)\|_{L^\infty(\Sigma)}^2$. The constant $c$ depends on $\Omega$ and on $Q$, but only on their length and on the distance between their boundaries. Furthermore,
\begin{equation}\label{flow:eq:apriorig20}
\begin{split}
\sup_{(t,s)\in S^1\times Q}&\left(\varepsilon^2\|B_s\|^2+\varepsilon^4\|C\|^2+\mathcal N_1+\mathcal N_2+\mathcal N_3+\mathcal N_4\right)\\
\leq &\varepsilon^2c \int_{S^1\times \Omega}\left(\|B_s\|^2+\varepsilon^2\|C\|^2\right) dt\,ds.
\end{split}
\end{equation}
\end{theorem}

\begin{remark}
 In the theorem we assume that the $L^4(\Sigma)$-norm of the curvature term $\partial_tA-d_A\Psi$ and the $L^\infty(\Sigma)$-norm of $\partial_sA-d_A\Phi$ are uniformly bounded; this condition is, for a Yang-Mills flow, always satisfied if we choose $\varepsilon$ small enough as we will see in section \ref{flow:section:linfty}.
\end{remark}

Before starting to prove the last theorem we consider the following three lemmas. The first one show a regularity result for the curvature terms $B_s$, $C$. The last two are the lemmas B.1. and B.4. that Salamon and Weber proved in \cite{MR2276534}. For an interval $Q\subset \mathbb R$ and a $0$- or $1$-form $\alpha$ we define the norm $\|\cdot\|_{1,2,2;2,\varepsilon,Q}$ by
\begin{align*}
\|&\alpha\|_{1,2,2;2,\varepsilon,Q}^2:=\int_{S^1\times Q}\left(\|\alpha\|^2+\varepsilon^4\|\nabla_s\alpha\|^2+\|d_{  A }\alpha\|^2+\|d_{  A }^*\alpha\|^2\right) dt\,ds\\
&+\int_{S^1\times Q}\left(\varepsilon^2\|\nabla_t \alpha\|^2+\|d_{  A }^*d_{  A }\alpha\|^2+\|d_{  A }d_{  A }^*\alpha\|^2+\varepsilon^4\|\nabla_t\nabla_t \alpha\|^2\right) dt\, ds.
\end{align*}

\begin{lemma}\label{flow:smooth}
We choose a positive constant $b$, then there are two positive constants $\varepsilon_0$, $c$ such that the following holds. For any perturbed Yang-Mills flow $\Xi=A+\Psi dt+\Phi ds \in \mathcal M^\varepsilon(\Xi_-,\Xi_+)$, with $\Xi_-, \Xi_+ \in \mathrm{Crit}^b_{\mathcal {YM}^{\varepsilon, H}}$ and $0<\varepsilon<\varepsilon_0$, satisfies
$$\| B_s \|_{1,2,2;2,\varepsilon ,\mathbb R}^2+ \varepsilon ^2\|C \|_{1,2,2;2,\varepsilon ,\mathbb R}^2 \leq c. $$
\end{lemma}

\begin{proof} By the Yang-Mills flow equation (\ref{flow:eq:apriori:eq2}), the Bianchi identity (\ref{flow:eq:apriori:eq3}) and the commutation formula (\ref{flow:eq:apriori:eq2}) we have 
\begin{align*}
0=&\nabla_s B_s +\frac 1{\varepsilon^2  }\nabla_sd_{  A }^*F_{  A }-\nabla_s\nabla_t B_t - \nabla_s*X_{t}(  A )\\
=&\nabla_s B_s+\frac 1{\varepsilon^2  }d_{  A }^*d_{  A } B_s -\nabla_t\nabla_t B_s -\nabla_td_{  A }C \\
&- d*X_{t}(  A )  B_s -\frac1{\varepsilon^2 }*[ B_s ,*F_{  A }]-[C , B_t ]\\
=&\nabla_s B_s +\frac 1{\varepsilon^2  }d_{  A }^*d_{  A } B_s -\nabla_t\nabla_t B_s -d_{  A }d_{  A }^* B_s \\
&- d*X_{t}(  A ) B_s -\frac1{\varepsilon^2 }*[ B_s ,*F_{  A }]-2[C , B_t ],
\end{align*}
\begin{align*}
0=&\nabla_sC -\frac 1{\varepsilon^2  }\nabla_sd_{  A }^* B_t \\
=&\nabla_sC +\frac 1{\varepsilon^2  }d_{  A }^*d_{  A }C +\frac 1{\varepsilon^2  } d_{  A }^*\nabla_t B_s +\frac 1{\varepsilon^2  }*[ B_s ,* B_t ]\\
=&\nabla_sC +\frac 1{\varepsilon^2  }d_{  A }^*d_{  A }C +\nabla_t\nabla_tC +\frac 2{\varepsilon^2  }*[ B_s ,* B_t ].
\end{align*}
Furthermore, choosing $s_0\in \mathbb R$ and a smooth cut-off function with support in $[s_0-1,s_0+2]$ and with value 1 on $[s_0,s_0+1]$, one can prove that, for $\Omega_1(s_0):=\Sigma\times S^1\times [s_0-1,s_0+2]$, 
\begin{equation}
\begin{split}
\| B_s &\|_{1,2,2;2,\varepsilon ,[s_0,s_0+1]}^2+ \varepsilon ^2\|C \|_{1,2,2;2,\varepsilon ,[s_0,s_0+1]}^2 \\
  \leq&\varepsilon^4 \left \|\nabla_s B_s +\frac 1{\varepsilon^2  }d_{  A }^*d_{  A } B_s -\nabla_t\nabla_t B_s +\frac 1{\varepsilon^2  }d_{  A }d_{  A }^* B_s \right\|_{L^2(\Omega_1(s_0))}^2\\
&+\varepsilon^6   \left\|\nabla_sC +\frac 1{\varepsilon^2  }d_{  A}^*d_{  A }C +\nabla_t\nabla_tC \right\|_{L^2(\Omega_1(s_0))}^2\\
&+\left\| B_s \right\|_{L^2(\Omega_1(s_0))}^2+\varepsilon ^2\|C \|^2_{L^2(\Omega_1(s_0))}\\
=& \left\|-\varepsilon^2  d*X_{t}(  A ) B_s -*[ B_s ,*F_{  A }]-2\varepsilon^2  [C , B_t ]\right\|_{L^2(\Omega_1(s_0))}^2\\
&+\varepsilon^2 \left\|*[ B_s ,* B_t ]\right\|_{L^2(\Omega_1(s_0))}^2+\left\| B_s \right\|_{L^2(\Omega_1(s_0))}^2+\varepsilon ^2\|C \|^2_{L^2(\Omega_1(s_0))}\\
\leq &\left\| B_s \right\|_{L^2(\Omega_1(s_0))}^2+\varepsilon ^2\|C \|^2_{L^2(\Omega_1(s_0))}\\
&+c\varepsilon \| B_s \|_{1,2,2;2,\varepsilon ,[s_0-1,s_0+2]}^2+ \varepsilon ^{3}\|C \|_{1,2,2;2,\varepsilon ,[s_0-1,s_0+2]}^2
%
%
\end{split}
\end{equation}
In order to prove the last estimate we use the energy bound
$$\int_{\Sigma\times S^1}\left(\frac 1{\varepsilon ^2}\|F_{A }\|^2+\| B_t \|^2\right) dt\leq b$$
combined with the Sobolev inequality for 1-forms on $\Sigma\times S^1$: For example for the term $\left\|*[ B_s ,* B_t ]\right\|_{L^2(\Omega_1(s_0))}^2$ we proceed in the following way.
\begin{align*}
\varepsilon ^2&\left\|*[ B_s ,* B_t ]\right\|_{L^2(\Omega_1)}^2\leq c\varepsilon ^2\int_{[s_0-1,s_0+2]} \left\| B_t \right\|_{L^2(\Sigma\times S^1)}^2\left\| B_s \right\|_{L^\infty(\Sigma\times S^1)}^2ds\\
\leq &c\varepsilon ^2\int_{[s_0-1,s_0+2]} \left\| B_s \right\|_{L^\infty(\Sigma\times S^1)}^2ds\\
\leq &c\varepsilon ^2\int_{[s_0-1,s_0+2]} \frac 1{\varepsilon }\left(\left\| B_s \right\|_{L^2(\Sigma\times S^1)}^2+  \left\|d_{A }^*d_{A } B_s \right\|_{L^2(\Sigma\times S^1)}^2\right) ds\\
&+c\varepsilon ^2\int_{[s_0-1,s_0+2]} \frac 1{\varepsilon }\left(\left\|d_{A }d_{A }^*  B_s \right\|_{L^2(\Sigma\times S^1)}^2+\varepsilon ^4\left\|\nabla_t\nabla_t  B_s \right\|_{L^2(\Sigma\times S^1)}^2\right)   ds\\
\leq& c\varepsilon  \| B_s \|_{1,2,2;2,\varepsilon ,[s_0-1,s_0+2]} .
\end{align*}
Finally since
\begin{align*}
\sum_{i=0}^\infty&\varepsilon^{\frac {|i|}2} \left(\| B_s \|_{1,2,2;2,\varepsilon  ,[s_0+i,s_0+i+1]}^2+ \varepsilon ^2\|C \|_{1,2,2;2,\varepsilon ,[s_0+i,s_0+i+1]}^2\right)\\
\leq &\sum_{i=0}^\infty  \varepsilon^{\frac {|i|}2} \left(\| B_s \|_{L^2(\Sigma\times S^1\times [s_0+i-1,s_0+i+2])}^2+ \varepsilon ^2\|C \|_{L^2(\Sigma\times S^1\times [s_0+i-1,s_0+i+2])}^2\right)\\
&+c \varepsilon \sum_{i=0}^\infty \varepsilon^{\frac {|i|}2} \left(\| B_s \|_{1,2,2;2,\varepsilon  ,[s_0+i-1,s_0+i+2]}^2+ \varepsilon ^2\|C \|_{1,2,2;2,\varepsilon ,[s_0+i-1,s_0+i+2]}^2\right)\\
\leq & 2 \|B_s\|_{L^2(\Omega_1(s_0))}^2+2\varepsilon^2\|C\|^2_{L^2(\Omega_1(s_0))}\\
&+c\varepsilon^{\frac 12}\sum_{i=0}^\infty\varepsilon^{\frac {|i|}2} \left(\| B_s \|_{1,2,2;2,\varepsilon  ,[s_0+i,s_0+i+1]}^2+ \varepsilon ^2\|C \|_{1,2,2;2,\varepsilon ,[s_0+i,s_0+i+1]}^2\right),
\end{align*}
we can conclude that 
\begin{align*}
\| B_s& \|_{1,2,2;2,\varepsilon  ,\mathbb R}^2+ \varepsilon ^2\|C \|_{1,2,2;2,\varepsilon ,\mathbb R}^2\\
\leq& \sum_{s_0\in \mathbb Z}\sum_{i=0}^\infty\varepsilon^{\frac {|i|}2} \left(\| B_s \|_{1,2,2;2,\varepsilon  ,[s_0+i,s_0+i+1]}^2+ \varepsilon ^2\|C \|_{1,2,2;2,\varepsilon ,[s_0+i,s_0+i+1]}^2\right)\\
 \leq& 5\left\| B_s \right\|_{L^2}^2+5\varepsilon ^2\|C \|^2_{L^2}\leq 5b
\end{align*}
for $\varepsilon$ small enough.
\end{proof}

\begin{remark}
We can prove that the curvature of a connection, which represents a Yang-Mills flow, is smooth with an analogous argument as for the previous lemma. Our connection satisfies the perturbed Yang-Mills equation and thus, by the Bianchi identity (\ref{flow:eq:apriori:eq3}) and the commutation formula (\ref{flow:eq:apriori:eq3}),
\begin{equation*}
\begin{split}
0=&d_AB_s+\frac 1{\varepsilon^2}d_Ad_A^*F_A-d_A\nabla_t B_t-d_A*X_t(A)\\
=&\nabla_s F_A+\frac 1 {\varepsilon^2}d_Ad_A^*F_A-\nabla_t d_AB_t+[B_t\wedge B_t]-d_A*X_t(A)\\
=&\nabla_s F_A+\frac 1 {\varepsilon^2}d_Ad_A^*F_A-\nabla_t\nabla_t F_A+[B_t\wedge B_t]-d_A*X_t(A),
\end{split}
\end{equation*}
\begin{align*}
0=&\nabla_t B_s +\frac 1{\varepsilon^2  }\nabla_td_{  A }^*F_{  A }-\nabla_t\nabla_t B_t - \nabla_t*X_{t}(  A )\\
=&\nabla_s B_t+d_{  A }C +\frac 1{\varepsilon^2  }d_{  A }^*\nabla_tF_A-\nabla_t\nabla_t B_t  \\
&- d*X_{t}(  A )  B_t-*\dot X_t(A) -\frac1{\varepsilon^2 }*[ B_t ,*F_{  A }]\\
=&\nabla_s B_t+\frac 1{\varepsilon^2  }d_{  A }d_{  A }^* B_t +\frac 1{\varepsilon^2  }d_{  A }^*d_{  A } B_t-\nabla_t\nabla_t B_t \\
&- d*X_{t}(  A )  B_t-*\dot X_t(A) -\frac1{\varepsilon^2 }*[ B_t ,*F_{  A }]
\end{align*}
Thus in the same way as for the last lemma we can estimate all the first derivatives of $F_A$ and $B_t$ in a set $\Omega_{s_0}:=\Sigma\times S^1\times [s_0,s_0+1]$, $s_0\in\mathbb R$. Then, by a bootstrapping argument one can prove that the curvature terms are in $W^{k,2}$ for any $k$.
\end{remark}

We denote by $\Delta:=\partial_1^2+\partial_2^2+\dots+\partial_n^2$ the standard Laplacian on $\mathbb R^n$, $n>0$, and we define $P_r:=(-r^2,0)\times B_r(0)$.

\begin{lemma}\label{flow:lemma:SWB1}
For every $n\in\mathbb N$ there is a constant $c_n>0$ such that the following holds for every $r\in (0,1]$. If $a\geq0$ and $w:\mathbb R\times \mathbb R^n\supset P_r\to \mathbb R$ is $C^1$ in the $s$-variable and $C^2$ in the $x$-variable such that
\begin{equation}
(\Delta-\partial_s) w\geq -a w,\quad w\geq0,
\end{equation}
then
\begin{equation}
w(0)\leq \frac {c_n e^{ar^2}}{r^{n+2}}\int_{P_r}w.
\end{equation}
\end{lemma}

\begin{lemma}\label{flow:lemma:SWB4}
Let $R, r>0$ and $u:\mathbb R\times \mathbb R^n\supset P_{R+r}\to \mathbb R$ be $C^1$ in the s-variable and $C^2$ in the x-variable and $f, g: P_{R+r}\to \mathbb R$ be continuous functions such that
\begin{equation}
(\Delta-\partial_s) u \geq g-f,\quad u\geq0,\quad f\geq0,\quad g\geq0.
\end{equation}
Then
\begin{equation}
\int_{P_R}g\leq \int_{P_{R+r}}f+\left(\frac 4{r^2}+\frac 1{Rr}\right)\int_{P_{R+r}\backslash P_R}u.
\end{equation}
\end{lemma}

\begin{proof}[Proof of the theorem \ref{flow:thm:apriori}]
The proof will be divided in seven steps. In the first one we will prove that the $L^2$-norm over $\Sigma$ of $F_A$ can be bounded by any positive constant provided we choose $\varepsilon$ small enough. This allows us to apply the lemmas \ref{lemma76dt94} and \ref{lemma82dt94} for $p=2$. The next three steps provide bounds of the $L^2$-norm over $\Sigma$ for some curvature terms (step 2), their derivatives (step 3) and their second derivatives (step 4). In the last two steps we will prove the estimates (\ref{flow:eq:apriorig1}) and (\ref{flow:eq:apriorig2}).\\

\noindent{\bf Step 1.} We choose a positive constant $\delta_0<1$. There is a constant $\varepsilon_0>0$ such that the following holds. If $0<\varepsilon<\varepsilon_0$, then $\sup_{(t,s)\in S^1\times\mathbb R } \|F_A\|_{L^2(\Sigma)}\leq \delta_0$.

\begin{proof}[Proof of step 1.]
The idea is to use lemma \ref{flow:lemma:SWB1} and therefore we need an estimate from below of $\left(\partial_t^2-\partial_s\right)\| F_A\|^2$. First, using the Bianchi identity (\ref{flow:eq:apriori:eq3}) in the second and in the fourth equality, the commutation formula (\ref{flow:eq:apriori:eq4}) in the third and the Yang-Mills flow equation (\ref{flow:eq:apriori:eq2}) in the fourth, we obtain that
\begin{equation}\label{flow:eq:estmkc}
\begin{split}
\frac 12\big(\partial_t^2&-\partial_s\big)\| F_A\|^2=\|\nabla_t F_A\|^2+\langle F_A, \nabla_t\nabla_t F_A\rangle-\langle F_A,\nabla_s F_A\rangle\\
=&  \|\nabla_t F_A\|^2+\langle F_A, \nabla_t d_A B_t\rangle-\langle F_A,\nabla_s F_A\rangle\\
=&  \|\nabla_t F_A\|^2+\langle F_A, d_A\nabla_t B_t\rangle+\langle F_A, [B_t\wedge B_t]\rangle-\langle F_A,\nabla_s F_A\rangle\\
=&  \|\nabla_t F_A\|^2+\frac 1{\varepsilon^2}\langle F_A, d_Ad_A^*F_A\rangle+\langle F_A, d_A B_s\rangle \\
&-\langle F_A,d_A*X_t(A)\rangle +\langle F_A, [B_t\wedge B_t]\rangle-\langle F_A,d_A B_s\rangle\\
=&  \|\nabla_t F_A\|^2+\frac 1{\varepsilon^2}\|d_A^*F_A\|^2  +\langle F_A, [B_t\wedge B_t]\rangle-\langle d_A^*F_A,*X_t(A)\rangle;
\end{split}
\end{equation}
therefore, applying the Cauchy-Schwarz inequality
$$|\langle d_A^*F_A,*X_t(A)\rangle|\leq c\|d_A^*F_A\|\leq \frac 1{2\varepsilon^2}\|d_A^*F_A\|^2+2c^2\varepsilon^2,$$
for any positive $\gamma_1$ we have
\begin{equation}
\begin{split}
\frac 12\left(\partial_t^2-\partial_s\right)\| F_A\|^2\geq &-\|B_t\|_{L^4(\Sigma)}^2\| F_A\|-c \varepsilon^2\geq-\frac c{\gamma_1}\|F_A\|^2-\gamma_1-c \varepsilon^2\\
\geq&-\frac c{\gamma_1}\left(\|F_A\|^2+\frac{\gamma_1^2}c-\varepsilon^2\gamma_1\right)
\end{split}
\end{equation}
where $c\geq 1$ depends on $\|B_t\|_{L^4(\Sigma)}$ and on $\|X_t\|_{L^2(\Sigma)}$. Thus, by lemma \ref{flow:lemma:SWB1}, for every $r\in (0,1]$, $P_r:=(-r^2,0)\times B_r(0)$,
\begin{equation}
\begin{split}
\|F_A\|^2\leq&\frac {c_1 e^{\frac c{\gamma_1}r^2}}{r^3}\int_{P_r}\left(\|F_A\|^2+\frac{\gamma_1^2}c+\varepsilon^2\gamma_1\right) dt\,ds\\
\leq & \frac {4c_1 e^{\frac c{\gamma_1}r^2}b\varepsilon^2}{r}+ 2c_1e^{\frac c{\gamma_1}r^2}\left(\frac{\gamma_1^2}c+\varepsilon^2\gamma_1\right)
\end{split}
\end{equation}
and where we use that $\int_0^1\|F_A\|^2 dt\leq 2b\varepsilon^2$; next, we choose $\gamma_1:=\frac {\delta_0}{2(c_1e)^{\frac12}}$ and $r=\left(\frac{\gamma_1}c\right)^{\frac12}$, then $\|F_A\|^2<4\sqrt 2 c_1^{\frac54}c^{\frac12} e^{\frac54}b\delta_0^{-\frac12}\varepsilon^2+ \frac 12\delta_0^2+{(c_1e)^{\frac12} \delta_0\varepsilon^2 }$ and finally with
 $\varepsilon^2 <\frac12 \frac{\delta_0^\frac 52}{4\sqrt 2 c_1^{\frac54}c^{\frac12} e^{\frac54}b+(c_1e)^{\frac12} \delta_0^{\frac32}}$, it follows that $\|F_A\|_{L^2(\Sigma)}<\delta_0$ and we end the proof of the first step.
\end{proof}
For the rest of the proof we choose $\delta_0$ satisfying the condition of the theorems \ref{lemma76dt94} and \ref{lemma82dt94} for $p=2$.\\

\noindent{\bf Step 2.} There are two constants $\varepsilon_0, c>0$ such that the following holds. If $0<\varepsilon<\varepsilon_0$, then
\begin{equation*}
\begin{split}
\sup_{(s,t)\in S^1\times Q}&\left(\|F_{ A}\|^2 +\varepsilon^2 \|B_t \|^2+\varepsilon^2\|B_s\|^2\right) \\
\leq &c\int_{S^1\times\Omega } \left(\varepsilon^2 \|B_t \|^2+\varepsilon^2c_{ \dot X_t} +\|F_{ A}\|^2+\varepsilon^2\|B_s\|^2+\varepsilon^4\|C\|^2\right)dt\,ds,
\end{split}
\end{equation*}
\begin{equation*}
\begin{split}
\int_{S^1\times Q}\left(\|F_A\|^2+ \varepsilon^2\|\nabla_t F_A\|^2+\|d_A^*F_A\|^2\right)\leq c\varepsilon^2\int_{S^1\times \Omega}\left(\|F_A\|^2+\varepsilon^2\|B_t\|^2\right) dt,
\end{split}
\end{equation*}
\begin{equation*}
\begin{split}
\int_{S^1\times Q}&\left( \varepsilon^2\|\nabla_tB_t\|^2+\|d_AB_t\|^2+\|d_A^*B_t\|^2\right) dt\, ds\\
\leq &c \int_{S^1\times \Omega} \left(\|F_A\|^2+\varepsilon^2\|B_t\|^2+\varepsilon^2 c_{\dot X_t(A)}\right) dt\,ds,
\end{split}
\end{equation*}
\begin{equation*}
\begin{split}
\int_{S^1\times Q}&\left( \varepsilon^2\|\nabla_tB_s\|^2+\|d_AB_s\|^2+\|d_A^*B_s\|^2\right) dt\, ds\\
\leq& c \int_{S^1\times \Omega} \left( \|F_A\|^2+\varepsilon^2 \|B_s\|^2+\varepsilon^2\|B_t\|^2+\varepsilon^4\|C\|^2\right) dt\,ds.
\end{split}
\end{equation*}

\begin{proof}[Proof of step 2.]
Analogously to the first step we need to compute $\frac 12 \left(\partial_t^2-\partial_s\right)\|B_t\|^2$, $\frac 12 \left(\partial_t^2-\partial_s\right)\|B_s\|^2$ and $\frac 12 \left(\partial_t^2-\partial_s\right)\|C\|^2$ in order to apply the lemmas \ref{flow:lemma:SWB1} and \ref{flow:lemma:SWB4}. On the one hand, we consider the following computation
\begin{align*}
\frac 12 (\partial_t^2&-\partial_s)\|B_t \|^2= \|\nabla_tB_t\|^2-\langle\nabla_s B_t,B_t\rangle+\langle \nabla_t \nabla_t B_t,B_t\rangle\\
=&\|\nabla_tB_t\|^2-\langle\nabla_s B_t,B_t\rangle+\frac 1{\varepsilon^2 }\langle \nabla_t d_{ A}^*F_{ A} ,B_t\rangle\\
&+\langle \nabla_t B_s,B_t\rangle-\langle\nabla_t *X_{t}( A),B_t\rangle\\
=&\|\nabla_tB_t\|^2+\frac 1{\varepsilon^2}\langle d_{ A}^*\nabla_t F_{ A} ,B_t\rangle
+\frac 1{\varepsilon^2 }\langle *[B_t,*F_{ A}] ,B_t\rangle\\
&+\langle d_{ A}C,B_t\rangle-\langle d*X_{t}( A)B_t+\dot X_t( A),B_t\rangle\\
\intertext{where for the second equality we use the Yang-Mills flow equation (\ref{flow:eq:apriori:eq2}) and for the third the commutation formula (\ref{flow:eq:apriori:eq4}) and the Bianchi identity (\ref{flow:eq:apriori:eq3}). By the Yang-Mills flow equation (\ref{flow:eq:apriori:eq2}) and the Bianchi identity (\ref{flow:eq:apriori:eq3}) we finally obtain that}
=&\|\nabla_tB_t\|^2+\frac 1{\varepsilon^2 }\| d_{ A}B_t\|^2+\frac 1{\varepsilon^2 }\| d_{ A}^*B_t\|^2
+\frac 1{\varepsilon^2 }\langle *[B_t,*F_{ A}] ,B_t\rangle\\
&-\langle d*X_{t}( A)B_t+\dot X_{t}( A),B_t\rangle.
\end{align*}
Thus, since by the Cauchy-Schwarz inequality and the Sobolev estimate
\begin{align*}
|\langle *[B_t,*F_{ A}] ,B_t\rangle|
\leq& c\left (\|B_t\|+\|d_AB_t\|+\|d_A^*B_t\|\right)\|B_t\|_{L^4(\Sigma)}\|F_A\|\\
\leq&  c\varepsilon^2\|B_t\|^2+\frac 1{\varepsilon^2} \|F_A\|^2+ \frac12\|d_AB_t\|^2+\frac12\|d_A^*B_t\|^2
\end{align*}
$$|\langle d*X_{t}( A)B_t+\dot{X}_{t}( A),B_t\rangle|\leq 
c\|B_t\|^2+\|\dot X_t(A)\|^2,$$
we have
\begin{equation}\label{flow:est1:apriori}
\begin{split}
\left(\partial_t^2-\partial_s\right) \|B_t\|^2
\geq &\|\nabla_tB_t\|^2+\frac 1{\varepsilon^2}\|d_AB_t\|^2+\frac 1{\varepsilon^2}\|d_A^*B_t\|^2\\
&- \frac c{\varepsilon^4}\|F_A\|^2-c\|B_t\|^2-c\|\dot X_t(A)\|^2.
\end{split}
\end{equation}
On the other hand, using the Bianchi identity (\ref{flow:eq:apriori:eq3}) in the second equality and the commutation formula (\ref{flow:eq:apriori:eq4}) in the third, it follows that
\begin{align*}
\frac 12(\partial_t^2&-\partial_s)\|B_s\|^2= \|\nabla_t B_s\|^2+\langle B_s,\nabla_t\nabla_t B_s\rangle-\langle B_s,\nabla_s B_s\rangle\\
=&\|\nabla_t B_s\|^2+\langle B_s,\nabla_t\nabla_s B_t\rangle+\langle B_s,\nabla_t d_AC\rangle-\langle B_s,\nabla_s B_s\rangle\\
=&\|\nabla_t B_s\|^2+\langle B_s,\nabla_s\nabla_t B_t\rangle-\langle B_s,[C,B_t]\rangle\\
&+\langle d_A^*B_s,\nabla_t C\rangle+ \langle B_s, [B_t,C]\rangle-\langle B_s,\nabla_s B_s\rangle\\
\intertext{the next equality follows from the Yang-Mills flow equation (\ref{flow:eq:apriori:eq2}) and the after one from the commutation formula (\ref{flow:eq:apriori:eq4})}
=&\|\nabla_t B_s\|^2+\frac 1{\varepsilon^2}\langle B_s,\nabla_s d_A^*F_A\rangle-\langle B_s,\nabla_s*X_t(A)\rangle+\langle B_s,\nabla_s B_s\rangle\\
&-2\langle B_s,[C,B_t]\rangle+\frac 1{\varepsilon^2}\langle d_A^*B_s,\nabla_t d_A^*B_t\rangle-\langle B_s,\nabla_s B_s\rangle\\
=&\|\nabla_t B_s\|^2+\frac 1{\varepsilon^2}\langle d_AB_s,\nabla_s F_A\rangle-\frac 1{\varepsilon^2}\langle B_s,*[B_s,*F_A]\rangle\\
&-\langle B_s,\nabla_s*X_t(A)\rangle-2\langle B_s,[C,B_t]\rangle+\frac 1{\varepsilon^2}\langle d_A^*B_s,d_A^*\nabla_t B_t\rangle\\
\intertext{finally, applying the Bianchi identity (\ref{flow:eq:apriori:eq3}) and the Yang-Mills flow equation (\ref{flow:eq:apriori:eq2}) one more time, we can conclude that}
=&\|\nabla_t B_s\|^2+\frac 1{\varepsilon^2}\|\nabla_s F_A\|-\frac 1{\varepsilon^2}\langle B_s,*[B_s,*F_A]\rangle-\langle B_s,\nabla_s*X_t(A)\rangle\\
&-2\langle B_s,[C,B_t]\rangle+\frac 1{\varepsilon^2}\langle d_A^*B_s,d_A^*B_s\rangle\\
&+\frac 1{\varepsilon^4}\langle d_A^*B_s,d_A^*d_A^*F_A\rangle+\frac 1{\varepsilon^2}\langle d_A^*B_s,d_AX_t(A)\rangle\\
=&\|\nabla_t B_s\|^2+\frac 1{\varepsilon^2}\|\nabla_s F_A\|+\frac 1{\varepsilon^2}\|d_A^*B_s\|^2-2\langle B_s,[C,B_t]\rangle\\
&-\frac 1{\varepsilon^2}\langle B_s,*[B_s,*F_A]\rangle-\langle B_s,d*X_t(A)B_s\rangle
\end{align*}
where we use that $[F_A,*F_A]=0$ and $d_AX_t(A)=0$. Thus, since $\|B_s\|_{L^\infty(\Sigma)}$ is bounded by a constant by the assumptions,
\begin{equation}\label{flow:est1:apriori2}
\begin{split}
\left(\partial_t^2-\partial_s\right) \|B_s\|^2
\geq &\|\nabla_tB_s\|^2+\frac 1{\varepsilon^2}\|d_AB_s\|^2+\frac 1{\varepsilon^2}\|d_A^*B_s\|^2\\
&- \frac c{\varepsilon^4}\|F_A\|^2-c\|B_s\|^2-c \|B_t\|^2-\|C\|^2.
\end{split}
\end{equation}
If we consider
\begin{align*}
\frac 12 \big(\partial_t^2&-\partial_s\big)\varepsilon^4\|C\|^2=\frac 12 \big(\partial_t^2-\partial_s\big)\|d_A^*B_t\|^2\\
=& \|\nabla_td_A^*B_t\|^2-\langle \nabla_s d_A^*B_t, d_A^*B_t \rangle
+\langle \nabla_t\nabla_td_A^*B_t,d_A^*B_t\rangle\\
\intertext{by the Yang-Mills flow equation (\ref{flow:eq:apriori:eq2}) and the commutation formula (\ref{flow:eq:apriori:eq4}) we have}
=& \|\nabla_td_A^*B_t\|^2-\langle \nabla_s d_A^*B_t, d_A^*B_t \rangle+\frac 1{\varepsilon^2}\langle \nabla_td_A^*d_A^*F_A,d_A^*B_t\rangle\\
&+\langle \nabla_td_A^*B_s,d_A^*B_t\rangle-\langle \nabla_td_A^**X_t(A),d_A^*B_t\rangle-\langle *\nabla_t[B_t\wedge *B_t],d_A^*B_t\rangle\\
\intertext{and using the commutation formula (\ref{flow:eq:apriori:eq2}) one more time}
=& \|\nabla_td_A^*B_t\|^2-\langle \nabla_s d_A^*B_t, d_A^*B_t \rangle+\langle d_A^*\nabla_tB_s,d_A^*B_t\rangle\\
&-\langle *[B_t\wedge *B_s],d_A^*B_t\rangle-\langle *\nabla_td_A X_t(A),d_A^*B_t\rangle\\
=& \|\nabla_td_A^*B_t\|^2+\frac 2{\varepsilon^2}\|d_Ad_A^*B_t\|^2\\
&-2\langle *[B_t\wedge *B_s],d_A^*B_t\rangle-\langle *\nabla_td_A X_t(A),d_A^*B_t\rangle
\end{align*}
where the last step follows from the Bianchi identity (\ref{flow:eq:apriori:eq3}) and the commutation formula (\ref{flow:eq:apriori:eq4}). Using the Cauchy-Schwarz inequality and that $\|B_s\|_{L^\infty(\Sigma)}$ is uniformly bounded, one can easily see that
\begin{equation}\label{flow:apr:kjdas}
\big(\partial_t^2-\partial_s\big)\varepsilon^4\|C\|^2\geq {\varepsilon^2}\|d_AC\|^2+c\varepsilon^2\|B_t\|^2+c\varepsilon^2\|\dot X_t(A)\|^2
\end{equation}
and therefore, by (\ref{flow:est1:apriori}), (\ref{flow:est1:apriori2}), (\ref{flow:apr:kjdas}) and $\|C\|\leq c\|d_AC\|$ by the lemma \ref{lemma76dt94} we have for two positive constants $c$, $c_0$
\begin{equation}\label{flow:ineq:apriori:22}
\begin{split}
(\partial_t^2&-\partial_s)\left(\|B_t\|^2+\|B_s\|^2+c_0\varepsilon^2\|C\|^2\right)\\
\geq &\|\nabla_tB_t\|^2+\frac 1{\varepsilon^2 }\| d_{ A}B_t\|^2+\frac 1{\varepsilon^2 }\| d_{ A}^*B_t\|^2+\|\nabla_t B_s\|^2+\frac 1{\varepsilon^2}\|d_AB_s\|\\
&+\frac 1{\varepsilon^2}\|d_A^*B_s\|^2
-\frac c {\varepsilon^4 }\|d_A^*F_{ A}\|^2- c\|B_t\|^2- c \|B_s\|^2-c\|\dot X_t(A)\|^2.
\end{split}
\end{equation}
Since $\|F_{ A }\|\leq \delta$, $\|F_A\|_{L^4(\Sigma)}\leq c\|d_A^*F_A\|$ for a constant $c$ by the theorems \ref{lemma76dt94} and \ref{lemma82dt94} there is  a $ A_1 \in \mathcal A^0(P)$ such that $\| A- A_1 \|\leq \|F_{ A}\|$ and thus we can write
\begin{equation}
\begin{split}
d_{ A}*&X_{t}( A)
= d_{ A }\left(*X_{t}( A)-*X_{t}( A_1)\right)\\
&+d_{ A_1 }*X_{t}( A_1)+\left[(A-A_1)\wedge * X_{t}( A_1)\right]
\end{split}
\end{equation}
where $d_{ A_1 }*X_{t}( A_1)=0$. Therefore, by (\ref{flow:eq:estmkc})
\begin{equation}\label{flow:ineq:apriori:11}
\begin{split}
\frac 12\left(\partial_t^2-\partial_s\right)\| F_A\|^2\geq &\frac 1{4\varepsilon^2}\| d_A^*F_A\|^2+\frac 14\|\nabla_tF_A\|^2-c\varepsilon^2\|B_t\|^2\\
\end{split}
\end{equation}
and with (\ref{flow:ineq:apriori:22}) it follows that for a constant $c_0$ big enough
\begin{equation*}
\begin{split}
\frac 12 \left(\partial_t^2-\partial_s\right)&\left( c_0\|F_{ A}\|^2+\varepsilon^2\|B_t \|^2+\varepsilon^2\|B_s\|^2+\varepsilon^2c_{\dot X_t} +c_0\varepsilon^4\|C\|^2\right)\\
\geq& -c \left(  c_0\|F_{ A}\|^2+\varepsilon^2 \|B_t \|^2+\varepsilon^2\|B_s\|^2+\varepsilon^2c_{\dot X_t}+c_0\varepsilon^4\|C\|^2 \right).
\end{split}
\end{equation*}
Finally by lemma \ref{flow:lemma:SWB1}, for a fix $r$ we can conclude that 
\begin{equation*}
\begin{split}
\sup_{(t,s)\in S^1\times Q}&\left(\|F_{ A}\|^2+\varepsilon^2 \|B_t \|^2+\varepsilon^2\|B_s\|^2+c_0\varepsilon^4\|C\|^2\right)\\
\leq &c \int_{S^1\times \Omega } \left(\varepsilon^2 \|B_t \|^2+\varepsilon^2c_{\dot X_t}+ \|F_{ A}\|^2+\varepsilon^2 \|B_s\|^2+c_0\varepsilon^4\|C\|^2\right)dt\,ds.
\end{split}
\end{equation*}
Next, we can apply lemma \ref{flow:lemma:SWB4} to the inequality (\ref{flow:ineq:apriori:11}) and we obtain
\begin{equation}\label{llkdskjgfj}
\begin{split}
\int_{S^1\times Q}&\left(\|F_A\|^2+\varepsilon^2\|\nabla_t F_A\|^2+\|d_A^*F_A\|^2\right)dt\,ds\\
\leq & c \varepsilon^2 \int_{S^1\times \Omega} \left(\|F_A\|^2+\varepsilon^2 \|B_t\|^2\right) dt\,ds \leq c\varepsilon^4,
\end{split}
\end{equation}
where the estimate of $\|F_A\|_{L^2(S^1\times Q)}^2$ follows from $\|F_A\|\leq c \|d_A^*F_A\|$. Using the lemma \ref{flow:lemma:SWB4} for the inequalities (\ref{flow:est1:apriori}) and (\ref{flow:ineq:apriori:22}) we can conclude the last two estimates of the third step:
\begin{equation}
\begin{split}
\int_{S^1\times Q}&\left( \varepsilon^2\|\nabla_tB_t\|^2+\|d_AB_t\|^2+\|d_A^*B_t\|^2\right) dt\, ds\\
\leq &c \int_{S^1\times \Omega} \left(\frac 1{\varepsilon^2}\|F_A\|^2+\varepsilon^2\|B_t\|^2+\varepsilon^2c_{\dot X_t(A)}\right) dt\,ds,
\end{split}
\end{equation}
\begin{equation}
\begin{split}
\int_{S^1\times Q}&\left( \varepsilon^2\|\nabla_tB_s\|^2+\|d_AB_s\|^2+\|d_A^*B_s\|^2\right) dt\, ds\\
\leq& c \int_{S^1\times \Omega} \left( \frac 1{\varepsilon^2}\|F_A\|^2+\varepsilon^2 \|B_s\|^2+\varepsilon^2\|B_t\|^2+c_0\varepsilon^4\|C\|^2\right) dt\,ds
\end{split}
\end{equation}
and the second step follows combining the last two estimates with (\ref{llkdskjgfj}).
\end{proof}
\noindent{\bf Step 3.}There are two constants $\varepsilon_0, c>0$ such that the following holds. If $0<\varepsilon<\varepsilon_0$, then
\begin{equation}\label{flow:eq:apriori:djdk2}
\begin{split}
\int_{S^1\times Q}& \left(\varepsilon^2\|\nabla_t d_AB_t\|^2+ \|d_A^*d_AB_t\|^2+  \varepsilon^2\|\nabla_td_A^*B_t\|^2+ \|d_Ad_A^*B_t\|^2\right)\\
\leq & c\varepsilon^2\int_{S^1\times \Omega}\left( \|F_A\|^2+\varepsilon^2\|B_t\|^2+\varepsilon^2c_{\dot X_t(A)}\right) dt\, ds,
\end{split}
\end{equation}
\begin{equation}\label{flow:eq:apriori:djdkss2}
\begin{split}
\int_{S^1\times Q}& \left(\varepsilon^2\|\nabla_t d_AB_s\|^2+ \|d_A^*d_AB_s\|^2+  \varepsilon^2\|\nabla_td_A^*B_s\|^2+ \|d_Ad_A^*B_s\|^2\right)\\
\leq & c\int_{S^1\times \Omega}\left( \|F_A\|^2+\varepsilon^2\|B_t\|^2+\varepsilon^2\|B_s\|^2+\varepsilon^4 c_{\dot X_t}+\varepsilon^6\|C\|^2\right) dt\, ds.
\end{split}
\end{equation}

\begin{proof}[Proof of step 3] Like in the previous steps we will prove this one using the lemmas \ref{flow:lemma:SWB1} and \ref{flow:lemma:SWB4} and therefore we need to compute $\frac 12 \big(\partial_t^2-\partial_s\big)||d_AB_t||^2$, $\frac 12 \big(\partial_t^2-\partial_s\big)||d_A^*B_t||^2$, $\frac 12 \big(\partial_t^2-\partial_s\big)||d_AB_s||^2$, $\frac 12 \big(\partial_t^2-\partial_s\big)||d_A^*B_t||^2$ and $\frac 12 \big(\partial_t^2-\partial_s\big)||d_A^*F_A||^2$. First, analogously to the previous steps, we obtain that
\begin{equation}\label{flow:eq:oihds1}
\begin{split}
\frac 12 \big(\partial_t^2&-\partial_s\big)\left(\|d_AB_t\|^2+\|d_A^*B_t\|^2\right)\geq \frac 12 \|\nabla_td_AB_t\|^2+ \frac 1{2\varepsilon^2}\|d_A^*d_AB_t\|^2\\
&+\frac 12 \|\nabla_td_A^*B_t\|^2+ \frac 1{2\varepsilon^2}\|d_Ad_A^*B_t\|^2
-\frac c{\varepsilon^2}\|d_A^*F_A\|^2\\
&-c \varepsilon^2 \|B_t\|^2-c\varepsilon^2\|\dot X_t(A)\|_{L^\infty }^2-c \varepsilon^2 \|\nabla_tB_t\|^2.
\end{split}
\end{equation}
and combined with (\ref{flow:est1:apriori}) yield to
\begin{equation}
\begin{split}
\big(\partial_t^2&-\partial_s\big)\left(  \|d_AB_t\|^2+  \|d_A^*B_t\|^2+c_1\|F_A\|^2+c_0\varepsilon^2\|B_t\|^2\right)\\
\geq&\|\nabla_td_AB_t\|^2+ \frac 1{\varepsilon^2}\|d_A^*d_AB_t\|^2+ \|\nabla_td_A^*B_t\|^2+ \frac 1{\varepsilon^2}\|d_Ad_A^*B_t\|^2\\
& - c\varepsilon^2\|B_t\|^2-c\varepsilon^2\|\dot X_t(A)\|_{L^\infty(\Sigma)}^2
\end{split}
\end{equation}
for two positive constants $c_0$ and $c_1$.
Therefore by lemma \ref{flow:lemma:SWB4} we have, for an open set $\Omega_1$ with $Q\subset \Omega_1\subset\subset \Omega$
\begin{equation}\label{flow:eq:apriori:djdk}
\begin{split}
\int_{S^1\times Q}& \left(\varepsilon^2\|\nabla_t d_AB_t\|^2+ \|d_A^*d_AB_t\|^2+  \varepsilon^2\|\nabla_td_A^*B_t\|^2+ \|d_Ad_A^*B_t\|^2\right)\\
\leq & c\varepsilon^2\int_{S^1\times \Omega_1}\left( \|d_AB_t\|^2+ \|d_A^*B_t\|^2+\|F_A\|^2+\varepsilon^2\|B_t\|^2+\varepsilon^2c_{\dot X_t(A)}\right) dt\, ds\\
\leq & c\varepsilon^2\int_{S^1\times \Omega}\left( \|F_A\|^2+\varepsilon^2\|B_t\|^2+\varepsilon^2c_{\dot X_t(A)}\right) dt\, ds
\end{split}
\end{equation}
where the second estimate follows from step 2.
Moreover, by the Bianchi identity $d_AB_s=\nabla_sF_A$ and by the identity $\nabla_tC=\frac 1{\varepsilon^2}d_A^*B_s$ which follows from the Yang-Mills flow equation, for a positive $c_1$
\begin{equation}
\begin{split}
\frac 12(\partial_t^2-\partial_s)&(\|d_AB_s\|^2+c_1\|d_A^*B_s\|^2)\geq \|\nabla_td_AB_s\|^2+\frac 1{\varepsilon^2}\|d_A^*d_AB_s\|^2\\
&+\|\nabla_td_A^*B_s\|^2+\frac 1{\varepsilon^2}\|d_Ad_A^*B_s\|^2\\
&- \frac {c}{\varepsilon^2} \| d_A^*B_s\|^2-c\varepsilon^2\|\nabla_tB_s\|^2-c\varepsilon^2\|\nabla_tB_t\|^2-c\frac 1{\varepsilon^2} \|d_Ad_A^*B_t\|^2\\
&-c\varepsilon^2 \|B_t\|^2-c\varepsilon^2 \|B_s\|^2-\frac c{\varepsilon^2}\|d_A^*F_A\|^2-c\frac {\|F_A\|^2}{\varepsilon^2}\|d_Ad_A^*B_s\|^2.
\end{split}
\end{equation}

Collecting all the estimates, for a constant $c_0>0$, we obtain
\begin{equation*}
\begin{split}
&\frac 12 \big(\partial_t^2-\partial_s\big)\left(c_0 \|F_A\|^2+c_0 \varepsilon^2\|B_t\|^2 +c_0\varepsilon^2 \|B_s\|^2\right)\\
&+\frac 12 \big(\partial_t^2-\partial_s\big)\left( c_1\|d_A^*B_s\|^2+ \|d_AB_s\|^2+\|d_A^*B_t\|^2\right)\\
&\qquad\geq \frac 1{\varepsilon^2}\|d_Ad_A^*B_s\|^2+\frac 1{\varepsilon^2}\|d_A^*d_AB_s\|^2- c\varepsilon^2\|B_t\|^2\\
&-c\varepsilon^2\|\dot X_t(A)\|_{L^\infty }-c\varepsilon^2 \|B_s\|^2\\
\end{split}
\end{equation*}
%
By the lemma \ref{flow:lemma:SWB4} we have then,
\begin{equation}
\begin{split}
\int_{ S^1\times Q}&\left( \|d_Ad_A^*B_s\|^2+ \|d_A^*d_AB_s\|^2\right)\\
\leq& c\varepsilon^2\int_{S^1\times  \Omega_1 } \left(\varepsilon^2\|B_t\|^2+\varepsilon^2c_{\dot X_t}+ \varepsilon^2\|B_s\|^2\right)\\
&+ c \varepsilon^2\int_{S^1\times  \Omega_1 } \left(\|F_A\|^2+ \|d_A^*B_s\|^2+ \|d_AB_s\|^2+\|d_A^*B_t\|^2\right)\\
\leq& c \int_{S^1\times \Omega } \left(\varepsilon^2\|B_t\|^2+\varepsilon^4 c_{\dot X_t}+\varepsilon^2 \|B_s\|^2+\varepsilon^2\|F_A\|^2+\varepsilon^6\|C\|^2\right).
\end{split}
\end{equation}

\end{proof}

\noindent{\bf Step 4.} There are two constants $\varepsilon_0, c>0$ such that the following holds.If $0<\varepsilon<\varepsilon_0$, then
\begin{equation*}
\begin{split}
\sup_{(t,s)\in S^1\times Q}&\left(\|d_Ad_A^*B_t\|^2+\|d_A^*d_AB_t\|^2\right)\\
\leq &c \int_{S^1\times \Omega}\left(\|F_A\|^2+\varepsilon^2\|B_t\|^2+\varepsilon^2c_{\dot X_t(A)}\right) dt\,ds,
\end{split}
\end{equation*}
\begin{equation*}
\begin{split}
\sup_{(t,s)\in S^1\times Q }&\left (\|d_A^*d_AB_s\|^2+\|d_Ad_A^*B_s\|^2\right)\\
\leq &\int_{S^1\times \Omega }(c\|F_A\|^2+c\varepsilon^2\|B_t\|^2+\varepsilon^2\|B_s\|^2+\varepsilon^2c_{\dot X_t(A)}+\varepsilon^4\|C\|^2) dt\,ds,
\end{split}
\end{equation*}
\begin{equation*}
\begin{split}
\sup_{(t,s)\in S^1\times Q}&\|d_A^*d_Ad_A^*B_t\|^2\\
\leq&c \int_{S^1\times \Omega} \left(\varepsilon^2\|B_s\|^2+\varepsilon^2\|B_t\|^2+\|F_A\|^2+\varepsilon^4c_{\dot X_t}+\varepsilon^4\|C\|^2\right)dt\,ds.
\end{split}
\end{equation*}

\begin{proof}[Proof of step 4] Analogously to the first two steps we can estimate  $$\big(\partial_t^2-\partial_s\big)
\left(\|d_Ad_A^*B_t\|^2+\|d_A^*d_AB_t\|^2\right)$$ and we obtain that there are two positive constants $c$ and $c_0$ such that
\begin{equation*}
\begin{split}
\big(\partial_t^2-\partial_s\big)&
\left(\|d_Ad_A^*B_t\|^2+\|d_Ad_A^*B_t\|^2+{c_0}\|F_A\|^2+c_0\varepsilon^2\|B_t\|+c_0\varepsilon^2\|B_s\|^2\right)\\
\geq &-c \varepsilon^2\|B_t\|^2-c\varepsilon^2\|B_s\|^2+\varepsilon c_{\dot X_t}.
\end{split}
\end{equation*}
We can therefore conclude by the lemma \ref{flow:lemma:SWB1} and the previous step that
\begin{equation*}
\begin{split}
\sup_{(t,s)\in S^1\times Q}&\left(\|d_Ad_A^*B_t\|^2+\|d_A^*d_AB_t\|^2\right)\\
\leq &c \int_{S^\times \Omega}\left(\|F_A\|^2+\varepsilon^2\|B_t\|^2+\varepsilon^2c_{\dot X_t(A)}\right) dt\,ds.
\end{split}
\end{equation*}
Furthermore, for two positive constants $c$ and $c_0$
\begin{equation*}
\begin{split}
\frac 12 \left(\partial_t^2-\partial_s\right) (&\|d_A^*d_AB_s\|^2+\|d_Ad_A^*B_s\|^2+c_0\|F_A\|^2+c_0\|d_AB_t\|^2+c_0\|d_Ad_A^*B_t\|^2\\
&+c_0\|d_A^*d_AB_t\|^2+c_0\|d_A^*B_t\|^2+c_0\varepsilon^2 \|B_t\|^2+c_0\varepsilon^2\|B_s\|^2)\\
\geq &c\left(\|F_A\|^2+\varepsilon^2 \|B_t\|^2+\varepsilon^2\|B_s\|^2+c_{\dot X_t(A)}\right)
\end{split}
\end{equation*}
Thus, by the lemma \ref{flow:lemma:SWB1} and the previous steps we can comnclude that
\begin{equation*}
\begin{split}
\sup_{(t,s)\in S^1\times Q }& (\|d_A^*d_AB_s\|^2+\|d_Ad_A^*B_s\|^2)\\
\leq &c\int_{S^1\times \Omega_1 }\left(\|d_A^*d_AB_s\|^2+\|d_Ad_A^*B_s\|^2+\|F_A\|^2+\|d_AB_t\|^2\right) dt\,ds\\
\leq &c\int_{S^1\times \Omega_1 }\left(\|d_Ad_A^*B_t\|^2+\|d_A^*B_t\|^2\right) dt\,ds\\
\leq &c\int_{S^1\times \Omega_1 }\left(\varepsilon^2\|B_t\|^2+\varepsilon^2\|B_s\|^2+\varepsilon^2c_{\dot X_t(A)}\right) dt\,ds\\
\leq &c\int_{S^1\times \Omega }\left(\|F_A\|^2+\varepsilon^2\|B_t\|^2+\varepsilon^2\|B_s\|^2+\varepsilon^2c_{\dot X_t(A)}\right) dt\,ds.
\end{split}
\end{equation*}
Moreover,

\begin{equation*}
\begin{split}
\frac 12 \big(\partial_t^2&-\partial_s\big)\|d_A^*d_Ad_A^*B_t\|^2
\geq  \|\nabla_td_A^*d_Ad_A^*B_t\|^2+\frac 1{\varepsilon^2}\|d_Ad_A^*d_Ad_A^*B_t\|^2\\
&-c\varepsilon^2 \left(\|B_s\|^2+\|d_AB_s\|^2+\|d_AB_t\|^2\right)-c\varepsilon^2\|d_A d_A^*B_s\|-\frac c{\varepsilon^2}\|d_A^*F_A\|^2
\end{split}
\end{equation*}
and hence by the lemma \ref{flow:lemma:SWB1} we can conclude the proof of the fourth step:
\begin{equation*}
\begin{split}
\sup_{(t,s)\in S^1\times Q}&\|d_A^*d_Ad_A^*B_t\|^2\\
\leq&c \int_{S^1\times \Omega_1} \left(\varepsilon^2\|B_s\|^2+\varepsilon^2\|d_AB_s\|^2+\varepsilon^2\|d_AB_t\|^2\right)dt\,ds\\
&+c \int_{S^1\times \Omega_1} \left(\varepsilon^2\|B_t\|^2+\|F_A\|^2+\|d_A^*d_Ad_A^*B_t\|^2\right)dt\,ds\\
\leq&c \int_{S^1\times \Omega} \left(\varepsilon^2\|B_s\|^2+\varepsilon^2\|B_t\|^2+\|F_A\|^2+\varepsilon^4c_{\dot X_t}+\varepsilon^4\|C\|^2\right)dt\,ds.
\end{split}
\end{equation*}
\end{proof}

\noindent{\bf Step 5.} There are two constants $\varepsilon_0, c>0$ such that the following holds. If $0<\varepsilon<\varepsilon_0$, then

\begin{equation}\label{flow:apriori:esti1}
\sup_{S^1\times Q}\left(\|F_A\|+\varepsilon\|\nabla_t F_A\|+\|d_Ad_A^* F_A\|\right)\leq c\varepsilon^{2-\frac1p}.
\end{equation}

\begin{proof}[Proof of step 5.] In order to prove the fifth step we need the following observation. For
$$f(t,s):=\|F_A\|^2+\varepsilon^2\|\nabla_t F_A\|^2+\varepsilon^{4}\|d_Ad_A^* F_A\|^2+\varepsilon^2\|d_A^*B_t\|^2$$
we have that
$$\varepsilon^2(\partial_t^2-\partial_s)f\geq {c_0}f-c\varepsilon^4$$
and since for $p\in \mathbb N$, $p\leq2$
$$\partial_t^2f^p=pf^{p-1}\partial_t^2f+p(p-1)\left(\partial_tf \right)^2\geq pf^{p-1}\partial_t^2f,$$

\begin{align*}
\varepsilon^2(\partial_t^2-\partial_s)f^p\geq &pf^{p-1}(\partial_t^2-\partial_s)f 
\geq pc_0f^p-pc\varepsilon^4f^{p-1}
\geq {c_0}f^p-c\varepsilon^{4p}
\end{align*}
where we used that $ab\leq \frac {(p-1)a^{\frac p{p-1}}}p+\frac{b^p}p$ for any positive numbers $a$ and $b$. Therefore by lemma \ref{flow:lemma:SWB4}, for a sequence of open sets $\Omega_i\subset \mathbb R$, $i=1,\dots,2p$, $Q\subset\Omega_0\subset\subset \Omega_1\subset\subset \Omega_2\subset\subset \dots \subset\subset \Omega_{2p}\subset\subset\Omega$
$$\int_{S^1\times \Omega_0}f^{p}dt\,ds\leq c\varepsilon^{4p}+c\varepsilon^2\int_{S^1\times \Omega_1}f^p dt\,ds\leq c\varepsilon^{4p}+c\varepsilon^{4p}\int_{S^1\times \Omega_{2p}}f^p dt\,ds\leq c\varepsilon^{4p}$$
where the last step follows by iterating the first one and from $\sup_{\Omega_{2p}}f\leq c\varepsilon$. By lemma \ref{flow:lemma:SWB1}
$$\sup_{S^1\times Q}\varepsilon^2f^p\leq c\varepsilon^{4p}+\varepsilon^2\int_{S^1\times \Omega_0}f^{p}dt\,ds \leq c\varepsilon^{4p}$$
and therefore
$$\sup_{S^1\times Q}\left(\|F_A\|+\varepsilon\|\nabla_t F_A\|+\|d_Ad_A^* F_A\|\right)\leq c\varepsilon^{2-\frac1p}.$$
\end{proof}

\noindent{\bf Step 6.}  There are two constants $\varepsilon_0, c>0$ such that the following holds. If $0<\varepsilon<\varepsilon_0$, then

\begin{equation}\label{flow:apriori:esti11}
\sup_{S^1\times Q}\left(\|F_A\|_{L^\infty(\Sigma)}+\|d_Ad_A^* F_A\|+\varepsilon^2\|\nabla_t\nabla_t F_A\|\right)\leq c\varepsilon^{2}.
\end{equation}

\begin{proof}
One can show that, by the previous steps, there are three constants $c$, $c_0$ and $c_1$ such that
\begin{align*}
\frac 12\left(\partial_t^2-\partial_s\right)&\left(\|\nabla_t\nabla_tF_A\|^2+\frac {c_0}{\varepsilon^2}\|F_A\|^2+c_1\|B_t\|^2\right)\\
\geq &-c\left(\|\nabla_t\nabla_tF_A\|^2+\frac {c_0}{\varepsilon^2}\|F_A\|^2+c_1\|B_t\|^2\right)
\end{align*}
and thus by the lemma \ref{flow:lemma:SWB1}
$$\sup_{s\in Q}\|\nabla_t\nabla_tF_A\|^2\leq c\int_{S^1\times \Omega} \left(\|\nabla_t\nabla_tF_A\|^2+\frac {c_0}{\varepsilon^2}\|F_A\|^2+c_1\|B_t\|^2\right) dt\,ds\leq c.$$
Furthermore, by the Bianchi identity (\ref{flow:eq:apriori:eq3}) and the Yang-Mills flow equation (\ref{flow:eq:apriori:eq2})
\begin{align*}
\sup_{s\in Q}\|d_Ad_A^*F_A\|^2\leq &\sup_{s\in Q}\varepsilon^4\|d_A\nabla_tB_t\|^2+c\varepsilon^4\\
\leq &\sup_{s\in Q}\varepsilon^4\|\nabla_td_AB_t\|^2+\varepsilon^4\|[B_t\wedge B_t]\|^2+c\varepsilon^4\\
\leq &\sup_{s\in Q}\varepsilon^4\|\nabla_t\nabla_tF_A\|^2+c\varepsilon^4\leq c\varepsilon^4
\end{align*}
and by the lemma \ref{lemma76dt94}
$$\|F_A\|_{L^\infty(\Sigma)}\leq c \|d_Ad_A^*F_A\|_{L^(\Sigma)}\leq c\varepsilon^2.$$
\end{proof}

\noindent{\bf Step 7.}  There are two constants $\varepsilon_0, c>0$ such that the following holds. If $0<\varepsilon<\varepsilon_0$, then
\begin{align*}
\sup_{(t,s)\in S^1\times Q}&\left(\varepsilon^2\|B_s\|^2+\varepsilon^4\|C\|^2+\mathcal N_1+\mathcal N_2+\mathcal N_3+\mathcal N_4\right)\\
\leq &\varepsilon^2c \int_{S^1\times \Omega}\left(\|B_s\|^2+\varepsilon^2\|C\|^2\right) dt\,ds.
\end{align*}

\begin{proof}
Since we know now that, for a positive constant $c_0$,
$$\varepsilon^2\|B_t\|_{L^\infty(\Sigma)}+\|F_A\|_{L^\infty(\Sigma)}\leq c_0\varepsilon^2,$$
using the computations of the second step we can obtain the following estimate a positive constant $c$
\begin{align*}
\frac 12\left(\partial_t^2-\partial_s\right)&\left(\|B_s\|^2+\varepsilon^2\|C\|^2\right)\\
\geq& \|\nabla_tB_s\|^2+\frac 1{\varepsilon^2}\|d_AB_s\|^2+\frac 1{\varepsilon^2}\|d_A^*B_s\|+\varepsilon^2\|\nabla_tC\|^2+\|d_AC\|^2\\
&-c\left(\|B_s\|^2+\varepsilon^2\|C\|^2\right).
\end{align*}
Analogously, for seven positive constants $c_i$, $i=1,\dots,7$, we have
\begin{align*}
\left(\partial_t^2-\partial_s\right)&\left(\mathcal N_1+c_1\varepsilon^2\|B_s\|^2+c_1\varepsilon^4\|C\|^2\right)\\
\geq& \frac {c_5}{\varepsilon^2}\mathcal N_2-c\left(\mathcal N_1+c_1\varepsilon^2\|B_s\|^2+c_1\varepsilon^4\|C\|^2\right),
\end{align*}
\begin{align*}
\left(\partial_t^2-\partial_s\right)&\left(\mathcal N_2+c_2\mathcal N_1+c_2c_1\varepsilon^2\|B_s\|^2+c_2c_1\varepsilon^4\|C\|^2\right)\\
\geq &\frac {c_6}{\varepsilon^2}\mathcal N_3
-c\left(\mathcal N_2+c_2\mathcal N_1+c_2c_1\varepsilon^2\|B_s\|^2+c_2c_1\varepsilon^4\|C\|^2\right),
\end{align*}
\begin{align*}
\left(\partial_t^2-\partial_s\right)&\left(\mathcal N_3+c_3\mathcal N_2+c_3c_2\mathcal N_1+c_3c_2c_1\varepsilon^2\|B_s\|^2+c_3c_2c_1\varepsilon^4\|C\|^2\right)\\
\geq &\frac {c_7}{\varepsilon^2}\mathcal N_4
-c\left(\mathcal N_4+c_4\mathcal N_3+c_4c_3\mathcal N_2+c_4c_3c_2\mathcal N_1\right)\\
&-c\left(c_4c_3c_2c_1\varepsilon^2\|B_s\|^2+c_4c_3c_2c_1\varepsilon^4\|C\|^2\right)
\end{align*}
\begin{align*}
\left(\partial_t^2-\partial_s\right)&\left(\mathcal N_4+c_4\mathcal N_3+c_4c_3\mathcal N_2+c_4c_3c_2\mathcal N_1\right)\\
&+\left(\partial_t^2-\partial_s\right)\left(c_4c_3c_2c_1\varepsilon^2\|B_s\|^2+c_4c_3c_2c_1\varepsilon^4\|C\|^2\right)\\
\geq &
-c\left(\mathcal N_4+c_4\mathcal N_3+c_4c_3\mathcal N_2+c_4c_3c_2\mathcal N_1\right)\\
&-c\left(c_4c_3c_2c_1\varepsilon^2\|B_s\|^2+c_4c_3c_2c_1\varepsilon^4\|C\|^2\right).
\end{align*}
The seventh step follows then from the lemmas \ref{flow:lemma:SWB1} and \ref{flow:lemma:SWB4} and the last estimates.
\end{proof}

\noindent With the seventh step we concluded the proof of the theorem \ref{flow:thm:apriori}.
\end{proof}

%
%
\section{$L^\infty$-bound for a Yang-Mills flow}\label{flow:section:linfty}

In this section we want to show that for any value $b>0$, any Yang-Mills flow $\Xi^\varepsilon\in \mathcal M^\varepsilon(\Xi^\varepsilon_-,\Xi^\varepsilon_+)$, $\Xi^\varepsilon_\pm\in \mathrm{Crit}_{\mathcal{YM}^{\varepsilon,H}}^b$ satisfies the $L^\infty$-estimate (\ref{flow:est:linfty}). This result allows us to apply the theorem \ref{flow:thm:apriori} needed in the proof of the surjectivity in the section \ref{flow:section:surj}.

\begin{theorem}\label{flow:thm:linf}
We choose a regular value $b>0$. Then there are two positive constants $c$ and $\varepsilon_0$ such that the following holds. For any $\varepsilon$, $0<\varepsilon<\varepsilon_0$,  any $\Xi_-^\varepsilon$, $\Xi_+^\varepsilon \in \mathrm{Crit}_{\mathcal {YM}^{\varepsilon,H}}^b$ and any $\Xi^\varepsilon=A^\varepsilon+\Psi^\varepsilon dt+\Phi^\varepsilon ds \in \mathcal M^\varepsilon(\Xi^\varepsilon_-,\Xi_+^\varepsilon)$
\begin{equation}\label{flow:est:linfty}
\left\|\partial_s A^\varepsilon-d_{A^\varepsilon}\Phi^\varepsilon\right\|_{L^\infty(\Sigma) }+\left\|\partial_t A^\varepsilon-d_{A^\varepsilon}\Psi^\varepsilon\right\|_{L^\infty(\Sigma) }\leq c.
\end{equation} 
\end{theorem}

\begin{proof}
Since the two estimates (\ref{oggi1}) and (\ref{oggi2}) in theorem \ref{flow:thm:apriori} hold also is we assume
\begin{equation}\label{oggi3}
\begin{split}
\sup_{(t,s)\in S^1\times \mathbb R}\Big(&\left\|\partial_s A^\varepsilon-d_{A^\varepsilon}\Phi^\varepsilon\right\|_{L^\infty(\Sigma) }+\left\|\partial_t A^\varepsilon-d_{A^\varepsilon}\Psi^\varepsilon\right\|_{L^4(\Sigma) }\Big)\leq c
\end{split}
\end{equation} 
then by the Sobolev estimate for 1-forms on $\Sigma$, (\ref{oggi3}) implies (\ref{flow:est:linfty}). In order to prove the theorem we assume therefore that there are sequences $\varepsilon_\nu\to 0$ and $\Xi^\nu:=A^\nu+\Psi^\nu dt+\Phi^\nu ds \in \mathcal M^{\varepsilon_\nu}(\Xi^{\varepsilon_\nu}_-,\Xi^{\varepsilon_\nu}_+)$, $\Xi_\pm^{\varepsilon_\nu}\in \mathrm{Crit}_{\mathcal YM^{\varepsilon_\nu,H}}^b$ such that
\begin{equation}
\begin{split}
m_\nu:=&\sup_{(t,s)\in S^1\times \mathbb R}\Big(\left\|\partial_s A^\nu-d_{A^\nu }\Phi^\nu\right\|_{L^\infty(\Sigma) }^{\frac 12 }+\left\|\partial_t A^\nu-d_{A^\nu }\Psi^\nu\right\|_{L^4(\Sigma)}\Big)\to \infty;
\end{split}
\end{equation}
furthermore we assume that there is a sequence $(t_\nu,s_\nu)$ such that
\begin{equation}
\begin{split}
&\left\|\partial_s A^\nu(t_\nu,s_\nu)-d_{A^\nu(t_\nu,s_\nu) }\Phi^\nu(t_\nu,s_\nu)\right\|_{L^\infty(\Sigma) }^{\frac 12 }\\
&\quad+\left\|\partial_t A^\nu(t_\nu,s_\nu)-d_{A^\nu(t_\nu,s_\nu) }\Psi^\nu(t_\nu,s_\nu)\right\|_{L^4(\Sigma)}=m_\nu.
\end{split}
\end{equation}

We will consider the following three cases. We denote by $\|\cdot\|$ the $L^2$-norm on $\Sigma$.\\

\noindent{\bf Case 1: $\lim_{\nu\to\infty}\varepsilon_\nu m_\nu=0$.} We take the sequence of connections $\bar \Xi^\nu=\bar A^\nu+\bar\Psi^\nu dt+\bar\Phi^\nu ds$ defined by
\begin{equation}\label{flow:linfty:w1}
\begin{split}
&\bar A^\nu(t,s):= A^\nu\left(\frac 1{m_\nu }t+t_\nu,\frac 1{m_\nu^2}s+s_\nu\right),\\
&\bar \Psi^\nu(t,s):= \frac 1{m_\nu}\Psi^\nu\left(\frac 1{m_\nu }t+t_\nu,\frac 1{m_\nu^2}s+s_\nu\right),\\
&\bar \Phi^\nu(t,s):=\frac 1{m_\nu^2} \Phi^\nu\left(\frac 1{m_\nu }t+t_\nu,\frac 1{m_\nu^2}s+s_\nu\right);
\end{split}
\end{equation}
then $\bar\Xi^\nu$ satisfies the Yang-Mills flow equations
\begin{equation}\label{flow:linfty:w2}
\begin{split}
&\partial_s \bar A^\nu-d_{ \bar A^\nu} \bar\Phi^\nu+\frac 1{\varepsilon_\nu^2m_\nu^2}d_{\bar A^\nu}^*F_{\bar A^\nu}-\nabla_t^{\bar \Psi^\nu }\left(\partial_t \bar A^\nu-d_{\bar A^\nu}\bar\Psi^\nu\right)-*\frac 1{m_\nu^2}X_{\frac 1{m_\nu}t+t_\nu }(\bar A^\nu)=0,\\
&\partial_s\bar \Psi^\nu-\nabla_t^{\bar \Psi^\nu}\bar\Phi^\nu-\frac 1{\varepsilon_\nu^2m_\nu^2 }d_{\bar A^\nu}^*\left(\partial_t\bar A^\nu-d_{\bar A^\nu }\bar\Psi^\nu\right)=0.
\end{split}
\end{equation}
If we define $\bar B_t^\nu:=\partial_t \bar A^\nu-d_{\bar A^\nu}\bar \Psi^\nu$, $\bar B_s^\nu:=\partial_s \bar A^\nu-d_{\bar A^\nu}\bar \Phi^\nu$ and $\bar C^\nu:=\partial_s\bar\Psi^\nu-\partial_t\bar \Phi^\nu-[\bar \Psi^\nu,\bar \Phi^\nu]$, then we have the following estimates for the norms:
\begin{equation}\label{flow:ap:codsu1}
\begin{split}
\sup_{(t,s)\in S^1\times \mathbb R}\Big(\left\|\bar B_s^\nu\right\|_{L^\infty(\Sigma)}^{\frac 12 }&+\left\|\bar B_t^\nu\right\|_{L^4(\Sigma) }\Big)=1,
\end{split}
\end{equation}
\begin{equation}\label{flow:ap:codsu2}
\begin{split}
\left\|\bar B_s^\nu\right\|_{L^2}^2=&\int_{\mathbb R}\int_{-\frac{m_\nu}2 }^{\frac {m_\nu}2}\left\|\bar B_s^\nu\right\|_{L^2(\Sigma) }^2dt\,ds\\
=&\int_{\mathbb R}\int_{-\frac12 }^{\frac 12}\frac 1{m_\nu^4 }\left\|\partial_s A^\nu-d_{A^\nu} \Phi^\nu\right\|_{L^2(\Sigma) }^2m_\nu^3\, dt\, ds\\
=&\frac 1{m_\nu}\left\|\partial_s A^\nu-d_{ A^\nu} \Phi^\nu\right\|_{L^2}^2\leq \frac c{m_\nu},
\end{split}
\end{equation}
\begin{equation}\label{flow:ap:codsu3}
\begin{split}
\varepsilon^2_\nu m_\nu^2\left\|\bar C^\nu\right\|_{L^2}^2=&\int_{\mathbb R}\int_{-\frac{m_\nu}2 }^{\frac {m_\nu}2}\left\|\bar C^\nu\right\|_{L^2(\Sigma) }^2dt\,ds\\
=&\int_{\mathbb R}\int_{-\frac12 }^{\frac 12}\frac {\varepsilon^2_\nu m_\nu^2}{m_\nu^6 }\left\|\partial_s \Psi^\nu-\partial_t\Phi-[\Psi,\Phi]\right\|_{L^2(\Sigma) }^2m_\nu^3\, dt\, ds\\
=&\frac {\varepsilon_\nu^2}{m_\nu}\left\|\partial_s \Psi^\nu-\partial_t\Phi-[\Psi,\Phi]\right\|_{L^2}^2\leq \frac c{m_\nu}.
\end{split}
\end{equation}

The theorem \ref{flow:thm:apriori} tell us that for every open interval $\Omega\subset \mathbb R$, $0\in\Omega$, and every compact set $Q\in \Omega$ there is a positive constant $c$ such that
\begin{equation*}
\begin{split}
\sup_{(t,s)\in S^1\times Q}&\left(\varepsilon_\nu^2 m_\nu^2\|\bar B_t^\nu \|^2+\|d_{\bar A^\nu}\bar B_t^\nu \|^2+\|d_{\bar A^\nu}^*\bar B_t^\nu \|^2\right)\\
&+\sup_{(t,s)\in S^1\times Q}\left(\varepsilon_\nu^2 m_\nu^2\|\bar B_s^\nu \|^2+\|d_{\bar A^\nu}d_{\bar A^\nu}^*\bar B_s^\nu \|^2+\|d_{\bar A^\nu}d_{\bar A^\nu}^*\bar B_s^\nu \|^2\right)\\
\leq &c \int_{S^1\times \Omega} \left(\|F_{\bar A^\nu}\|^2+\varepsilon_\nu^2 m_\nu^2\|\bar B_t^\nu \|^2+{\varepsilon_\nu^2 m_\nu^2}c_{\frac 1{m_\nu^2}\dot X_{\frac 1{m_\nu}t+t_\nu}(\bar A^\nu)}\right) dt\\
&+c \int_{S^1\times \Omega} \left(\varepsilon_\nu^2m_\nu^2\|\bar B_s^\nu\|^2+\varepsilon_\nu^4m_\nu^4\|C^\nu\|^2\right) dt\\
\leq & c{\varepsilon_\nu^2 m_\nu^2}\left( c_{\frac 1{m_\nu^2}\dot X_{\frac 1{m_\nu}t+t_\nu}(\bar A^\nu)}+\frac 1{m_\nu }\right)
\end{split}
\end{equation*}
where for the last inequality we used that 
$$\int_{S^1\times \Omega}\|F_{\bar A^\nu}\|^2 dt\leq m_\nu c\sup {s\in \Omega}\int_{S^1}\|F_{A^\nu}\|^2 dt\leq c\varepsilon_\nu^2 m_\nu.$$
Since 
\begin{equation}\label{flow:ap:codsu4}
\left |\left |\frac 1{m_\nu^2}\dot X_{\frac 1{m_\nu}t +t_\nu}({\bar A^\nu})\right|\right |_{L^\infty}=\frac 1{m_\nu^3}\|\dot X_t(A^\nu)\|_{L^\infty}\leq \frac c{m_\nu^3}=:\left(c_{\frac 1{m_\nu^2}\dot X_{\frac 1{m_\nu}t+t_\nu}(\bar A^\nu)}\right)^{\frac12},
\end{equation}
it follows that
\begin{equation*}
\begin{split}
\sup_{(t,s)\in S^1\times \mathbb R} \|\bar B_t^\nu \|_{L^4(\Sigma)}\leq& c\sup_{(t,s)\in S^1\times \mathbb R}\left( \|\bar B_t^\nu\|+\|d_{\bar A^\nu}\bar B_t^\nu\|+\|d_{\bar A^\nu}^*\bar B_t^\nu\|\right)\\
\leq& \frac c{\sqrt {m_\nu} }\to 0 \quad (\nu\to \infty),\\
\sup_{(t,s)\in S^1\times \mathbb R}\|\bar B_s^\nu\|_{L^\infty(\Sigma)}\leq& c\sup_{(t,s)\in S^1\times \mathbb R}\big( \|\bar B_s^\nu\|+\|d_{\bar A^\nu}\bar B_s^\nu\|+\|d_{\bar A^\nu}^*\bar B_s^\nu\|\\
&\qquad \qquad\qquad+\|d_{\bar A^\nu}d_{\bar A^\nu}^*\bar B_s^\nu\|+\|d_{\bar A^\nu}^*d_{\bar A^\nu}\bar B_s^\nu\|\big)\\
\leq& \frac c{\sqrt {m_\nu}  }\to 0 \quad (\nu\to \infty),
\end{split}
\end{equation*}
and this is a contradiction.\\

\noindent{\bf Case 2: $\lim_{\nu\to \infty} \varepsilon_\nu m_\nu=c_1 >0$. } We consider the sequence of connections $\bar \Xi^\nu=\bar A^\nu+\bar\Psi^\nu dt+\bar\Phi^\nu ds$ defined by
\begin{equation*}
\begin{split}
&\bar A^\nu(t,s):= A^\nu\left(\varepsilon_\nu t+t_\nu,\varepsilon_\nu ^2s+s_\nu\right),\quad 
\bar \Psi^\nu(t,s):= \varepsilon_\nu \Psi^\nu\left(\varepsilon_\nu t+t_\nu,\varepsilon_\nu ^2s+s_\nu\right),\\
&\bar \Phi^\nu(t,s):=\varepsilon_\nu ^2 \Phi^\nu\left(\varepsilon_\nu t+t_\nu,\varepsilon_\nu ^2s+s_\nu\right);
\end{split}
\end{equation*}
then $\Xi^\nu$ satisfies the Yang-Mills flow equations
\begin{equation*}
\begin{split}
&\partial_s \bar A^\nu-d_{A^\nu} \Phi^\nu+d_{\bar A^\nu}^*F_{\bar A^\nu}-\nabla_t^{\bar \Psi^\nu }\left(\partial_t \bar A^\nu-d_{\bar A^\nu}\bar\Psi^\nu\right)-\varepsilon_\nu^2*X_{\varepsilon_\nu t+t_\nu }(\bar A^\nu)=0,\\
&\partial_s\bar \Psi^\nu-\nabla_t^{\bar \Psi^\nu}\bar\Phi^\nu- d_{\bar A^\nu}^*\left(\partial_t\bar A^\nu-d_{\bar A^\nu }\bar\Psi^\nu\right)=0,
\end{split}
\end{equation*}
In addition, if define $\bar B_t^\nu:=\partial_t \bar A^\nu-d_{\bar A^\nu}\bar \Psi^\nu$, $\bar B_s^\nu:=\partial_s \bar A^\nu-d_{\bar A^\nu}\bar \Phi^\nu$ and $\bar C^\nu:=\partial_s\bar\Psi^\nu-\partial_t\bar \Phi^\nu-[\bar \Psi^\nu,\bar \Phi^\nu]$, for $\nu\to \infty$, we have that
\begin{equation*}
\begin{split}
&\sup_{(t,s)\in S^1\times \mathbb R}\Big(\left\|\bar B_s^\nu\right\|_{L^2(\Sigma) }^{\frac 12 }+\left\|\bar B_t^\nu\right\|_{L^2(\Sigma) }\Big)=c_\nu\to c_1,
\end{split}
\end{equation*}
\begin{equation*}
\begin{split}
\left\|\partial_s \bar A^\nu-d_{\bar A^\nu}\bar \Phi^\nu\right\|_{L^2}^2=&\int_{\mathbb R}\int_{-\frac 1{2\varepsilon_\nu} }^{\frac 1{2\varepsilon_\nu} }\left\|\partial_s \bar A^\nu-d_{\bar A^\nu}\bar \Phi^\nu\right\|_{L^2(\Sigma) }^2dt\,ds\\
=&\int_{\mathbb R}\int_{-\frac12 }^{\frac 12}\varepsilon_\nu^4\left\|\partial_s A^\nu-d_{A^\nu} \Phi^\nu\right\|_{L^2(\Sigma) }^2\frac 1{\varepsilon_\nu^3}\, dt\, ds\\
=&\varepsilon_\nu \left\|\partial_s A^\nu-d_{ A^\nu} \Phi^\nu\right\|_{L^2}^2\leq c\varepsilon_\nu,
\end{split}
\end{equation*}
\begin{equation*}
\begin{split}
\left\|\bar C^\nu\right\|_{L^2}^2=&\int_{\mathbb R}\int_{-\frac 1{2\varepsilon_\nu} }^{\frac 1{2\varepsilon_\nu} }\left\|\bar C^\nu\right\|_{L^2(\Sigma) }^2dt\,ds\\
=&\int_{\mathbb R}\int_{-\frac12 }^{\frac 12}\varepsilon_\nu^6\left\|\partial_s \Psi^\nu-\partial_t \Phi^\nu-[\Psi^\nu,\Phi^\nu]\right\|_{L^2(\Sigma) }^2\frac 1{\varepsilon_\nu^3}\, dt\, ds\\
=&\varepsilon_\nu^3 \left\|\partial_s \Psi^\nu-\partial_t \Phi^\nu-[\Psi^\nu,\Phi^\nu]\right\|_{L^2}^2\leq c\varepsilon_\nu
\end{split}
\end{equation*}
We can compute the estimates (\ref{oggi1}) and (\ref{oggi2}) of the theorem \ref{flow:thm:apriori} also for $\varepsilon=1$ exactly in the same way as we did in the proof of that theorem. Thus, for every open interval $\Omega\subset \mathbb R$, $0\in \Omega$, and every compact set $Q\in \Omega$ there is a positive constat $c$ such that
 
\begin{equation*}
\begin{split}
\sup_{(t,s)\in S^1\times Q}&\left(\|\bar B_t^\nu \|^2+\|d_{\bar A^\nu}\bar B_t^\nu \|^2+\|d_{\bar A^\nu}^*\bar B_t^\nu \|^2\right)\\
&+\sup_{(t,s)\in S^1\times Q}\left(\|\bar B_s^\nu \|^2+\|d_{\bar A^\nu}d_{\bar A^\nu}^*\bar B_s^\nu \|^2+\|d_{\bar A^\nu}^*d_{\bar A^\nu}\bar B_s^\nu \|^2\right)\\
\leq &c \int_{S^1\times \Omega} \left(\|F_{\bar A^\nu}\|^2+\|\bar B_t^\nu \|^2+c_{\varepsilon_\nu^2\dot X_{\varepsilon_\nu t +t_\nu}({\bar A^\nu}) }+\|\bar B_s^\nu\|^2+\|\bar C^\nu\|^2\right) dt\\
\leq &c \varepsilon_\nu+c\frac 1 {\varepsilon_\nu}c_{\varepsilon_\nu^2\dot X_{\varepsilon_\nu t+t_\nu }({\bar A^\nu}) }
\end{split}
\end{equation*}
where for the last estimate we used that
$$\int_{S^1\times \Omega}\|F_{\bar A^\nu}\|^2dt\leq \frac c{\varepsilon_\nu}\sup_{s\in \Omega}\int_{S^1}\|F_{A^\nu}\|^2dt\leq c\varepsilon_\nu.$$

Next, we consider
$$\left |\left |\varepsilon_\nu^2\dot X_{\varepsilon_\nu t+t_\nu }\right|\right |_{L^\infty}=\varepsilon_\nu^3  \|\dot X_t(A)\|_{L^\infty}\leq  c\varepsilon_\nu^3=:\left(c_{\varepsilon_\nu^2\dot X_{\varepsilon_\nu t+t_\nu }({\bar A^\nu}) }\right)^{\frac12},$$
then
\begin{equation*}
\begin{split}
\sup_{(t,s)\in S^1\times \mathbb R}\|\bar B_t^\nu \|_{L^4(\Sigma)}\leq& c\sup_{(t,s)\in S^1\times \mathbb R}\left( \|\bar B_t^\nu\|+\|d_{\bar A^\nu}\bar B_t^\nu\|+\|d_{\bar A^\nu}^*\bar B_t^\nu\|\right)\\
\leq&c\varepsilon_\nu^{\frac12}\to 0 \quad (\nu\to \infty),\\
\sup_{(t,s)\in S^1\times \mathbb R}\|\bar B_s^\nu\|_{L^\infty(\Sigma)}\leq& c\sup_{(t,s)\in S^1\times \mathbb R}\big( \|\bar B_s^\nu\|+\|d_{\bar A^\nu}\bar B_s^\nu\|+\|d_{\bar A^\nu}^*\bar B_s^\nu\|\\
&\qquad \qquad\qquad+\|d_{\bar A^\nu}d_{\bar A^\nu}^*\bar B_s^\nu\|+\|d_{\bar A^\nu}^*d_{\bar A^\nu}\bar B_s^\nu\|\big)\\
\leq& \varepsilon_\nu^{\frac12}\to 0 \quad (\nu\to \infty),
\end{split}
\end{equation*}
and this is a contradiction.\\

\noindent {\bf Case 3: $\lim_{\nu\to \infty} \varepsilon_\nu m_\nu=\infty$. }We consider the substitution (\ref{flow:linfty:w1}) as in the case 1 and hence the new connection satisfies the Yang-Mills equations (\ref{flow:linfty:w2}) and the estimates (\ref{flow:ap:codsu1})-(\ref{flow:ap:codsu1}). In addition, we denote $ B_s^\nu:=\partial_s A^\nu-d_{ A^\nu}\Phi^\nu$, $\bar B_t^\nu:=\partial_t\bar A^\nu-d_{\bar A^\nu}\bar\Psi^\nu$, $\bar B_s^\nu:=\partial_s  \bar A^\nu-d_{\bar A^\nu} \bar\Phi^\nu$. We recall that by the computations in the first and in the second step of the proof of theorem \ref{flow:thm:apriori} we have

\begin{equation*}
\begin{split}
\frac 12 (\partial_t^2-\partial_s)\|\bar B_t^\nu \|^2=&\|\nabla_t^{\nu }\bar B_t^\nu\|^2+\frac 1{\varepsilon_\nu^2m_\nu^2 }\| d_{\bar A^\nu}\bar B_t^\nu\|^2+\frac 1{\varepsilon_\nu^2m_\nu^2 }\| d_{\bar A^\nu}^*\bar B_t^\nu\|^2\\
&+\frac 1{\varepsilon_\nu^2m_\nu^2 }\langle *[\bar B_t^\nu,*F_{\bar A^\nu}] ,\bar B_t^\nu\rangle-\langle \frac 1{m_\nu^2}d*X_{\frac 1{m_\nu}t}(\bar A^\nu)\bar B_t^\nu\\
&+\frac 1{m_\nu^2}\dot X_{\frac 1{m_\nu}t+t_\nu}(\bar A^\nu),\bar B_t^\nu\rangle,\\
\intertext{ }
\frac 12 \left(\partial_t^2-\partial_s\right)\|F_{\bar A^\nu}\|^2=&\|\nabla_t^\nu F_{\bar A^\nu}\|^2+\frac 1{\varepsilon_\nu^2 m_\nu^2}\|d_{\bar A^\nu}^*F_{\bar A^\nu}\|^2\\
&+\langle F_{\bar A^\nu},[\bar B_t^\nu\wedge \bar B_t^\nu]\rangle-\langle d_{\bar A^\nu }^*F_{\bar A^\nu},*\frac 1{m_\nu^2}X_{\frac 1{m_\nu}t+t_\nu}(\bar A^\nu)\rangle,\\
\intertext{ }
\frac 12\big(\partial_t^2-\partial_s\big)\|\bar B_s^\nu \|^2=&\|\nabla_t^\nu \bar B_s\|^2+\frac 1{\varepsilon_\nu^2m_\nu^2} \|d_{\bar A^\nu}\bar B_s^\nu\|+\frac 1{\varepsilon_\nu^2m_\nu^2}\|d_{\bar A^\nu}^*\bar B_s^\nu\|^2\\
&+\frac 1{\varepsilon_\nu^2m_\nu^2}\langle \bar B_s^\nu,[d_{\bar A^\nu}^*\bar B_t,\bar B_t^\nu]\rangle-\frac 2{\varepsilon_\nu^2 m_\nu^2}\langle \bar B_s^\nu,*[\bar B_s^\nu,*F_{\bar A^\nu}]\rangle\\
&-2\langle \bar B_s^\nu,\nabla_s^\nu*\frac 1{m_\nu^2}X_{\frac 1{m_\nu}t+t_\nu}(\bar A^\nu)\rangle,
\end{split}
\end{equation*}
thus for a constant $c_0>0$
\begin{equation*}
\begin{split}
 \left(\partial_t^2-\partial_s\right)&\left( \frac {c_0}{\varepsilon_\nu^2 m_\nu^2}\|F_{\bar A^\nu}\|^2+\|\bar B_t^\nu \|^2+\|\bar B_s^\nu\|^2\right)\\
\geq &- c\frac 1{\varepsilon_\nu^2 m_\nu^2} \|F_{\bar A^\nu}\|^2-c\|\bar B_t^\nu \|^2 -c\|\bar B_s^\nu \|^2- \frac c{m_\nu^2}
\end{split}
\end{equation*}
and hence by the lemma \ref{flow:lemma:SWB1} there is, for any open set $\Omega\subset \mathbb R$, $0\in\Omega$, and any compact interval $Q\subset \Omega$, a positive constant $c$ such that
\begin{equation*}
\begin{split}
 \sup_{(t,s)\in S^1\times Q}&\left(\|\bar B_s^\nu\|^2+\|\bar B_t\|^2\right)\\
\leq& c\int _{S^1\times \Omega} \left(\|\bar B_s^\nu\|^2+\|\bar B_t^\nu\|^2+\frac {1}{\varepsilon_\nu^3 m_\nu^3}\|F_{\bar A^\nu}\|^2+\frac 1{m_\nu^2}\right) dt\,ds\\
 \leq &\frac c{m_\nu}+c \sup_{s\in \Omega} \int_{S^1}\|\bar B_t^\nu\|^2 dt\leq \frac c{m_\nu}.
 \end{split}
\end{equation*}
Analogously, using the computations of the proof of theorem \ref{flow:thm:apriori} we can obtain that there is a constant $c_0$ such that 
\begin{equation}
 \begin{split}
 \left(\partial_t^2-\partial_s\right)&\left( \|d_{\bar A^\nu}\bar B_t^\nu \|^2+\|d_{\bar A^\nu}^*d_{\bar A^\nu}\bar B_s^\nu\|^2+\|d_{\bar A^\nu}^*\bar B_t^\nu \|^2+\|d_{\bar A^\nu}d_{\bar A^\nu}^*\bar B_s^\nu\|^2\right)\\
&+\left(\partial_t^2-\partial_s\right)\left( \frac {c_0}{\varepsilon_\nu^2 m_\nu^2}\|F_{\bar A^\nu}\|^2+\|\bar B_t^\nu \|^2+\|\bar B_s^\nu\|^2\right)\\
\geq &- \frac c{\varepsilon_\nu^2 m_\nu^2}\|F_{\bar A^\nu}\|^2-c\|d_{\bar A^\nu}\bar B_t^\nu \|^2-c\|d_{\bar A^\nu}^*d_{\bar A^\nu}\bar B_s^\nu\|^2-c\|\bar B_t^\nu\|^2\\
&-c \|d_{\bar A^\nu}^*\bar B_t^\nu \|^2-c\|d_{\bar A^\nu}d_{\bar A^\nu}^*\bar B_s^\nu\|^2-c\|\bar B_s^\nu\|^2
 \end{split}
\end{equation}
thus, by the Sobolev estimates for 1-forms on $\Sigma$ and the lemma \ref{flow:lemma:SWB1} we get, for $\Omega_1=\Sigma\times S^1\times \Omega$,
\begin{equation*}
 \begin{split}
 \sup_{(t,s)\in S^1\times Q}&\left( \|\bar B_t^\nu\|^2_{L^4(\Sigma)}+\|\bar B_s^\nu\|^2_{L^\infty(\Sigma)}\right)\\
\leq  &\sup_{(t,s)\in S^1\times Q}\left( \|\bar B_s^\nu\|^2+\|d_{\bar A^\nu}^*d_{\bar A^\nu}\bar B_s^\nu\|^2+\|d_{\bar A^\nu}d_{\bar A^\nu}^*\bar B_s^\nu\|^2\right)\\
    &+\sup_{(t,s)\in S^1\times Q}\left( \|\bar B_t^\nu \|^2+ \|d_{\bar A^\nu}\bar B_t^\nu \|^2+\|d_{\bar A^\nu}^*\bar B_t^\nu \|^2\right)\\
\leq &c\int_{S^1\times \Omega}  \left( \|d_{\bar A^\nu}^*d_{\bar A^\nu}\bar B_s^\nu\|^2+\|d_{\bar A^\nu}d_{\bar A^\nu}^*\bar B_s^\nu\|^2\right) dt\,ds\\ 
&+c\int_{S^1\times \Omega}  \left( \|d_{\bar A^\nu}\bar B_t^\nu \|^2+\|d_{\bar A^\nu}^*\bar B_t^\nu \|^2\right) dt\,ds+\frac c{ m_\nu}\\
\leq &c\frac 1{m_\nu}\left( \|d_{ A^\nu}^*d_{A^\nu} B_s^\nu\|^2_{L^2(\Omega_1)}+\|d_{ A^\nu}d_{A^\nu}^* B_s^\nu\|^2_{L^2(\Omega_1)}\right) \\ 
&+c\varepsilon_\nu\|d_{ A^\nu} B_s^\nu \|^2_{L^2(\Omega_1)}+\frac {c\varepsilon_\nu^2}{m_\nu}\|\nabla_t B_s^\nu \|^2_{L^2(\Omega_1)}+\frac c{ m_\nu}+c\varepsilon_\nu\\
\leq &\frac c{ m_\nu}+ c\varepsilon_\nu \to 0
\end{split}
\end{equation*}
where the last estimate follows from the lemma \ref{flow:smooth} and the third by 
\begin{equation}\label{lkllla}
\begin{split}
\|d_{\bar A^\nu} \bar B_t^\nu\|^2_{L^2(\Omega_1)}&+\|d_{\bar A^\nu}^*\bar B_t^\nu\|^2_{L^2(\Omega_1)}\\
\leq& \frac {c}{m_\nu}+c\varepsilon_\nu+c\varepsilon_\nu\|d_{A^\nu} B_s^\nu\|_{L^2(\Omega_1)}^2+c\frac {\varepsilon_\nu^2}{m_\nu}\|\nabla_t B_s^\nu\|_{L^2(\Omega_1)}^2
\end{split}
\end{equation}
and therefore we have a contradiction. In order to finish the proof of the theorem it remains only to prove (\ref{lkllla}). By the identities
\begin{equation*}
\begin{split}
&\varepsilon^2\left\|\bar B_s^\nu+\frac 1{\varepsilon^2}d_{\bar A^\nu}^*F_{\bar A^\nu}-\nabla_t \bar B_t^\nu-*X_{\frac 1{m_\nu}t+t_\nu}(\bar A^\nu)\right\|^2_{L^2(\Sigma\times S^1)}\\
&+\|\nabla_t F_{\bar A^\nu}-d_{\bar A^\nu}\bar B_t^\nu\|^2_{L^2(\Sigma\times S^1)}\\
=&\|\bar B_s^\nu\|^2+
\frac 1{\varepsilon^2}\left\|d_{\bar A^\nu}^*F_{\bar A^\nu}\right\|^2_{L^2(\Sigma\times S^1)}+\varepsilon^2\left\|\nabla_t \bar B_t^\nu\right\|^2_{L^2(\Sigma\times S^1)}\\
&+\varepsilon^2
\left\|X_{\frac 1{m_\nu}t+t_\nu}({\bar A^\nu})\right\|^2_{L^2(\Sigma\times S^1)}+\left\|\nabla_t F_{\bar A^\nu}\right\|^2_{L^2(\Sigma\times S^1)}+\left\|d_{\bar A^\nu}\bar B_t^\nu\right\|^2_{L^2(\Sigma\times S^1)}\\
&-2\varepsilon^2\left\langle *X_{\frac 1{m_\nu}t+t_\nu}(\bar A^\nu),\frac 1{\varepsilon^2}d_{\bar A^\nu}^*F_{\bar A^\nu}-\nabla_t \bar B_t^\nu\right\rangle\\
&+2\left\langle \bar B_s^\nu,d_{\bar A^\nu}^*F_{\bar A^\nu}-\varepsilon^2\nabla_t\bar B_t^\nu-\varepsilon^2*X_{\frac 1{m_\nu}t+t_\nu}({\bar A^\nu})\right\rangle\\
& -2 \left\langle d_{\bar A^\nu}^*F_{\bar A^\nu},\varepsilon^2\nabla_t \bar B_t^\nu\right\rangle - \left\langle\nabla_t F_{\bar A^\nu},d_{\bar A^\nu}\bar B_t^\nu\right\rangle,
\end{split}
\end{equation*}
\begin{equation*}
-2\left\langle d_{\bar A^\nu}^*F_{\bar A^\nu},\nabla_t \bar B_t^\nu\right\rangle -2\left\langle\nabla_t F_{\bar A^\nu},d_{\bar A^\nu}\bar B_t^\nu\right\rangle
=2\left\langle F_{\bar A^\nu},[\bar B_t^\nu\wedge \bar B_t^\nu]\right\rangle,
\end{equation*}
\begin{equation*}
2\left\langle \bar B_s^\nu,d_{\bar A^\nu}^*F_{\bar A^\nu}-\varepsilon^2\nabla_t\bar B_t^\nu\right\rangle=
2\left\langle d_{\bar A^\nu}\bar B_s^\nu, F_{\bar A^\nu}\right\rangle-\varepsilon^2\left\langle\nabla_t\bar B_s^\nu,\bar B_t^\nu\right\rangle,
\end{equation*}
by the Bianchi identity $\nabla_t F_{\bar A^\nu}-d_A\bar B_t=0$ and by the perturbed Yang-Mills equation (\ref{intro:YMeq1}) we have:
\begin{align*}
&\|d_{\bar A^\nu} \bar B_t^\nu\|^2_{L^2(\Omega_1)}+\|d_{\bar A^\nu}^*\bar B_t^\nu\|^2_{L^2(\Omega_1)}+\varepsilon^2_\nu m_\nu^2\|\nabla_t\bar B_t^\nu\|^2_{L^2(\Omega_1)}\\
\leq& \frac c{m_\nu^2}\|*X_{\frac 1{m_\nu}t+t_\nu}(\bar A^\nu)\|_{L^\infty}\|F_{\bar A^\nu}\|^2_{L^2(\Sigma\times S^1)}+\varepsilon^2_\nu |\langle *\nabla_t X_{\frac 1{m_\nu}t+t_\nu}({\bar A^\nu}),\bar B_t^\nu\rangle| \\
&+c\int_{\Omega} \left( {\varepsilon_\nu m_\nu}\|d_{\bar A^\nu}\bar B_s^\nu\|_{\Sigma\times S^1}^2+\frac 1{\varepsilon_\nu m_\nu}\|F_{\bar A^\nu}\|_{\Sigma\times S^1}^2\right) ds\\
&+c\int_{\Omega} \left(\varepsilon^2_\nu m_\nu^2\|\nabla_t\bar B_s^\nu\|_{\Sigma\times S^1}^2+\sup_{s\in \Omega}\|\bar B_t^\nu\|_{\Sigma\times S^1}^2\right) ds\\
&+c\int_{\Omega} \left(\varepsilon^2_\nu\|\bar B_s^\nu\|_{\Sigma\times S^1}^2+c\varepsilon_\nu^2\right) ds\\
&+c\int_{\Omega}\|F_{\bar A^\nu}\|_{L^2(\Sigma\times S^1)}(\varepsilon_\nu m_\nu)^{-\frac 12}\left( \|\bar B_t^\nu\|^2_{L^2(\Sigma\times S^1)}+\|d_{\bar A^\nu}\bar B_t^\nu\|^2_{L^2(\Sigma\times S^1)}\right)\\
&+c\int_{\Omega}\|F_{\bar A^\nu}\|_{L^2(\Sigma\times S^1)}(\varepsilon_\nu m_\nu)^{-\frac 12}\left( \|d_{\bar A^\nu}^*\bar B_t^\nu\|^2_{L^2(\Sigma\times S^1)}+\varepsilon^2_\nu m_\nu^2\|\nabla_t\bar B_t^\nu\|^2_{L^2(\Sigma\times S^1)}\right)\\
\leq& c\frac {\varepsilon_\nu}{m_\nu^2}+\varepsilon^2_\nu  +c\varepsilon_\nu\|d_{A^\nu} B_s^\nu\|_{L^2(\Omega_1)}^2+c\varepsilon_\nu\\
&+c\frac {\varepsilon_\nu^2}{m_\nu}\|\nabla_t B_s^\nu\|_{L^2(\Omega_1)}^2+\frac c{m_\nu}+c\varepsilon_\nu^2\\
&+c\varepsilon_\nu ^{\frac 12}\sup_{s\in \Omega} \|\bar B_t^\nu\|^2_{L^2(\Sigma\times S^1)}\\
&+c\varepsilon_\nu ^{\frac 12}\left(\|d_{\bar A^\nu}\bar B_t^\nu\|^2_{L^2(\Omega_1)}+ \|d_{\bar A^\nu}^*\bar B_t^\nu\|^2_{L^2(\Omega_1)}+\varepsilon^2_\nu m_\nu^2\|\nabla_t\bar B_t^\nu\|^2_{L^2(\Omega_1)}\right)
\end{align*}
where we use the H\"older inequality and the Sobolev estimate in the first estimate. Thus choosing $\varepsilon_\nu$ small enough

\begin{align*}
\|d_{\bar A^\nu} \bar B_t^\nu\|^2_{L^2(\Omega_1)}&+\|d_{\bar A^\nu}^*\bar B_t^\nu\|^2_{L^2(\Omega_1)}\\
\leq& \frac {c}{m_\nu}+c\varepsilon_\nu+c\varepsilon_\nu\|d_{A^\nu} B_s^\nu\|_{L^2(\Omega_1)}^2+c\frac {\varepsilon_\nu^2}{m_\nu}\|\nabla_t B_s^\nu\|_{L^2(\Omega_1)}^2.
\end{align*}
With this last estimate we conclude the discussion of the third case and thus, also the proof of the theorem \ref{flow:thm:linf}.
\end{proof}

%
%
\section{Exponential convergence}\label{flow:section:expconv}

In this section we will prove the exponential convergence of the curvature terms $B_s$ and $C$  stated in the next theorem.  In order to simplify the exposition, for this section, we denote by $\|\cdot\|$ the norm $\|\cdot\|_{L^2(\Sigma\times S^1)}$.

\begin{theorem}\label{flow:thm:expconv}
We assume that every perturbed closed geodesic on $\mathcal M^g(P)$ is non-degenerate. Then for every $c_0,b >0$ there are four positive constants $\delta, \varepsilon_0, c, \rho$ such that the following holds. If $\Xi\in \mathcal M^{\varepsilon }(\Xi^\varepsilon_-,\Xi^\varepsilon_+)$, $\Xi^\varepsilon_-,\Xi^\varepsilon_+\in \mathrm{Crit}^b_{\mathcal{YM}^{\varepsilon,H} }$ with $0<\varepsilon \leq \varepsilon_0$, satisfies, for $B_t:=\partial_tA-d_A\Psi$, $B_s:=\partial_sA-d_A\Phi$, $C:=\partial_s\Psi-\partial_t\Phi-[\Psi,\Phi]$,
\begin{equation}
\| B_s\|_{L^\infty(\Sigma\times S^1\times \mathbb R) }+\|B_t\|_{L^\infty(\Sigma\times S^1\times \mathbb R) }+\varepsilon \| C\|_{L^\infty(\Sigma\times S^1\times \mathbb R) }\leq c_0
\end{equation}
and
\begin{equation}
\|B_s\|^2_{L^2(\Sigma\times S^1\times [0,\infty))}+\varepsilon^2\|C\|^2_{L^2(\Sigma\times S^1\times [0,\infty))}\leq \delta,
\end{equation}
then, for $S\geq1$,
\begin{equation}
\|B_s\|_{L^2(\Sigma\times S^1\times [S,\infty))}^2+\varepsilon^2\|C\|_{L^2(\Sigma\times S^1\times [S,\infty))}^2\leq c e^{-\rho S }.
\end{equation}

\end{theorem}

\begin{lemma}\label{flow:lemma:expconv}
We choose a positive constant $c_0$ and we assume that every perturbed closed geodesic on $\mathcal M^g(P)$ is non-degenerate. Then there are positive constants $\delta_0, \varepsilon_0$ and $c$ such that the following holds. If $A+\Psi dt$ is a connection on $P\times S^1$ which satisfies, for $0<\varepsilon\leq \varepsilon_0$, $B_t:=\partial_tA-d_A\Psi$,
\begin{equation}
\left \| \frac 1{\varepsilon^2}d_A^*F_A-\nabla_t B_t-*X_t(A)\right\|_{L^\infty}+\frac 1{\varepsilon}\left \| |d_A^*B_t\right\|_{L^\infty}+\sup_{t\in S^1} \|F_A\|_{L^2(\Sigma) }\leq \delta_0,
\end{equation}
\begin{equation}
\left \|B_t\right\|_{L^\infty}\leq c_0,
\end{equation}
where $B_t:=\partial_tA-d_A\Psi$, then
\begin{equation}\label{flow:lemma:expconv:eq33}
\begin{split}
\|\alpha+\psi dt\|_{2,2,\varepsilon } \leq &c\left( \varepsilon\left\| \mathcal D^\varepsilon(A,\Psi)(\alpha,\psi)\right\|_{0,2,\varepsilon}+ \left\| \pi_{A^0}\mathcal D^\varepsilon(A,\Psi)(\alpha,\psi)\right\|_{L^2}\right),
\end{split}
\end{equation}
\begin{equation}\label{flow:lemma:expconv:eq33x}
\begin{split}
\|(1-\pi_{A^0})\alpha+\psi dt\|_{2,2,\varepsilon } \leq &c\varepsilon^2\left\| \mathcal D^\varepsilon(A,\Psi)(\alpha,\psi)\right\|_{0,2,\varepsilon}\\
&+ c\varepsilon\left\| \pi_{A^0}\mathcal D^\varepsilon(A,\Psi)(\alpha,\psi)\right\|_{L^2}
\end{split}
\end{equation}
for every 1-form $\alpha+\psi dt$ on $P\times S^1$.
\end{lemma}

\begin{proof}[Proof of lemma \ref{flow:lemma:expconv}.]
We suppose that the lemma does not hold. Then there are sequences $A^\nu+\Psi^\nu dt$, $\varepsilon_\nu\to 0$, such that, for $B_t^\nu:=\partial_tA^\nu-d_{A^\nu}\Psi^\nu$, and
\begin{equation*}
\delta_\nu:=\left\|\frac 1{\varepsilon^2}d_{A^\nu}^*F_{A^\nu}-\nabla_t^{\Psi^\nu} B_t^\nu-*X_t(A^\nu) \right\|_{L^\infty}+\left\|\frac 1{\varepsilon_\nu}d_{A^\nu}^*B_t^\nu\right\|_{L^\infty}+\sup_{t\in S^1} \|F_{A^\nu}\|_{L^2(\Sigma)},
\end{equation*}
with $\delta_\nu\to 0$ and (\ref{flow:lemma:expconv:eq33}) or (\ref{flow:lemma:expconv:eq33x}) are not satisfied for $A+\Psi dt=A^\nu+\Psi^\nu dt$ and $\varepsilon=\varepsilon_\nu$.\\

{\bf Claim:} The following two estimates hold:
\begin{equation*}
\begin{split}
\|F_A\|_{1,2,\varepsilon}\leq &
\left\|d_A^*F_A-\varepsilon^2\nabla_t B_t-\varepsilon^2*X_t(A)\right\|+c\varepsilon^2,\\
\|B_t \|_{1,2,1}\leq &\left\|d_A^*B_t\right\|+
\left\|\frac 1{\varepsilon^2}d_A^*F_A-\nabla_t B_t-*X_t(A)\right\|
+\|B_t\|+c.
\end{split}
\end{equation*}

\begin{proof} By the identity
\begin{equation*}
\begin{split}
&\left\|d_A^*F_A-\varepsilon^2\nabla_t B_t-\varepsilon^2*X_t(A)\right\|^2
+\varepsilon^2\|\nabla_t F_A-d_AB_t\|^2\\
=&\left\|d_A^*F_A\right\|^2+\varepsilon^4\left\|\nabla_t B_t\right\|^2+
\varepsilon^4\left\|X_t(A)\right\|^2+\varepsilon^2 \left\|\nabla_t F_A\right\|^2\\
&+\varepsilon^2\left\|d_AB_t\right\|^2-2\left\langle \varepsilon^2*X_t(A),d_A^*F_A-\varepsilon^2\nabla_t B_t\right\rangle\\
&-2\varepsilon^2 \left\langle d_A^*F_A,\nabla_t B_t\right\rangle -2\varepsilon^2 \left\langle\nabla_t F_A,d_AB_t\right\rangle,
\end{split}
\end{equation*}
\begin{equation*}
\begin{split}
\left\|d_A^*B_t\right\|^2&+
\left\|\frac 1{\varepsilon^2}d_A^*F_A-\nabla_t B_t-*X_t(A)\right\|^2
+\frac 1{\varepsilon^2}\|\nabla_t F_A-d_AB_t\|^2\\
=&\left\|d_A^*B_t\right\|^2+
\frac 1{\varepsilon^4}\left\|d_A^*F_A\right\|^2+\left\|\nabla_t B_t\right\|^2+
\left\|X_t(A)\right\|^2\\
&+\frac 1{\varepsilon^2} \left\|\nabla_t F_A\right\|^2+\frac 1{\varepsilon^2}\left\|d_AB_t\right\|^2-2\left\langle *X_t(A),\frac 1{\varepsilon^2}d_A^*F_A-\nabla_t B_t\right\rangle\\
&-2\frac 1{\varepsilon^2} \left\langle d_A^*F_A,\nabla_t B_t\right\rangle -2\frac 1{\varepsilon^2} \left\langle\nabla_t F_A,d_AB_t\right\rangle,
\end{split}
\end{equation*}
\begin{equation*}
-2\left\langle d_A^*F_A,\nabla_t B_t\right\rangle -2\left\langle\nabla_t F_A,d_AB_t\right\rangle
=2\left\langle F_A,[B_t\wedge B_t]\right\rangle
\end{equation*}
and since $\nabla_t F_A-d_AB_t=0$ by the Bianchi identity, the lemma holds.
\end{proof}

Thus, the $L^p$-norm of the curvatures $F_{A^\nu}+B_t^\nu dt$ is uniformly bounded for any $p$ which satisfies the Sobolev's condition $-\frac 3p< 1-\frac 32$ and hence for $p<6$. Therefore, by the weak Uhlenbeck compactness theorem (see \cite{MR648356} or theorem 7.1. in \cite{MR2030823}), we can assume that $A^\nu+\Psi^\nu dt$ converges weakly to a $A^0+\Psi^0dt$ in $W^{1,p}$ and hence also strongly in $L^\infty$. In addition we have that, for $B_t^0:=\partial_tA^0-d_{A^0}\Psi^0$,
\begin{equation*}
F_{A^0}=0,\quad d_{A^0}^*B_t^0=d_{A^0}B_t^0=0,\quad \nabla_t^0B_t^0+*X_t(A^0)\in \textrm{im } d_{A^0}^*;
\end{equation*}
thus, $A^0+\Psi^0dt$ satisfies the equations of a perturbed closed geodesic and therefore for any $1$-form $\alpha$ on $P\times S^1$ with $\alpha(t)\in \Omega^1(\Sigma,\mathfrak g_P)$
\begin{equation*}
\|\pi_{A^0}(\alpha)\|_{1,2,1}\leq c\|\mathcal D^0(A^0)(\pi_{A^0}(\alpha),\psi_0)\|_{L^2}
\end{equation*}
where $\psi$ satisfies $-2*\left[\pi_{A^0}(\alpha) \wedge *B_t^0\right]-d_{A^0}^*d_{A^0}\psi_0=0$. Then 
\begin{equation*}
\begin{split}
\|\pi_{A^0}(\alpha)\|_{1,2,1}\leq c\Big|\Big|\pi_{A^0}\Big(&2[\psi_0,B_t^\nu]+d*X_t(A^\nu)\pi_{A^0}(\alpha)\\
&+\nabla_t^{\Psi^\nu}\nabla_t^{\Psi^\nu}\pi_{A^0}(\alpha)+\frac 1{\varepsilon_\nu^2}*[\pi_{A^0}(\alpha)\wedge *F_{A^\nu }]\Big)\Big|\Big|_{L^2}
\end{split}
\end{equation*}
where we used 
\begin{equation*}
\begin{split}
&\left\|\left[\pi_{A_0}(\alpha)\wedge*\left(d_{A^0}(d_{A^0}^*d_{A^0})^{-1}(\nabla_t^0B_t^0-*X_t(A^0))-\frac 1{\varepsilon^2_\nu}F_{A^\nu}\right)\right]\right\|\\
 & \qquad\qquad\leq\|\pi_{A_0}(\alpha)\|_{L^\infty}
 \left\|\frac 1{\varepsilon_\nu^2}d_{A^\nu}^*F_{A^\nu}-\nabla_t^0 B_t^0-*X_t(A^0)\right\|\\
& \qquad\qquad\quad+\|\pi_{A_0}(\alpha)\|_{L^\infty}
  \left\|\nabla_t^0 B_t^0+*X_t(A^0)\right\|_{L^\infty}\left\|A^\nu-A^0\right\| \\ 
 & \qquad\qquad\leq\|\pi_{A_0}(\alpha)\|_{L^\infty}\left\|\frac 1{\varepsilon_\nu^2}d_{A^\nu}^*F_{A^\nu}-\nabla_t^\nu B_t^\nu-*X_t(A^\nu)\right\|\\
 & \qquad\qquad\quad+\|\pi_{A_0}(\alpha)\|_{L^\infty}
 \left\|\nabla_t^0B_t^0+*X_t(A^0)-\nabla_t^\nu B_t^\nu-*X_t(A^\nu)\right\|\\
  & \qquad\qquad\leq \frac 12\|\pi_{A_0}(\alpha)\|_{1,2,1}
\end{split}
\end{equation*}
for $\nu$ big enough. Therefore, analogously to the theorems 7.1 and 7.2 of \cite{remyj6}, one can prove that

\begin{equation*}
\begin{split}
\|\alpha\|_{2,2,\varepsilon_\nu }\leq c\Big(&\varepsilon_\nu \left\|\mathcal D^{\varepsilon_\nu}(A^\nu,\Psi^\nu)(\alpha,\psi)\right\|_{0,2,\varepsilon_\nu }
+\left\|\pi_{A^0}\mathcal D_1^{\varepsilon_\nu} (A^\nu,\Psi^\nu)(\alpha,\psi)\right\|_{L^2}\Big),
\end{split}
\end{equation*}

\begin{equation*}
\begin{split}
\|\alpha+\psi ds-\pi_{A^0}(\alpha)\|_{2,2,\varepsilon_\nu }\leq c\Big(&\varepsilon_\nu^2 \left\|\mathcal D^{\varepsilon_\nu}(A^\nu,\Psi^\nu)(\alpha,\psi)\right\|_{0,2,\varepsilon_\nu }\\
&+\varepsilon_\nu \left\|\pi_{A^0}\mathcal D_1^{\varepsilon_\nu} (A^\nu,\Psi^\nu)(\alpha,\psi)\right\|_{L^2}\Big).
\end{split}
\end{equation*}
Thus, we have a contradiction and hence we finished the proof of the lemma \ref{flow:lemma:expconv}.
\end{proof}

\begin{proof}[Proof of theorem \ref{flow:thm:expconv}.]

To prove this theorem we proceed as Dostoglou and Salamon did for the theorem 7.4 in \cite{MR1283871}. The idea is to find a positive bounded function $f(s)$ such that it satisfies 
\begin{equation}\label{flow:proof:eqfdsjapo}
f''(s)\geq \rho^2 f(s)
\end{equation}
 for $s\geq 1$. Then, this implies that $f$ has an exponential decay, because, since
\begin{equation*}
 \partial_s\left( e^{-\rho s}\left(f'(s)+\rho f(s)\right)\right)= e^{-\rho s}\left(-\rho^2 f(s)+ f''(s)\right) \geq 0,
\end{equation*} 
$f'(s)+\rho f(s)<0$ (otherwise $e^{-\rho s}\left(f'(s)+\rho f(s)\right)$ would be positive and increase; thus, since $f(s)$ is bounded, $e^{-\rho s }f(s)$ would decrease and hence $f'(s)$ would increase. Therefore $f(s)$ would be unbounded which is a contradiction.) and hence $e^{\rho s}f(s)$ is decreasing. Therefore, if a function $f$ satisfies (\ref{flow:proof:eqfdsjapo}), then
\begin{equation}\label{flow:proof:eqfdsjapo2}
f(s)\leq e^{-\rho (s-1) }c_1.
\end{equation}
with $c_1=f(1)$. By the a priori estimate (\ref{flow:eq:apriorig20}) and the lemma \ref{lemma76dt94} for any $\delta$  
\begin{equation}
\|B_s\|_{L^\infty(\Sigma\times S^1\times \{s\})}+\varepsilon \|C\|_{L^\infty(\Sigma\times S^1\times \{s\})}\leq \delta
\end{equation}
holds for $s$ sufficiently big. We define
\begin{equation}
f(s):=\frac 12 \int_0^1\left(\|B_s(t,s)\|_{L^2(\Sigma)}^2+ \varepsilon^2\|C(t,s)\|_{L^2(\Sigma)}^2\right) dt;
\end{equation}
then, as we will show later,
\begin{equation}\label{flow:thm:expconv:eq1}
\begin{split}
f''(s)=&\frac12 \partial_s^2\big(\|B_s\|^2_{L^2(\Sigma\times S^1)}+\varepsilon^2\|C\|^2_{L^2(\Sigma\times S^1)}\big)\\
\geq& \Big\|\frac 1{\varepsilon^2}\left(d_A^*d_AB_s+d_Ad_A^*B_s*+*[B_s,*F_A]\right)\\
&\qquad\qquad\qquad-\nabla_t\nabla_t B_s-d*X_t(A)B_s-2[C,B_t]\Big\|^2\\
&+ \frac 1{\varepsilon^2}\left \| d_A^*d_AC-\varepsilon^2\nabla_t\nabla_t C+*[B_s\wedge *B_t]\right \|^2.
\end{split}
\end{equation}
Next, for $s\geq 1$ and $\delta$ sufficiently small we apply the lemma \ref{flow:thm:expconv} for $\alpha+\psi dt:= B_s+C dt$ 
and thus,
\begin{equation*}
\begin{split}
f(s)=&\frac 12\left(\|B_s\|^2_{2,2,\varepsilon}+\varepsilon^2\|C\|^2_{2,2,\varepsilon}\right)\\
\leq& c\Big\|\frac 1{\varepsilon^2}\left(d_A^*d_AB_s+d_Ad_A^*B_s*+*[B_s,*F_A]\right)\\
&\qquad\qquad\qquad-\nabla_t\nabla_t B_s-d*X_t(A)B_s-2[C,B_t]\Big\|^2\\
&+c\varepsilon^6\left \| \frac 1{\varepsilon^2} d_A^*d_AC-\nabla_t\nabla_t C+\frac 2{\varepsilon^2}*[B_s\wedge *B_t]\right \|^2\\
\leq& c f''(s).
\end{split}
\end{equation*}
Therefore, $\rho^2 f(s)\leq  f''(s)$ for $\rho>0$ small enough. Thus, by (\ref{flow:proof:eqfdsjapo2}),
\begin{equation*}
\int_S^\infty\left(\|B_s\|_{L^2(\Sigma\times S^1) }^2+\varepsilon^2\|C\|_{L^2(\Sigma\times S^1)}^2\right) ds\leq c e^{-\rho S }
\end{equation*}
for $S\geq1$.

\begin{proof}[Proof of (\ref{flow:thm:expconv:eq1})] First, we consider the following two short computations that are consequence of the commutation formula (\ref{flow:eq:apriori:eq4}) and of the Yang-Mills flow equation (\ref{flow:eq:apriori:eq2}):  
\begin{equation}\label{flow:eq:expconvv1}
\begin{split}
\nabla_t d_AC=&d_A\nabla_t C+[B_t,C]=\frac 1{\varepsilon^2} d_A\nabla_t d_A^*B_t+[B_t,C]\\
=&\frac 1{\varepsilon^2} d_Ad_A^*\nabla_t B_t+[B_t,C]=\frac 1{\varepsilon^2} d_Ad_A^*B_s+[B_t,C],
\end{split}
\end{equation}
\begin{equation}\label{flow:eq:expconvv2}
\begin{split}
d_A^*\nabla_t B_s=& \nabla_t d_A^*B_s+*[B_t\wedge*B_s]
=\nabla_t d_A^*\nabla_t B_t+*[B_t\wedge *B_s]\\
=&\nabla_t\nabla_t d_A^*B_t+*[B_t\wedge*B_s]
=\varepsilon^2\nabla_t\nabla_t C+*[B_t\wedge*B_s],
\end{split}
\end{equation}
for in both cases we use $[B_t\wedge *B_t]=0$ and $[F_A,*F_A]=0$. In the following we use the notation
\begin{align*}
D_1:=&\frac 1{\varepsilon^2}d_A^*d_AB_s-\nabla_t\nabla_t B_s+\nabla_t d_AC-d*X_t(A)B_s-
\frac 1{\varepsilon^2}*[B_s,*F_A]+[B_t,C],\\
D_2:=&\varepsilon^2\nabla_t\nabla_tC-d_A^*d_AC-*2[B_s\wedge *B_t].
\end{align*}
Next, we can compute the second derivative of $\|B_s\|^2+\|C\|^2$, i.e.
\begin{align*}
\frac12 \partial_s^2\big(\|B_s\|^2&+\varepsilon^2\|C\|^2\big)\\
=&\|\nabla_s B_s\|^2+\varepsilon^2\|\nabla_s C\|^2+\langle \nabla_s\nabla_s B_s,B_s\rangle+\varepsilon^2\langle \nabla_s\nabla_s C,C\rangle\\
=&\left\|\frac 1{\varepsilon^2}\nabla_s d_A^*F_A-\nabla_s\nabla_t B_t-*\nabla_s X_t(A)\right\|^2+\frac1{\varepsilon^2}\|\nabla_s d_A^*B_t\|^2\\
&-\left\langle \nabla_s\nabla_s\left(\frac 1{\varepsilon^2}d_A^*F_A-\nabla_t B_t-*X_t(A)\right),B_s\right\rangle+\langle \nabla_s\nabla_s d_A^*B_t,C\rangle\\
\intertext{where in the second step we use the Yang-Mills flow equation (\ref{flow:eq:apriori:eq2}). Then by the commutation formula (\ref{flow:eq:apriori:eq4})}
=& \left\|\frac 1{\varepsilon^2}d_A^*\nabla_s F_A-\nabla_t\nabla_s B_t-d*X_t(A)B_s-\frac 1{\varepsilon^2}*[B_s,*F_A]-[C,B_t]\right\|^2\\
&+\frac 1{\varepsilon^2}\|d_A^*\nabla_s B_t-*[B_s\wedge * B_t]\|^2\\
&-\left\langle \nabla_s\left(\frac 1{\varepsilon^2}d_A^*\nabla_s F_A-\nabla_t\nabla_s B_t-d*X_t(A)B_s\right),B_s\right\rangle\\
&-\left\langle \nabla_s\left(-\frac 1{\varepsilon^2}*[B_s,*F_A]-[C,B_t]\right),B_s\right\rangle\\
&+\left\langle \nabla_s\left(d_A^*\nabla_s B_t-*[B_s\wedge B_t]\right),C\right\rangle\\
\intertext{and by the Bianchi identity}
=& \left\|D_1\right\|^2+\frac 1{\varepsilon^2}\|d_A^*\nabla_ t B_s-d_A^*d_AC-*[B_s\wedge *B_t]\|^2\\
&-\left\langle \nabla_s\left(\frac 1{\varepsilon^2}d_A^*d_AB_s-\nabla_t\nabla_ t B_s+\nabla_t d_AC-d*X_t(A)B_s\right),B_s\right\rangle\\
&-\left\langle \nabla_s\left(-\frac 1{\varepsilon^2}*[B_s,*F_A]-[C,B_t]\right),B_s\right\rangle\\
&+\left\langle \nabla_s\left(d_A^*\nabla_t B_s-d_A^*d_AC-*[B_s\wedge B_t]\right),C\right\rangle\\
\intertext{in addition, if we apply (\ref{flow:eq:expconvv1}) and (\ref{flow:eq:expconvv2}), then}
=& \left\|D_1\right\|^2+\frac 1{\varepsilon^2}\|D_2\|^2-\left\langle \nabla_s D_1,B_s\right\rangle+\left\langle \nabla_s D_2,C\right\rangle\\
\intertext{if we permute the derivatives in $D_1$ and $D_2$ with $\nabla_s$ and we apply the partial integration, then}
=& \left\|D_1\right\|^2+\frac 1{\varepsilon^2}\|D_2\|^2-\left\langle \nabla_sB_s, D_1\right\rangle+\left\langle \nabla_sC, D_2\right\rangle\\
&-\left\langle \left[\nabla_s,\left(\frac 1{\varepsilon^2}d_A^*d_A-\nabla_t\nabla_ t+\frac 1{\varepsilon^2}d_Ad_A^*\right)\right]B_s,B_s\right\rangle\\
&+\left\langle \frac 1{\varepsilon^2}*\left[B_s,*\nabla_sF_A\right] +2\left[C,\nabla_sB_t\right]+d^2*X_t(A)[B_s,B_s],B_s\right\rangle\\
&+\left\langle \left[\nabla_s,\left(\varepsilon^2\nabla_t\nabla_t-d_A^*d_A\right)\right]C,C\right\rangle-2\left\langle *\left[B_s\wedge *\nabla_s B_t\right],C\right\rangle
\end{align*}
The last three lines can be estimates by
\begin{align*}
\left|\left\langle \left[\nabla_s,\frac 1{\varepsilon^2}d_A^*d_A\right]B_s,B_s\right\rangle\right|=&\frac 1{\varepsilon^2}\left| 2\left \langle \left[B_s\wedge B_s\right],d_A B_s\right\rangle\right|\\
\leq &\frac {c \delta}{\varepsilon^2}  \|B_s\|\cdot \|d_AB_s\|
\leq \frac {c \delta}{\varepsilon^4}  \|d_AB_s\|^2+\delta \|B_s\|^2,\\
\left|\left\langle \left[\nabla_s,\nabla_t\nabla_ t\right]B_s,B_s\right\rangle\right|=&\left| \left \langle \left[C,\nabla_t B_s\right],B_s\right\rangle
+\left\langle \nabla_t\left[C,B_s\right], B_s\right\rangle\right|\\
=&\left|2 \left \langle \left[C,\nabla_t B_s\right],B_s\right\rangle\right|\\
\leq &c\delta \| C\|\cdot \|B_s\| \leq c\delta\|C\|^2+ \delta \|B_s\|^2, \\
\left|\left\langle \left[\nabla_s,\frac 1{\varepsilon^2}d_Ad_A^*\right]B_s,B_s\right\rangle\right|
=&0,\\
\left|\left\langle \frac 1{\varepsilon^2}*\left[B_s,*\nabla_sF_A\right],B_s\right\rangle\right|
\leq &\frac {c\delta}{\varepsilon^2} \|d_AB_s\|\cdot\|B_s\| 
\leq \frac {c\delta}{\varepsilon^4} \|d_AB_s\|^2+\delta\|B_s\|^2, \\
\left|\left\langle 2\left[C,\nabla_sB_t\right],B_s\right\rangle\right|
\leq& c\delta \|d_AB_s\|\cdot\|C\| \leq \frac {c\delta}{\varepsilon^2} \|d_AB_s\|^2+\varepsilon^2\delta\|C\|^2, \\
\left|\left\langle \left[\nabla_s,\varepsilon^2\nabla_t\nabla_t\right]C,C\right\rangle\right|=&0,\\
\left|\left\langle \left[\nabla_s,d_A^*d_A\right]C,C\right\rangle\right|
=& \left| 2\left\langle \left[C,d_AC\right],B_s \right\rangle\right|\leq c\delta\|d_AC\|^2,\\
\left|\left\langle 2*\left[B_s\wedge *\nabla_s B_t\right],C\right\rangle\right|\leq & c\delta\|C\|\cdot \|d_AB_s\|\leq \varepsilon^2\delta\|C\|^2+\frac {c\delta}{\varepsilon^2} \|d_AB_s\|^2,\\
|\langle d^2*X_t(A)[B_s,B_s],B_s\rangle |\leq &c\delta \|B_s\|^2;
\end{align*}
we can therefore conclude that
\begin{align*}
\frac12 \partial_s^2\big(\|B_s\|^2+\varepsilon^2\|C\|^2\big)
\geq& 2\left\|D_1\right\|^2+\frac 2{\varepsilon^2}\|D_2\|^2
-\frac {c\delta}{\varepsilon^4}  \|d_AB_s\|^2\\
&- \delta\|B_s\|^2- c\delta \|C\|^2+c\delta\|d_AC\|\\
\geq& \left\|D_1\right\|^2+\frac 1{\varepsilon^2}\|D_2\|^2
\end{align*}
where the last step follows from the lemma \ref{flow:lemma:expconv} and choosing $\delta$ and $\varepsilon_0$ small enough and thus, we concluded the proof of the identity (\ref{flow:thm:expconv:eq1}).
\end{proof}
\noindent We concluded therefore the proof of the theorem \ref{flow:thm:expconv}.
\end{proof}
\noindent Next, we use the notation of the section \ref{flow:section:apriori}.
\begin{theorem}\label{flow:thm:apriori22}
We choose four constants $b, c_0>0$, $p, s_1\geq 2$. There are three positive constants $\varepsilon_0$, $c$ and $\rho$ such that the following holds. If a perturbed Yang-Mills flow $\Xi=A+\Psi dt+\Phi ds \in \mathcal M^\varepsilon(\Xi_-,\Xi_+)$, with $\Xi_\pm=A_\pm+\Psi_\pm dt \in \mathrm{Crit}^b_{\mathcal {YM}^{\varepsilon, H}}$ and $0<\varepsilon<\varepsilon_0$, satisfies
\begin{equation}\label{flow:apriori7422}
\|\partial_t A-d_A\Psi\|_{L^4(\Sigma)}+\|\partial_s A-d_A\Phi\|_{L^\infty(\Sigma)}\leq c_0,
\end{equation} 
then
\begin{equation}\label{flow:eq:apriorig122}
\sup_{(t,s)\in S^1\times[s_0,\infty)}\left(\|B_s\|_{L^\infty(\Sigma)}+\varepsilon\|C\|_{L^\infty(\Sigma)}\right)\leq c e^{-\rho s_0},
\end{equation}
\begin{equation}\label{flow:eq:apriorig1221}
\sup_{(t,s)\in S^1\times[s_0,\infty)}\left(\mathcal N_1+\mathcal N_2+\mathcal N_3+\mathcal N_4\right)\leq c\varepsilon^2 e^{-\rho s_0},
\end{equation}
\begin{equation}\label{flow:eq:apriorig1222}
\sup_{(t,s)\in S^1\times[s_0,\infty)}\left(\|\alpha\|_{L^\infty(\Sigma)}+\varepsilon\|\nabla_t\alpha\|_{L^\infty(\Sigma)}
\right)\leq c e^{-\rho s_0},
\end{equation}
\begin{equation}\label{flow:eq:apriorig1223}
\sup_{(t,s)\in S^1\times[s_0,\infty)}\left(\|d_A^*\alpha\|_{L^\infty(\Sigma)}+\|d_A\alpha\|_{L^\infty(\Sigma)}
\right)\leq c e^{-\rho s_0},
\end{equation}
\begin{equation}\label{flow:eq:apriorig1224}
\sup_{(t,s)\in S^1\times[s_0,\infty)}\left(\varepsilon\|\nabla_td_A\alpha\|_{L^p(\Sigma)}+\varepsilon^2\|\nabla_sd_A^*\alpha\|_{L^p(\Sigma)}\right)\leq ce^{-\rho s_0} ,
\end{equation}
\begin{equation}\label{flow:eq:apriorig1225}
\sup_{(t,s)\in S^1\times[s_0,\infty)}\left(\varepsilon\|\psi\|_{L^\infty(\Sigma)}
+\varepsilon\|d_A\psi\|_{L^\infty(\Sigma)}
+\varepsilon^2\|\nabla_t\psi\|_{L^\infty(\Sigma)}\right)\leq c e^{-\rho s_0},\\
\end{equation}
\begin{equation}\label{flow:eq:apriorig1226}
\sup_{(t,s)\in S^1\times[s_0,\infty)}\left(\varepsilon^2\|\nabla_td_A\psi\|_{L^p(\Sigma)}+\varepsilon^3\|\nabla_sd_A\phi\|_{L^p(\Sigma)}\right)\leq ce^{-\rho s_0} ,\\
\end{equation}
\begin{equation}\label{flow:eq:apriorig1227}
\sup_{S^1\times [s_0,\infty) }\left(\|F_A-F_{A_+}(s)\|_{L^\infty(\Sigma)}+\varepsilon \|\nabla_t (F_A-F_{A_+}(s))\|_{L^p(\Sigma)}\right)\leq c \varepsilon^{2-\frac 1p}e^{-\rho s_0},\\
\end{equation}
\begin{equation}\label{flow:eq:apriorig1228}
\sup_{S^1\times [s_0,\infty) }\left(\varepsilon^{1-\frac 1p} \|\nabla_s F_A\|_{L^p(\Sigma}+\varepsilon^2 \|\nabla_t\nabla_t (F_A-F_{A_+}(s))\|_{L^p(\Sigma)}\right)\leq c \varepsilon^{2-\frac 1p}e^{-\rho s_0},
\end{equation}
\begin{equation}\label{flow:eq:apriorig1229}
\sup_{S^1\times [s_0,\infty) } \left(\left\|B_t- B_t^+\right\|_{L^\infty}+\varepsilon\left\|\nabla_t(B_t- B_t^+)\right\|_{L^p}\right)\leq ce^{-\rho s_0},
\end{equation}
\begin{equation}\label{flow:eq:apriorig12210}
\sup_{S^1\times [s_0,\infty) } \varepsilon^2\left\|\nabla_s(B_t- B_t^+)\right\|_{L^p}\leq ce^{-\rho s_0},
\end{equation}
where $s_0>s_1$ and, for $g\in \mathcal G_0^{2,2}(\Sigma\times S^1\times \mathbb R)$ defined by $g^{-1}\partial_sg=\Phi$,
$$A_+(s)+\Psi_+(s) dt:=g(s)^*\left(A_++\Psi_+ dt\right),$$
$$\alpha(s)+\phi(s) dt:= (A(s)+\Psi(s) dt)-(A_+(s)+\Psi_+(s) dt),$$
$$B_t^+(s):=\partial_tA_+(s)+d_{A_+(s)}\Psi_+(s).$$
\end{theorem}

\begin{proof} By the estimate (\ref{flow:eq:apriorig20}), the theorem \ref{flow:thm:expconv} and the lemma \ref{lemma76dt94} we know that there two constants $\rho$ and $c$ such that
\begin{equation}\label{kdsfaogh}
\begin{split}
\sup_{(t,s)\in S^1\times[s_0,\infty)}\left(\|B_s\|_{L^\infty(\Sigma)}+\varepsilon\|C\|_{L^\infty}\right)\leq c e^{-\rho s_0},\\
\sup_{(t,s)\in S^1\times[s_0,\infty)}\left(\mathcal N_1+\mathcal N_2+\mathcal N_3+\mathcal N_4\right)\leq c\varepsilon^2 e^{-\rho s_0}.
\end{split}
\end{equation}
Thus, if we integrate the first estimate of $\ref{kdsfaogh}$ we have
$$\sup_{(t,s)\in S^1\times[s_0,\infty)}\left(\|\alpha\|_{L^\infty(\Sigma)}+\varepsilon\|\psi\|_{L^\infty(\Sigma)}\right)\leq c e^{-\rho s_0}$$
and if we pick $s_2\in [s_0,\infty)$, then the third estimate of the theorem follows from the computation
\begin{align*}
\varepsilon\|\nabla_t\alpha(s_2)\|_{L^\infty(\Sigma)}\leq& \int_{\infty}^{s_2}\varepsilon\|\nabla_s\nabla_t(A(s)-A_+(s)\|_{L^\infty(\Sigma)} ds\\
\leq& \int_{\infty}^{s_2}\varepsilon\|C\|_{L^\infty}\|A(s)-A_+(s)\|_{L^\infty(\Sigma)} ds\\
&+ \int_{\infty}^{s_2}\varepsilon\|\nabla_t\nabla_s(A(s)-g(s)^*A_+)\|_{L^\infty(\Sigma)} ds\\
\leq& ce^{-\rho s_2}+\int_{\infty}^{s_2}\varepsilon\|\nabla_t(\partial_sA(s)+[\Psi(s),A(s)]\\
&\qquad\qquad\qquad\quad-[\Phi(s),g(s)^*A_+]-d_{g(s)^*A_+}\Phi(s))\|_{L^\infty(\Sigma)}ds\\
\leq& ce^{-\rho s_2}+\int_{\infty}^{s_2}\varepsilon\|\nabla_tB_s\|_{L^\infty(\Sigma)}ds\\
\leq& ce^{-\rho s_2}+c\int_{\infty}^{s_2}\left(\varepsilon\|\nabla_tB_s\|_{L^2(\Sigma)}+\varepsilon\|d_A^*d_A\nabla_tB_s\|_{L^2(\Sigma)}\right) ds\\
 &+\int_{\infty}^{s_2}\varepsilon\|d_Ad_A^*\nabla_tB_s\|_{L^2(\Sigma)}ds\\
\leq& ce^{-\rho s_2}
\end{align*}
where the constant $c$ does not depend on $(t,s_2)$ for $(t,s_2)\in S^1\times [s_0,\infty)$. The second step of the computation follows from the commutation formula (\ref{flow:eq:apriori:eq4}), the third by the definition of $g(s)$ and the previous estimates, the fifth by the lemma \ref{lemma76dt94} and the last one by \ref{kdsfaogh}. The estimates (\ref{flow:eq:apriorig1223})-(\ref{flow:eq:apriorig1226}) follows in the same way. Next we prove the first part of the (\ref{flow:eq:apriorig1227}). By (\ref{flow:eq:apriorig1})
$$\sup_{(t,s)\in S^1\times[s_0,\infty) }\|F_A\|_{L^\infty(\Sigma)}\leq c \varepsilon^{2}$$
and by the Bianchi identity $d_AB_s=\nabla_sF_A$ and by the lemma \ref{lemma76dt94}:
\begin{align*}
\sup_{(t,s)\in S^1\times[s_0,\infty) }\|\nabla_sF_A\|_{L^\infty(\Sigma)}=&c\sup_{(t,s)\in S^1\times[s_0,\infty) }\|d_Ad_A^*d_AB_s\|_{L^2(\Sigma)}
\leq c\varepsilon e^{-\rho s_0}.
\end{align*}
Thus, integrating the last estimate, $\sup_{(t,s)\in S^1\times[s_0,\infty) }\|F_A-F_{A_+}\|_{L^\infty(\Sigma)}\leq c\varepsilon e^{-\rho s_0}$ and hence
\begin{align*}
\sup_{(t,s)\in S^1\times[s_0,\infty) }&\|F_A-F_{A_+}\|^p_{L^\infty(\Sigma)}\\
\leq& c\varepsilon e^{-\rho s_0}\sup_{(t,s)\in S^1\times[s_0,\infty) }\left(\|F_A\|_{L^\infty(\Sigma)}+\|F_{A_+}\|_{L^\infty(\Sigma)}\right)^{p-1}\\
\leq &c\varepsilon^{2p-1} e^{-\rho s_0}
\end{align*}
and finally we obtain
$$\sup_{(t,s)\in S^1\times[s_0,\infty) }\|F_A-F_{A_+}\|_{L^\infty(\Sigma)}\leq c\varepsilon^{2-\frac 1p} e^{-\frac\rho p s_0}.$$
The other estimates of (\ref{flow:eq:apriorig1227})-(\ref{flow:eq:apriorig12210}) follow in the same way using the Bianchi identity, the Yang-Mills equation (\ref{flow:eq}) in order to commute the operators and the estimates (\ref{kdsfaogh}).
\end{proof}

\section{Relative Coulomb gauge}\label{ch:rcg}

\begin{theorem}\label{surj:thm:prop6.2}
Assume $q\geq p>2$, $q>4$ and $qp/(q-p)>4$. We choose $\Xi_0=A_0+\Psi_0 dt+\Phi_0 ds \in \mathcal A^{1,p}(\Xi_-,\Xi_+)$ such that $F_{A_0}=0$. Then for every constant $c_0>0$ there exist constants $\delta>0$ and $c>0$ such that the following holds for $0<\varepsilon \leq 1$. If $\Xi\in \mathcal A^{1,p}(\Xi_-,\Xi_+)$ satisfies
\begin{equation}\label{surj:prop62eq1}
\varepsilon^2\left\|d_{\Xi_0}^{*_\varepsilon}\left(\Xi-\mathcal K_2^\varepsilon(\Xi_0)\right)\right\|_{L^p}\leq c_0 \varepsilon^{\frac 3p}, \quad \left\|\Xi-\mathcal K_2^\varepsilon(\Xi_0)\right\|_{0,q,\varepsilon}\leq \delta \varepsilon^{\frac 3q },
\end{equation}
then there exists a gauge transformation $g\in\mathcal G^{2,p}_0(P\times S^1\times \mathbb R)$ such that $$d_{\Xi_0}^{*_\varepsilon }\left(g^*\Xi-\mathcal K_2^\varepsilon(\Xi_0)\right)=0$$ and 
\begin{equation}\label{surj:prop62eq2}
\left\|g^*\Xi-\Xi\right\|_{1,p,\varepsilon}\leq c\varepsilon^2\left(1+\varepsilon^{-\frac 3p }\left\|\Xi-\mathcal K_2^\varepsilon(\Xi_0)\right\|_{1,p,\varepsilon}\right)\left\|d_{\Xi_0}^{*_\varepsilon}\left(\Xi-\mathcal K_2^\varepsilon(\Xi_0)\right)\right\|_{L^p}.
\end{equation}
\end{theorem}

This last theorem is analogous to the proposition 6.2 in \cite{MR1283871}, but if we compare them, we will remark some differences. The first one is given by different rescaling in the $s$ direction induced by the equations (we have an $\varepsilon^2$ factor instead of $\varepsilon$) and this also induces a difference in the Sobolev's estimates and it causes the change in some exponents: $\frac 2p$, $\frac 2q$ and $-\frac 2p$ become $\frac 3p$, $\frac 3q$ and $-\frac 3p$.
The second difference is an extra $\varepsilon^2$ factor in the first estimate of (\ref{surj:prop62eq1}) and in (\ref{surj:prop62eq2}); this is given by the difference in the definitions of $d_{\Xi_0}^{*_\varepsilon}$. In  \cite{MR1283871} it is defined by, with $\Xi_0=:A+\Psi dt+\Phi ds$,
\begin{equation}
d_{\Xi_0}^{*_\varepsilon}(\alpha+\psi dt+ \phi ds)=d_{A}^*\alpha-\varepsilon^2\nabla_t\psi-\varepsilon^2\nabla_s\phi,
\end{equation}
our definition is instead
\begin{equation}
\varepsilon^2 d_{\Xi_0}^{*_\varepsilon}(\alpha+\psi dt+ \phi ds)=d_{A}^*\alpha-\varepsilon^2\nabla_t\psi-\varepsilon^4\nabla_s\phi.
\end{equation}
The third difference is that we use the difference $\Xi-\mathcal K_2^\varepsilon(\Xi_0)$ instead of $\Xi-\Xi_0$ and this is needed in order to have finite norms.\\

For any $\sigma\in \mathbb R$ we define $\rho_\sigma:\mathbb R^2\to \mathbb R^2$ by $\rho_\sigma(t,s)=(t,s+\sigma)$.

\begin{theorem}\label{flow:thm:timeshift}
We choose $p>10$ and $b>0$. Let $\Xi_0\in \mathcal M^0(\Xi_-,\Xi_+)$, $\Xi_\pm \in \mathrm{Crit}_{E^H}^b$ with index difference $1$. Then there exist three positive constants $\varepsilon_0, \delta$ and $c$ such that the following holds. If $0<\varepsilon <\varepsilon_0$ and $\Xi\in \mathcal M^\varepsilon(\mathcal T^{\varepsilon, b }(\Xi_-),\mathcal T^{\varepsilon,b}(\Xi_+))$ such that
\begin{equation}
\left\|\Xi-\mathcal K_2^\varepsilon(\Xi_0)\right\|_{1,p,\varepsilon}\leq \delta \varepsilon^{1-\frac 4p },\quad \varepsilon^2\left\|\nabla_s \Xi\right\|_{0,p,\varepsilon}\leq c \varepsilon^{1+\frac 7p }
\end{equation}
then there exist $\sigma\in \mathbb R$ and $g\in \mathcal G^{2,p}_0(P\times S^1\times \mathbb R)$ such that $\Xi^\varepsilon=g^*\left(\Xi \circ \rho_\sigma\right)$ satisfies
\begin{equation}
d_{\Xi_0}^{*_\varepsilon}(\Xi^\varepsilon-\mathcal K_2^\varepsilon(\Xi_0))=0,\quad \Xi^\varepsilon-\mathcal K_2^\varepsilon(\Xi_0) \in \textrm{im }\left(\mathcal D^\varepsilon(\mathcal K_2^\varepsilon(\Xi_0))\right)^*
\end{equation}
and
\begin{equation}
\left\|\Xi^\varepsilon-\mathcal K_2^\varepsilon(\Xi_0)\right\|_{1,p,\varepsilon}\leq c\left\|\Xi-\mathcal K_2^\varepsilon(\Xi_0)\right\|_{1,p,\varepsilon}.
\end{equation}
Furthermore, for $\Xi^\varepsilon-\mathcal K_2^\varepsilon(\Xi_0):=\alpha^\varepsilon+\psi^\varepsilon dt+\phi^\varepsilon ds $ and $\Xi-\mathcal K_2^\varepsilon(\Xi_0):=\alpha+\psi dt+\phi ds $, then
\begin{equation}\label{flow:eq:789}
\left\|\nabla_t\alpha^\varepsilon\right\|_{L^p}\leq \left\|\nabla_t\alpha\right\|_{L^p}+c \left\|\Xi-\mathcal K_2^\varepsilon(\Xi_0)\right\|_{1,p,\varepsilon},
\end{equation}
\begin{equation}
\left\|\nabla_s\alpha^\varepsilon \right\|_{L^p}\leq \left\|\nabla_s\alpha\right\|_{L^p}+c \left\|\Xi-\mathcal K_2^\varepsilon(\Xi_0)\right\|_{1,p,\varepsilon},
\end{equation}
\begin{equation}
\left\|\nabla_t\psi^\varepsilon\right\|_{L^p}\leq \left\|\nabla_t\psi\right\|_{L^p}+c \left\|\Xi-\mathcal K_2^\varepsilon(\Xi_0)\right\|_{1,p,\varepsilon},
\end{equation}
\begin{equation}
\left\|\nabla_s\psi^\varepsilon \right\|_{L^p}\leq \left\|\nabla_s\psi\right\|_{L^p}+c \left\|\Xi-\mathcal K_2^\varepsilon(\Xi_0)\right\|_{1,p,\varepsilon},
\end{equation}
\begin{equation}
\left\|\nabla_t\phi^\varepsilon\right\|_{L^p}\leq \left\|\nabla_t\phi\right\|_{L^p}+c \left\|\Xi-\mathcal K_2^\varepsilon(\Xi_0)\right\|_{1,p,\varepsilon},
\end{equation}
\begin{equation}\label{flow:eq:7888899}
\left\|\nabla_s\phi^\varepsilon \right\|_{L^p}\leq \left\|\nabla_s\phi\right\|_{L^p}+c \left\|\Xi-\mathcal K_2^\varepsilon(\Xi_0)\right\|_{1,p,\varepsilon}.
\end{equation}
\end{theorem}

\noindent In order to prove the theorem \ref{surj:thm:prop6.2} we need the following lemma

\begin{lemma}\label{flow:lemma:dkdk}
Assume $q\geq p >2$, $q>4$, and $pq/(q-p)>4$. Given $c_0>0$ there exists a constant $c>0$ such that, if $\|\eta\|_{L^\infty}\leq c_0$ and $g=\exp (\eta)$, then
\begin{equation}
\begin{split}
\varepsilon^2\left\|d_{\Xi_0}^{*_\varepsilon}\left(g^*\Xi-\Xi -d_\Xi\eta\right)\right\|_{L^p}\leq & 
c\varepsilon^{-\frac 3q}\left(\|\eta\|_{1,q,\varepsilon}+\|\Xi-\mathcal K_2^\varepsilon(\Xi_0)\|_{0,q,\varepsilon}+\varepsilon^2\right)\|\eta\|_{2,p,\varepsilon}\\
&+c\varepsilon^{-\frac 3q}\left(\varepsilon^2\left\|d_{\Xi_0}^{*_\varepsilon}\left(\Xi-\mathcal K_2^\varepsilon(\Xi_0)\right)\right\|_{L^p}+\varepsilon^2\right)\|\eta\|_{1,q,\varepsilon}
\end{split}
\end{equation}
and if $\|\eta\|_{1,q,\varepsilon}+\|\Xi-\Xi_0\|_{0,q,\varepsilon}\leq c_0 \varepsilon^{\frac 3q}$, then
\begin{equation}
 \left\|g^*\Xi-\Xi\right\|_{0,q,\varepsilon}\leq c\|\eta\|_{1,q,\varepsilon},
\end{equation}
\begin{equation}
 \left\|g^*\Xi-\Xi\right\|_{1,p,\varepsilon}
\leq c\left(\|\eta\|_{2,p,\varepsilon}+\varepsilon^{-\frac 2q}\|\Xi-\Xi_0\|_{1,p,\varepsilon}\|\eta\|_{1,q,\varepsilon}\right).
\end{equation}
\end{lemma}

\begin{proof}
The lemma follows exactly as the lemma 6.6 in \cite{MR1283871} using the estimate (\ref{flow:k2:leqw3}).
\end{proof}

\begin{proof}[Proof of theorem \ref{surj:thm:prop6.2}] We choose $\Xi_1=\Xi$ and we define the sequence $\Xi_\nu$, for $\nu\geq 2$, by 
$$\Xi_{\nu+1}=g_\nu^*\Xi_\nu,\quad g_\nu=\exp(\eta_\nu),\quad d_{\Xi_0}^{*_\varepsilon}\left(d_{\Xi_0}\eta_\nu+\Xi_\nu-\mathcal K_2^\varepsilon(\Xi_0)\right)=0$$
by the definition of $\eta_\nu$ and the lemma 6.4 in \cite{MR1283871} and the Sobolev theorem \ref{flow:thm:sob} we have that
\begin{equation}\label{flow:thm62:2a}
\|\eta_\nu\|_{2,p,\varepsilon}+\varepsilon^{\frac 3p-\frac 3q}\|\eta_\nu\|_{1,q,\varepsilon}\leq c_1\varepsilon^2\left\| d_{\Xi_0}^{*_\varepsilon}(\Xi_\nu-\mathcal K_2^\varepsilon(\Xi_0))\right\|_{L^p},
\end{equation}
\begin{equation}\label{flow:thm62:2b}
\|\eta_\nu\|_{1,q,\varepsilon}\leq c_1 \|\Xi_\nu-\mathcal K_2^\varepsilon(\Xi_0)\|_{0,q,\varepsilon}.
\end{equation}
In order to conclude the proof of the theorem we need first to show by induction that there are three positive constants $c_2$, $c_3$ and $c_4$ such that the following estimates hold.
\begin{equation}\label{flow:thm62:3}
\left\|\Xi_\nu-\mathcal K_2^\varepsilon(\Xi_0)\right\|_{0,q,\varepsilon}\leq c_2 \left\|\Xi-\mathcal K_2^\varepsilon(\Xi_0)\right\|_{0,q,\varepsilon}+c_2\left\|\Pi_{\textrm{im } d_{\Xi_0}}(\Xi-\mathcal K_2^\varepsilon(\Xi_0))\right\|_{0,p,\varepsilon},
\end{equation}

\begin{equation}\label{flow:thm62:4}
\begin{split}
\varepsilon^2\left\|d_{\Xi_0}^{*_\varepsilon}\left(\Xi_\nu-\mathcal K_2^\varepsilon(\Xi_0)\right)\right\|_{L^p}\leq& c_3 \varepsilon^{2-\frac 3q}\left\|\Xi_{\nu-1}-\mathcal K_2^\varepsilon(\Xi_0)\right\|_{0,q,\varepsilon}\left\| d_{\Xi_0}^{*_\varepsilon}\left(\Xi_{\nu-1}-\mathcal K_2^\varepsilon(\Xi_0)\right)\right\|_{L^p}\\
&+c_3\varepsilon^2\|\eta_{\nu-1}\|_{2,p,\varepsilon},
\end{split}
\end{equation}

\begin{equation}\label{flow:thm62:5}
\left\|  d_{\Xi_0}^{*_\varepsilon}\left(\Xi_\nu-\mathcal K_2^\varepsilon(\Xi_0)\right) \right\|_{L^p}\leq 2^{1-\nu} \left\| d_{\Xi_0}^{*_\varepsilon}\left(\Xi-\mathcal K_2^\varepsilon(\Xi_0)\right)\right\|_{L^p},
\end{equation}

\begin{equation}\label{flow:thm62:6}
\|\eta_\nu\|_{1,q,\varepsilon}\leq c_4 2^{-\nu}\left( \left\|\Xi-\mathcal K_2^\varepsilon(\Xi_0)\right\|_{0,q,\varepsilon}+\left\|\Pi_{\textrm{im } d_{\Xi_0}}(\Xi-\mathcal K_2^\varepsilon(\Xi_0))\right\|_{0,p,\varepsilon}\right).
\end{equation} 
For $\nu=1$  (\ref{flow:thm62:3}) and (\ref{flow:thm62:5}) are satisfied by definition and with $c_2\geq1$, (\ref{flow:thm62:2b}) implies (\ref{flow:thm62:6}) for $c_4\geq 2c_1$  and (\ref{flow:thm62:4}) is empty. Next, we consider $\nu\geq 2$. By the assumptions of the theorem and by (\ref{flow:thm62:2b}), for $\delta$ small enough, we have that
$$\|\eta_j\|_{1,q,\varepsilon}+\|\Xi_j-\mathcal K_2^\varepsilon(\Xi_0)\|_{0,q,\varepsilon}\leq \varepsilon^{\frac 3q},\quad j=1,...,\nu-1.$$
By lemma \ref{flow:lemma:dkdk} and (\ref{flow:thm62:6})
\begin{align*}
\left\| \Xi_{j+1}-\Xi_j\right\|_{0,q,\varepsilon}\leq& c_5 \|\eta_j\|_{1,q,\varepsilon}\\
\leq &c_5c_4 2^{-\nu}\left(\left\| \Xi-\mathcal K_2^\varepsilon(\Xi)\right\|_{0,q,\varepsilon}+\left\|\Pi_{\textrm{im } d_{\Xi_0}}(\Xi-\mathcal K_2^\varepsilon(\Xi_0))\right\|_{0,p,\varepsilon}\right)
\end{align*}
and thus we have (\ref{flow:thm62:3}). Next,
\begin{align*}
d_{\Xi_0}^{*_\varepsilon}&\left(\Xi_{\nu+1}-\mathcal K_2^\varepsilon(\Xi_0)\right)
=d_{\Xi_0}^{*_\varepsilon}\left(g_\nu^*\Xi_\nu-\Xi_\nu-d_{\Xi_0}\eta_\nu\right)\\
=&d_{\Xi_0}^{*_\varepsilon}\left(g_\nu^*\Xi_\nu-\Xi_\nu-d_{\Xi_\nu}\eta_\nu\right)+\left[d_{\Xi_0}^{*_\varepsilon}\left(\Xi_\nu-\mathcal K_2^\varepsilon(\Xi_0)\right)\wedge \eta_\nu\right]\\
&+\left[d_{\Xi_0}^{*_\varepsilon}\left(\mathcal K_2^\varepsilon(\Xi_0)-\Xi_0\right)\wedge \eta_\nu\right]
-*_\varepsilon\left[*_\varepsilon\left(\Xi_\nu-\mathcal K_2^\varepsilon(\Xi_0)\right)\wedge d_{\Xi_0}\eta_\nu\right]\\
&-*_\varepsilon\left[*_\varepsilon\left(\mathcal K_2^\varepsilon(\Xi_0)-\Xi_0\right)\wedge d_{\Xi_0}\eta_\nu\right]
\end{align*}
and hence by the lemma (\ref{flow:thm62:2a}) and by (\ref{flow:thm62:2b}), (\ref{flow:thm62:4}), we can conclude (\ref{flow:thm62:4}). By (\ref{flow:thm62:2a}) and (\ref{flow:thm62:4}) we get
\begin{align*}
\|\eta_\nu\|_{2,p,\varepsilon}&+\varepsilon^{\frac 3p-\frac 3q}\|\eta_\nu\|_{1,q,\varepsilon}\\
\leq& c_1c_3\varepsilon^{2-\frac 3q}\left\|\Xi_{\nu-1}-\mathcal K_2^\varepsilon(\Xi_0)\right\|_{0,q,\varepsilon}\left\|d_{\Xi_0}^{*_\varepsilon}\left(\Xi_{\nu-1}-\mathcal K_2^\varepsilon(\Xi_0)\right)\right\|_{L^p}\\
&+c_2c_3\varepsilon^{2-\frac 3p}\|d_{\Xi_0}\eta_\nu\|_{0,p,\varepsilon}
\end{align*}
and hence 
\begin{align*}
\|\eta_\nu\|_{2,p,\varepsilon}&+\varepsilon^{\frac 3p-\frac 3q}\|\eta_\nu\|_{1,q,\varepsilon}\\
\leq &c_1c_3\varepsilon^{-\frac 3q}\left\|\Xi_{\nu-1}-\mathcal K_2^\varepsilon(\Xi_0)\right\|_{0,q,\varepsilon}\left\|d_{\Xi_0}^{*_\varepsilon}\left(\Xi_{\nu-1}-\mathcal K_2^\varepsilon(\Xi_0)\right)\right\|_{L^p}\\
&+c_2c_3\varepsilon^{2-\frac 3p}\|\Pi_{d_{\Xi_0}}(\Xi_{\nu-1}-\mathcal K_2^\varepsilon(\Xi_0))\|_{0,p,\varepsilon}
\end{align*}
and 
\begin{equation}\label{flow:thm62:4b}
\begin{split}
\varepsilon^2&\left\|d_{\Xi_0}^{*_\varepsilon}\left(\Xi_\nu-\mathcal K_2^\varepsilon(\Xi_0)\right)\right\|_{L^p}\\
\leq& 2c_3 \varepsilon^{2-\frac 3q}\left\|\Xi_{\nu-1}-\mathcal K_2^\varepsilon(\Xi_0)\right\|_{0,q,\varepsilon}\left\| d_{\Xi_0}^{*_\varepsilon}\left(\Xi_{\nu-1}-\mathcal K_2^\varepsilon(\Xi_0)\right)\right\|_{L^p}\\
&+c_2c_3\varepsilon^{2-\frac 3p}\|\Pi_{d_{\Xi_0}}(\Xi_{\nu-1}-\mathcal K_2^\varepsilon(\Xi_0))\|_{0,p,\varepsilon}.
\end{split}
\end{equation}
By (\ref{flow:thm62:2b}) and (\ref{flow:thm62:3}) 
$$\|\eta_2\|_{1,q,\varepsilon}\leq c_0 c_3 \left\|\Xi-\mathcal K_{2}^\varepsilon(\Xi_0)\right\|_{0,q,\varepsilon}$$
using (\ref{flow:thm62:4b}) 
\begin{align*}
\|\eta_\nu\|_{1,q,\varepsilon}\leq &16c_1c_2c_3^2\delta 2^{-\nu}\left\|\Xi_{\nu-1}-\mathcal K_2^\varepsilon(\Xi_0)\right\|_{0,q,\varepsilon}\\
&+c_2c_3\varepsilon^{2-\frac 3p}\|\Pi_{d_{\Xi_0}}(\Xi_{\nu-1}-\mathcal K_2^\varepsilon(\Xi_0))\|_{0,p,\varepsilon}.
\end{align*}
(\ref{flow:thm62:6}) therefore holds for $c_4=1$ whenever $\delta$ and $\varepsilon$ are small enough.
The lemma \ref{flow:lemma:dkdk} with (\ref{flow:thm62:2a}) and (\ref{flow:thm62:2b}) implies 
$$\left\|\Xi_{\nu+1}-\Xi_\nu\right\|_{1,p,\varepsilon}\leq c_6\varepsilon^2\left(1+\varepsilon^{-\frac 3p}\left\|\Xi_\nu-\mathcal K_2^\varepsilon(\Xi_0)\right\|_{1,p,\varepsilon}\right)\left\|d_{\Xi_0}^{*_\varepsilon}\left(\Xi_\nu-\Xi_0\right)\right\|_{L^p}$$
and thus, for $\delta$ sufficiently small,
$$\left\|\Xi_\nu-\Xi_0\right\|_{1,p,\varepsilon}\leq \varepsilon^{\frac 3p}+2\left\|\Xi-\mathcal K_2^\varepsilon(\Xi_0)\right\|_{1,p,\varepsilon}.$$
The sequence converges therefore in $W^{1,p}$ to a connection $\Xi_\varepsilon$ which satisfies the condition $d_{\Xi_0}^{*_\varepsilon}\left(\Xi_\varepsilon-\Xi_0\right)=0$ and the estimate (\ref{surj:prop62eq2}). In addition, the sequence $h_\nu:=g_1g_2...g_\nu$ satisfies $h_\nu^*\Xi=\Xi_\nu$ and converges in $\mathcal G^{2,p}_0(P\times S^1\times \mathbb R)$ to a gauge transformation $g$ which satisfies $g^*\Xi=\Xi_\varepsilon$. 
\end{proof}

\begin{proof}[Proof of theorem \ref{flow:thm:timeshift}] We follows the proof of the proposition 6.3 in \cite{MR1283871} adapting it for our needs. First, we consider the estimates
\begin{equation}\label{flow:eq:78999}
\begin{split}
\left\|\Xi\circ \tau_\sigma- \mathcal K_2^\varepsilon(\Xi_0)\right\|_{1,p,\varepsilon}
\leq & \left\|\Xi-\mathcal K_2^\varepsilon(\Xi_0)\right\|_{1,p,\varepsilon}+\left\|\mathcal K_2^\varepsilon(\Xi_0)\circ\tau_\sigma-\mathcal K_2^\varepsilon(\Xi_0)\right\|_{1,p,\varepsilon}\\
\leq & \left\|\Xi-\mathcal K_2^\varepsilon(\Xi_0)\right\|_{1,p,\varepsilon}+|\sigma|\cdot\left\|\partial_s\mathcal K_2^\varepsilon(\Xi_0)\right\|_{1,p,\varepsilon},\\
\left\|\nabla_s\left(\Xi\circ \tau_\sigma-\mathcal K_2^\varepsilon(\Xi_0)\right)\right\|_{0,p,\varepsilon}
\leq & \left\|\nabla_s\left(\Xi-\mathcal K_2^\varepsilon(\Xi_0)\right)\right\|_{0,p,\varepsilon}
+c|\sigma|\cdot\left\|\Xi-\mathcal K_2^\varepsilon(\Xi_0)\right\|_{L^p}\\
&+\left\|\nabla_s\left(\mathcal K_2^\varepsilon(\Xi_0)\circ\tau_\sigma\right)-\nabla_s\mathcal K_2^\varepsilon(\Xi_0)\right\|_{0,p,\varepsilon}\\
\leq &\left\|\nabla_s\left(\Xi-\mathcal K_2^\varepsilon(\Xi_0)\right)\right\|_{0,p,\varepsilon}
+c|\sigma|\cdot\left\|\Xi-\mathcal K_2^\varepsilon(\Xi_0)\right\|_{L^p}\\
&+c|\sigma|\cdot \left\|\partial_s\nabla_s \mathcal K_2^\varepsilon(\Xi_0)\right\|_{0,p,\varepsilon},\\
\left\|\nabla_t\left(\Xi\circ \tau_\sigma-\mathcal K_2^\varepsilon(\Xi_0)\right)\right\|_{0,p,\varepsilon}
\leq & \left\|\nabla_t\left(\Xi-\mathcal K_2^\varepsilon(\Xi_0)\right)\right\|_{0,p,\varepsilon}
+c|\sigma|\cdot\left\|\Xi-\mathcal K_2^\varepsilon(\Xi_0)\right\|_{L^p}\\
&+\left\|\nabla_t\left(\mathcal K_2^\varepsilon(\Xi_0)\circ\tau_\sigma\right)-\nabla_t\mathcal K_2^\varepsilon(\Xi_0)\right\|_{0,p,\varepsilon}\\
\leq &\left\|\nabla_t\left(\Xi-\mathcal K_2^\varepsilon(\Xi_0)\right)\right\|_{0,p,\varepsilon}
+c|\sigma|\cdot\left\|\Xi-\mathcal K_2^\varepsilon(\Xi_0)\right\|_{L^p}\\
&+c|\sigma|\cdot \left\|\partial_s\nabla_t \mathcal K_2^\varepsilon(\Xi_0)\right\|_{0,p,\varepsilon},
%
%
%
\end{split}
\end{equation}
where, by the definitions of the section \ref{flow:section:firstapprox},
\begin{equation}\label{flow:eq:78999c}
\begin{split}
\left\| \partial_s\mathcal K_2^\varepsilon(\Xi_0)\right\|_{1,p,\varepsilon}
\leq& \left\|\partial_s \Xi_0\right\|_{1,p,\varepsilon}+\left\|\nabla_s(\mathcal K_2^\varepsilon(\Xi_0)-\Xi_0)\right\|_{1,p,\varepsilon}\\
\leq& \left\|\partial_s\Xi_0\right\|_{1,p,\varepsilon}+c\varepsilon^2,\\
\left\|\partial_s\nabla_s \mathcal K_2^\varepsilon(\Xi_0)\right\|_{0,p,\varepsilon}\leq &
\left\|\partial_s\nabla_s \Xi_0\right\|_{0,p,\varepsilon}+\left\|\nabla_s\nabla_s (\mathcal K_2^\varepsilon(\Xi_0)-\Xi_0)\right\|_{L^p}\\
&+\left\|\nabla_s(\mathcal K_2^\varepsilon(\Xi_0)-\Xi_0)\right\|_{L^p}+c\varepsilon^2\leq c,\\
\left\|\partial_s\nabla_t \mathcal K_2^\varepsilon(\Xi_0)\right\|_{0,p,\varepsilon}\leq &
\left\|\partial_s\nabla_t \Xi_0\right\|_{0,p,\varepsilon}+\left\|\nabla_s\nabla_t (\mathcal K_2^\varepsilon(\Xi_0)-\Xi_0)\right\|_{L^p}\\
&+\left\|\nabla_t(\mathcal K_2^\varepsilon(\Xi_0)-\Xi_0)\right\|_{L^\infty}+c\varepsilon^2\leq c.
\end{split}
\end{equation}
Therefore for $|\sigma|\leq \delta \varepsilon^{\frac3p}$, $\left\|\Xi\circ \tau_\sigma-\mathcal K_2^\varepsilon(\Xi_0)\right\|_{1,p,\varepsilon}\leq \delta \varepsilon^{\frac 3p}$ for $\varepsilon$ small enough, and thus by theorem \ref{surj:thm:prop6.2}, there is a gauge transformation $g_\sigma\in \mathcal G_0^{2,p}(P\times S^1\times \mathbb R)$ such that for $\Xi_\sigma=g_\sigma^*\left(\Xi\circ \tau_\sigma\right)$, $d_{\Xi_0}^{*_\varepsilon}\left(\Xi_\sigma-\mathcal K_2^\varepsilon(\Xi_0)\right)=0$ and 
$$\left\|\Xi_\sigma-\mathcal K_2^\varepsilon(\Xi_0)\right\|_{1,p,\varepsilon}\leq c_1\left(|\sigma|+\left\| \Xi-\mathcal K_2^\varepsilon(\Xi_0)\right\|_{1,p,\varepsilon}\right).$$ 
We assume that $g_0= 1$. We need to show that there is a $\sigma$ such that $\Xi_\sigma-\mathcal K_2^\varepsilon(\Xi_0)\in \textrm{im } \mathcal D^\varepsilon(\mathcal K_2^\varepsilon(\Xi_0))^*$ and $|\sigma|\leq c_2 \left\| \Xi-\mathcal K_2^\varepsilon(\Xi_0)\right\|_{1,p,\varepsilon}$. The estimates (\ref{flow:eq:789})-(\ref{flow:eq:7888899}) follows as (\ref{flow:eq:78999}) and (\ref{flow:eq:78999c}) computing separately every component.\\

$\mathcal D^0(\Xi_0)$ is onto and has index $1$; therefore its kernel is spanned by $\xi_0=\partial_s\Xi_0\in W^{1,p}$. Then $\mathcal D^\varepsilon (\mathcal K_2^\varepsilon(\Xi_0))$ has index $1$ with the kernel spanned by 
$$\xi_\varepsilon=\xi_0-\mathcal D^\varepsilon(\mathcal K_2^\varepsilon(\Xi_0))^*\left(\mathcal D^\varepsilon(\mathcal K_2^\varepsilon(\Xi_0))\mathcal D^\varepsilon(\mathcal K_2^\varepsilon(\Xi_0))^*\right)^{-1}\mathcal D^\varepsilon(\mathcal K_2^\varepsilon(\Xi_0))\xi_0.$$
Consider now the function $\theta(\sigma)=\langle \xi_\varepsilon,\Xi_\sigma-\mathcal K_2^\varepsilon(\Xi_0)\rangle_\varepsilon$; thus, if $\theta(\sigma)=0$, then $\Xi_\sigma-\mathcal K_2^\varepsilon(\Xi_0)\in \textrm{im } \mathcal D^\varepsilon(\mathcal K_2^\varepsilon(\Xi_0))^*$. We assume that there are positive constants $\delta_0$, $\varepsilon_0$ and $\rho_0$ such that, for $0<\varepsilon<\varepsilon_0$,
\begin{equation}\label{flow:surj:thetaestsa}
|\sigma|+\left\|\Xi-\mathcal K_2^\varepsilon(\Xi_0)\right\|_{1,p,\varepsilon}\leq \delta_0\varepsilon^{1-\frac 4p}\quad \Rightarrow\quad \theta'(\sigma)\geq \rho_0.
\end{equation}
Then the existence of a zero for $\theta(\sigma)$ follows from
\begin{equation*}
\begin{split}
|\theta(0)|=\left|\langle \xi_\varepsilon, \Xi-\mathcal K_2^\varepsilon(\Xi_0)\rangle_\varepsilon \right|\leq \left\|\xi_\varepsilon\right\|_{0,q,\varepsilon}\left\|\Xi-\mathcal K_2^\varepsilon(\Xi_0)\right\|_{0,p,\varepsilon}\leq c_3\delta \varepsilon^{1-\frac 4p}
\end{split}
\end{equation*}
where $q=\frac p{p-1}$. In fact, if $c_3\delta<\frac 12\delta_0\rho_0$, $\delta\leq \frac 12 \delta_0$, then $\left\|\Xi-\mathcal K_2^\varepsilon(\Xi_0)\right\|_{1,p,\varepsilon}\leq \frac 12 \delta_0\varepsilon^{1-\frac 4p}$. Therefore, by (\ref{flow:surj:thetaestsa}), there is a $\sigma\in \mathbb R$ with $|\sigma|\leq \frac {\theta(0)}{\rho_0}\leq \frac 12 \delta_0\varepsilon^{1-\frac 4p}$ such that $\theta(\sigma)=0$. For this $\sigma$, we have
$$|\sigma|\leq c\left\|\Xi-\mathcal K_2^\varepsilon(\Xi_0)\right\|_{1,p,\varepsilon},\quad \Xi_\sigma-\mathcal K_2^\varepsilon(\Xi_0)\in \textrm{im } \mathcal D^\varepsilon(\mathcal K_2^\varepsilon(\Xi_0))^*.$$
Thus, in order to finish the proof of the theorem we need only to show (\ref{flow:surj:thetaestsa}).
\begin{proof}[Proof of (\ref{flow:surj:thetaestsa})]
We define $\eta_\sigma=g_\sigma^{-1}\left(\partial_\sigma g_\sigma-\partial_s g_\sigma\right)$, then 
\begin{equation}\label{flow:relgeq2}
\theta'(\sigma)=\langle \xi_\varepsilon,\partial_s\Xi_\sigma+d_{\Xi_\sigma}\eta_\sigma\rangle_\varepsilon,\quad \partial_\sigma d_{\Xi_0}^{*_\varepsilon}\left(\Xi_\sigma-\mathcal K_2^\varepsilon(\Xi_0)\right)=0.
\end{equation}
Thus,
\begin{equation}\label{flow:surj:eqthete5}
d_{\Xi_0}^{*_\varepsilon}\partial_s\Xi_\sigma+d_{\Xi_0}^{*_\varepsilon}d_{\Xi_0}\eta_\sigma
+d_{\Xi_0}^{*_\varepsilon}\left[\left(\Xi_\sigma-\Xi_0\right)\wedge \eta_\sigma\right]=0.
\end{equation}
If $\varepsilon^{-\frac3p}\left\|\Xi_\sigma-\mathcal K_2^\varepsilon(\Xi_0)\right\|_{1,p,\varepsilon}$ and $\left\|\mathcal K_2^\varepsilon(\Xi_0)-\Xi_0\right\|_{1,\infty,\varepsilon}$ are sufficiently small, then there is a unique $\eta_\sigma$ which satisfies (\ref{flow:surj:eqthete5}), furthermore
\begin{equation}\label{flow:relgeq1}
\begin{split}
\left\|\eta_\sigma\right\|_{1,p,\varepsilon}\leq &c\left\|\partial_s\Xi_\sigma\right\|_{0,p,\varepsilon}
\leq c\left(1+\left\|\partial_s\left(\Xi_\sigma-\Xi_0\right)\right\|_{0,p,\varepsilon}\right)\\
\leq &c\left(1 +\left\|\nabla_s\left(\Xi_\sigma-\mathcal K_2^\varepsilon(\Xi_0)\right)\right\|_{0,p,\varepsilon}+\left\|\Xi_\sigma-\mathcal K_2^\varepsilon(\Xi_0)\right\|_{0,p,\varepsilon}\right).
\end{split}
\end{equation}
where the last step follows from the definition of $\mathcal K_2^\varepsilon$ and the estimate (\ref{flow:k2:leqw3}). Since $d_{\Xi_0}^{*_\varepsilon}\xi_\varepsilon=0$, we have
\begin{equation*}
\begin{split}
\left|\langle \xi_\varepsilon,d_{\Xi_\sigma}\eta_\sigma\rangle_\varepsilon\right|=&\left|\langle \xi_\varepsilon, \left[\left(\Xi_\sigma-\Xi_0\right)\wedge \eta_\sigma\right]\rangle_\varepsilon\right|\\
\leq & c \left\|\Xi_\sigma-\Xi_0\right\|_{\infty,\varepsilon}\left\|\eta_\sigma\right\|_{0,p,\varepsilon}\\
\leq &c\varepsilon^{-\frac 3p}\left\|\Xi_\sigma-\Xi_0\right\|_{1,p,\varepsilon}\left(1 +\left\|\nabla_s\left(\Xi_\sigma-\mathcal K_2^\varepsilon(\Xi_0)\right)\right\|_{0,p,\varepsilon}\right)\\
&+c\varepsilon^{-\frac 3p}\left\|\Xi_\sigma-\Xi_0\right\|_{1,p,\varepsilon}\left\|\Xi_\sigma-\mathcal K_2^\varepsilon(\Xi_0)\right\|_{0,p,\varepsilon}\\
\leq & c\delta+c\delta \varepsilon^{-\frac 3p}\varepsilon^{1-\frac 4p}\left\|\nabla_s\left(\Xi_\sigma-\mathcal K_2^\varepsilon(\Xi_0)\right)\right\|_{0,p,\varepsilon}.
\end{split}
\end{equation*}
where the second inequality follows from the Sobolev theorem \ref{flow:thm:sob} and from (\ref{flow:relgeq1}) and the last one is a consequence of the assumptions. Next we consider
$$\langle \xi_\varepsilon,\partial_s\Xi_\sigma\rangle_\varepsilon =\langle \partial_s\xi_\varepsilon,\mathcal K_2^\varepsilon(\Xi_0)-\Xi_\sigma\rangle_\varepsilon+\langle \xi_\varepsilon,\partial_s\mathcal K_2^\varepsilon(\Xi_0)\rangle_\varepsilon,$$
Since $\left\|\partial_s\Xi_0\right\|_{0,2,\varepsilon}\geq 3\rho_0>0$ for some $\rho_0$, $\langle \xi_\varepsilon,\partial_s\mathcal K_2^\varepsilon (\Xi_0)\rangle_\varepsilon\geq 2\rho_0$
by the definition of $\xi_\varepsilon$ and thus $\langle \xi_\varepsilon,\partial_s\Xi_\sigma+d_{\Xi_\sigma}\eta_\sigma\rangle_\varepsilon >\rho_0$ for $\delta_0$ small enough. Hence, by (\ref{flow:relgeq2}) $\theta'(\sigma)>\rho_0$.
\end{proof}
\end{proof}

%
%
\section{Surjectivity of $\mathcal R^{\varepsilon,b }$}\label{flow:section:surj}

In this section we will show that the map $\mathcal R^{\varepsilon,b}$ defined in the section \ref{flow:section:themap} is also surjective.

\begin{theorem}\label{flow:thm:surj}
We assume that the energy functional $E^H$ is Morse-Smale and we choose $b>0$ to be a regular value of $E^H$. Then there is a constant $\varepsilon_0>0$ such that the following holds. For every $\varepsilon\in (0,\varepsilon_0)$, every pair $\Xi_\pm:=A_\pm+\Psi_\pm dt\in \mathrm{Crit}^b_{E^H}$, the map
\begin{equation}
\mathcal R^{\varepsilon,b}: \mathcal M^0\left(\Xi_-,\Xi_+\right)\to \mathcal M^\varepsilon\left(\mathcal T^{\varepsilon,b}(\Xi_-),\mathcal T^{\varepsilon,b}(\Xi_+)\right)
\end{equation}
is surjective.
\end{theorem}
\begin{figure}[ht]
\begin{center}
\begin{tikzpicture} 

\draw (-1,-1) .. controls (-1,0) and (0,0.7) .. (0.5,1); 
\draw (6.5,-1) .. controls (6.5,-0.5) and (7,0.1) .. (7.2,0.2); 
\draw (-1,-1).. controls (3,-1) and (3,-1) .. node[near start,above] {$F_A=0$}(6.5,-1); 

\filldraw (0,0) circle (1pt) node[below] {$\Xi_-$}
(6,0) circle (1pt) node[below] {$\Xi_+$}
(0,2) circle (1pt) node[above, left] {$\mathcal T^{\varepsilon_\nu,b}(\Xi_-)$}
(6,1) circle (1pt) node[above] {$\mathcal T^{\varepsilon_\nu,b}(\Xi_+)$}; 
\draw (0,0) .. controls (1,0.5) and (2,0.2) .. (3,0); 
\draw (3,0) .. controls (4,-0.2)  and (5,-0.2) .. node[near start,below] {$\Xi^\nu_2$}(6,0); 
 
\draw (0,0) .. controls (0,1) and (0,1) .. node[left] {$\alpha_{-}^{\nu}+\psi_-^{\nu} dt$}  (0,2); 
\draw (6,0) .. controls (6,0.5) and (6,0.5) .. node[right] {$\alpha_{+}^{\nu}+\psi_+^{\nu} dt$}  (6,1); 

\draw[blue] (0,0) .. node[above] {$\Xi^\nu_1$} controls (1,1) and (2,0.4) ..(2.5,0.2); 
\draw[blue] (2.5,0.2) .. controls (3,0) and (2.7,1.9) ..  (3,1.8); 
\draw[blue] (3,1.8) .. controls (3.3,1.7) and (3,0.7) .. (3.5,0.5); 
\draw[blue] (3.5,0.5) .. controls (4,0.3)  and (5,0.2) .. (6,0); 

\draw[magenta] (0,2) .. controls (1,3) and (2,2.4) ..(3,1.8); 
\draw[magenta] (3,1.8) .. controls (4,1.2)  and (5,1.2) ..node[above] {$\Xi^\nu$} (6,1); 
\end{tikzpicture} 
\caption{Idea of the proof of theorem \ref{flow:thm:surj}.}
\end{center}
\end{figure}
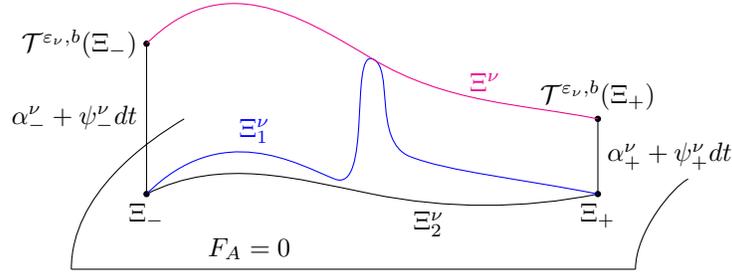
\begin{proof} 
We prove the theorem indirectly. We assume that there is a sequence $\Xi^\nu\in\mathcal M^{\varepsilon_\nu}\left(\mathcal T^{\varepsilon_\nu,b}(\Xi_-),\mathcal T^{\varepsilon_\nu,b}(\Xi_+)\right)$, $\varepsilon_\nu\to 0$, that is not in the image of $\mathcal R^{\varepsilon_\nu,b}$. Hence by the theorems \ref{flow:thm:apriori} and \ref{flow:thm:linf}, for a positive constant $c_0$,
\begin{equation}\label{flow:est:surj:bo}
\begin{split}
\|\partial_s A^\nu-d_{A^\nu}\Phi^\nu\|_{L^\infty(\Sigma)}+\varepsilon_\nu \|\partial_s\Psi-\partial_t\Phi-[\Psi,\Phi]\|_{L^\infty(\Sigma)}\leq c_0,\\
\varepsilon_\nu^2\|\partial_t A^\nu-d_{A^\nu}\Psi^\nu\|_{L^\infty(\Sigma)}+\|F_{A^\nu}\|_{L^\infty(\Sigma)}+\|d_{A^\nu}d_{A^\nu}^*F_{A^\nu}\|_{L^2(\Sigma)}\leq c_0\varepsilon_\nu^2
\end{split}
\end{equation}
and all the estimates of the theorem \ref{flow:thm:apriori22}. The proof is structured in the following way. In the first step, see figure \ref{flow:section:surj}.1, we will define a sequence $\Xi_1^\nu$ which converges to $\Xi_\pm$ for $s\to \pm\infty$ and in the following step we will project it on the space $\mathcal A_0(P)$ defining a new sequence $\Xi_2^\nu$; then by the implicit function theorem \ref{flow:thm:implfunctthm} there is a subsequence of $\Xi_2^\nu$ which converges to a geodesic flow  (step 3). Finally, after choosing appropriate gauge transformations and time shifts (steps 4-6) we can show that the sequence satisfies the assumptions of the local uniqueness theorem \ref{flow:thm:locuniq} (step 7) and therefore a subsequence turns out to lie in the image of $\mathcal R^{\varepsilon,b}$. This yields to a contradiction.\\

Furthermore, we assume that there is a positive $S_0$ such that $\Phi^\nu=0$ for $|s|\geq S_0$. In the general case, the $\Phi^\nu$ converge to $0$ for $|s|\to\infty$ in an exponential way and hence we can find a sequence $\{\tilde g_\nu\}_{\nu\in \mathbb N }$ of gauge transformations such that $\tilde g_\nu^*\Xi^\nu$ has the above property. First, we pick the sequence $\{g_\nu\}_{\nu\in\mathbb N}$ of gauge transformations defined by $g_{\nu }^{-1}\partial_s g_\nu:=\Phi_\nu$ and we define
$$A^\nu_\pm(s)+\Psi^\nu_\pm(s) dt:=g(s)^*(A^\nu_\pm+\Psi^\nu_\pm dt),$$
for $\Xi^\nu_{\pm}:=A^\nu_\pm+\Psi^\nu_\pm dt :=\lim_{s\to \pm\infty} \Xi^\nu$. Second, like in the section \ref{flow:section:themap} we choose a smooth positive function $\theta(s)=0$ for $s\leq1$ and $\theta(s)=1$ for $s\geq2$, such that $0\leq \theta\leq 1$ and $0\leq \partial_s\theta\leq c_0$ with $c_0>0$ and we define a family of 1-forms $\alpha_0^\nu+\psi_0^\nu dt$ as  
\begin{equation}\label{flow:surj:dsiabp}
\begin{split}
\xi_0^\nu=\alpha_0^\nu+\psi_0^\nu dt:=&\theta(-s)\left((A^\nu_-+\Psi^\nu_- dt)-\left(\mathcal T^{\varepsilon_\nu,b }\right)^{-1}(A^\nu_-+\Psi^\nu_- dt)\right)\\
&+\theta(s)\left((A^\nu_++\Psi^\nu_+ dt)-\left(\mathcal T^{\varepsilon_\nu,b }\right)^{-1}(A^\nu_++\Psi^\nu_+ dt)\right);
\end{split}
\end{equation}  
in addition we denote 
$$\alpha^\nu+\psi^\nu dt+\phi^\nu ds:=\theta(-s)\left(\Xi^\nu-\Xi^\nu_-\right)+\theta(s)\left(\Xi^\nu-\Xi^\nu_+\right)$$
which satisfies the uniformly exponential convergence estimates of the theorem \ref{flow:thm:apriori22}.\\

\noindent{\bf Step 1.} We define $\Xi^\nu_1=A^\nu_1+\Psi^\nu_1 dt+\Phi^\nu_1 ds:=\Xi^\nu-\xi_0^\nu$ then there is a constant $c>0$ such that
\begin{equation}\label{flow:est:surj:bo2}
\begin{split}
\|\partial_t A^\nu_1-d_{A^\nu_1}\Psi^\nu_1\|_{L^\infty(\Sigma)}+\|\partial_s A^\nu_1-d_{A^\nu_1}\Phi^\nu_1\|_{L^\infty(\Sigma)}\leq c
\end{split}
\end{equation}
\begin{equation}\label{flow:est:surj:bodfs}
\begin{split}
\|F_{A_1^\nu}\|_{L^\infty(\Sigma)}+\varepsilon_\nu^2\left\|\nabla_t^{\Psi_1^\nu}\left(\partial_tA_1^\nu-d_{A_1^\nu}\Psi_1^\nu\right)\right\|_{L^p(\Sigma)}\leq c\varepsilon^2
\end{split}
\end{equation}
and for $s\geq 0$
\begin{equation*}
\begin{split}
\partial_s& A^\nu_1-d_{A^\nu_1}\Phi^\nu_1-\nabla^{\Psi^\nu_1}_t\left(\partial_t A^\nu_1-d_{A^\nu_1}\Psi^\nu_1\right)-*X_t(A^{\nu}_1)\\
=&-\frac 1{\varepsilon_\nu^2}d_{A^\nu_1}^*d_{A^-}\bar\alpha_{0,+}^\nu
+\left[\psi_0^\nu,\left(\left(\partial_t A^\nu-d_{A^\nu}\Psi^\nu\right)- \left(\partial_t A_+^\nu-d_{A^\nu_+}\Psi^\nu_+\right)\right)\right]\\
&-\frac1{\varepsilon_\nu^2}d_{A^\nu}^*\left(F_{A^\nu}-F_{A^\nu_+}\right)+\frac 1{\varepsilon_\nu^2}*\left[\alpha^\nu,\left(d_{A_+}\left(\alpha_0^\nu-\bar\alpha_{0,+}^\nu\right)+\frac 12[\alpha_0^\nu\wedge\alpha_0^\nu]\right)\right]\\
&+\left(*X_t(A_+)-*X_t(A^\nu_1)\right)+\left(*X_t(A^\nu)-*X_t(A^\nu_+)\right)\\
&+\left[\psi^\nu,\left(\left(\partial_t A^\nu-d_{A^\nu}\Psi^\nu\right)- \left(\partial_t A_1^\nu-d_{A^\nu_1}\Psi^\nu_1\right)\right)\right]
+\nabla_t^{\Psi_+}\left([\psi^\nu,\alpha_0^\nu]-[\alpha^\nu,\psi_0^\nu]\right)
\end{split}
\end{equation*}
where $\bar \alpha_{0,\pm}^\nu(s)\in \textrm{im } d_{A_\pm(s)}^*$ are defined uniquely by $$d_{A_\pm}^*d_{A_\pm}\bar\alpha_{0,\pm}^\nu=\varepsilon_\nu^2\nabla_t^{\Psi_\pm}\left(\partial_t A_\pm-d_{A_\pm}\Psi_\pm\right).$$

\begin{proof}[Proof of step 1.] The estimates (\ref{flow:est:surj:bo2}) and (\ref{flow:est:surj:bodfs}) follow from (\ref{flow:est:surj:bo}), from the inequalities of the theorem \ref{thm:mainthm}, the remark \ref{thm:existence:crit:est} and the Sobolev theorem \ref{lemma:sobolev}:
$$\|\partial_s A^\nu_1-d_{A^\nu_1}\Phi^\nu_1\|_{L^\infty(\Sigma)}=\|\partial_s A^\nu-d_{A^\nu}\Phi^\nu\|_{L^\infty(\Sigma)}\leq c,$$
\begin{align*}
\|\partial_t A^\nu_1-d_{A^\nu_1}\Psi^\nu_1\|_{L^\infty(\Sigma)}\leq &
\|\partial_t A^\nu-d_{A^\nu}\Psi^\nu\|_{L^\infty(\Sigma)}\\
&+\|\nabla_t^{\Psi^\nu}\alpha_0^\nu\|_{L^\infty(\Sigma)}+\|d_{A_1^\nu}\psi_0^\nu\|_{L^\infty(\Sigma)}
\leq c,
\end{align*}
$$\|F_{A_1^\nu}\|_{L^\infty(\Sigma)}\leq \|F_{A_1^\nu}\|_{L^\infty(\Sigma)}+\|d_{A_1^\nu}\alpha_0^\nu\|_{L^\infty(\Sigma)}+c\|\alpha_0^\nu\|_{L^\infty(\Sigma)}\leq c,$$
\begin{align*}
\Big\|\nabla_t^{\Psi_1^\nu}&\left(\partial_tA_1^\nu-d_{A_1^\nu}\Psi_1^\nu\right)\Big\|_{L^p(\Sigma)}\leq 
\left\|\nabla_t^{\Psi^\nu}\left(\partial_tA^\nu-d_{A^\nu}\Psi^\nu\right)\right\|_{L^p(\Sigma)}+c\\
&\qquad\leq \frac 1{\varepsilon_\nu^2}\left\|d_{A^\nu}^*F_{A^\nu}\right\|_{L^p(\Sigma)}+\left\|\partial_s A^\nu-d_{A^\nu}\Psi^\nu\right\|_{L^p(\Sigma)}+c
\leq c
\end{align*}
where for the last estimate we use also the Yang-Mills flow equation (\ref{flow:eq}). In order to prove the identity, we first remark that
\begin{equation}\label{flow:surj:eq.slsl}
\begin{split}
\partial_s A^\nu_1&-d_{A^\nu_1}\Phi^\nu_1-\nabla^{\Psi^\nu_1}_t\left(\partial_t A^\nu_1-d_{A^\nu_1}\Psi^\nu_1\right)-*X_t(A^{\nu}_1)\\
=&\partial_s A^\nu -d_{A^\nu}\Phi^\nu-\nabla_t^{\Psi^\nu} \left(\partial_t A^\nu-d_{A^\nu}\Psi^\nu\right)-*X_t(A^\nu)\\
&+*X_t(A^\nu)-*X_t(A^\nu_1)+\left[\psi_0^\nu,\left(\partial_t A^\nu-d_{A^\nu}\Psi^\nu\right)\right]\\
&+\nabla_t^{\Psi_1^\nu}\left(\left(\partial_t A^\nu-d_{A^\nu}\Psi^{\nu}\right)-\left(\partial_t A^\nu_1-d_{A^\nu_1}\Psi^\nu_1\right)\right);
\end{split}
\end{equation}
next, in order to simplify the exposition, we consider $s\geq 0$, for a negative $s$ the proof is analogous. Since $\Xi^\nu$ is a Yang-Mills flow, we have
\begin{align*}
\partial_s A^\nu_1&-d_{A^\nu_1}\Phi^\nu_1-\nabla^{\Psi^\nu_1}_t\left(\partial_t A^\nu_1-d_{A^\nu_1}\Psi^\nu_1\right)-*X_t(A^{\nu}_1)\\
=&-\frac 1{\varepsilon_\nu^2}d_{A^\nu}^*F_{A^\nu}+\left[\psi_0^\nu,\left(\partial_t A^\nu_+-d_{A^\nu_+}\Psi^\nu_+\right)\right]\\
&+\left[\psi_0^\nu,\left(\left(\partial_t A^\nu-d_{A^\nu}\Psi^\nu\right)- \left(\partial_t A_+^\nu-d_{A^\nu_+}\Psi^\nu_+\right)\right)\right]\\
&+\left[\psi^\nu,\left(\left(\partial_t A^\nu-d_{A^\nu}\Psi^\nu\right)- \left(\partial_t A_1^\nu-d_{A^\nu_1}\Psi^\nu_1\right)\right)\right]\\
&+\nabla_t^{\Psi_+}\left(\left(\partial_t A^\nu-d_{A^\nu}\Psi^\nu\right)- \left(\partial_t A_1^\nu-d_{A^\nu_1}\Psi^\nu_1\right)\right)\\
&+\left(-*X_t(A^\nu_1)+*X_t(A^\nu)\right).
\end{align*}
Furthermore, since 
\begin{align*}
\left[\psi_0^\nu,\left(\partial_t A^\nu_+-d_{A^\nu_+}\Psi^\nu_+\right)\right]
&+\nabla_t^{\Psi_+}\left(\left(\partial_t A^\nu-d_{A^\nu}\Psi^\nu\right)- \left(\partial_t A_1^\nu-d_{A^\nu_1}\Psi^\nu_1\right)\right)\\
=&+\nabla_t^{\Psi_+^\nu}\left(\partial_t A^\nu_+-d_{A^\nu_+}\Psi^\nu_+\right)-\nabla_t^{\Psi_+}\left(\partial_t A_+-d_{A_+}\Psi_+\right)\\
&+\nabla_t^{\Psi_+}\left([\psi^\nu,\alpha_0^\nu]-[\alpha^\nu,\psi_0^\nu]\right)\\
=&\frac 1{\varepsilon_\nu^2}d_{A_+^\nu}^*F_{A_+^\nu}-*X_t(A^\nu_+)-\frac1{\varepsilon_\nu^2}d_{A_+}^*d_{A_+}\bar\alpha_{0,+}^\nu\\
&+*X_t(A_+)+\nabla_t^{\Psi_+}\left([\psi^\nu,\alpha_0^\nu]-[\alpha^\nu,\psi_0^\nu]\right)
\end{align*}
where the last identity follows from the equations for the perturbed geodesics (\ref{intro:eqg0}) and for the perturbed Yang-Mills connections (\ref{intro:YMeq1}). Thus,
\begin{align*}
\partial_s A^\nu_1&-d_{A^\nu_1}\Phi^\nu_1-\nabla^{\Psi^\nu_1}_t\left(\partial_t A^\nu_1-d_{A^\nu_1}\Psi^\nu_1\right)-*X_t(A^{\nu}_1)\\
=&-\frac 1{\varepsilon_\nu^2}d_{A^\nu}^*F_{A^\nu}+\frac 1{\varepsilon_\nu^2}d_{A_+^\nu}^*F_{A_+^\nu}
-\frac1{\varepsilon_\nu^2}d_{A_+}^*d_{A_+}\bar\alpha_{0,+}^\nu\\
&+\nabla_t^{\Psi_+}\left([\psi^\nu,\alpha_0^\nu]-[\alpha^\nu,\psi_0^\nu]\right)\\
&+\left[\psi_0^\nu,\left(\left(\partial_t A^\nu-d_{A^\nu}\Psi^\nu\right)- \left(\partial_t A_+^\nu-d_{A^\nu_+}\Psi^\nu_+\right)\right)\right]\\
&+\left(-*X_t(A^\nu_1)+*X_t(A_+)\right)+\left(*X_t(A^\nu)-*X_t(A^\nu_+)\right)\\
&+\left[\psi^\nu,\left(\left(\partial_t A^\nu-d_{A^\nu}\Psi^\nu\right)- \left(\partial_t A_1^\nu-d_{A^\nu_1}\Psi^\nu_1\right)\right)\right]
\end{align*}
Next, the first step then follows directely using the identities
$$d_{A^\nu}^*F_{A^\nu}
= d_{A^\nu}^*(F_{A^\nu}-F_{A^\nu_+})-*\left[\alpha^\nu,*\left(d_{A_+}\alpha_0^\nu
+\frac12[\alpha_0^\nu\wedge\alpha_0^\nu]\right)\right]+d_{A^\nu_+}^*F_{A^\nu_+},$$
$$-\frac1{\varepsilon_\nu^2}d_{A_+}^*d_{A_+}\bar\alpha_{0,+}^\nu
=-\frac1{\varepsilon_\nu^2}d_{A_1^\nu}^*d_{A_+}\bar\alpha_{0,+}^\nu
-\frac 1{\varepsilon_\nu^2}*[\alpha^\nu,*d_{A_+}\bar\alpha_{0,+}^\nu].$$
\end{proof}
Since $F_{A_1^\nu}+d_{A_1^\nu}\alpha_0^\nu+\frac 12[\alpha_0^\nu\wedge \alpha_0^\nu]=F_{A^\nu }$, $d_{A_+}\alpha_0^\nu+\frac 12[\alpha_0^\nu\wedge \alpha_0^\nu]=F_{A^\nu_+}$ for $s\geq 2$,   we have
\begin{equation}\label{flow:est:913}
F_{A_1^\nu }=F_{A^\nu }-F_{A^\nu_+ }-\left[\alpha^\nu\wedge\alpha_0^\nu\right],\quad \textrm{for }s\geq 2, 
\end{equation}
and thus we can estimate the norm of $F_{A_1^\nu}$, for any $q\geq2$, by
\begin{equation}\label{flow:est:913a}
\left\| F_{A_1^\nu }\right\|_{L^q(\Sigma) }\leq \left\| F_{A^\nu }-F_{A^\nu_+ }\right\|_{L^q(\Sigma)}+\left\|\left[\alpha^\nu\wedge\alpha_0^\nu\right]\right\|_{L^q(\Sigma) }
\end{equation}
for $s\geq 2$.\\

\noindent{\bf Step 2.} There are two positive constants $c$ and $\delta$ such that the following holds. For every $\Xi^\nu$ there is a connection $\Xi^\nu_2:= A^\nu_2+\Psi^\nu_2 dt+\Phi^\nu_2 ds= \Xi^\nu_1+\alpha_1^\nu+\psi^\nu_1 dt+\phi^\nu_1 ds$, $\alpha_1^\nu\in\textrm{im } d_{A_1^\nu}^*$, which satisfies
$$\begin{array}{llll}
i)&F_{A^{\nu}_2}=0,
&ii)\,\,& d_{ A^{\nu}_2}^*(\partial_t A^{\nu}_2-d_{ A^{\nu}_2}\Psi^{\nu}_2)=0,\\[5pt]
iii)\,\,&d_{ A^{\nu}_2}^*(\partial_s A^{\nu}_2-d_{A^{\nu}_2}\Phi^{\nu}_2)=0,\quad
&iv)& \|\alpha_1^\nu(s)\|_{L^\infty(\Sigma)}\leq ce^{(\theta(-s)-\theta(s))\delta s}\varepsilon_\nu^{2-\frac1p}\\
v)&\left\|\pi_{ A^{\nu}_2}\left(\mathcal F^0\left(\Xi^{\nu}_2\right)\right)\right\|_{L^p}\leq c\varepsilon_\nu^{1-\frac1p},
&vi)&\left\|\pi_{ A^{\nu}_2}\left(\mathcal F^0\left(\Xi^{\nu}_2\right)\right)\right\|_{L^p(\Sigma)}\leq c\varepsilon_\nu^{1-\frac1p},\\
vii)&\lim_{s\to\pm\infty }\Xi^\nu_2=\Xi_\pm.
&&
\end{array}$$

\begin{proof}[Proof of step 2]
By the the first step we know that $\|F_{A_1^\nu}\|_{L^\infty}\leq c\varepsilon^{2}$ and thus, for $\varepsilon$ small enough, the lemma \ref{lemma82dt94} allows us to find a positive constant $c$ such that for any $A^\nu_1$ there is a unique 0-form $\gamma^\nu$ such that 

\begin{equation}\label{flow:est:914}
F_{A^\nu_1+*d_{A^\nu_1}\gamma^\nu}=0,\quad \left\| d_{A^\nu_1}\gamma^\nu\right\|_{L^\infty(\Sigma)}\leq c \left\|F_{A^\nu_1}\right\|_{L^{\infty}(\Sigma)}\leq c\varepsilon_\nu^2.
\end{equation}
and $\Psi^\nu_2$, $\Phi^\nu_2$ are then uniquely given by
\begin{equation}
d_{A^\nu_2}^*\left(\partial_t A^\nu_2-d_{A^\nu_2}\Psi^\nu_2\right)=0,\quad
d_{A^\nu_2}^*\left(\partial_s A^\nu_2-d_{A^\nu_2}\Phi^\nu_2\right)=0.
\end{equation}
Therefore $i)$, $ ii)$ and $iii)$ are satisfied. By (\ref{flow:est:913a}), (\ref{flow:est:914}) one can remark that 
$$\|\alpha^\nu_1(s)\|_{L^\infty(\Sigma)}\leq ce^{(\theta(-s)-\theta(s))\delta s}\varepsilon_\nu^{2-\frac1p}$$
using the a priori estimate of the theorem \ref{flow:thm:apriori22}. In order to show the inequalities $v)$ and $vi)$ of the second step, we consider
\begin{equation*}
\begin{split}
\partial_s A^\nu_2&-d_{A^\nu_2}\Phi^\nu_2
-\nabla_t^{\Psi^\nu_2}\left(\partial_t A^\nu_2-d_{A^\nu_2}\Psi^\nu_2\right)-*X_t(A^\nu_2)\\
=&\partial_s A^\nu_1-d_{A^\nu_1}\Phi^\nu_1
-\nabla_t^{\Psi^\nu_1}\left(\partial_t A^\nu_1-d_{A^\nu_1}\Psi^\nu_1\right)-*X_t(A^\nu_1)\\
&+\nabla_s^{\Phi^\nu_1}\alpha_1^\nu-d_{A^\nu_1}\phi^\nu_1-[\alpha_1^\nu,\phi^\nu_1]
-\left[\psi^\nu_1,\left(\partial_t A^\nu_1-d_{A^\nu_1}\Psi^\nu_1\right)\right]\\
&-\nabla_t^{\Psi^\nu_1}\left(\nabla_t^{\Psi^\nu_1}\alpha_1^\nu-d_{A^\nu_1}\psi^\nu_1-[\alpha^\nu_1,\psi^\nu_1]\right)
-\left(*X_t(A^\nu_2)-*X_t(A^\nu_1)\right)
\end{split}
\end{equation*}
and we remark that
\begin{equation*}
\begin{split}
\nabla_t^{\Psi^\nu_1}\nabla_t^{\Psi^\nu_1}\alpha_1^\nu=&
\nabla_t^{\Psi^\nu_1}\nabla_t^{\Psi^\nu_1}*d_{A^\nu_1}\gamma^\nu\\
=&*d_{A^\nu_2}\nabla_t^{\Psi^\nu_1}\nabla_t^{\Psi^\nu_1}\gamma^\nu
+*\nabla_t^{\Psi^\nu_1}\left[\left(\partial_t A^\nu_1-d_{A^\nu_1}\Psi^\nu_1\right),\gamma^\nu\right]\\
&+*\left[\left(\partial_t A^\nu_1-d_{A^\nu_1}\Psi^\nu_1\right),\nabla_t^{\Psi^\nu_1}\gamma^\nu\right]-*\left[\alpha^\nu_1, \nabla_t^{\Psi^\nu_1}\nabla_t^{\Psi^\nu_1}\gamma^\nu\right],
\end{split}
\end{equation*}
\begin{equation*}
\nabla_t^{\Psi^\nu_1}d_{A^\nu_1}\psi^\nu_1=
d_{A^\nu_2}\nabla_t^{\Psi^\nu_1}\psi^\nu_1-\left[\alpha^\nu_1,\nabla_t^{\Psi^\nu_1}\psi^\nu_1\right]+\left[\left(\partial_t A^\nu_1-d_{A^\nu_1}\Psi^\nu_1\right),\psi^\nu_1\right],
\end{equation*}
$$\nabla_s^{\Phi_1^\nu}\alpha_1^\nu=*d_{A_2^\nu}\nabla_s^{\Phi_1^\nu}\gamma^\nu
+*\left[\left(\partial_sA_1^\nu-d_{A_1^nu}\Phi_1^\nu\right),\gamma^\nu\right]
+*\left[\alpha_1^\nu,\nabla_s^{\Phi_1^\nu}\gamma^\nu\right].$$
Using the first step, the next lemma and the the uniformly exponential convergence estimates of the theorem \ref{flow:thm:apriori22}, we can conclude that
\begin{equation*}
\left\|\pi_{A^\nu_2}\left(\partial_s A^\nu_2-d_{A^\nu_2}\Psi^\nu_2-\nabla_t^{\Psi^\nu_2}\left(\partial_t A^\nu_2-d_{A^\nu_2}\Psi^\nu_2\right)-*X_t(A^\nu_2)\right)\right\|_{L^p}\leq c \varepsilon_\nu^{1-\frac1p},
\end{equation*}
\begin{equation*}
\left\|\pi_{A^\nu_2}\left(\partial_s A^\nu_2-d_{A^\nu_2}\Psi^\nu_2-\nabla_t^{\Psi^\nu_2}\left(\partial_t A^\nu_2-d_{A^\nu_2}\Psi^\nu_2\right)-*X_t(A^\nu_2)\right)\right\|_{L^p(\Sigma)}\leq c \varepsilon_\nu^{1-\frac1p}.
\end{equation*}
\end{proof}

\begin{lemma}\label{flow:lemmasurj1}
There are two positive constants $c$ and $\varepsilon_0$ such that
$$\varepsilon_\nu \left\|\nabla_t^{\Psi^\nu_2}\alpha^\nu_1\right\|_{L^p}+\varepsilon_\nu \left\|\nabla_t^{\Psi^\nu_2}\gamma^\nu\right\|_{L^p}+\varepsilon_\nu^2\left\|\nabla_t^{\Psi^\nu_2}\nabla_t^{\Psi^\nu_2}\gamma^\nu\right\|_{L^p}       \leq c\varepsilon_\nu^{2-\frac 1p},$$
$$\varepsilon_\nu \left\|\nabla_t^{\Psi^\nu_2}\alpha^\nu_1\right\|_{L^p(\Sigma)}+\varepsilon_\nu \left\|\nabla_t^{\Psi^\nu_2}\gamma^\nu\right\|_{L^p(\Sigma)}+\varepsilon_\nu^2\left\|\nabla_t^{\Psi^\nu_2}\nabla_t^{\Psi^\nu_2}\gamma^\nu\right\|_{L^p(\Sigma)}       \leq c\varepsilon_\nu^{2-\frac 1p},$$
$$\left\|d_{A_2^\nu}\phi^\nu_1\right\|_{L^p}+\left\|\nabla_s^{\Phi^\nu_2}\alpha_1^\nu\right\|_{L^p}+\left\|\nabla_s^{\Phi^\nu_2}\gamma^\nu\right\|_{L^p}\leq c \varepsilon_\nu,$$
$$
\left\|d_{A_2^\nu}\phi^\nu_1\right\|_{L^p(\Sigma)}+\left\|\nabla_s^{\Phi^\nu_2}\alpha_1^\nu\right\|_{L^p(\Sigma)}+\left\|\nabla_s^{\Phi^\nu_2}\gamma^\nu\right\|_{L^p(\Sigma)}\leq c \varepsilon_\nu,$$
$$
\|\psi^\nu_1\|_{L^\infty(\Sigma)}+\|d_{A^\nu_2}\psi^\nu_1\|_{L^\infty(\Sigma)}+\varepsilon_\nu\left\|\nabla_t^{\Psi^\nu_2}\psi^\nu_1\right\|_{L^p(\Sigma)}+\varepsilon_\nu\left\|\nabla_t^{\Psi^\nu_2}\psi^\nu_1\right\|_{L^p}\leq c\varepsilon_\nu^{1-\frac 1p},$$
$$\left\|\partial_sA_2^\nu-d_{A_2^\nu}\Phi_2^\nu\right\|_{L^p(\Sigma)}\leq c e^{\delta (\theta(-s)-\theta(s))s},\quad\left\|\partial_tA_2^\nu-d_{A_2^\nu}\Psi_2^\nu\right\|_{L^p(\Sigma)}\leq c,$$

for any $0<\varepsilon_\nu<\varepsilon_0$.
\end{lemma}

\begin{proof}[Proof of lemma \ref{flow:lemmasurj1}]
The first estimate of the lemma can be obtained deriving by $\nabla_t^{\Psi^\nu_1}$ the identity
\begin{equation}\label{id:ndsao}
d_{A_1^\nu}*d_{A_1^\nu}\gamma^\nu=F_{A_1^\nu}-\frac 12 \left[*d_{A_1^\nu}\gamma^\nu\wedge *d_{A_1^\nu}\gamma^\nu\right];
\end{equation}
Then, by the commutation formula \ref{commform} and the theorem \ref{flow:thm:apriori} 
\begin{equation}\label{lemma79:1}
\left \| d_{A_1^\nu}*d_{A_1^\nu}\nabla_t^{\Psi_1^\nu}\gamma^\nu\right \|_{L^2(\Sigma)}
\leq c\left(\left\| \nabla_t^{\Psi_1^\nu}F_{A_1^\nu}\right\|_{L^2(\Sigma)}+\left\| \alpha_1^\nu \right\|_{L^\infty(\Sigma)}\right)\leq c\varepsilon^{1-\frac 1p}_\nu
\end{equation}
or by the theorem \ref{flow:thm:apriori22} for $|s|\geq2$
\begin{equation}\label{lemma79:2}
\begin{split}
\left \| d_{A_1^\nu}*d_{A_1^\nu}\nabla_t^{\Psi_1^\nu}\gamma^\nu\right \|_{L^2(\Sigma)}
\leq &c\left(\left\| \nabla_t^{\Psi_1^\nu}F_{A_1^\nu}\right\|_{L^2(\Sigma)}+\left\| \alpha_1^\nu \right\|_{L^\infty(\Sigma)}\right)\\
\leq & c\left\| \nabla_t^{\Psi_1^\nu}\left(F_{A^\nu}-F_{A^\nu_+}\right)\right\|_{L^2(\Sigma)}\\
&+c\left(\left\|\nabla_t^{\Psi_1^\nu}[\alpha\wedge \alpha_0^\nu] \right\|_{L^2(\Sigma)}+\left\| \alpha_1^\nu \right\|_{L^\infty(\Sigma)}\right)\\
\leq & ce^{(\theta(-s)-\theta(s))\delta s}\varepsilon^{1-\frac 1p}_\nu,
\end{split}
\end{equation}
where for the second estimate we used (\ref{flow:est:913}) and for the last one from the theorem \ref{flow:thm:apriori22}. Analogously, if we derive (\ref{id:ndsao}) twice by $\nabla_t^{\Psi^\nu_1}$, then we obtain
\begin{equation}\label{lemma79:3}
\left \| d_{A_1^\nu}*d_{A_1^\nu}\nabla_t^{\Psi_1^\nu}\nabla_t^{\Psi_1^\nu}\gamma^\nu\right \|_{L^2(\Sigma)}
\leq   c\varepsilon^{-\frac 1p}_\nu,
\end{equation}
\begin{equation}\label{lemma79:4}
\left \| d_{A_1^\nu}*d_{A_1^\nu}\nabla_t^{\Psi_1^\nu}\nabla_t^{\Psi_1^\nu}\gamma^\nu\right \|_{L^2(\Sigma)}
\leq   ce^{(\theta(-s)-\theta(s))\delta s}\varepsilon^{-\frac 1p}_\nu.
\end{equation}
In fact since  
 $$\left\| \nabla_t^{\Psi_1^\nu}\nabla_t^{\Psi_1^\nu}F_{A_1^\nu}\right\|_{L^2(\Sigma)}=\left\| \nabla_t^{\Psi_1^\nu}\nabla_t^{\Psi_1^\nu}\left(F_{A^\nu}-F_{A^\nu_+}-[\alpha^\nu\wedge \alpha_0^\nu]\right)\right\|_{L^2(\Sigma)},$$
 $$\left\| \nabla_t^{\Psi_1^\nu}\nabla_t^{\Psi_1^\nu}F_{A_1^\nu}\right\|_{L^2(\Sigma)}\leq c$$ 
 by theorem \ref{flow:thm:apriori} and by the estimates on the $1$-forms $\alpha_0+\psi_0 dt$ and $\alpha_1$ and if we consider the estimates of the theorem \ref{flow:thm:apriori22} we get 
 $$\left\| \nabla_t^{\Psi_1^\nu}\nabla_t^{\Psi_1^\nu}F_{A_1^\nu}\right\|_{L^2(\Sigma)}\leq ce^{(\theta(-s)-\theta(s))\delta s}\varepsilon_\nu^{-\frac 1p}.$$
Thus , we can also conclude that by (\ref{lemma79:1}) and (\ref{lemma79:3})  and the commutation formula (\ref{commform})
\begin{equation}\label{lemma79:6}
\varepsilon_\nu\left\|\nabla_t^{\Psi_1^\nu}\alpha_1^\nu\right\|_{L^p(\Sigma)}+\varepsilon_\nu^2\left\|\nabla_t^{\Psi_1^\nu}\nabla_t^{\Psi_1^\nu}\alpha_1^\nu\right\|_{L^p(\Sigma)}\leq c\varepsilon_\nu^{2-\frac 1p}
\end{equation} 
and  by (\ref{lemma79:2}) and (\ref{lemma79:4})  and the commutation formula (\ref{commform})
$$\varepsilon_\nu\left\|\nabla_t^{\Psi_1^\nu}\alpha_1^\nu\right\|_{L^p(\Sigma)}+\varepsilon_\nu^2\left\|\nabla_t^{\Psi_1^\nu}\alpha_1^\nu\right\|_{L^p(\Sigma)}\leq ce^{(\theta(-s)-\theta(s))\delta s}\varepsilon^{2-\frac 1p}_\nu. $$ 
Next, we can derive
\begin{equation*}
F_{A^\nu}+d_{A^\nu}\left(*d_{A^\nu_1}\gamma^\nu-\alpha_0^\nu\right)
+\frac12\left[\left(*d_{A^\nu_1}\gamma^\nu-\alpha_0^\nu\right)\wedge\left(*d_{A^\nu_1}\gamma^\nu-\alpha_0^\nu\right)\right]=0
\end{equation*}
by $\nabla_s$ and we obtain
\begin{equation*}
\begin{split}
d_{A^\nu}*d_{A^\nu}\nabla_s^{\Psi^\nu}\gamma^\nu=&
-\nabla_s^{\Psi^\nu}F_{A^\nu}-d_{A^\nu}\nabla_s^{\Psi^\nu}*\left[\alpha_0^\nu,\gamma^\nu\right]\\
&+\left[\left(\partial_sA^\nu-d_{A^\nu}\Phi^\nu\right),\alpha_0^\nu\right]\\
&-\left[\left(\partial_s A^\nu-d_{A^\nu}\Phi^\nu\right)\wedge * d_{A^\nu}\gamma^\nu\right]
-d_{A^\nu}*\left[\left(\partial_s A^\nu-d_{A^\nu}\Phi^\nu\right),\gamma^\nu\right]\\
&-\left[*\left(\nabla_s^{\Psi^\nu}d_{A^\nu_1}\gamma^\nu\right)\wedge \left(*d_{A^\nu_1}\gamma^\nu-\alpha_0^\nu\right)\right]
\end{split}
\end{equation*}
and thus we can conclude that
\begin{equation*}
\left\|d_{A^\nu}*d_{A^\nu}\nabla_s^{\Psi^\nu}\gamma^\nu\right\|_{L^2(\Sigma)}
\leq c \left\|\nabla_s F_{A^\nu}\right\|_{L^2(\Sigma)}+c \varepsilon_\nu^{2-\frac 1p}e^{\delta s}\leq c \varepsilon_\nu e^{\theta(-s)\delta s}e^{-\theta(s)\delta s}
\end{equation*}
by theorem \ref{flow:thm:apriori22} and hence $\|\nabla_s\alpha_1^\nu\|_{L^p}\leq c\varepsilon_\nu$. In order to prove the other estimates we need to use the definitions of $\Psi_2$ and $\Phi_2$ and to expand the identity. On the one hand,
\begin{equation}\label{flow:dsoaaaa}
\begin{split}
0=& d_{A^\nu_2}^*\left(\partial_t A^\nu_2-d_{A^\nu_2}\Psi^\nu_2\right)\\
=&-d_{A^\nu_2}^*d_{A^\nu_2}\psi_1^\nu
-*\left[\alpha_1^\nu\wedge * \left(\partial_t A^\nu_1-d_{A^\nu_1}\Psi^\nu_1\right)\right]\\
&+d_{A^\nu_2}^*\nabla_t^{\Psi^\nu_1}\alpha^\nu_1+d_{A^\nu_1}^*\left(\partial_t A_1^\nu-d_{A^\nu_1}\Psi^\nu_1\right);
\end{split}
\end{equation}
where the last term can be written in the following way:
\begin{align*}
d_{A^\nu_1}^*&\left(\partial_t A^\nu_1-d_{A^\nu_1}\Psi^\nu_1\right)\\
=&
d_{A^\nu}^*\left(\partial_t A^\nu-d_{A^\nu}\Psi^\nu\right)+d_{A^\nu_1}^*\left(\nabla_t^{\Psi^\nu}\alpha_0^\nu-d_{A^\nu}\psi_0^\nu-[\alpha_0^\nu,\psi_0^\nu]\right)\\
&-*\left[\alpha_0^\nu\wedge *\left(\partial_t A^\nu-d_{A^\nu}\Psi^\nu\right)\right]\\
=&d_{A^\nu}^*\left(\partial_t A^\nu-d_{A^\nu}\Psi^\nu\right)
-*\left[\alpha^\nu\wedge *\left(\nabla_t^{\Psi^\nu_-}\alpha_0^\nu-d_{A^\nu_-}\psi_0^\nu-[\alpha_0^\nu,\psi_0^\nu]\right)\right]\\
&+d_{A^\nu_-}^*\left(-\nabla_t^{\Psi^\nu_-}\alpha_0^\nu+d_{A^\nu_-}\psi_0^\nu-[\alpha_0^\nu,\psi_0^\nu]\right)+*\left[\alpha_0^\nu\wedge*\left(\partial_t A^\nu_--d_{A^\nu_-}\Psi^\nu_-\right)\right]\\
&+*\left[\alpha_0^\nu\wedge *\left(-\nabla_t^{\Psi^\nu_-}\alpha_0^\nu+d_{A^\nu_-}\psi_0^\nu-[\alpha_0^\nu,\psi_0^\nu]\right)\right]\\
&+d_{A^\nu_1}^*\left(-[\psi^\nu,\alpha_0^\nu]+[\alpha^\nu,\psi_0^\nu]\right)\\
&+*\left[\alpha_0^\nu\wedge *\left(\left(\partial_t A^\nu-d_{A^\nu}\Psi^\nu\right)-\left(\partial_t A^\nu_--d_{A^\nu_-}\Psi^\nu_-\right)\right)\right]
\end{align*}
where the second and the third line of the last expression vanish because they can be written as 
$$d_{A^\nu_-}^*\left(\partial_t A^\nu_--d_{A^\nu_-}\Psi^\nu_-\right)
-d_{A_-}^*\left(\partial_t A_--d_{A_-}\Psi_-\right)=0$$
by the condition for the perturbed geodesics (\ref{cond:eqg0}) and the equation for the perturbed Yang-Mills connections (\ref{intro:YMeq2}). Thus
\begin{align*}
d_{A^\nu_1}^*&\left(\partial_t A^\nu_1-d_{A^\nu_1}\Psi^\nu_1\right)\\
=& d_{A^\nu}^*\left(\partial_tA^\nu-d_{A^\nu}\Psi^\nu\right)-*\left[\alpha^\nu\wedge *\left(\nabla_t^{\Psi^\nu_-}\alpha_0^\nu-d_{A^\nu_-}\psi_0^\nu-[\alpha_0^\nu,\psi_0^\nu]\right)\right]\\
&+d_{A^\nu_1}^*\left(-[\psi^\nu,\alpha_0^\nu]+[\alpha^\nu,\psi_0^\nu]\right)\\
&+*\left[\alpha_0^\nu\wedge *\left(\left(\partial_t A^\nu-d_{A^\nu}\Psi^\nu\right)-\left(\partial_t A^\nu_--d_{A^\nu_-}\Psi^\nu_-\right)\right)\right]\\
=:& d_{A^\nu}^*\left(\partial_tA^\nu-d_{A^\nu}\Psi^\nu\right)+D^\nu
\end{align*}
and
\begin{equation}\label{flow:clsdhf}
\left\|d_{A^\nu_1}^*\left(\partial_t A^\nu_1-d_{A^\nu_1}\Psi^\nu_1\right)\right\|_{L^p}\leq c\varepsilon_\nu^{1-\frac 1p}.
\end{equation}
Furthermore for the term $d_{A^\nu_2}^*\nabla_t^{\Psi^\nu_1}\alpha^\nu_1$ we have
\begin{align*}
d_{A^\nu_2}^*\nabla_t^{\Psi^\nu_1}\alpha^\nu_1=&
\nabla_t^{\Psi^\nu}d_{A^\nu}^**d_{A^\nu}\gamma^\nu-*\left[\alpha_1^\nu\wedge *\nabla_t^{\Psi^\nu_1}\alpha_1^\nu\right]-*\left[\left(\partial_t A^\nu_1-d_{A^\nu_1}\Psi^\nu_1\right)\wedge *\alpha_1^\nu\right]\\
=&
*[F_{A^\nu},\nabla_t^{\Psi^\nu}\gamma^\nu]+*[\nabla_t^{\Psi^\nu}F_{A^\nu},\gamma^\nu]-*\left[\alpha_1^\nu\wedge *\nabla_t^{\Psi^\nu_1}\alpha_1^\nu\right]\\
&-*\left[\left(\partial_t A^\nu_1-d_{A^\nu_1}\Psi^\nu_1\right)\wedge *\alpha_1^\nu\right].
\end{align*}
Therefore, estimating term by term (\ref{flow:dsoaaaa}) we obtain
\begin{equation}\label{lemma79:5}
\left\| d_{A_1^\nu}^*d_{A_1^\nu}\psi_1^\nu \right\|_{L^p(\Sigma)}\leq c\varepsilon_\nu^{1-\frac 1p} \quad \textrm{or}\quad \left\| d_{A_1^\nu}^*d_{A_1^\nu}\psi_1^\nu \right\|_{L^p(\Sigma)}\leq ce^{(\theta(-s)-\theta(s))\delta s}\varepsilon_\nu^{1-\frac 1p}
\end{equation}
and hence by (\ref{flow:est:surj:bo2}), (\ref{flow:est:914}), (\ref{lemma79:6}) and (\ref{lemma79:5})
$$\|\partial_t A_2^\nu-d_{A_2^\nu}\Psi_2^\nu\|_{L^p(\Sigma)}\leq c.$$
Analogously, deriving $d_{A_2^\nu}^*\left(\partial_tA_2^\nu-d_{A_2^\nu}\Psi_2^\nu\right)$ by $\nabla_t^{\Psi^\nu_2}$ we can obtain $\varepsilon_\nu^2\left\|\nabla_t^{\Psi^\nu_2}\psi^\nu\right\|_{L^p}\leq \varepsilon^{2-\frac 1p}_\nu$ using 
\begin{align*}
\left\|\nabla_t^{\Psi^\nu}d_{A^\nu}^*\left(\partial_t A^\nu-d_{A^\nu}\Psi^\nu\right)\right\|_{L^p(\Sigma)}\leq& \left\|d_{A^\nu}^*\nabla_t^{\Psi^\nu}\left(\partial_t A^\nu-d_{A^\nu}\Psi^\nu\right)\right\|_{L^p(\Sigma)}\\
= &\left\|d_A^*\left(\partial_s A^\nu-d_{A^\nu}\Phi^\nu\right)\right\|_{L^p(\Sigma)}\\
\leq& ce^{(\theta(-s)-\theta(s))\delta s}\varepsilon_\nu
\end{align*}
by the commutation formula (\ref{commform2}), by the identities
$$\left[\left(\partial_s A^\nu-d_{A^\nu}\Phi^\nu\right)\wedge*\left(\partial_s A^\nu-d_{A^\nu}\Phi^\nu\right)\right]=0, \quad \left[F_{ A^\nu},*F_{A^\nu}\right]=0,$$
by the Yang-Mills flow equation (\ref{flow:eq}) and by the theorem \ref{flow:thm:apriori22}. On the other hand, since $\Xi^\nu$ is a Yang-Mills flow
and $d_{A_\nu^2}^*\left(\partial_s A^\nu_2-d_{A^\nu_2}\Phi^\nu_2\right)=0$, we can estimate
$$\left\|d_{A_\nu}^*\left(\partial_s A^\nu-d_{A^\nu}\Phi^\nu\right)-d_{A_\nu^2}^*\left(\partial_s A^\nu_2-d_{A^\nu_2}\Phi^\nu_2\right)\right\|_{L^p(\Sigma)}\leq ce^{(\theta(-s)-\theta(s))\delta s}\varepsilon_\nu$$
by the theorem \ref{flow:thm:apriori22} and the estimates computed so far; thus
\begin{equation*}
\begin{split}
d_{A^\nu_2}^*d_{A^\nu_2}\phi_1^\nu= &-d_{A_\nu}^*\left(\partial_s A^\nu-d_{A^\nu}\Phi^\nu\right)+d_{A_\nu^2}^*\left(\partial_s A^\nu_2-d_{A^\nu_2}\Phi^\nu_2\right)\\
&-*\left[\left(\alpha_0^\nu+\alpha_1^\nu\right)\wedge* \left(\partial_s A^\nu-d_{A^\nu}\Phi^\nu\right)\right]+d_{A^\nu_2}^*\nabla_s^{\Phi^\nu}\alpha_1^\nu\\
=&-d_{A_\nu}^*\left(\partial_s A^\nu-d_{A^\nu}\Phi^\nu\right)\\
&-*\left[\left(\alpha_0^\nu+\alpha_1^\nu\right)\wedge* \left(\partial_s A^\nu-d_{A^\nu}\Phi^\nu\right)\right]\\
&+d_{A^\nu_1}^*[\psi_0,\alpha_1]+\left[\alpha_1^\nu\wedge*\nabla_s^{\Phi^\nu}\alpha_1^\nu\right]\\
&+\nabla_s^{\Phi^\nu_1}*\left[F_{A_1^\nu},\gamma^\nu\right]
+*\left[\left(\partial_sA_1^\nu-d_{A_1^\nu}\Phi_1^\nu\right)\wedge*\alpha_1^\nu\right]
\end{split}
\end{equation*}
and this implies that $\|d_A\phi^\nu_1\|_{L^p}\leq c\varepsilon^{1-\frac 1p}$, $\|d_A\phi^\nu_1\|_{L^p}\leq ce^{(\theta(-s)-\theta(s))\delta s}\varepsilon^{1-\frac 1p}$ for $|s|\geq2$ and $\|\partial_sA_2^\nu-d_{A_2^\nu}\Phi_2^\nu\|_{L^p(\Sigma)}\leq c e^{(\theta(-s)-\theta(s))\delta s}$. 
\end{proof}

\noindent Weber proved the following theorem (cf. \cite{Weberhab}, theorem 1.12)

\begin{theorem}[Implicit function theorem]\label{flow:thm:implfunctthm}
Fix a perturbation $H:\mathcal{LM} \to \mathbb R$ that satisfies (\ref{flow:intro:cond}). Assume $E^H$ is Morse and that $\mathcal D_u^0$ is onto for every $u\in \mathcal M^0(x_-,x_+; H)$ and every pair $x_\pm\in \mathrm{Crit}_{E^H}^b$. Fix two critical points $x_\pm\in  \mathrm{Crit}_{E^H}^b$ with Morse index difference one. Then, for all $c_0>0$ and $p>2$, there exist positive constants $\delta_0$ and $c$ such that the following holds. If $u:\mathbb R\times S^1\to M$ is a smooth map such that $\lim_{s\to\pm\infty}u(s,\cdot)=x_\pm(\cdot)$ exists, uniformly in $t$, and such that
\begin{equation}\label{flow:ift:eq1}
|\partial_s u(s,t)|\leq \frac{c_0}{1+s^2},\quad |\partial_t u(s,t)|\leq c_0
,\quad |\nabla_t\partial_t u(s,t)|\leq c_0
\end{equation}
for all $(s,t)\in \mathbb R\times S^1$ and
\begin{equation}
\left\|\partial_s u-\nabla_t\partial_t u-\mathrm{grad}H(u)\right\|_{L^p}\leq \delta_0.
\end{equation}
Then there exist elements $u_0\in \mathcal M^0(x_-,x_+;H)$ and $\xi\in \textrm{im }(\mathcal D_{u_0}^0)^*\cap \mathcal W_{u_0}$ satisfying\footnote{The norm $ \|\xi\|_{\mathcal W_{u_0}}$ sums the the norms $ \|\xi\|_{L^p}$, $\|\nabla_t\xi\|_{L^p}$,$ \|\nabla_t\nabla_t\xi\|_{L^p}$,$ \|\nabla_s\xi\|_{L^p}$ and the space $\mathcal W_{u_0}$ is completion of the space of smooth, compactly supported, vector fields along $u_0$ respect to the norm $\|\cdot\|_{\mathcal W_{u_0}}$.}

\begin{equation}
u=\exp_{u_0}(\xi),\quad \|\xi\|_{\mathcal W_{u_0}}\leq c\|\partial_s u-\nabla_t\partial_t u-\mathrm{grad} H(u)\|_{L^p}.
\end{equation}

\end{theorem}

\begin{remark}
The third condition of  (\ref{flow:ift:eq1}) follows from the first one and, for a positive constant $c_1$,
$$|\partial_s u-\nabla_t\partial_t u-\mathrm{grad} H(u)|\leq c_1;$$
therefore, in our case all the assumptions are satisfied by the second step and by the lemma \ref{flow:lemmasurj1}.
\end{remark}

\noindent {\bf Step 3.} We choose $p>4$. There are $\varepsilon_0, c>0$ such that the following holds. If $\varepsilon_\nu<\varepsilon_0$, then there is a smooth map 
$A^\nu_3:\mathbb R^2\to \mathcal A_0(P)$ such that $[A^\nu_3]\in \mathcal M^0(\Xi_-,\Xi_+)$,
\begin{equation}\label{flow:surj:step3:cond11}
d_{A^\nu_3}^*\left( A_3^\nu-A_2^\nu\right)=0,
\end{equation}
\begin{equation}
\left\|A^\nu_3-A^\nu_2\right\|_{L^p}+\left\|A^\nu_3-A^\nu_2\right\|_{L^\infty}\leq c\varepsilon_\nu^{1-\frac1p}
\end{equation}
\begin{equation}
\left\|\left(\partial_t A^\nu_3-d_{A^\nu_3}\Psi_3^\nu\right)-\left(\partial_t A^\nu_2-d_{A_2^\nu }\Psi_2^\nu\right)\right\|_{L^p}\leq c\varepsilon_\nu^{1-\frac1p }
\end{equation}
\begin{equation}
\left\|\left(\partial_s A^\nu_3-d_{A^\nu_3}\Phi_3^\nu\right)-\left(\partial_s A^\nu_2-d_{A_2^\nu }\Phi_2^\nu\right)\right\|_{L^p}\leq c\varepsilon_\nu^{1-\frac1p }
\end{equation}
where $\Psi_3^\nu$ and $\Phi_3^\nu$ are defined uniquely by $$d_{A^\nu_3}^*\left(\partial_t A^\nu_3-d_{A^\nu_3}\Psi_3^\nu\right)=0\textrm{ and }d_{A^\nu_3}^*\left(\partial_s A^\nu_3-d_{A^\nu_3}\Phi_3^\nu\right)=0.$$

\begin{proof}[Proof of step 3.]
The third step follows directly from the theorem \ref{flow:thm:implfunctthm}. The condition (\ref{flow:surj:step3:cond11}) can be reached using the local slice theorem (theorem 8.1 in \cite{MR2030823}).
\end{proof}


\noindent {\bf Step 4.} We choose $p>4$. There are $\varepsilon_0, c>0$ such that the following holds. If $\varepsilon_\nu<\varepsilon_0$, then there is a smooth map 
$A^\nu_4:\mathbb R^2\to \mathcal A_0(P)$ such that $[A^\nu_4]\in \mathcal M^0(\Xi_-,\Xi_+)$,
\begin{equation}\label{flow:surj:step4:cond1}
d_{A^\nu_4}^*\left( A_4^\nu-A^\nu_1\right)=0,
\end{equation}
\begin{equation}
\left\|A^\nu_4-A^\nu_1\right\|_{L^p}+\left\|A^\nu_4-A^\nu_1\right\|_{L^\infty}\leq c\varepsilon_\nu^{1-\frac1p}
\end{equation}
\begin{equation}\label{flow:surj:step4:eq2d}
\left\|\left(\partial_t A^\nu_4-d_{A^\nu_4}\Psi_4^\nu\right)-\left(\partial_t A^\nu_1-d_{A_1^\nu }\Psi_1^\nu\right)\right\|_{L^p}\leq c\varepsilon_\nu^{1-\frac1p }
\end{equation}
\begin{equation}\label{flow:surj:step4:eq2}
\left\|\left(\partial_s A^\nu_4-d_{A^\nu_4}\Phi_4^\nu\right)-\left(\partial_s A^\nu_1-d_{A_1^\nu }\Phi_1^\nu\right)\right\|_{L^p}\leq c\varepsilon_\nu^{1-\frac1p }
\end{equation}
where $\Psi_4^\nu$ and $\Phi_4^\nu$ are defined uniquely by 
$$d_{A^\nu_4}^*\left(\partial_t A^\nu_4-d_{A^\nu_4}\Psi_4^\nu\right)=0\textrm{ and }d_{A^\nu_4}^*\left(\partial_s A^\nu_4-d_{A^\nu_4}\Phi_4^\nu\right)=0.$$

\begin{proof}[Proof of step 4.]
By the previous two steps and the lemma \ref{flow:lemmasurj1} we can conclude that
\begin{equation*}
\left\|A^\nu_3-A^\nu_1\right\|_{L^p}+\left\|A^\nu_3-A^\nu_1\right\|_{L^\infty}\leq c\varepsilon_\nu^{1-\frac1p},
\end{equation*}
\begin{equation*}
\left\|\left(\partial_s A^\nu_3-d_{A^\nu_3}\Phi_3^\nu\right)-\left(\partial_s A^\nu_1-d_{A_1^\nu }\Phi_1^\nu\right)\right\|_{L^p}\leq c\varepsilon_\nu^{1-\frac1p }.
\end{equation*}
Since
\begin{equation*}
\begin{split}
d_{A^\nu_3}^*\left( A_3^\nu-A^\nu_1\right)=&d_{A^\nu_3}^*\left( A_3^\nu-A^\nu_2\right)+d_{A^\nu_3}^*\alpha_1^\nu\\
=&*d_{A^\nu}d_{A^\nu }\gamma^\nu-*[(A^\nu_3-A^\nu_2)\wedge \alpha_1^\nu],
\end{split}
\end{equation*}
\begin{equation*}
d_{A^\nu}\left(A^\nu_1-A^\nu_2\right)=F_{A_1^\nu}-\frac 12\left[\left(A^\nu_1-A_3^\nu\right)\wedge\left(A^\nu_1-A^\nu_3\right)\right]
\end{equation*}
hold, we obtain
\begin{equation*}
\left\|d_{A^\nu_3}^*\left( A_3^\nu-A^\nu_1\right)\right\|_{L^p(\Sigma)}+\varepsilon_\nu\left\|d_{A^\nu_3}\left( A_3^\nu-A^\nu_1\right)\right\|_{L^p(\Sigma)}\leq c \varepsilon_\nu^{3-\frac2p }.
\end{equation*}
Thus, by the local gauge theorem there are maps $g_\nu:\mathbb R^2\to \mathcal G_0^{2,p}(P)$ such that
\begin{equation*}
d_{A^\nu_1}^*\left( g_\nu^*A_3^\nu-A^\nu_1\right)=0,\quad
\left\|g_\nu^*A_3^\nu-A^\nu_1\right\|_{W^{1,p}(\Sigma) }\leq c\left\|A_3^\nu-A^\nu_1\right\|_{W^{1,p}(\Sigma) };
\end{equation*}
then we conclude the proof of the fourth step defining $A^\nu_4:=g_\nu^*A^\nu_3$.
\end{proof}

\noindent {\bf Step 5.} For two positive constants $c$, $\varepsilon_0$, $0<\varepsilon<\varepsilon_0$, $\Xi^\nu_4:=A^\nu_4+\Psi^\nu_4dt+\Phi^\nu_4ds\in \mathcal M^0\left(\Xi_-,\Xi_+\right)$ satisfies
\begin{equation}\label{setpdsp0}
\left\|\left(1-\pi_{A_4^\nu}\right)\left(\Xi^\nu_1-\Xi^\nu_4\right)\right\|_{\Xi_4^\nu,1,p,\varepsilon_\nu}+\varepsilon_\nu\|d_{A_4^\nu}\nabla_t^{\Psi_4^\nu}(\Psi_1^\nu-\Psi_4^\nu)\|\leq c \varepsilon_\nu^{2-\frac 2p},
\end{equation}
\begin{equation}
\left\|\pi_{A_4^\nu}\left(A^\nu_1-A^\nu_4\right)\right\|_{\Xi_4^\nu,1,p,1}\leq c \varepsilon_\nu^{1-\frac 1p}.
\end{equation}

\begin{proof}[Proof of step 5.] Since $d_{A_4^\nu}^*(A_1^\nu-A_4^\nu)=0$ and 
\begin{equation*}
d_{A^\nu_4}\left(A^\nu_1-A^\nu_4\right)=F_{A^\nu_1}-\frac 12 \left[\left(A^\nu_1-A^\nu_4\right)\wedge \left(A^\nu_1-A^\nu_4\right)\right],
\end{equation*}
by lemma \ref{lemma76dt94}
\begin{equation*}
\left |\left|(1-\pi_{A_4^\nu})\left(A^\nu_1-A^\nu_4\right) \right|\right | _{L^p}+\left |\left|d_{A^\nu_4}\left(A^\nu_1-A^\nu_4\right) \right|\right | _{L^p}\leq c\varepsilon_\nu^{2-\frac 2p }.
\end{equation*}
By (\ref{flow:surj:step4:cond1}) we have
\begin{equation*}
\begin{split}
d_{A_4^\nu}^*\nabla_t^{\Psi_4^\nu}\left(A^\nu_1-A^\nu_4\right)=&*\left[\left(\partial_t A_4^\nu-d_{A^\nu_4}\Psi^\nu_4\right)\wedge * \left(A^\nu_1-A^\nu_4\right)\right]\\
d_{A_4^\nu}^*\nabla_s^{\Phi_4^\nu}\left(A^\nu_1-A^\nu_4\right)=&*\left[\left(\partial_s A_4^\nu-d_{A^\nu_4}\Psi^\nu_4\right)\wedge * \left(A^\nu_1-A^\nu_4\right)\right],
\end{split}
\end{equation*}
and by the properties and definitions of the connections
$$
d_{A_4^\nu}^*\left(\partial_t A_{4}^\nu-d_{A_4^\nu}\Psi_4^\nu\right)=0, \quad d_{A_4^\nu}^*\left(\partial_s A_{4}^\nu-d_{A_4^\nu}\Phi_4^\nu\right)=0, $$
$$d_{A_1^\nu}^*\left(\partial_t A_{1}^\nu-d_{A_1^\nu}\Psi_1^\nu\right)=d_{A^\nu}^*\left(\partial_tA^\nu+d_{A^\nu}{\Psi^\nu}\right)+D^\nu,$$
$$d_{A_1^\nu}^*\left(\partial_s A_{1}^\nu-d_{A_1^\nu}\Phi_1^\nu\right)=d_{A^\nu}^*\left(\partial_sA^\nu+d_{A^\nu}{\Phi^\nu}\right)-\left[\alpha_0^\nu,\left(\partial_s A^\nu-d_{A^\nu}\Phi^\nu\right)\right],$$
where $D^\nu$ is defined in the proof of the lemma \ref{flow:lemmasurj1}. Hence we have
\begin{equation}\label{flow:surj:step5:es3}
\begin{split}
d_{A_4^\nu}^*d_{A_4^\nu}\left(\Psi_4^\nu-\Psi_1^\nu\right)
=&d_{A_1^\nu}^*\left(\partial_t A_{1}^\nu-d_{A_1^\nu}\Psi_1^\nu\right)\\
&+*\left[\left(A^\nu_1-A^\nu_4\right)\wedge *\left(\partial_t A_{1}^\nu-d_{A_1^\nu}\Psi_1^\nu\right)\right]\\
&-d_{A_4^\nu}^*\left(\nabla_t^{\Psi_4^\nu} \left(A^\nu_1-A^\nu_4\right)-\left[\left(A^\nu_1-A^\nu_4\right),\left(\Psi_1^\nu-\Psi_4^\nu\right)\right]\right)\\
=&d_{A_1^\nu}^*\left(\partial_t A_{1}^\nu-d_{A_1^\nu}\Psi_1^\nu\right)\\
&+*\left[\left(A^\nu_1-A^\nu_4\right)\wedge *\left(\partial_t A_{1}^\nu-d_{A_1^\nu}\Psi_1^\nu\right)\right]\\
&-*\left[\left(\partial_t A_4^\nu-d_{A^\nu_4}\Psi^\nu_4\right)\wedge * \left(A^\nu_1-A^\nu_4\right)\right]\\
&+*\left[*\left(A^\nu_1-A^\nu_4\right)\wedge d_{A_4^\nu}\left(\Psi_1^\nu-\Psi_4^\nu\right)\right],
\end{split}
\end{equation}
\begin{equation}\label{flow:surj:step5:esfgd}
\begin{split}
d_{A_4^\nu}^*d_{A_4^\nu}\left(\Phi_4^\nu-\Phi_1^\nu\right)
=&d_{A_1^\nu}^*\left(\partial_s A_{1}^\nu-d_{A_1^\nu}\Phi_1^\nu\right)\\
&+*\left[\left(A^\nu_1-A^\nu_4\right)\wedge *\left(\partial_s A^\nu-d_{A^\nu}\Phi^\nu\right)\right]\\
&-*\left[\left(\partial_s A_4^\nu-d_{A^\nu_4}\Phi^\nu_4\right)\wedge * \left(A^\nu_1-A^\nu_4\right)\right]\\
&+*\left[*\left(A^\nu_1-A^\nu_4\right)\wedge d_{A_4^\nu}\left(\Phi_1^\nu-\Phi_4^\nu\right)\right]
\end{split}
\end{equation}
and thus by the first step, (\ref{flow:clsdhf}), the a priori estimates for a geodesics flow (\ref{flow:web:w1})-(\ref{flow:web:w1}) and the a priori estimates of the theorem \ref{flow:thm:apriori22}, we obtain
\begin{equation*}
\left\| d_{A_4^\nu}^*d_{A_4^\nu}\left(\Psi_4^\nu-\Psi_1^\nu\right)\right\|_{L^p}+\left\| d_{A_4^\nu}^*d_{A_4^\nu}\left(\Phi_4^\nu-\Phi_1^\nu\right)\right\|_{L^p}\leq c\varepsilon_\nu^{1-\frac 1p }
\end{equation*}
By the lemma \ref{lemma76dt94}, the estimates (\ref{flow:surj:step4:eq2d}), (\ref{flow:surj:step4:eq2}) and the triangular inequality, we have also that
$$\varepsilon_\nu\left\| \Psi_4^\nu-\Psi_1^\nu\right\|_{L^p}+\varepsilon_\nu\left\| \Phi_4^\nu-\Phi_1^\nu\right\|_{L^p}\leq c\varepsilon_\nu^{2-\frac 1p },$$
$$\varepsilon_\nu\left\| d_{A_4^\nu}\left(\Psi_4^\nu-\Psi_1^\nu\right)\right\|_{L^p}+\varepsilon_\nu\left\| d_{A_4^\nu}\left(\Phi_4^\nu-\Phi_1^\nu\right)\right\|_{L^p}\leq c\varepsilon_\nu^{2-\frac 1p },$$
$$\varepsilon_\nu\left\| \nabla_t^{\Psi_4^\nu}\left(A_4^\nu-A_1^\nu\right)\right\|_{L^p}+\varepsilon_\nu\left\| \nabla_s^{\Psi_4^\nu}\left(A_4^\nu-A_1^\nu\right)\right\|_{L^p}\leq c\varepsilon_\nu^{2-\frac 1p }.$$
Furthermore, deriving by $\nabla_s^{\Phi_4^\nu}$ and by $\nabla_t^{\Psi_4^\nu}$ the identities (\ref{flow:surj:step5:es3}) and (\ref{flow:surj:step5:esfgd}), we can obtain the other estimates needed for (\ref{setpdsp0}).
\end{proof}

\noindent {\bf Step 6.} We choose $p>10$. Then there are $\varepsilon_0, c >0$ such that the following holds. There are two sequences $g_{\nu}\in \mathcal G_0^{2,p}(P\times S^1\times \mathbb R)$ and $s_\nu\in \mathbb R$ such that $\Xi^\nu_5:=g_\nu^*\Xi_4^\nu(t,s+s_\nu)=A_5^\nu+\Psi_5^\nu dt+\Phi_5^\nu ds$ satisfy
\begin{equation}
d_{\Xi_5^\nu }^{*_{\varepsilon_\nu} }\left(\Xi^\nu-\mathcal K_2^{\varepsilon_\nu}\left(\Xi_5^\nu\right)\right)=0,\quad \Xi^\varepsilon-\mathcal K_2^{\varepsilon_\nu}(\Xi_5^\nu) \in \textrm{im } \mathcal D^{\varepsilon_\nu}\left(\mathcal K_2^{\varepsilon_\nu}\left(\Xi_5^\nu\right)\right)^*
\end{equation}
\begin{equation}
\left\|\left(1-\pi_{A_5^\nu}\right)\left(\Xi^\nu_1-\Xi^\nu_5\right)\right\|_{\Xi_5^\nu,1,p,\varepsilon_\nu}\leq c \varepsilon_\nu^{2-\frac 2p},
\end{equation}
\begin{equation}
\left\|\pi_{A_5^\nu}\left(A^\nu_1-A^\nu_5\right)\right\|_{\Xi_5^\nu,1,p,1}\leq c \varepsilon_\nu^{1-\frac 1p}.
\end{equation}

\begin{remark}
In the sixth step we use the connection $\mathcal K_2(\Xi^\nu_4)$ introduced in the section \ref{flow:section:firstapprox}; the definition of the $1$-form $\alpha_0^\varepsilon+\psi_0^\varepsilon dt$ in that section is not the same as (\ref{flow:surj:dsiabp}) even is we consider that this holds. In fact, one can replace the definition (\ref{flow:k2:eqq1}) by
{\small \begin{equation*}
\begin{split}
\alpha_0^\varepsilon(s)+\psi_0^\varepsilon(s) dt:=&\theta(-s) (h(s)g(s))^{-1}(\mathcal T^{\varepsilon,b }(A_-+\Psi_- dt)-(A_-+\Psi_- dt))h(s)g(s)\\
&+\theta(s) (h(s)g(s))^{-1}(\mathcal T^{\varepsilon,b }(A_++ \Psi_+ dt)-(A_++\Psi_+ dt))h(s)g(s),
\end{split}
\end{equation*}}
where $g(s)$ is defined as in (\ref{flow:k2:eqq2}) and $h$ by $h^{-1}\partial_sh=g(\Phi^\nu-\Phi_4^\nu)g^{-1}$. In this case, 
$(hg)^{-1}\partial_s(hg)=g^{-1}(h^{-1}\partial_sh)g+g^{-1}\partial_sg=\Psi^\nu$. With this change all the theorems proved for $\mathcal K_2^\varepsilon$ continues to hold.
\end{remark}

\begin{proof}[Proof of step 6.]
Step 5 and the theorem \ref{flow:lemma:firstappr} tells us that
\begin{equation*}
\left\| \Xi^\nu-\mathcal K_2(\Xi^\nu_4)\right\|_{\Xi_4^\nu,1,p,\varepsilon_\nu}
\leq c\varepsilon_\nu^{1-\frac3p }\leq \delta \varepsilon_\nu^{1-\frac 4p},
\end{equation*}
\begin{equation*}
\varepsilon^2\left\| \nabla_s^{\Phi_4^\nu}\left(\Xi^\nu-\mathcal K _2(\Xi^\nu_4)\right)\right\|_{0,p,\varepsilon_\nu}\leq c \varepsilon_\nu^{2-\frac 3p}\leq c\varepsilon_\nu^{1+\frac 7p }
\end{equation*}
for $c\varepsilon^{\frac 1p}\leq \delta$ where $\delta$ is given by the theorem \ref{flow:thm:timeshift}. Then by theorem \ref{flow:thm:timeshift}, there is a sequence $g_\nu\in \mathcal G^{2,p}_0(P\times S^1\times \mathbb R)$, $\sigma_\nu\in \mathbb R$ such that
\begin{equation*}
d_{\Xi_4^\nu }^{*_\varepsilon}(g_\nu^*(\Xi^\nu\circ\rho_{\sigma_\nu})-\mathcal K_2(\Xi^\nu_4))=0,
\end{equation*}
\begin{equation*}
g_\nu^*(\Xi^\nu\circ\rho_{\sigma_\nu})-\mathcal K_2(\Xi^\nu_4) \in \textrm{im }\left(\mathcal D^\varepsilon(\mathcal K_2(\Xi_4^\nu))\right)^*.
\end{equation*}
We define $\Xi_5^\nu$ by $(g_{\nu}^{-1})^*\Xi_4^\nu\circ \rho_{-\sigma_\nu}$. By the step 5, the theorem \ref{flow:thm:timeshift} and the triangular inequality we have
\begin{equation*}
\left\|\Xi^\nu_1-\Xi^\nu_5\right\|_{\Xi_5^\nu,1,p,\varepsilon_\nu}\leq c \varepsilon_\nu^{1-\frac 2p},
\end{equation*}

\begin{equation*}
\varepsilon_\nu\left\|\nabla_t^{\Psi_5^\nu}(A^\nu_1-A^\nu_5)\right\|_{L^p}+ \varepsilon_\nu^2\left \|\nabla_t^{\Psi_5^\nu}\left(\Psi^\nu_1-\Psi_5^\nu\right)\right \|_{L^p}+\varepsilon_\nu^3\left \|\nabla_t^{\Psi_5^\nu}\left(\Phi^\nu_1-\Phi_5^\nu\right)\right\|_{L^p}\leq c \varepsilon_\nu^{2-\frac 2p},
\end{equation*}

\begin{equation*}
\varepsilon_\nu^3\left \|\nabla_t^{\Psi_5^\nu}\left(\Phi^\nu_1-\Phi_5^\nu\right)\right\|_{L^p}+\varepsilon_\nu^2\left\|\nabla_s^{\Phi_5^\nu}(A^\nu_1-A^\nu_5)\right\|_{L^p}\leq c \varepsilon_\nu^{2-\frac 2p},
\end{equation*}
\begin{equation*}
 \varepsilon_\nu^3\left \|\nabla_s^{\Phi_5^\nu}\left(\Psi^\nu_1-\Psi_5^\nu\right)\right|\|_{L^p}+\varepsilon_\nu^4\left\|\nabla_s^{\Phi_5^\nu}\left(\Phi^\nu_1-\Phi_5^\nu\right)\right\|_{L^p}\leq c \varepsilon_\nu^{2-\frac 2p},
\end{equation*}

\begin{equation*}
\left\|\pi_{A_5^\nu}\left(\Xi^\nu_1-\Xi^\nu_5\right)\right\|_{\Xi_5^\nu,1,p,1}\leq c \varepsilon_\nu^{1-\frac 1p};
\end{equation*}
in order to improve the estimates for the non-harmonic part, we use the identity $$d_{\Xi_5^\nu }^{*_{\varepsilon_\nu} }\left(\Xi_1^\nu-\mathcal K_2^{\varepsilon_\nu}\left(\Xi_5^\nu\right)\right)=0,$$
i.e. if we define $\Xi_1^\nu-\Xi_5^\nu=:\bar\alpha^\nu+\bar\psi^\nu dt+ \bar\psi^\nu ds$ and  $\Xi^\nu-\mathcal K_2^{\varepsilon_\nu}\left(\Xi_5^\nu\right)=:\alpha^\nu+\psi^\nu dt+ \psi^\nu ds $, by the definition of $\mathcal K_2^{\varepsilon_\nu}$
$$\|d_{A_5^\nu}^*\bar\alpha^\nu \|_{L^p}= \|d_{A_5^\nu}^*\alpha^\nu \|_{L^p}\leq \varepsilon_\nu^2\|\nabla_t\psi^\nu\|_{L^p}+\varepsilon^4_\nu\|\nabla_s\phi^\nu\|_{L^p}\leq c\varepsilon^{2-\frac 1p}; $$
and since $F_{A_1^\nu}=d_{A_5^\nu}\bar\alpha^\nu+\frac 12 \left[\bar\alpha^\nu \wedge \bar\alpha^\nu\right]$ and 
$$\|d_{A_5^\nu}\bar\alpha\|_{L^p}\leq c\varepsilon^{2-\frac 1p}+ \|\bar\alpha^\nu\|_{L^{2p}}^2\leq c\varepsilon^{2-\frac2p}.$$
\end{proof}

\noindent{\bf Step 7.} We choose $p>13$. There are three positive constants $\delta_1, \varepsilon_0, c$ such that for any $\varepsilon_\nu<\varepsilon_0$
\begin{equation}
\left\|\pi_{A_5^\nu}\left(A^\nu_1-A^\nu_5\right)\right\|_{L^p}+\left\|\pi_{A_5^\nu}\left(A^\nu_1-A^\nu_5\right)\right\|_{L^\infty }\leq c \varepsilon_\nu^{1+\delta_1}.
\end{equation}

\noindent {\bf End of the proof: }Finally we can apply the theorem \ref{flow:thm:locuniq} choosing $\varepsilon_0$ such that $c\varepsilon_0^{\delta_1}<\delta$ for the $\delta$ needed in the theorem \ref{flow:thm:locuniq}; thus, we can conclude that for $\nu$ big enough $\Xi^\nu=\mathcal R^{\varepsilon_\nu,b}(\Xi^\nu_5)$ which is a contradiction to the fact that the $\Xi^\nu$ are not in the image of $\mathcal R^{\varepsilon_\nu,b}$. Therefore the proof of theorem \ref{flow:thm:surj} is concluded.

\begin{proof}[Proof of step 7] The idea is to consider the situation in the figure \ref{picsur1f} in order to improve the norm of $\pi_{A_5^\nu}(A^\nu-\mathcal K_2^{\varepsilon_\nu}(\Xi_5^\nu))$. In particular, we use that $\alpha^\nu_1\in \textrm{im } d_{A_2^\nu}^*$ and the fact that the norm of $\Pi_{\textrm{im } d_{A^\nu_5}^*}(\bar\alpha^\nu)$ can be estimate using the identity $d_{A_5^\nu}\bar\alpha^\nu=-\frac 12 [\bar\alpha^\nu\wedge \bar\alpha^\nu]$ deduced from $F_{A_2^\nu}=F_{A^0}+d_{A_5^\nu}\bar\alpha^\nu+\frac 12 [\bar\alpha^\nu\wedge \bar\alpha^\nu]$. We denote by $A^\nu_k+\Psi_k^\nu dt+\Phi_k^\nu ds$ the connection $\mathcal K_2^{\varepsilon_\nu}(\Xi_5^\nu)$.

\begin{figure}[ht]
\begin{center}
\begin{tikzpicture} 

\draw (-1,-1) .. controls (-1,0) and (0,0.7) .. (0.5,1); 
\draw (6.5,-1) .. controls (6.5,-0.5) and (7,0.1) .. (7.2,0.2); 
\draw (-1,-1).. controls (0,-1) and (1,-1) .. node[near start,above] {$F_A=0$}(2,-1); \draw (2,-1).. controls (3,-1) and (4,-1) .. (6.5,-1); 

\draw (1,2) .. controls (3,1.75) and (3,1.75) ..node[above] {$\tilde \alpha^\nu(t)$}(5,1.5); 

\filldraw (1,0) circle (1pt) node[below] {$A_2^\nu(t,s)$}
(5,0) circle (1pt) node[below] {$A_5^\nu(t,s)$}
(5,1.5) circle (1pt) node[above,right] {$A_k^\nu(t,s)$}
(1,2) circle (1pt) node[above, left] {$A^\nu(t,s)$}; 
\draw[red] (1,0) .. controls (2,0) and (4,0) .. node[below] {$\bar \alpha^\nu(t,s)$}(5,0); 
\draw[violet] (5,0) .. controls (5,1) and (5,1) .. node[right] {$\bar\alpha_0^\nu(t,s)$}(5,1.5); 
\draw[blue] (1,0) .. controls (1,1) and (1,1) .. node[right] {$\alpha^\nu_1(t,s)$}  (1,2); 

\end{tikzpicture} 
\caption{The splitting of the seventh step.}\label{picsur1f}
\end{center}
\end{figure}
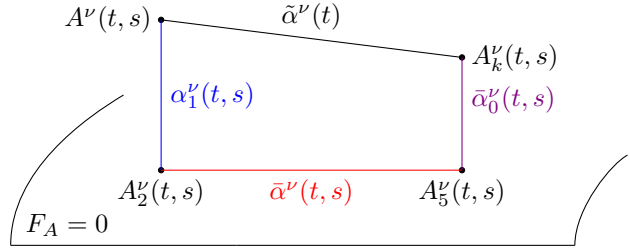

Let $\nabla_t:=\nabla_t^{\Psi_5^\nu}$ and $\nabla_s:=\nabla_s^{\Phi_5^\nu}$. By the lemmas \ref{flow:lemma:saweb2}, \ref{flow:lemma:diffD} and the estimates (\ref{flow:k2fhods2}) and (\ref{flow:k2fhods}) we have
\begin{equation}\label{flowste7o1}
\begin{split}
\Big\|\pi_{A_5^\nu}&\left(A^\nu-\mathcal K_2^{\varepsilon_\nu}(\Xi_5^\nu)\right)\Big\|_{L^p}
+\left\|\nabla_t \pi_{A_5^\nu}\left(A^\nu-\mathcal K_2^{\varepsilon_\nu}(\Xi_5^\nu)\right)\right\|_{L^p}\\
&+\left\|\nabla_s \pi_{A_5^\nu}\left(A^\nu-\mathcal K_2^{\varepsilon_\nu}(\Xi_5^\nu)\right)\right\|_{L^p}
+\left\|\nabla_t\nabla_t \pi_{A_5^\nu}\left(A^\nu-\mathcal K_2^{\varepsilon_\nu}(\Xi_5^\nu)\right)\right\|_{L^p}
\\
\leq& c\left\| \pi_{A_5^\nu}\mathcal D_1^{\varepsilon_\nu}(\mathcal K_2^{\varepsilon_\nu}(\Xi_5^\nu))(\Xi^\nu-\mathcal K_2^{\varepsilon_\nu}(\Xi_5^\nu))\right\|_{L^p}\\
&+c\left\|\left(1-\pi_{A_5^\nu}\right)\left(A^\nu-A_k^\nu\right)\right\|_{1,p,\varepsilon_\nu}+c\varepsilon^{3-\frac 2p}\\
&+c \left\|\nabla_t\left(1-\pi_{A_5^\nu}\right)\left(A^\nu-A_k^\nu\right)\right\|_{L^p}\\
&+c \varepsilon_\nu^2\left\| \mathcal D_2^{\varepsilon_\nu}(\mathcal K_2^{\varepsilon_\nu}(\Xi_5^\nu))(\Xi^\nu-\mathcal K_2^{\varepsilon_\nu}(\Xi_5^\nu))\right\|_{L^p}.
\end{split}
\end{equation}
and thus our task is to estimate all the norms on the right hand side of the inequality; the second one can be estimate by $c\varepsilon_\nu^{2-\frac2p}$ by the previous step and the lemma \ref{flow:lemma:firstappr}. The last term of (\ref{flowste7o1}) can be estimate by
\begin{equation}\label{flowste7o3}
\begin{split}
 \varepsilon_\nu^2\big\| \mathcal D_2^{\varepsilon_\nu}(\mathcal K_2^{\varepsilon_\nu}(\Xi_5^\nu))&(\Xi^\nu-\mathcal K_2^{\varepsilon_\nu}(\Xi_5^\nu))\big\|_{L^p}\\
\leq &\varepsilon_\nu^2\left\| \mathcal F_2^{\varepsilon_\nu}(\mathcal K_2^{\varepsilon_\nu}(\Xi_5^\nu))\right\|_{L^p}\\
&+\varepsilon_\nu^2\left\| \mathcal C_2^{\varepsilon_\nu}(\mathcal K_2^{\varepsilon_\nu}(\Xi_5^\nu))(\Xi^\nu-\mathcal K_2^{\varepsilon_\nu}(\Xi_5^\nu))\right\|_{L^p}
\leq  c\varepsilon^{2-\frac 5p}
\end{split}
\end{equation}
where the first inequality follows from 
$$\mathcal D_2^{\varepsilon_\nu}(\mathcal K_2^{\varepsilon_\nu}(\Xi_5^\nu))(\Xi^\nu-\mathcal K_2^{\varepsilon_\nu}(\Xi_5^\nu))=-\mathcal F_2^{\varepsilon_\nu}(\mathcal K_2^{\varepsilon_\nu}(\Xi_5^\nu))-\mathcal C_2^{\varepsilon_\nu}(\mathcal K_2^{\varepsilon_\nu}(\Xi_5^\nu))(\Xi^\nu-\mathcal K_2^{\varepsilon_\nu}(\Xi_5^\nu))$$
because of $\mathcal F_2^{\varepsilon_\nu}(\Xi^\nu)=0$ and the second estimating $\mathcal C_2^{\varepsilon_\nu}(\mathcal K_2^{\varepsilon_\nu}(\Xi_5^\nu))(\Xi^\nu-\mathcal K_2^{\varepsilon_\nu}(\Xi_5^\nu))$ term by term using the formula (\ref{qestnvdsopapa}). Next, we define $\tilde \alpha^\nu+\tilde \psi^\nu dt+\tilde \phi^\nu ds:=\Xi^\nu-\mathcal K_2^{\varepsilon_\nu}(\Xi_5^\nu)$ and 
\begin{equation}
A^\nu-A^\nu_k=(A^\nu-A_2^\nu)+(A_2^\nu-A_5^\nu)+(A_5^\nu-A^\nu_k)=\alpha^\nu_1+\bar\alpha^\nu-\bar\alpha_0^\nu
\end{equation}
where $\bar \alpha_0^\nu:=\mathcal K_2^{\varepsilon_\nu}(\Xi_5^\nu)-\Xi_5^\nu$. We remark that $0=F_{A_5^\nu+\bar\alpha^\nu }=d_{A_5^\nu}\bar\alpha^\nu+\frac 12\left[\bar \alpha_\nu\wedge \bar\alpha^\nu\right]$. Furthermore, since by $\mathcal F_1^{\varepsilon_\nu}(\Xi^\nu)=0$
$$\mathcal D_1^{\varepsilon_\nu}(\mathcal K_2^{\varepsilon_\nu}(\Xi_5^\nu))(\Xi^\nu-\mathcal K_2^{\varepsilon_\nu}(\Xi_5^\nu))=-\mathcal F_1^{\varepsilon_\nu}(\mathcal K_2^{\varepsilon_\nu}(\Xi_5^\nu))-\mathcal C_1^{\varepsilon_\nu}(\mathcal K_2^{\varepsilon_\nu}(\Xi_5^\nu))(\Xi^\nu-\mathcal K_2^{\varepsilon_\nu}(\Xi_5^\nu))$$
and by (\ref{flow:k2:leqw4}) $\left\|\pi_{A_5^\nu}\mathcal F_1^{\varepsilon_\nu}(\mathcal K_2^{\varepsilon_\nu}(\Xi_5^\nu))\right\|_{L^p}\leq c\varepsilon^2$,
\begin{equation}\label{flowste7o2}
\begin{split}
&\left\|\pi_{A_5^\nu}\mathcal D_1^{\varepsilon_\nu}(\mathcal K_2^{\varepsilon_\nu}(\Xi_5^\nu))(\tilde \alpha^\nu+\tilde\psi^\nu dt+\tilde \phi^\nu ds)\right\|_{L^p}\\
\leq& \left\|\pi_{A_5^\nu}\mathcal F_1^{\varepsilon_\nu}(\mathcal K_2^{\varepsilon_\nu}(\Xi_5^\nu))\right\|_{L^p}\\
&+\left\|\pi_{A_5^\nu}\mathcal C_1^{\varepsilon_\nu}(\mathcal K_2^{\varepsilon_\nu}(\Xi_5^\nu))(\tilde\alpha^\nu+\tilde\psi^\nu dt+\tilde \phi^\nu ds)\right\|_{L^p}\\
\leq & c \varepsilon_\nu^{2-\frac 7p }+\frac 1{\varepsilon_\nu^2}\left\|\pi_{A_5^\nu}\left(\left[\tilde\alpha^\nu\wedge *\left(d_{A_5^\nu}\tilde\alpha^\nu+\frac 12 \left[\tilde\alpha^\nu\wedge\tilde\alpha^\nu\right] \right)\right]\right)\right\|_{L^p}\\
\leq & c \varepsilon_\nu^{2-\frac 7p }+\frac 1{\varepsilon_\nu^2}\left\|\pi_{A_5^\nu}\left(\left[\tilde\alpha^\nu\wedge *\left(d_{A_5^\nu}\left(\tilde\alpha^\nu-\bar\alpha^\nu\right)\right)\right]\right)\right\|_{L^p}+c\varepsilon_\nu^{1-\frac 7p }\left\|\pi_{A^\nu_5}(\tilde\alpha^\nu)\right\|_{L^p}
\end{split}
\end{equation}
where the second step follows estimating (\ref{xvasvdb}) term by term. Next, we consider the following operator
\begin{equation}
\begin{split}
Q^{\varepsilon_\nu}\left(\Xi_5^\nu\right)(\Xi^\nu-\mathcal K_2^{\varepsilon_\nu}(\Xi_5^\nu)):=& \mathcal D^{\varepsilon_\nu}\left(\Xi_5^\nu\right)(\Xi^\nu-\mathcal K_2^{\varepsilon_\nu}(\Xi_5^\nu))+\frac 1{2\varepsilon_\nu^2}d_{A_5^\nu}^*\left[\bar\alpha^\nu\wedge \bar\alpha^\nu\right] 
\end{split}
\end{equation}
whose first component can be written as
\begin{equation}\label{flow:surj:step7:eqqq}
\begin{split}
Q^{\varepsilon_\nu}_1&\left(\Xi_5^\nu\right)(\tilde \alpha^\nu+\tilde\psi^\nu dt+\tilde \phi^\nu ds)
= \nabla_s\left(\tilde\alpha^\nu-\Pi_{\textrm{im } d_{A_5^\nu}^*}(\bar \alpha^\nu)\right)-d_{A_5^\nu}\tilde\phi^\nu\\
&+\frac 1{\varepsilon_\nu^2}d_{A_5^\nu}^*d_{A_5^\nu}\left(\tilde\alpha^\nu-\bar\alpha^\nu\right)
-d*X_t(A_5^\nu)\left(\tilde\alpha^\nu-\Pi_{\textrm{im } d_{A_5^\nu}^*}(\bar \alpha^\nu)\right)\\
&-\nabla_t\nabla_t\left(\tilde\alpha^\nu-\Pi_{\textrm{im } d_{A_5^\nu}^*}(\bar \alpha^\nu)\right)
-2\left[\tilde\psi^\nu,\left(\partial_t A_5^\nu-d_{A_5^\nu}\Psi_5^\nu\right)\right]\\
&+\nabla_s\left(\Pi_{\textrm{im } d_{A_5^\nu}^*}(\bar \alpha^\nu)\right)-d*X_t(A_5^\nu)\Pi_{\textrm{im } d_{A_5^\nu}^*}(\bar \alpha^\nu)-\nabla_t\nabla_t\left(\Pi_{\textrm{im } d_{A_5^\nu}^*}(\bar \alpha^\nu)\right).
\end{split}
\end{equation}
By theorem \ref{flow:thm:3eqbasic} we obtain that
\begin{align*}
&\left\| d_{A_5^\nu}^*d_{A_5^\nu}\left(\tilde\alpha^\nu-\bar\alpha^\nu\right)\right\|_{L^p}
+\varepsilon_\nu^2\left\| \nabla_t\nabla_t(1-\pi_{A_5^\nu})\left(\tilde\alpha^\nu-\Pi_{\textrm{im } d_{A_5^\nu}^*}(\bar \alpha^\nu)\right)\right\|_{L^p}\\
&+\varepsilon_\nu\left\| \nabla_t(1-\pi_{A_5^\nu})\left(\tilde\alpha^\nu-\Pi_{\textrm{im } d_{A_5^\nu}^*}(\bar \alpha^\nu)\right)\right\|_{L^p}\\
\leq& c\varepsilon_\nu^2\left\|  \mathcal D_1^{\varepsilon_\nu}(\Xi_5^\nu)\left(\tilde\alpha^\nu-\Pi_{\textrm{im } d_{A_5^\nu}^*}(\bar \alpha^\nu),\tilde\psi^\nu,\tilde\phi^\nu\right)\right\|_{L^p}
+\varepsilon^2\left\|\nabla_s\pi_{A_5^\nu}\left(\tilde\alpha^\nu-\bar\alpha^\nu\right)\right\|_{L^p}\\
&+\varepsilon^2\left\|\nabla_t\nabla_t\pi_{A_5^\nu}\left(\tilde\alpha^\nu-\bar\alpha^\nu\right)\right\|_{L^p}+c\varepsilon_\nu^{3-\frac 1p}\\
\intertext{and by (\ref{flow:surj:step7:eqqq}), step 2, the lemma \ref{flow:lemmasurj1} and the theorem \ref{flow:thm:existence}}
\leq& c\varepsilon_\nu^2\left\|  Q_1^{\varepsilon_\nu}(\Xi_5^\nu)(\tilde\alpha^\nu,\tilde\psi^\nu,\tilde\phi^\nu)\right\|_{L^p}+c\varepsilon_\nu^{3-\frac 2p}\\
&+\varepsilon_\nu^2\left\|\nabla_t\nabla_t\left(\Pi_{\textrm{im } d_{A_5^\nu}^*}(\bar \alpha^\nu)\right)\right\|_{L^p}+\varepsilon_\nu^2\left\|\nabla_s\left(\Pi_{\textrm{im } d_{A_5^\nu}^*}(\bar \alpha^\nu)\right)\right\|_{L^p}\\
\intertext{since $\mathcal F_1^{\varepsilon_\nu}\left(\mathcal K_2^{\varepsilon_\nu}\left(\Xi_5^\nu\right)\right)+ Q_1^{\varepsilon_\nu}\left(K_2^{\varepsilon_\nu}\left(\Xi_5^\nu\right)\right)(\tilde\alpha^\nu,\tilde\psi^\nu,\tilde\phi^\nu)+ \mathcal C_1^{\varepsilon_\nu}\left(K_2^{\varepsilon_\nu}\left(\Xi_5^\nu\right)\right)(\tilde\alpha^\nu,\tilde\psi^\nu,\tilde\phi^\nu)-\frac {1}{2\varepsilon_\nu^2}d_{A_5^\nu}^*\left[\bar\alpha^\nu\wedge\bar\alpha^\nu\right]=0$,}
\leq& c\varepsilon_\nu^2\left\|  \mathcal F_1^{\varepsilon_\nu}(\mathcal K_2^{\varepsilon_\nu}(\Xi_5^\nu))\right\|_{L^p}+c\varepsilon_\nu^{3-\frac 2p}\\
&+c\varepsilon_\nu^2\left\|  \mathcal C_1^{\varepsilon_\nu}(\mathcal K_2^{\varepsilon_\nu}(\Xi_5^\nu))(\tilde\alpha^\nu,\tilde\psi^\nu,\tilde\phi^\nu)-\frac 1{2\varepsilon_\nu^2}d_{A_5^\nu}^*\left[\bar\alpha^\nu\wedge\bar\alpha^\nu\right]\right\|_{L^p}\\
&+\varepsilon_\nu^2\left\|\nabla_t\nabla_t\left(\Pi_{\textrm{im } d_{A_5^\nu}^*}(\bar \alpha^\nu)\right)\right\|_{L^p}+\varepsilon_\nu^2\left\|\nabla_s\left(\Pi_{\textrm{im } d_{A_5^\nu}^*}(\bar \alpha^\nu)\right)\right\|_{L^p}\\
\leq& c\varepsilon_\nu^{3-\frac 7p}+\varepsilon_\nu^2\left\|\nabla_t\nabla_t\left(\Pi_{\textrm{im } d_{A_5^\nu}^*}(\bar \alpha^\nu)\right)\right\|_{L^p}+\varepsilon_\nu^2\left\|\nabla_s\left(\Pi_{\textrm{im } d_{A_5^\nu}^*}(\bar \alpha^\nu)\right)\right\|_{L^p}\\.
\end{align*}
where the last step follows estimating (\ref{xvasvdb}) term by term. Hence, by the last estimate, (\ref{flowste7o1}), (\ref{flowste7o3}) (\ref{flowste7o2}) and the next claim:
\begin{equation*}
\begin{split}
\left\|\pi_{A_5^\nu}\left(A^\nu-\mathcal K_2^{\varepsilon_\nu}(\Xi_5^\nu)\right)\right\|_{L^p}
&+\left\|\nabla_t \pi_{A_5^\nu}\left(A^\nu-\mathcal K_2^{\varepsilon_\nu}(\Xi_5^\nu)\right)\right\|_{L^p}\\
&+\left\|\nabla_t\nabla_t \pi_{A_5^\nu}\left(A^\nu-\mathcal K_2^{\varepsilon_\nu}(\Xi_5^\nu)\right)\right\|_{L^p}\\
&+\left\|\nabla_s \pi_{A_5^\nu}\left(A^\nu-\mathcal K_2^{\varepsilon_\nu}(\Xi_5^\nu)\right)\right\|_{L^p}
\\
\leq& c\varepsilon_\nu^{2-\frac {10}p }+\varepsilon_\nu^{1-\frac 1p}\left\|\nabla_t\nabla_t\left(\Pi_{\textrm{im } d_{A_5^\nu}^*}(\bar \alpha^\nu)\right)\right\|_{L^p}\\
&+\varepsilon_\nu^{\frac 1p}\left\|\nabla_s\left(\Pi_{\textrm{im } d_{A_5^\nu}^*}(\bar \alpha^\nu)\right)\right\|_{L^p}
\leq c\varepsilon_\nu^{2-\frac {10}p }.
\end{split}
\end{equation*}
Therefore, for $p>10$ and by the Sobolev's theorem \ref{flow:thm:sob} for $\varepsilon=1$, there is a $\delta_1>0$, such that
\begin{equation*}
\left\|\pi_{A_5^\nu}\left(A^\nu_1-A^\nu_5\right)\right\|_{L^p}+\left\|\pi_{A_5^\nu}\left(A^\nu_1-A^\nu_5\right)\right\|_{L^\infty }\leq c \varepsilon_\nu^{1+\delta_1}
\end{equation*}
holds for $\varepsilon_\nu$ small enough. Thus we concluded the proof of the seventh step.
\end{proof}

\noindent{\bf Claim. }There are two positive constants $c$ and $\varepsilon_0$ such that, for $0<\varepsilon<\varepsilon_0$,
\begin{align*}
\left\|\nabla_t\nabla_t\left(\Pi_{\textrm{im } d_{A_5^\nu}^*}(\bar \alpha^\nu)\right)\right\|_{L^p}&+\left\|\nabla_s\left(\Pi_{\textrm{im } d_{A_5^\nu}^*}(\bar \alpha^\nu)\right)\right\|_{L^p}\\
\leq& c\varepsilon^{1-\frac 2p}\left(1+\|\nabla_t\nabla_t\pi_{A_5^\nu}(\bar\alpha^\nu)\|_{L^p}+\|\nabla_s\pi_{A_5^\nu}(\bar\alpha^\nu)\|_{L^p}\right)
\end{align*}

\begin{proof}[Proof of the claim]
We write $\left(\Pi_{\textrm{im } d_{A_5^\nu}^*}(\bar \alpha^\nu)\right)=d_{A_5^\nu}^*\omega^\nu$ for $2$-form $\omega^\nu$ and hence
\begin{align*}
\|\nabla_t\nabla_td_{A_5^\nu}^*\omega^\nu\|_{L^p}
\leq & \|d_{A_5^\nu}\nabla_t\nabla_td_{A_5^\nu}^*\omega^\nu\|_{L^p}
+\left\|\left(1-\Pi_{\textrm{im } d_{A_5^\nu}^*}\right)\nabla_t\nabla_td_{A_5^\nu}^*\omega^\nu\right\|_{L^p}\\
\intertext{using the commutation formulas, the $L^\infty$-bound on the curvature terms and the lemma \ref{lemma76dt94}, we obtain}
\leq & \|\nabla_t\nabla_td_{A_5^\nu}\bar\alpha^\nu\|_{L^p}
+c\|d_{A_5^\nu}^*\omega^\nu\|_{L^p}+c\|\nabla_t d_{A_5^\nu}^*\omega^\nu\|_{L^p}\\
\intertext{and by the identity $d_{A_5^\nu}\bar \alpha^\nu+\frac 12[\bar\alpha^\nu\wedge \bar \alpha^\nu]$}
\leq & \frac 12\|\nabla_t\nabla_t[\bar\alpha^\nu\wedge \bar \alpha^\nu]\|_{L^p}
+c\|\bar\alpha^\nu\|_{L^p}+c\|\nabla_t \bar\alpha^\nu\|_{L^p}\\
\leq & c\| \bar \alpha^\nu\|_{L^\infty}
\|\nabla_t\nabla_t\bar\alpha^\nu\|_{L^p}
+ c\| \nabla_t\bar \alpha^\nu\|_{L^{2p}}
\|\nabla_t\bar\alpha^\nu\|_{L^{2p}}\\
&+c\|\bar\alpha^\nu\|_{L^p}+c\|\nabla_t \bar\alpha^\nu\|_{L^p}\\
\leq & c\varepsilon^{1-\frac 2p}\left(1+\left\|\nabla_t\nabla_t\left(\Pi_{\textrm{im } d_{A_5^\nu}^*}(\bar \alpha^\nu)\right)\right\|_{L^p}\right)\\
&+ c\varepsilon^{1-\frac 2p}\left\|\nabla_t\nabla_t\left(\pi_{A_5^\nu}(\bar \alpha^\nu)\right)\right\|_{L^p}.
\end{align*}
In the same way, we can show that
\begin{align*}
\left\|\nabla_s\left(\Pi_{\textrm{im } d_{A_5^\nu}^*}(\bar \alpha^\nu)\right)\right\|_{L^p}\leq&  c\varepsilon^{1-\frac 2p}\left(1+\left\|\nabla_s\left(\Pi_{\textrm{im } d_{A_5^\nu}^*}(\bar \alpha^\nu)\right)\right\|_{L^p}\right)\\
& +c\varepsilon^{1-\frac 2p}\left\|\nabla_s\left(\pi_{A_5^\nu}(\bar \alpha^\nu)\right)\right\|_{L^p}
\end{align*}
and thus the claim holds for $\varepsilon$ sufficiently small.
\end{proof}
\noindent With the last claim we concluded also the proof of the theorem \ref{flow:thm:surj}.
\end{proof}

\section{Proofs of the main theorems}\label{c:flow:mt}

The definition \ref{defiR} of the map $\mathcal R^{\varepsilon,b}$ and the theorem \ref{flow:thm:surj}, 
which assures its surjectivity, allow us to conclude that the theorem \ref{flow:thm:main1} holds. Thus, we need only to explain the proof of theorem \ref{thm:main}.

\begin{proof}[Proof of theorem \ref{thm:main}]

If we fix a regular value $b$ of $E^H$, then by the theorem \ref{thm:mainthm} there 
is a positive constant $\varepsilon_0$ such that the map $\mathcal T^{\varepsilon,b}$ 
is a bijection for $0<\varepsilon<\varepsilon_0$. In addition, 
since $\mathcal G_0(P)$ acts freely, ${\mathcal T}^{\varepsilon, b}$ descends to the map 
$$\bar {\mathcal T}^{\varepsilon, b}:
\mathrm{Crit}_{E^H}^b/\mathcal G_0(P\times S^1)\to 
\mathrm{Crit}_{\mathcal {YM}^{\varepsilon,H}}^b/\mathcal G_0(P\times S^1)$$
which maps 
perturbed closed geodesics to orbits of perturbed Yang-Mills connections with the 
same Morse index and therefore we can see it as chain complex homeomorphism
$$\bar{\mathcal T}^{\varepsilon,b}: C_*^{E^H,b}\to C_*^{\mathcal {YM}^{\varepsilon,H},b}.$$
For any two perturbed geodesics $\gamma_\pm\in [\mathrm{Crit}_{E^H}^b]$ with index 
difference $1$, the map $\mathcal R^{\varepsilon,b}$ is, for two lifts 
$\Xi_\pm\in \mathrm{Crit}_{E^H}^b$ with $[\Xi_\pm]=\gamma_\pm$, by theorem 
\ref{flow:thm:main1}, bijective and thus
$$\sharp_{\mathbb Z_2}\bar{\mathcal M}^0\left(\Xi_-,\Xi_+\right)/\mathbb R=\sharp_{\mathbb Z_2}\bar{\mathcal M}^\varepsilon\left(\mathcal T^{\varepsilon,b}(\Xi_-),\mathcal T^{\varepsilon,b}(\Xi_+)\right)/\mathbb R$$
which yields that the following diagram commutes 
 $$\begin{CD}
 \dots @>>> C_{k+1}^{E^H,b} @>{\partial_k^{E^H}}>>C_{k}^{E^H,b} @>{\partial_{k-1}^{E^H}}>>C_{k-1}^{E^H,b} @>>> \dots\\
@.  @VV{\bar{\mathcal T}^{\varepsilon,b}}V  @VV{\bar{\mathcal T}^{\varepsilon,b}}V @VV{\bar{\mathcal T}^{\varepsilon,b}}V\\
  \dots @>>> C_{k+1}^{\mathcal {YM}^{\varepsilon,H},b} @>{\partial_k^{\mathcal {YM}^{\varepsilon,H}}}>>C_{k}^{\mathcal {YM}^{\varepsilon,H},b} @>{\partial_{k-1}^{\mathcal {YM}^{\varepsilon,H}}}>>C_{k-1}^{\mathcal {YM}^{\varepsilon,H},b} @>>> \dots\\
  @.  @VV{\left(\bar{\mathcal T}^{\varepsilon,b}\right)^{-1}}V  @VV{\left(\bar{\mathcal T}^{\varepsilon,b}\right)^{-1}}V @VV{\left(\bar{\mathcal T}^{\varepsilon,b}\right)^{-1}}V\\
  \dots @>>> C_{k+1}^{E^H,b} @>{\partial_k^{E^H}}>>C_{k}^{E^H,b} @>{\partial_{k-1}^{E^H}}>>C_{k-1}^{E^H,b} @>>> \dots
 \end{CD}$$
 and hence
 $$\left(\bar{\mathcal T}^{\varepsilon,b}\right)_*: HM_*\left(\mathcal L^b\mathcal M^g(P), \mathbb Z_2\right)\to HM_*\left(\mathcal A^{\varepsilon,b}\left(P\times S^1\right)/\mathcal G_0\left( P\times S^1\right), \mathbb Z_2\right)$$
 is an isomorphism.
\end{proof}

\begin{appendix}

\section{Norms for $1$-forms on $\Sigma\times S^1$}\label{appendix:norm}

We fix a connection $\Xi_{0}=A_{0}+\Psi_{0} dt\,\in\mathcal A(\Sigma\times S^1)$ and we define the following norms on $\Omega^i(\Sigma\times S^1,\mathfrak g_P)$, $i=1,2$. Let $\xi(t)=\alpha(t)+\psi (t)\wedge dt$ such that $\alpha(t)\in \Omega^1(\Sigma,\mathfrak g_{P})$ and $\psi(t)\in \Omega^0(\Sigma,\mathfrak g_{P})$ or $\alpha(t)\in \Omega^2(\Sigma,\mathfrak g_{P})$ and $\psi(t)\in \Omega^1(\Sigma,\mathfrak g_{P})$, then
\begin{equation*}
\|\xi\|_{0,p,\varepsilon,\Sigma\times S^1}^p:=\int_{0}^1\left(\|\alpha\|_{L^p(\Sigma)}^p
+\varepsilon^{p}\|\psi\|_{L^p(\Sigma)}^p\right)dt,
\end{equation*}
\begin{equation*}
\|\xi\|_{\infty,\varepsilon,\Sigma\times S^1}:=\|\alpha\|_{L^\infty(\Sigma\times S^1)}
+\varepsilon\|\psi\|_{L^\infty(\Sigma\times S^1)}
\end{equation*}
and
\begin{equation*}
\begin{split}
\|\xi&\|_{\Xi_0,1,p,\varepsilon,\Sigma\times S^1}^p:=\int_{0}^1\left(\|\alpha\|_{L^p(\Sigma)}^p
+\|d_{A_{0}}\alpha\|_{L^p(\Sigma)}^p+
\|d_{A_{0}}^*\alpha\|_{L^p(\Sigma)}^p\right)dt\\
&+\int_{0}^1\varepsilon^p\left(\|\nabla_{t}\alpha\|_{L^p(\Sigma)}^p+\|\psi\|_{L^p(\Sigma)}^p+\|d_{A_{0}}\psi\|_{L^p(\Sigma)}^p
+\varepsilon^{p}\|\nabla_{t}\psi\|_{L^p(\Sigma)}^p\right) dt.
\end{split}
\end{equation*}
Inductively,
\begin{equation*}
\begin{split}
\|\xi\|_{\Xi_0,k+1,p,\varepsilon,\Sigma\times S^1}^p:=&\|\alpha+\psi \,dt\|_{\Xi_0,k,p,\varepsilon,\Sigma\times S^1}^p+\|d_{A_{0}}\alpha\|_{\Xi_0,k,p,\varepsilon,\Sigma\times S^1}^p\\
&+\|d_{A_{0}}^*\alpha\|_{\Xi_0,k,p,\varepsilon,\Sigma\times S^1}^p+\varepsilon^p\|\nabla_{t}\alpha\|_{\Xi_0,k,p,\varepsilon,\Sigma\times S^1}^p\\
&+\|d_{A_{0}}\psi \wedge dt\|_{\Xi_0,k,p,\varepsilon,\Sigma\times S^1}^p
+\varepsilon^{p}\|\nabla_{t}\psi\, dt\|_{\Xi_0,k,2,\varepsilon,\Sigma\times S^1}^p.
\end{split}\end{equation*}
For $i=1,2$, we can define by $W^{k,p}(\Sigma\times S^1,\Lambda^iT^*(\Sigma\times S^1)\otimes\mathfrak g_{P\times S^1 })$ the Sobolev space of the $i$-forms respect to the norm $\|\cdot\|_{\Xi_0,k,p,1,\Sigma\times S^1}$. We now choose a reference connection $\Xi_0$ and all the Sobolev inequalities 
hold as follows by the Sobolev embedding theorem (cf. \cite{remyj6}).
\begin{theorem}[Sobolev estimates]\label{lemma:sobolev}
We choose $1\leq p,q<\infty$ and $l\leq k$. Then there is a constant $c_s$ such that for every 
$\xi\in  W^{k,p}(\Sigma\times S^1,\Lambda^iT^*(\Sigma\times S^1)\otimes\mathfrak g_{P\times S^1})$, $i=1,2$, and any reference connection $\Xi_0$:
\begin{enumerate}
\item If $l-\frac 3q\leq k-\frac3p$, then
\begin{equation}\label{eq:sobolev1}
\|\xi\|_{\Xi_0,l,q,\varepsilon,\Sigma\times S^1}
\leq c_s\varepsilon^{1/q-1/p}\,\|\xi\|_{\Xi_0,k,p,\varepsilon,\Sigma\times S^1}.
\end{equation}
\item If $0< k-\frac3p$, then
\begin{equation}\label{eq:sobolev2}
\|\xi\|_{\Xi_0,\infty,\varepsilon,\Sigma\times S^1}
\leq c_s\varepsilon^{-1/p}\,\|\xi\|_{\Xi_0,k,p,\varepsilon,\Sigma\times S^1}.
\end{equation}
\end{enumerate}
\end{theorem}
%

\section{Estimates on the surface}

The first two lemmas were proved by Dostoglou and Salamon (cf. \cite{MR1283871}, lemma 7.6 and lemma 8.2) for $p> 2$ and $q=\infty$; the proofs in the case $p=2$ and $2\leq q <\infty$ are similar.

\begin{lemma}\label{lemma76dt94}
We choose $p> 2$ and $q=\infty$ or $p=2$ and $2\leq q <\infty$. Then there exist two positive constants $\delta$ and $c$ such that for every connection 
$A\in\mathcal A(P)$ with
$$\|F_A\|_{L^p(\Sigma) }\leq \delta$$
there are estimates 
$$\|\psi\|_{L^q(\Sigma)}\leq c\|d_A\psi\|_{L^p(\Sigma)},\qquad 
\|d_A\psi\|_{L^q(\Sigma)}\leq c\|d_A*d_A\psi\|_{L^p(\Sigma)},$$
for $\psi\in \Omega^0(\Sigma,\mathfrak g_P)$.
\end{lemma}

\begin{lemma}\label{lemma82dt94}
We choose $p> 2$ and $q=\infty$ or $p=2$ and $2\leq q <\infty$. Then there exist two positive constants $\delta$ and $c$ such that the following holds. 
For every connection $A\in \mathcal A(P)$ with 
$$\|F_A\|_{L^p(\Sigma)}\leq \delta$$
there exists a unique section $\eta \in \Omega^0(\Sigma,\mathfrak g_P)$ such that
$$F_{A+*d_A\eta}=0,\qquad \|d_A\eta\|_{L^q(\Sigma)}\leq c\|F_A\|_{L^p(\Sigma)}.$$
\end{lemma}

The following lemma is a simplified version of the lemma B.2. in \cite{MR1736219} where Salamon allows also to modify the complex structure on $\Sigma$ if it is $C^1$-closed to a fixed one.
\begin{lemma}\label{flow:lemma:lpl2}
Fix a connection $A^0\in \mathcal A_0(P)$. Then, for every $\delta>0$, $C>0$, and $p\geq2$, there exists a constant $c=c(\delta, C,A^0)\geq 1$ such that, if $A\in\mathcal A(P)$ satisfy $\|A-A^0\|_{L^\infty(\Sigma)}\leq C$ then, for every $\psi\in \Omega^0(\Sigma,\mathfrak g_P)$ and every $\alpha\in \Omega^1(\Sigma,\mathfrak g_P)$,
\begin{equation}
\|\psi\|_{L^p(\Sigma)}^p\leq\delta \|d_A\psi\|_{L^p(\Sigma)}^p+c\|\psi\|_{L^2(\Sigma)}^p,
\end{equation}
\begin{equation}
\|\alpha\|_{L^p(\Sigma)}^p\leq \delta\left(\|d_A\alpha\|_{L^p(\Sigma)}^p+\|d_A*\alpha\|_{L^p(\Sigma)}^p\right)+c\|\alpha\|_{L^2(\Sigma)}^p.
\end{equation}

\end{lemma}

\begin{lemma}\label{flow:lemma40} We choose $p\geq2$. There is a positive constant $c$ such that the following holds. For any connection $A\in \mathcal A_0(P)$ and any $\alpha\in \Omega^1(\Sigma,\mathfrak g_P)$
\begin{equation}
\begin{split}
\|\alpha\|_{L^p(\Sigma)}+\|d_A\alpha\|_{L^p(\Sigma)}&+\|d_A^*\alpha\|_{L^p(\Sigma)}+\|d_A^*d_A\alpha\|_{L^p(\Sigma)}+\|d_A^*d_A^*\alpha\|_{L^p(\Sigma)}\\
\leq& c\|(d_Ad_A^*+d_A^*d_A)\alpha\|_{L^p}+\|\pi_A(\alpha)\|_{L^p(\Sigma) }.
\end{split}
\end{equation}
\end{lemma}

\begin{proof}
For any flat connection $A$, the orthogonal splitting of $\Omega^1(\Sigma)=\textrm{im } d_A\oplus \textrm{im } d_A^*\oplus H_A^1(\Sigma,\mathfrak g_P)$ implies that there is a positive constant $c_0$ such that
\begin{equation*}
\|d_Ad_A^*\alpha\|_{L^p(\Sigma)}+\|d_A^*d_A\alpha\|_{L^p(\Sigma)}\leq c_0 \|(d_Ad_A^*+d_A^*d_A)\alpha\|_{L^p(\Sigma)};
\end{equation*}
thus, we can conclude the proof applying the lemma \ref{lemma76dt94}.
\end{proof}

\begin{lemma}\label{lemma:dAalphaest}
We choose $p\geq2$. There is a positive constant $c$ such that the following holds. For any $\delta>0$, any connection $A\in\mathcal A_0(P)$, $\alpha\in\Omega^1(\Sigma,\mathfrak g_P)$ and $\psi \in\Omega^0(\Sigma,\mathfrak g_P)$
\begin{equation}
\begin{split}
\left\| d_A\alpha\right\|_{L^p(\Sigma) }\leq c\left(\delta^{-1}\, \left\| \alpha  \right\|_{L^p(\Sigma)}+ \delta\, \left\| d_A^*d_A\alpha  \right\|_{L^p(\Sigma)}\right),\\
\left\| d_A^*\alpha\right\|_{L^p(\Sigma)}\leq c\left(\delta^{-1}\, \left\| \alpha  \right\|_{L^p(\Sigma)}+ \delta\, \left\| d_Ad_A^*\alpha  \right\|_{L^p(\Sigma)}\right),\\
\left\| d_A\psi\right\|_{L^p(\Sigma) }\leq c\left(\delta^{-1}\, \left\| \psi \right\|_{L^p(\Sigma)}+ \delta\, \left\| d_A^*d_A\psi  \right\|_{L^p(\Sigma)}\right).
\end{split}
\end{equation}
Furthermore, for any $\delta>0$, any connection $A+\Psi dt\in\mathcal A(P\times S^1)$, $\alpha+\psi dt\in\Omega^1(\Sigma\times S^1,\mathfrak g_P)$ 
\begin{equation}
\begin{split}
\varepsilon\left\| \nabla_t\alpha\right\|_{L^p(\Sigma\times S^1) }\leq c\left(\delta^{-1}\, \left\| \alpha  \right\|_{L^p(\Sigma\times S^1)}+ \delta\varepsilon^2 \left\|\nabla_t\nabla_t\alpha  \right\|_{L^p(\Sigma\times S^1)}\right),\\
\varepsilon^2\left\| \nabla_t\psi\right\|_{L^p(\Sigma\times S^1) }\leq c\left(\delta^{-1}\varepsilon\, \left\| \psi \right\|_{L^p(\Sigma\times S^1)}+ \delta\varepsilon^3 \left\| \nabla_t\nabla_t\psi  \right\|_{L^p(\Sigma\times S^1)}\right).
\end{split}
\end{equation}
\end{lemma}

\begin{proof}
The last two estimates follow analogously to the lemma D.4.  in \cite{MR2276534}. The first can be proved as follows. We choose $q$ such that $\frac 1p+\frac 1q=1$ then
\begin{equation*}
\begin{split}
\left\| d_A\alpha\right\|_{L^p(\Sigma)}=& \sup_{\bar \alpha} 
\frac {\langle d_A\alpha, \delta^{-1}\bar \alpha +\delta d_Ad_A^*\bar \alpha\rangle}{\left\| \delta^{-1}\bar\alpha+\delta d_Ad_A^*\bar\alpha\right\|_{L^q(\Sigma)}}\\
\leq & \sup_{\bar \alpha} 
\frac {c \langle\delta^{-1}\alpha+\delta d_A^*d_A\alpha, d_A^*\bar \alpha\rangle}{\delta^{-1}\left\|\bar\alpha\right\|_{L^q(\Sigma)}+\delta \left\|d_Ad_A^*\bar\alpha\right\|_{L^q(\Sigma)}+\left\| d_A^*\bar\alpha\right\|_{L^q}(\Sigma)}\\
\leq&\left(\delta^{-1}\left\|\alpha\right\|_{L^p(\Sigma)}+\delta \left\|d_A^*d_A\alpha\right\|_{L^p(\Sigma)}\right)
 \sup_{\bar \alpha}
 \frac {c \left\| d_A^*\bar \alpha\right\|_{L^q(\Sigma)}}{\left\| d_A^*\bar\alpha\right\|_{L^q(\Sigma)}}\\
 =&c\left(\delta^{-1}\left\|\alpha\right\|_{L^p(\Sigma)}+\delta \left\|d_A^*d_A\alpha\right\|_{L^p(\Sigma)}\right).
\end{split}
\end{equation*}
where the supremum is taken over all non-vanishing $1$-forms $\bar \alpha\in L^q$ with $d_Ad_A^*\bar \alpha\in L^q$. The norm $\left\| \delta^{-1}\bar\alpha+\delta d_Ad_A^*\bar\alpha\right\|_{L^q(\Sigma)}$ is never $0$ because
$$\left\| \delta^{-1}\bar\alpha+\delta d_Ad_A^*\bar\alpha\right\|_{L^2(\Sigma)}^2=\delta^{-2}\left\|\bar\alpha\right\|_{L^2(\Sigma)}^2+\delta^2 \left\|d_Ad_A^*\bar\alpha\right\|_{L^2(\Sigma)}^2+2\left\| d_A^*\bar\alpha\right\|_{L^2(\Sigma)}^2\neq 0,$$
otherwise we would have a contradiction by the H\"older inequality and the operator $\delta^{-1}+\delta d_Ad_A^*$ is surjective. The second and the third estimate of the lemma can be shown exactly in the same way.
\end{proof}

\end{appendix}

\bibliographystyle{amsplain}
\bibliography{bib}

\providecommand{\bysame}{\leavevmode\hbox to3em{\hrulefill}\thinspace}
\providecommand{\MR}{\relax\ifhmode\unskip\space\fi MR }
\providecommand{\MRhref}[2]{%
  \href{http://www.ams.org/mathscinet-getitem?mr=#1}{#2}
}
\providecommand{\href}[2]{#2}
\begin{thebibliography}{10}

\bibitem{MR2276948}
Alberto Abbondandolo and Pietro Majer, \emph{Lectures on the {M}orse complex
  for infinite-dimensional manifolds}, Morse theoretic methods in nonlinear
  analysis and in symplectic topology, NATO Sci. Ser. II Math. Phys. Chem.,
  vol. 217, Springer, Dordrecht, 2006, pp.~1--74. \MR{MR2276948 (2007m:58012)}

\bibitem{MR2276949}
Alberto Abbondandolo and Matthias Schwarz, \emph{Note on {F}loer homology and
  loop space homology}, Morse theoretic methods in nonlinear analysis and in
  symplectic topology, NATO Sci. Ser. II Math. Phys. Chem., vol. 217, Springer,
  Dordrecht, 2006, pp.~75--108. \MR{MR2276949 (2007k:53149)}

\bibitem{MR2190223}
\bysame, \emph{On the {F}loer homology of cotangent bundles}, Comm. Pure Appl.
  Math. \textbf{59} (2006), no.~2, 254--316. \MR{MR2190223 (2006m:53137)}

\bibitem{MR702806}
M.~F. Atiyah and R.~Bott, \emph{The {Y}ang-{M}ills equations over {R}iemann
  surfaces}, Philos. Trans. Roy. Soc. London Ser. A \textbf{308} (1983),
  no.~1505, 523--615. \MR{MR702806 (85k:14006)}

\bibitem{Davisthesis}
Thomas~James Davies, \emph{The {Y}ang-{M}ills functional over {R}iemann
  surfaces and the loop group}, PhD thesis, University of Warwick (1996).

\bibitem{MR1297130}
Stamatis Dostoglou and Dietmar Salamon, \emph{Instanton homology and symplectic
  fixed points}, Symplectic geometry, London Math. Soc. Lecture Note Ser., vol.
  192, Cambridge Univ. Press, Cambridge, 1993, pp.~57--93. \MR{MR1297130
  (96a:58065)}

\bibitem{MR1283871}
Stamatis Dostoglou and Dietmar~A. Salamon, \emph{Self-dual instantons and
  holomorphic curves}, Ann. of Math. (2) \textbf{139} (1994), no.~3, 581--640.
  \MR{MR1283871 (95g:58050)}

\bibitem{MR1715156}
Ying-Ji Hong, \emph{Harmonic maps into the moduli spaces of flat connections},
  Ann. Global Anal. Geom. \textbf{17} (1999), no.~5, 441--473. \MR{MR1715156
  (2001d:58013)}

\bibitem{remyj6}
R\'emi Janner, \emph{Perturbed geodesics on the moduli space of flat
  connections and {Y}ang-{M}ills theory}, Mathematische Zeitschrift
  \textbf{273} (2013), 653--710.

\bibitem{elliptic}
R\'emi Janner and Jan Swoboda, \emph{Elliptic yang-mills flow theory},
  preprint, arXiv:1303.1401v1 (2013).

\bibitem{MR1179335}
Johan R{\aa}de, \emph{On the {Y}ang-{M}ills heat equation in two and three
  dimensions}, J. Reine Angew. Math. \textbf{431} (1992), 123--163.
  \MR{MR1179335 (94a:58041)}

\bibitem{Rade:fv}
Johan Rade, \emph{Compactness theorems for invariant connections}, Preprint
  (2000).

\bibitem{MR2276534}
D.~A. Salamon and J.~Weber, \emph{Floer homology and the heat flow}, Geom.
  Funct. Anal. \textbf{16} (2006), no.~5, 1050--1138. \MR{MR2276534
  (2007k:53154)}

\bibitem{MR1736219}
Dietmar~A. Salamon, \emph{Quantum products for mapping tori and the
  {A}tiyah-{F}loer conjecture}, Northern {C}alifornia {S}ymplectic {G}eometry
  {S}eminar, Amer. Math. Soc. Transl. Ser. 2, vol. 196, Amer. Math. Soc.,
  Providence, RI, 1999, pp.~199--235. \MR{MR1736219 (2000m:53122)}

\bibitem{MR648356}
Karen~K. Uhlenbeck, \emph{Connections with {$L\sp{p}$} bounds on curvature},
  Comm. Math. Phys. \textbf{83} (1982), no.~1, 31--42. \MR{MR648356
  (83e:53035)}

\bibitem{vit}
Claude Viterbo, \emph{Functors and computations in floer cohomology. part 2},
  Preprint (1996), revised 2003.

\bibitem{Weberhab}
Joa Weber, \emph{The heat flow and the homology of the loop space},
  Habilitationsschrift, Humboldt-Universit\"at zu Berlin (2010).

\bibitem{MR2030823}
Katrin Wehrheim, \emph{Uhlenbeck compactness}, EMS Series of Lectures in
  Mathematics, European Mathematical Society (EMS), Z\"urich, 2004.
  \MR{MR2030823 (2004m:53045)}

\bibitem{MR943798}
D.~R. Wilkins, \emph{The {P}alais-{S}male conditions for the {Y}ang-{M}ills
  functional}, Proc. Roy. Soc. Edinburgh Sect. A \textbf{108} (1988), no.~3-4,
  189--200. \MR{MR943798 (89e:58035)}

\end{thebibliography}

\end{document}